\documentclass[11pt,francais]{smfart}

\usepackage[T1]{fontenc}
\usepackage[english,francais]{babel}
\usepackage{amssymb,url,xspace,smfthm}

\usepackage[latin1]{inputenc}
\usepackage{color} 
\usepackage{booktabs}
\usepackage[all]{xy}
\RequirePackage{mathrsfs}
\usepackage[colorlinks=true]{hyperref} 

\DeclareMathAlphabet{\mathpzc}{OT1}{pzc}{m}{it} 

\makeatletter
    \def\ps@copyright{\ps@empty
    \def\@oddfoot{\hfil\small\copyright 2012, \SMF}}
\makeatother

\newcommand{\SMF}{Société Ma\-thé\-ma\-ti\-que de France}
\newcommand{\BibTeX}{{\scshape Bib}\kern-.08em\TeX}
\newcommand{\T}{\S\kern .15em\relax }
\newcommand{\AMS}{$\mathcal{A}$\kern-.1667em\lower.5ex\hbox
        {$\mathcal{M}$}\kern-.125em$\mathcal{S}$}

\NumberTheoremsIn{subsubsection}
\NumberTheoremsAs{paragraph}

\newtheorem*{notationsstar}{Notations}
\newtheorem{exo}[paragraph]{Exercice} 
\newtheorem{nota}[paragraph]{Notation}

\AtBeginDocument{%
\let\oldtheo\theo
\renewcommand{\theo}{%
\oldtheo\hypertarget{\csname@currentHref\endcsname}{}}
\let\oldprop\prop
\renewcommand{\prop}{%
\oldprop\hypertarget{\csname@currentHref\endcsname}{}}
\let\oldlemm\lemm
\renewcommand{\lemm}{%
\oldlemm\hypertarget{\csname@currentHref\endcsname}{}}
\let\oldexem\exem
\renewcommand{\exem}{%
\oldexem\hypertarget{\csname@currentHref\endcsname}{}}
\let\olddefi\defi
\renewcommand{\defi}{%
\olddefi\hypertarget{\csname@currentHref\endcsname}{}}
\let\oldexo\exo
\renewcommand{\exo}{%
\oldexo\hypertarget{\csname@currentHref\endcsname}{}}
\let\oldnota\nota
\renewcommand{\nota}{%
\oldnota\hypertarget{\csname@currentHref\endcsname}{}}
\let\oldrema\rema
\renewcommand{\rema}{%
\oldrema\hypertarget{\csname@currentHref\endcsname}{}}
}

\def\SGAtrois{\cite{sga3}} 

\newcommand{\tooby}[1]{\buildrel #1\over\longrightarrow}
\newcommand{\oto}[1]{\overset{#1}\to}
\newcommand{\isoto}{\buildrel \raisebox{-0.25ex}{\ensuremath{\sim}}\over\to}
\newcommand{\isoby}[1]{\buildrel #1\over\simeq}

\newcommand{\doublecell}[2][l]{\begin{tabular}[l]{@{}l@{}}#2\end{tabular}} 

\def\cf{{cf.}\kern.3em}
\def\loccit{\textit{loc.\kern2pt cit.}\xspace}
\def\ie{{i.e.}\kern.3em}

\newcommand{\ZZ}{\mathbb{Z}} 
\newcommand{\QQ}{\mathbb{Q}} 
\newcommand{\NN}{\mathbb{N}} 
\newcommand{\FF}{\mathbb{F}} 

\let\emptyset\varnothing

\DeclareMathOperator{\id}{id}
\DeclareMathOperator{\Hom}{Hom}
\DeclareMathOperator{\End}{End}

\DeclareMathOperator{\coker}{coker}

\newcommand{\CatFais}{{\mathscr F\kern-0.15em aisc}} 

\newcommand{\cC}{\mathcal{C}} 
\newcommand{\cD}{\mathcal{D}} 

\newcommand{\cCat}[1]{\mathrm{Cat}(#1)} 
\newcommand{\cProd}[1]{\mathrm{\Pi}(#1)} 

\newcommand{\grpd}{\simeq} 
\newcommand{\cG}{\mathcal{G}} 

\def\setGL{\mathrm{GL}} 
\def\setSL{\mathrm{SL}} 
\def\setorthO{\mathrm{O}} 
\def\setSO{\mathrm{SO}} 
\def\setEnd{\mathrm{End}} 
\def\setHom{\mathrm{Hom}} 
\def\setAut{\mathrm{A\kern-0.06em ut}} 
\def\setM{\mathrm{M}} 

\newcommand{\fppf}{{\textup{fppf}}}
\newcommand{\fpqc}{{\textup{fpqc}}}

\def\faisDual{I} 

\def\faisAut{\mathbf{A\kern-0.06em ut}} 
\def\faisAutExt{\mathbf{A\kern-0.06em utExt}} 
\def\faisIso{\mathbf{Iso}} 
\def\faisHom{\mathbf{Hom}} 
\def\faisUn{\mathbf{1}} 
\def\faisEnd{\mathbf{End}} 
\def\faisM{\mathbf{M}} 

\def\faisCentre{\mathbf{Centre}} 
\def\faisStab{\mathbf{Stab}} 
\def\faisNorm{\mathbf{Norm}}
\def\faisAut{\mathbf{A\kern-0.06em ut}} 
\def\faisAutExt{\mathbf{A\kern-0.06em utExt}} 

\def\faisGm{\mathbf{G}_m} 
\newcommand{\faisGmS}[1]{{\mathbf{G}_m}_{\hspace{-.1ex},#1}} 
\def\faismu{\pmb{\mu}} 
\def\faisGa{\mathbf{G}_a} 
\def\faisconst{\mathbf{\ZZ/2}} 
\def\faisO{\mathbf{O}} 

\def\faisGL{\mathbf{GL}} 
\def\faisGLh{\mathbf{GL}^h} 
\def\faisDiag{\mathbf{Diag}} 
\def\faisSL{\mathbf{SL}} 
\def\faisSDiag{\mathbf{PDiag}} 
\def\faisTriSup{\mathbf{TriSup}} 
\def\faisPGL{\mathbf{PGL}} 
\def\faisPDiag{\mathbf{PDiag}} 

\def\faisU{\mathbf{U}} 
\def\faisSU{\mathbf{SU}} 

\def\faisorthO{\mathbf{O}} 
\def\faisSO{\mathbf{SO}} 
\def\faisSOdiag{\mathbf{DiagSO}} 
\def\faisOplus{\mathbf{O}^+} 

\def\faisPGO{\mathbf{PGO}} 
\def\faisPGOplus{\mathbf{PGO}^+} 
\def\faisPGOplusdiag{\mathbf{DiagPGO}^+} 
\def\faisGO{\mathbf{GO}} 
\def\faisGOdiag{\mathbf{DiagGO}} 

\def\faisGamma{\mathbf{\Gamma}}
\def\faisSGamma{\mathbf{S\Gamma}}
\def\faisSGammadiag{\mathbf{DiagS\Gamma}} 
\def\faisPin{\mathbf{Pin}}
\def\faisSpin{\mathbf{Spin}}
\def\faisSpindiag{\mathbf{DiagSpin}} 

\def\faisSp{\mathbf{Sp}}
\def\faisSpdiag{\mathbf{DiagSp}} 
\def\faisPGSp{\mathbf{PGSp}} 
\def\faisPGSpdiag{\mathbf{DiagPGSp}}

\def\Xstruc{\mathfrak{X}} 
\def\Ystruc{\mathfrak{Y}} 

\def\struc{\mathfrak{struc}} 
\def\gr{\mathfrak{gr}} 
\def\alg{\mathfrak{alg}} 
\def\alginv{\mathfrak{alginv}} 
\newcommand\actg[1]{#1\text{-}\mathfrak{act}} 
\newcommand\actd[1]{\mathfrak{act}\text{-}#1} 

\newcommand{\Final}{\mathpzc{Final}} 
\newcommand{\Oper}[2]{\mathpzc{Oper}(#1,#2)} 
\newcommand{\Tors}[1]{\mathpzc{Tors}(#1)} 
\newcommand{\Formes}[1]{\mathpzc{Formes}(#1)} 

\newcommand{\Etn}[1]{\mathpzc{Et}_{#1}} 
\newcommand{\Vect}{\mathpzc{Vec}} 
\newcommand{\Vecn}[1]{\mathpzc{Vec}_{#1}} 
\newcommand{\Sympl}[1]{\mathpzc{Alt}_{#1}} 
\newcommand{\AVec}[1]{\text{$#1$-}\mathpzc{Vec}} 
\newcommand{\AVecn}[2]{\text{$#1$-}\mathpzc{Vec}_{#2}} 
\newcommand{\Vectrivdetn}[1]{\mathpzc{VecTrivDet}_{#1}} 
\newcommand{\nTriv}[1]{\text{#1-}\mathpzc{Triv}} 
\newcommand{\Azumaya}{\mathpzc{Azumaya}} 
\newcommand{\TrAz}{\mathpzc{TrAzumaya}} 
\newcommand{\NrAzn}[1]{\mathpzc{NrAzumaya_{#1}}} 
\newcommand{\Azumayan}[1]{\mathpzc{Azumaya}_{#1}} 
\newcommand{\nABan}[2]{\text{$#1$-$#2$-}\mathpzc{Ban}} 
\newcommand{\GrAzumayan}[1]{\mathpzc{GrAzumaya}_{#1}} 
\newcommand{\GrVecn}[1]{\mathpzc{GrVec}_{#1}} 
\newcommand{\ZeroUn}{\mathpzc{ZeroUn}} 
\newcommand{\GrZeroVecn}[1]{\mathpzc{GrVec}^0_{#1}} 
\newcommand{\ZGrAzumaya}{\mathpzc{ZGrAzumaya}} 
\newcommand{\PairesQuadn}[1]{\mathpzc{PairesQuad}_{#1}} 
\newcommand{\PairesQuadArfTrivn}[1]{\mathpzc{PairesQuad}_{#1}} 
\newcommand{\Quad}{\mathpzc{Quad}} 
\newcommand{\RegQuadn}[1]{\mathpzc{RegQuad}_{#1}} 
\newcommand{\RegQuadArfTrivn}[1]{\mathpzc{RegQuadArfTriv}_{#1}} 
\newcommand{\ModDet}{\mathpzc{ModDet}} 
\newcommand{\QuadDetTrivn}[1]{\mathpzc{QuadDetTriv}_{#1}} 
\newcommand{\DeplCliff}[1]{\mathpzc{DeplCliff}_{\hspace{-.2ex}#1}} 
\newcommand{\GrDeplCliff}[1]{\mathpzc{GrDeplCliff}_{\hspace{-.2ex}#1}} 
\newcommand{\DeplCliffSn}[1]{\mathpzc{DeplCliffSn}_{\hspace{-.2ex}#1}} 

\newcommand{\CliffFonc}{\mathrm{C}} 
\newcommand{\DetFonc}{\Delta} 
\newcommand{\qDetFonc}{\mathrm{D}} 
\newcommand{\qDemiDetFonc}{\tfrac{1}{2}D} 
\newcommand{\EndFonc}{\mathrm{End}} 
\newcommand{\DegLocFonc}{\mathrm{DegLoc}} 
\newcommand{\SnFonc}{\mathrm{Sn}} 
\newcommand{\ArfFonc}{\mathrm{Arf}} 
\newcommand{\XiFonc}{\Xi} 

\def\dual{\vee} 
\newcommand{\bid}{\varpi} 
\newcommand{\totExt}{\Lambda^{\hspace{-0.5ex}*}} 
\newcommand{\totExtPair}{\Lambda^{\hspace{-0.5ex}p}} 
\newcommand{\totExtImpair}{\Lambda^{\hspace{-0.5ex}i}} 

\def\opp{\mathrm{op}} 
\def\Symalg{\mathbb{S}} 
\def\trace{\mathrm{tr}} 
\def\trd{\mathrm{trd}} 
\def\nrd{\mathrm{nrd}} 
\def\sand{\mathrm{sand}} 

\def\rad{\mathrm{rad}} 
\def\qrad{\mathrm{qrad}} 
\def\hyps{b^h} 
\def\hypa{a^h} 
\def\hypq{q^h} 
\def\orth{\perp} 
\newcommand{\fq}[1]{\langle #1\rangle} 
\newcommand{\Det}[1]{\Delta_{#1}} 
\newcommand{\bDetMorph}[1]{\mathrm{d}_{#1}} 
\newcommand{\qDemiDetMorph}[1]{\tfrac{1}{2}\mathrm{d}_{#1}} 
\newcommand{\bDetMod}[1]{D_{#1}} 
\newcommand{\qDetMod}[1]{D_{#1}} 
\newcommand{\qDemiDetMod}[1]{\tfrac{1}{2}D_{#1}} 
\newcommand{\sn}{\mathrm{sn}} 

\def\setCliff{\mathrm{C}} 
\def\faisCliff{\mathbf{C}} 
\def\signe{\varepsilon} 
\def\faisCBim{\mathbf{C}} 

\def\transp{{^t}} 
\renewcommand{\int}{\mathrm{int}} 

\def\exch{\tau} 
\def\faisSym{\mathbf{Sym}} 
\def\faisAlt{\mathbf{Alt}} 
\def\invstd{\sigma} 
\def\invadj{\eta} 

\def\alL{\mathbf{\frak{L}}} 
\def\faisTan{\mathbf{\bf{T}}} 
\def\faisLie{\mathbf{\bf{Lie}}} 
\def\faisDer{\mathbf{\bf{D\acute{e}r}}} 
\DeclareMathOperator{\mLie}{\mathscr{L}\hspace{-2pt}\mathit{ie}} 
\DeclareMathOperator{\Der}{D\acute{e}r} 

\def\Roots{\Sigma} 
\def\cdual{\vee} 

\newcommand{\adjoint}{\mathrm{adj}} 
\newcommand{\semisim}{\mathrm{ss}} 
\newcommand{\simco}{\mathrm{sc}} 
\newcommand{\derive}{\mathrm{d\acute{e}r}} 

\def\Hcoh{\mathrm{H}} 
\def\Het{\Hcoh_{\text{ét}}} 
\def\Hfppf{\Hcoh_{\text{fppf}}} 
\def\Htau{\Hcoh_\Ttau} 

\def\deuxbord{\partial} 

\newcommand{\Agerban}[1]{\text{$#1$-}\mathrm{GerBan}} 
\newcommand{\germuban}[2]{\text{$#1$-}\mu_{#2}\mathrm{GerBan}} 
\def\Brauer{\mathrm{Br}} 
\newcommand{\muBrauer}[1]{\mu\mathrm{Br}_{#1}} 

\def\Ttau{\mathbb{T}} 

\newcommand{\contr}[1]{\wedge^{#1}} 

\DeclareMathOperator{\Spec}{Spec}
\DeclareMathOperator{\Pic}{Pic}

\newcommand{\cO}{\mathcal{O}} 
\newcommand{\cF}{\mathcal{F}} 
\newcommand{\cM}{\mathcal{M}} 
\newcommand{\cA}{\mathcal{A}} 

\newcommand{\wW}{\mathbf{W}} 
\newcommand{\vV}{\mathbf{V}} 

\title{Groupes classiques}
\date{\today}
\author{Baptiste Calmès et Jean Fasel}


\begin{document}

\begin{abstract}
Nous introduisons les groupes algébriques linéaires dits ``classiques'' sur une base quelconque, puis nous les replaçons dans la classification des groupes réductifs établie dans \SGAtrois. Nous traitons les cas non déployés, et décrivons au passage plusieurs catégories de torseurs.
\end{abstract}

\begin{altabstract}
We introduce so-called ``classical'' linear algebraic groups over a general base, and we then place them where they belong in the classification of reductive groups established in \SGAtrois. We cover the non split cases, and we describe on the way several categories of torsors.
\end{altabstract}

\maketitle

\setcounter{tocdepth}{3}
\tableofcontents

\vskip-2cm

\section{Introduction}

L'objet principal de \SGAtrois\ est la classification des schémas en groupes réductifs sur une base quelconque. Rappelons tout d'abord qu'un groupe algébrique $G$ sur une base $S$ est réductif (au sens de \SGAtrois, voir Exp. XIX, déf. 2.7) s'il est lisse et affine sur $S$, et à fibres géométriques réductives et connexes au sens de la théorie sur les corps; sur les fibres géométriques, son radical est donc un tore (voir \cite[Exp. XIX, déf. 1.6.1]{sga3}). Il est de plus semi-simple si ces fibres sont semi-simples (radical trivial). Notons qu'être réductif ou semi-simple est en fait par descente une propriété locale pour toutes les topologies considérées dans cet article.

Le but de cette note est de faire le lien entre les groupes linéaires et quadratiques dits ``classiques'' et la classification de \SGAtrois. Cet exposé s'adressant en priorité aux lecteurs qui ne sont pas habitués à la terminologie employée dans l'ouvrage, de nombreux détails de calculs, exemples et traductions sont fournis. 

\begin{notationsstar}
Dans tout ce qui suit, $R$ désigne un anneau associatif commutatif et unitaire, $S$ un schéma de base quelconque, implicitement égal à $\Spec(R)$ lorsqu'il est affine.
\end{notationsstar}

Soit $X$ un schéma sur $S$, et soit $T\to S$ un morphisme de schémas. La notation $X_{T}$ désigne le produit fibré de $X$ sur $T$. Lorsque $T$ est le spectre d'un anneau $R'$, on s'autorisera également la notation $X_{R'}$ pour $X_{T}$. Un point $s\in S$ correspond à un morphisme de son corps résiduel $\kappa(s)$ vers $S$, et $X_s=X_{\kappa(s)}$ est alors appelé \emph{fibre} en $s$. Si $\bar s$ est le spectre d'une clôture algébrique de $\kappa(s)$, le schéma $X_{\bar s}$ est appelé \emph{fibre géométrique}.
\medskip

La méthode suivie dans ce texte comporte plusieurs étapes. Tout d'abord, il nous faut introduire et définir les groupes classiques que l'on veut replacer dans la classification. Il s'agit des groupes semi-simples adjoints ou simplement connexes, des séries $A_n$, $B_n$, $C_n$ et $D_n$, bien que nous en manipulions bien d'autres au passage. Conformément à la philosophie de \SGAtrois, les schémas (en groupes) sont vus comme des foncteurs de points représentables. Nos groupes classiques sont donc introduits comme des foncteurs de points, puis on montre qu'ils sont représentables par des schémas affines sur la base, ce qui ne pose pas de réelle difficulté pour deux raisons: beaucoup de nos groupes sont construits comme des produits fibrés de groupes obtenus précédemment, et cette opération préserve la représentabilité. Par ailleurs, la représentabilité par des schémas affines sur la base peut se tester localement pour la topologie \fppf\ ou étale. Ces deux topologies de Grothendieck sont d'ailleurs quasiment les seules que nous utiliserons, en dehors de la topologie de Zariski, bien entendu. 

Les groupes classiques se construisent à l'aide de structures algébriques comme des modules quadratiques, des algèbres d'Azumaya, ou des algèbres à involution. Nous discutons donc ces structures sur une base quelconque dans les sections \ref{AlgSepEtAz_sec}, \ref{modulesquadratiques_sec} et \ref{AlgInv_sec}. Les groupes sont alors introduits dans les parties \ref{schemasFond_sec} et \ref{An_sec} pour les groupes de type linéaire (liés au type $A_n$), ainsi que dans la partie \ref{Groupesquadratiques_sec} pour les groupes en rapport avec les modules quadratiques ou plus généralement les algèbres à involution. 

Une fois la plupart des groupes introduits, il faut identifier les groupes déployés, leur version sur $\ZZ$, dits groupes de Chevalley, et prouver qu'ils sont bien réductifs. Le principal problème est la lissité, qui doit essentiellement être obtenue à la main en vérifiant d'abord la lissité de chaque fibre géométrique, ce qui revient, parce qu'on est sur un corps, à comparer la dimension de Krull et celle de l'algèbre de Lie, puis on montre que la dimension de ces fibres ne varie pas. Les critères précis sont rappelés en section \ref{Lisse_sec}.

L'étape d'après, est de trouver la donnée radicielle de chacun des groupes. Cela consiste à décrire un tore maximal déployé, expliciter l'algèbre de Lie et la représentation adjointe, y trouver les racines, et constater qu'elles forment bien les données radicielles des types considérés. 

Ces deux dernières étapes forment les parties intitulées ``Groupe déployé adjoint'' et ``Groupe déployé simplement connexe'' de chacune des sections \ref{An_sec} (type $A_n$), \ref{Bn_sec} (type $B_n$), \ref{Cn_sec} (type $C_n$) et \ref{Dn_sec} (type $D_n$). 

Il reste alors le problème de la torsion. Si \SGAtrois\ nous dit que tous les groupes réductifs sont des formes étales des groupes déployés, il nous faut montrer pour être complet que les différents groupes classiques que nous avons introduits, en plus du cas déployé, constituent toutes les formes étales de ces groupes, qu'il n'y en a pas d'autres. Pour cela, nous avons recours au formalisme du produit contracté, décrit en section \ref{torseurs_sec}. Ce formalisme nous dit que toute forme est en fait un produit contracté du groupe par un torseur sous son groupe d'automorphismes. Il nous faut donc identifier les groupes d'automorphismes, ce qui est fait dans les sections du même nom dans les quatre parties \ref{An_sec}, \ref{Bn_sec}, \ref{Cn_sec} et \ref{Dn_sec}. 

Puis, il faut être capable de donner une description assez concrète des torseurs pour expliciter les produits contractés. C'est la raison pour laquelle nous utilisons des champs, car c'est un cadre très souple dans lequel on peut formaliser les produits contractés, et construire des catégories concrètes d'objets algébriques (par exemple les modules quadratiques), équivalentes à des catégories de torseurs. Le texte est donc parsemé de théorèmes qui comparent un champ bien concret (comme celui des modules quadratiques) avec une catégorie de torseurs (par exemple sous le groupe orthogonal). Comme nous savons que le mot ``champ'' provoque immanquablement chez certains un rhume des foins, nous nous empressons d'ajouter que nous n'utilisons aucune des subtilités qui font le bonheur des champistes. Le peu de choses que nous utilisons effectivement est rappelé dans les parties \ref{Structures_sec} et \ref{torseurs_sec}, et pourrait se résumer en quelque mots: un champ est une sorte de catégorie dans laquelle les objets et les morphismes peuvent se construire localement. Nous faisons un usage répétitif des produits fibrés de champs, pour en fabriquer de nouveaux sans peine. Les premiers champs que nous introduisons sont construits comme des structures sur les faisceaux, ce qui explique la présence de la partie \ref{Structures_sec}.
Nous avons également inclus à titre d'exemple quelques calculs de gerbes, qui correspondent aux morphisme de connexion vers le terme $H^2$ des suites exactes de cohomologie.

Faute de temps, nous avons laissé de côté un certain nombre de choses, qui pourraient probablement être traitées par les mêmes méthodes. Citons: la description systématique des catégories de torseurs sous tous les groupes mentionnés, y compris ceux qui ne sont pas déployés; la description des variétés projectives homogènes sous les groupes considérés; la description de types simples intermédiaires entre le cas simplement connexe et le cas adjoint, lorsqu'ils existent (par exemple en type $A_n$); la description de types mixtes; la description des types exceptionnels déjà connus sur les corps (ex: $G_2$, $F_4$); la description du type $D_4$; la description des isomorphismes exceptionnels de bas rang.

Les principales difficultés que nous avons rencontrées sont de deux natures: premièrement, bien souvent, dans la littérature, les constructions sont faites sur les corps, en distinguant la caractéristique $2$ des autres. Or, pour travailler sur une base quelconque, ces distinctions sont interdites, les constructions doivent être indépendantes d'une éventuelle caractéristique. Deuxièmement, sur un corps, tous les fibrés vectoriels sont triviaux, et certains invariants ne se voient donc pas. Nous avons donc dû en rajouter certains. Voir par exemple l'invariant $l_q$ qui intervient dans les formes strictement intérieures des types simplement connexes $B_n$ et $D_n$. 

Outre SGA3, les trois grandes références que nous avons utilisé sans vergogne sont le livre de Giraud sur les champs \cite{gir}, celui de Knus sur les formes quadratiques \cite{knus}, et le livre des involutions de Knus, Merkurjev, Rost et Tignol \cite{bookinv}, ce dernier principalement pour sa notion de paire quadratique, qui fonctionne en toute caractéristique, et tous les groupes qui en découlent. 

Vu le grand nombre de groupes et de champs considérés, pour que la lectrice ne soit pas perdue, nous avons ajouté des tables à la fin du texte qui listent les groupes les plus importants, ainsi que les équivalences entre certains champs et certaines catégories de torseurs, avec références aux endroits du texte concernés. 
\medskip

Nous remercions vivement Philippe Gille pour diverses explications sur des points techniques de SGA3, Cyril Demarche pour son aide sur les champs et les gerbes, Asher Auel pour ses précisions sur les algèbres de Clifford, et enfin Skip Garibaldi, Sylvain Brochard et les rapporteurs anonymes pour leurs suggestions d'amélioration à partir d'une version préliminaire de ce texte.
 
\section{Préliminaires}

\subsection{Structures} \label{Structures_sec}

\subsubsection{Structures élémentaires} \label{structureselem_sec}

Expliquons brièvement la notion de structure dans une catégorie $\cC$ munie de produits finis, puis la notion d'objet $\Xstruc$ en cette structure. Nous voulons que ce formalisme puisse contenir la notion d'objet en groupes dans une catégorie donnée, ou bien la notion de $\faisO_S$-modules dans la catgorie des $S$-foncteurs de points; voir les exemples \ref{structures_exem} ci-après. Dans tout les cas, il faut pouvoir spécifier des morphismes, comme la muliplication $G \times G \to G$ d'un objet en groupes, et des contraintes vérifiées par ces morphismes, comme l'associativité. Enfin, nous avons besoin d'objets et de morphismes que nous appellerons constants, comme par exemple l'objet neutre qui nous servira d'unité pour les groupes, ou bien l'objet en anneaux $\faisO_S$ et sa multiplication, dont on se sert pour définir chaque $\faisO_S$-modules. 

En termes précis, ces structures se formalisent à l'aide de deux constructions catégoriques classiques: \'Etant donné un graphe orienté $\gamma$, notons $\cCat{\gamma}$ la catégorie engendrée par ce graphe (la catégorie des chemins de ce graphe). Cette construction est fonctorielle et définit un adjoint à gauche du foncteur oubli de la catégorie des petites catégories vers celle des graphes (voir \cite[Ch. II, \S 7]{McL}). Notons également $\cProd{\cC}$ la complétion d'une petite catégorie $\cC$ par les produits finis. C'est également une construction fonctorielle, et qui définit un adjoint à gauche du foncteur d'oubli de la catégorie des petites catégories munies de produits finis (avec pour morphismes les foncteurs les respectant) vers la catégorie des petites catégories. Une autre manière de le dire est que pour toute petite catégorie $\cC$, il y a un foncteur canonique $\cC \to \cProd{\cC}$ et on a la propriété universelle suivante: tout foncteur de $\cC$ dans une catégorie $\cD$ munie de produits finis se factorise de manière unique par un foncteur $\cProd{\cC} \to \cD$ qui conserve les produits. La catégorie $\cProd{\cC}$ se construit aisément de manière combinatoire à partir de $\cC$.\footnote{Nous n'avons pas de bonne référence pour cette construction, mais on peut procéder comme suit: on considère la catégorie dont les objets sont des ensembles finis au-dessus de l'ensemble des objets de $\cC$, et avec pour morphismes de $\phi:A \to \mathrm{Ob}(\cC)$ vers $\psi: B \to \mathrm{Ob}(\cC)$ les applications $\theta:A\to B$ munies d'un étiquetage $a \mapsto f_a$ de chaque élément de $A$ par un morphisme $f_a: \phi(a) \to \psi(\theta(a))$ de $\cC$. On vérifie alors que la catégorie opposée de cette catégorie répond au problème.} 
\begin{defi}
Une \emph{structure dans $\cC$} est la donnée de: 
\begin{enumerate}
\item un graphe orienté fini $\gamma$ et un morphisme de graphes orientés $c:\gamma \to \cC$ (qui spécifie les constantes); 
\item un ensemble fini $I$ et un graphe $F$ contenant le graphe $\cProd{\cCat{I \sqcup \gamma}}$, ayant les mêmes sommets que celui-ci, et seulement un nombre fini d'arêtes supplémentaires (qui correspondent aux morphismes) 
\item une relation d'équivalence $R$ sur les morphismes de $\cProd{\cCat{F}}$ engendrée par un nombre fini de relations $f \sim g$ entre morphismes de $\cProd{\cCat{F}}$ (ce qui définit les contraintes sur les morphismes). 
\end{enumerate} 
\end{defi}
On obtient alors un morphisme de graphes $g$ donné par la composée 
$$\gamma \to I \sqcup \gamma \to \cCat{I \sqcup \gamma} \to \cProd{\cCat{I \sqcup \gamma}} \subseteq F \to \cProd{\cCat{F}}$$
(où les unités des deux adjonctions ont été utilisées) et on peut considérer la catégorie quotiente $\cProd{\cCat{F}}/R$, au sens de \cite[Ch. II, \S 8]{McL}.
\begin{defi}
Si $\struc$ est une structure dans $\cC$ au sens de la définition précédente, un \emph{objet en $\struc$} est un foncteur $\Xstruc:\cProd{\cCat{F}}/R \to \cC$ tel que $\Xstruc \circ g = c$. La catégorie $\cC^\struc$ des objets en $\struc$ est celle des foncteurs $\cProd{\cCat{F}}/R \to \cC$.
\end{defi}
Remarquons que la donnée d'un tel foncteur $\Xstruc$ est équivalente à la donnée d'un foncteur $\cProd{\cCat{F}}\to \cC$, compatible à la relation d'équivalence $R$, et que $c:\gamma \to \cC$ étant déjà fixé, par les propriétés universelles, il suffit de donner les objets images des $i\in I$ (qu'on peut noter $X_i$) puis les images des arêtes supplémentaires de $F$, en s'assurant qu'elles respectent les contraintes données par la relation d'équivalence.

\begin{defi} \label{strucpouss_defi}
Si $F:\cC \to \cD$ est foncteur respectant les produits finis, on peut pousser une structure $\struc$ dans $\cC$ vers une structure $F(\struc)$ dans $\cD$ de la manière évidente, et on obtient également un foncteur $F:\cC^{\struc} \to \cD^{F(\struc)}$. 
\end{defi}

\begin{exem} \label{structures_exem}
Tout d'abord, pour chaque ensemble $I$, la structure triviale sur un ensemble d'objets indicés par $I$ est celle pour laquelle les constantes, morphismes et relations sont vides. Mentionnons encore les structures suivantes (dans les $S$-foncteurs de points, ou dans les $S$-faisceaux pour une certaine topologie) que nous rencontrerons:
\begin{enumerate}
\item groupe $\gr$: le seul objet constant est $S$ (donc le graphe $\gamma$ a un sommet et aucune arête), les morphismes sont l'unité $S \to G$, la loi de groupe $G \times G \to G$ et l'inverse $G \to G$ (l'ensemble $I$ a donc un élément, et le graphe F a $3$ morphismes supplémentaires), et les relations imposent les contraintes bien connues, comme l'associativité. 
\item groupe abélien: une structure de groupe, avec la commutativité ajoutée aux relations;
\item \label{anneaustruc_item} anneau: une loi d'addition, de multiplication et les diagrammes évidents. En particulier, dans la catégorie des $S$-foncteurs, l'anneau que nous noterons $\faisO_S$. Il est construit comme tiré à $S$ de $\faisO=\Spec(\ZZ[x])$, et vérifie $\faisO_S(T)=\Gamma(T,\cO_{T})$, qu'on abrège en $\Gamma(T)$, pour tout schéma $T$ au-dessus de $S$. En particulier, $\faisO(\Spec(R))=R$ pour tout anneau $R$ (\cf \cite[Exp. I, 4.3.3]{sga3}). Les morphismes structuraux sont ceux qui induisent la structure d'anneau dessus; 
\item \label{OSmod_item} $\faisO_S$-module: les constantes sont l'anneau $\faisO_S$ et ses morphismes de structure, puis on se donne une structure de groupe abélien et d'action de $\faisO_S$ avec les relations habituelles;
\item $\faisO_S$-algèbre $\alg$: la structure de $\faisO_S$-module assorti d'une multiplication muni d'une unité avec les relations habituelles (y compris l'associativité);
\item $\faisO_S$-algèbre à involution $\alginv$: la structure de $\faisO_S$-algèbre avec un endomorphisme $\sigma$ du $\faisO_S$-module sous-jacent, qui anti-commute à la multiplication et qui vérifie $\sigma^2=\mathrm{Id}$ (et qui est donc un automorphisme de module);
\item \label{bilineaire_item} $\faisO_S$-forme bilinéaire (resp. symétrique, alternée), c'est-à-dire une structure de $\faisO_S$-module, avec un morphisme $M \times M \to \faisO_S$ qui est bilinéaire (resp. et symétrique ou alterné); 
\item \label{quadratique_item} $\faisO_S$-module quadratique: une structure de $\faisO_S$ module, à laquelle on rajoute un morphisme $q: M \to \faisO_S$ qui est quadratique en la multiplication par les scalaires de $\faisO_S$ et tel que le module polaire associé satisfait aux relations de bilinéarité de l'exemple précédent; 
\item \label{groupeaction_item} $G$-action à gauche $\actg{G}$ (ou à droite $\actd{G}$): étant donné un objet en groupes $G$, utilisé comme objet constant, un $G$-objet est un objet $X$ muni d'un morphisme $G\times X\to X$ satisfaisant les diagrammes d'associativité et de l'action triviale du neutre. Cette action peut être à gauche ou à droite, auquel cas le morphisme considéré est plutôt $X \times G \to X$.
\end{enumerate}
\end{exem}

\begin{rema}
Dans les exemples précédents, il est parfois utile d'ajouter des morphismes auxquels on ne pense pas immédiatement pour avoir les bonnes relations. Par exemple dans l'exemple des modules quadratiques, il faut construire la forme polaire. Pour cela, il est pratique d'utiliser le morphisme ``multiplication par $-1$'', endomorphisme de l'objet constant $\faisO_S$ et qu'on ajoutera donc aux constantes de la structure. 
\end{rema}

\subsubsection{Changement de base des structures} 

Afin de formaliser les changements de base, la première notion à introduire est celle de catégorie fibrée sur une catégorie de base, qui sera pour nous toujours celle des $S$-schémas, aussi nous nous limitons à ce cas.

\begin{defi} \label{catfib_defi}
Une \emph{catégorie fibrée scindée} sur les $S$-schémas, abrégé en \emph{$S$-catégorie fibrée} dans ce texte, est la donnée de:
\begin{itemize}
\item pour tout schéma $T$ sur $S$, une catégorie $\cC_{T}$ appelée fibre sur $T$;
\item pour tout $S$-morphisme $f:T' \to T$, un foncteur de changement de base $\phi_{f}:\cC_{T} \to \cC_{T'}$; 
\item pour tout $T$ sur $S$, un isomorphisme de foncteurs $\epsilon_T:\phi_{\id_T} \isoto \id_{\cC_T}$;
\item pour toute paire de morphismes composables $f$ et $g$, un isomorphisme de foncteurs $\iota_{f,g}:\phi_g \circ \phi_f \isoto \phi_{f \circ g}$.
\end{itemize}
Ces isomorphismes doivent vérifier une condition de compatibilité entre $\epsilon$ et $\iota$ lorsque l'un des deux morphismes est l'identité, et une relation d'associativité, toutes deux aisées à deviner. 
\end{defi}

Comme il est de coutume, si $X$ est un objet de $\cC_T$ et $T'$ un $T$-schéma, nous désignerons son changement de base $\phi_{T'\to T}(X)$ par $X_{T'}$.

\begin{defi} \label{fonctfib_defi}
Un foncteur $F$ d'une $S$-catégories fibrée $\cC$ vers une autre $\cD$ est la donnée, pour chaque $T$ sur $S$, d'un foncteur $F_{T}:\cC_{T} \to \cD_{T}$, de manière à commuter avec les foncteurs de changement de base. 
\end{defi}

\begin{defi} \label{produitfibre_defi}
Soient $F_1:\cC_1 \to \cD$ et $F_2:\cC_2 \to \cD$ deux foncteurs de $S$-catégories fibrées. Alors on définit le produit fibré de catégories fibrées $\cC_1 \times_\cD \cC_2$ comme la catégorie fibrée dont les objets sur $T$ sont les triplets $(x_1,x_2,\phi)$ avec $x_i \in (\cC_i)_T$, $\phi:F_1(x_1)\isoto F_2(x_2)$ et dont les morphismes $(x_1,x_2,\phi) \to (y_1,y_2,\psi)$ sont les paires de morphismes $f_1:x_1 \to y_1$ et $f_2:x_2 \to y_2$ tels que $\psi \circ F_1(f_1)=F_2(f_2)\circ \phi$.
\end{defi}

Notons $P_i$, $i=1,2$, les foncteurs de projection évidents $\cC_1 \times_\cD \cC_2 \to \cC_i$, et $c$ l'isomorphisme de foncteurs $F_1 P_1 \isoto F_2 P_2$ donné par $c_{(x_1,x_2,\phi)}=\phi$. Il est immédiat que le produit fibré $\cC_1 \times_\cD \cC_2$ satisfait à la propriété universelle suivante. 
\begin{prop} \label{univprodfib_prop}
\'Etant donné une $S$-catégorie fibrée $\cA$, deux foncteurs $G_1:\cA \to \cC_1$ et $G_2:\cA\to\cC_2$ et un isomorphisme de foncteurs $a:F_1 G_1 \isoto F_2 G_2$, il existe un unique foncteur $H:\cA \to \cC_1\times_\cD \cC_2$ tel que pour $i=1,2$, on ait $P_i H = G_i$ et l'égalité de morphismes de foncteurs $c H = a$.
\end{prop}

Un \emph{groupoïde} est une catégorie dans laquelle tous les morphismes sont des isomorphismes. A toute catégorie $\cC$, on peut associer un groupoïde $\cC_\grpd$, qui a les mêmes objets, mais dont les morphismes sont les seuls isomorphismes de $\cC$. De même, le groupoïde $\cC_\grpd$ associé à une catégorie fibrée $\cC$ est la catégorie fibrée de fibres $(\cC_\grpd)_T=(\cC_T)_\grpd$.
\medskip

Nous voulons maintenant pouvoir manipuler des objets en une structure, mais sur une catégorie fibrée, et de manière compatible au changement de base. En d'autres termes, étant donné une structure $\struc$ dans la catégorie $\cC_S$, nous allons définir la $S$-catégorie fibrée des objets en $\struc$. Nous supposerons toujours dans ce contexte que les fibres de $\cC$ admettent des produits finis, et que les morphismes de changement de base les respectent. Une struture $\struc$ dans $\cC_S$ fournit par changement de base une structure $\struc_T$ sur $\cC_T$ par \ref{strucpouss_defi}: les objets et morphismes constants sont obtenus par changement de base de ceux de $\cC_S$ à $\cC_T$, et les objets en $\struc_T$ doivent être munis de morphismes satisfaisant aux ``mêmes'' contraintes. Le changement de base définit ainsi un foncteur $\cC_S^{\struc} \to \cC_T^{\struc_T}$.

\begin{defi} \label{catfibreestruc_defi}
\'Etant données une $S$-catégorie fibrée $\cC$ (munie de produits finis respectés par les changements de base) et une structure $\struc$ dans $\cC_S$, on définit la catégorie fibrée $\cC^{\struc}$ de fibre $(\cC^{\struc})_T=(\cC_T)^{\struc_T}$ et dont les changements de base sont induits par ceux de $\cC$. 
\end{defi}
Cela permet de construire de nouvelles catégories fibrées à partir de certaines structures.

Par exemple, la catégorie fibrée des $\faisO$-modules, dont la $T$-fibre sera constituée des $\faisO_T$-modules, peut être obtenue à partir de la catégorie fibrée $\cC$ des faisceaux, dont la $T$-fibre est constituée des $T$-faisceaux (voir page \pageref{Sfais_defi}), en considérant l'objet constant $\faisO_S$ pour définir la structure de $\faisO_S$-module dans $\cC_S$.

\subsubsection{Champs}

Pour pouvoir décrire concrètement les catégories de torseurs sous des groupes classiques, nous avons besoin d'un cadre dans lequel on puisse tordre un objet d'une catégorie par un torseur sous son groupe d'automorphismes, et obtenir ainsi un nouvel objet de cette catégorie. Le bon cadre pour ce type d'opérations est celui des champs, tel que développé dans \cite{gir}, dont nous n'aurons en fait besoin que d'une petite part. Aussi, le but de ce qui suit est d'introduire et de rendre palpable la partie de la théorie dont nous aurons besoin pour le lecteur qui ne connaît pas les champs.

\begin{defi}
Un \emph{$S$-foncteur de points} est un préfaisceau d'ensembles sur la catégorie des $S$-schémas, autrement dit un foncteur contravariant des $S$-schémas vers les ensembles.  
\end{defi} 
\begin{defi}
On dit qu'un foncteur de points est \emph{représentable par un schéma $X$} s'il est isomorphe au foncteur $T \mapsto X(T)$, où $X(T)=\Hom_{\text{$S$-schéma}}(T,X)$.
\end{defi}
On peut restreindre un foncteur à la sous-catégorie pleine des schémas affines sur $S$. Or, le foncteur qui envoie un $S$-schéma sur son foncteur de points restreint aux $S$-schémas affines sur $S$ est pleinement fidèle (voir \cite[th. de comparaison, p. 18]{dg}). En d'autres termes, pour connaître un schéma (ou un morphisme de schémas), il suffit de connaître ses points sur les schémas affines sur $S$. Nous définirons donc de nombreux schémas en donnant leurs $S$-points sur les schémas affines sur $S$, puis en montrant que le foncteur est représentable, ce qui montre que le foncteur de points s'étend de manière unique à tous les schémas sur $S$. 
Pour décrire un morphisme de schémas, il suffit alors de donner des morphismes de foncteurs de points. Autrement dit, lorsque $X_1$ et $X_2$ sont deux $S$-schémas, la donnée d'un morphisme de $S$-schémas $X_1 \to X_2$ revient à la donnée d'applications d'ensembles $X_1(Y) \to X_2(Y)$ pour tout $S$-schéma $Y$, de manière fonctorielle en $Y$. Il suffit même de donner ces morphismes pour les schémas $Y$ affines sur $S$, par ce qui précède. Ce dernier point est particulièrement utile pour définir des morphismes de schémas sur $\Spec(\ZZ)$, car il suffit alors de décrire le morphisme sur tous les $R$-points, $R$ étant un anneau.
Par exemple, pour donner une structure de schéma en groupes sur un $S$-schéma $X$, on peut décrire les morphismes donnant la structure de groupe sur chaque $X(Y)$, de manière fonctorielle en $Y$.
\medskip

L'ajout d'une \emph{topologie de Grothendieck} $\Ttau$ sur la catégorie des $S$-schémas permet de parler de propriétés locales d'objets d'une $S$-catégorie fibrée, et de définir la notion de faisceau, que nous rappelons ci-dessous. Nous ne pouvons rappeler ici la définition exacte d'une topologie de Grothendieck, et renvoyons le lecteur à \cite[Exp. IV, \S 4.2]{sga3}. Nous la supposerons toujours associée à une prétopologie, c'est-à-dire la spécification de recouvrements de chaque $S$-schémas $T$, un recouvrement de $T$ étant un ensemble de morphismes à but $T$. Ces recouvrements doivent bien entendu satisfaire certains axiomes (\loccit\ déf. 4.2.5), qui miment ceux des ouverts d'une topologie classique. 
Cette topologie sera souvent sous-entendue, et en pratique les topologies considérées dans le cadre qui nous occupe sont les topologies de Zariski, étale, \fppf\ et dans une moindre mesure \fpqc\ (dans l'ordre de la moins fine à la plus fine).
\begin{defi} \label{Sfais_defi}
Un $S$-foncteur de points $X$, \ie un foncteur contravariant des $S$-schémas vers les ensembles, sera dit un \emph{$S$-faisceau} (pour la topologie $\Ttau$) s'il satisfait aux conditions habituelles:
\begin{enumerate}
\item $X(\emptyset)$ est un singleton.
\item Pour tout $S$-schéma $T$, et pour tout recouvrement $(T_i \to T)_{i \in I}$ de $T$, si on note $T_{ij}=T_i \times_T T_j$, le diagramme d'ensembles
$$\xymatrix{X(T) \ar[r]^-{i} & \displaystyle \prod_{i\in I} X(T_i) \ar@<2pt>[r]^-{p_1} \ar@<-2pt>[r]_-{p_2} & \displaystyle \prod_{(i,j)\in I^2} X(T_{ij})}$$
est exact, au sens que le morphisme de gauche est injectif et identifie $X(T)$ avec l'égalisateur des deux morphismes de droite, où le morphisme $p_1$ est induit par les projections sur le premier facteur $T_{ij} \to T_i$ et $p_2$ par celles sur le deuxième facteur $T_{ij} \to T_j$.
\end{enumerate}
\end{defi}
(La première condition découle de la seconde, lorsque le schéma vide $\emptyset$ admet le recouvrement vide (à ne pas confondre avec le recouvrement par le vide, qui est toujours présent) dans la topologie $\Ttau$, ce qui sera le cas de toutes les topologies que nous considérerons par la suite.)
Pour $T \to S$ donné, on vérifie, voir \cite[Exp. IV, Prop. 4.5.2]{sga3}, que la restriction des $S$-foncteur de points vers les $T$-foncteurs de points, envoyant $X$ sur $X_T$ donné par $X_{T}(T')=X(T')$ pour tout $T'\to T$, envoie bien les $S$-faisceaux sur les $T$-faisceaux. On peut donc définir la catégorie fibrée des faisceaux (pour la topologie $\Ttau$ sur $S$), en utilisant les restrictions comme foncteurs de changement de base. Nous noterons cette catégorie $\CatFais_{/S,\Ttau}$, voire $\CatFais_{/S}$. 
\medskip
  
De manière imagée, un champ est une catégorie fibrée (sur une catégorie de base munie d'une topologie) dans laquelle on peut définir les objets et les morphismes localement par recollement. En étant encore plus imprécis, c'est une sorte de faisceau de catégories. 

\begin{defi} \label{SChamp_defi}
Un \emph{$S$-champ} (pour la topologie $\Ttau$) est une $S$-catégorie fibrée (voir déf. \ref{catfib_defi}) $\cC$ qui vérifie pour tout $S$-schéma $T$:
\begin{enumerate}
\item \label{faisceauchamp_item} Pour tous $x,y$ objets de $\cC_T$, le $T$-foncteur de points $T' \mapsto \setHom_{\cC_{T'}}(x_{T'},y_{T'})$ est un $\Ttau$-faisceau. 
\item \label{descente_item} La descente sur les objets: Soit $(T_i)_{i\in I}$ un recouvrement de $T$. Notons $T_{ij}$ le produit fibré de $T_i \times_T T_j$. Soient des objets $x_i \in \cC_{T_i}$, et des isomorphismes $\psi_{ij}: (x_i)_{T_{ij}} \isoto (x_j)_{T_{ij}}$ vérifiant la condition de compatibilité évidente sur les produits fibrés triples. Alors il existe un unique (à isomorphisme unique près) objet $x \in \cC_T$ muni d'isomorphismes $\lambda_i: x_i\isoto x_{T_i}$ tels que $\psi_{ij}=(\lambda_j)_{T_{ij}}^{-1}\circ(\lambda_i)_{T_{ij}}$ pour tout $i,j$. 
\end{enumerate}
Un morphisme de $S$-champs est un foncteur de $S$-catégories fibrées.
\end{defi}

\begin{defi} \label{donneedescente_defi}
\'Etant donnée un recouvrement $(T_i)_{i\in I}$ de $T$, la donnée d'objets $x_i \in \cC_{T_i}$ et d'isomorphismes compatibles $\psi_{ij}$ au sens du point \ref{faisceauchamp_item} précédent est appelée \emph{donnée de descente}, relativement au recouvrement fixé. 
\end{defi} 
On peut définir la catégorie des données de descente relativement à un recouvrement donné. Tout objet $x$ de $\cC_T$ définit canoniquement une donnée de descente sur n'importe quel recouvrement de $T$, en posant $x_i=x_{T_i}$ et en définissant les isomorphismes $\psi_{ij}$ canoniquement par la structure de catégorie fibrée. Cela définit un foncteur de $\cC_T$ vers la catégorie des données de descente relativement à $(T_i)_{i\in I}$, et la définition d'un champ peut alors se reformuler en disant que ce foncteur est une équivalence de catégories: le point \ref{faisceauchamp_item} donne la pleine fidélité, et le point \ref{descente_item} l'essentielle surjectivité. 

On dit alors qu'on construit un objet de $\cC_T$ (resp. un morphisme) ``par descente'' lorsqu'on en construit une occurrence sur chaque $T_i$ d'un recouvrement de $T$, ainsi que des isomorphismes de compatibilité sur les produits doubles (resp. rien), et qu'on applique le point \ref{descente_item} (resp. le point \ref{faisceauchamp_item}) de la définition d'un champ pour l'obtenir sur $T$. Autrement dit, on construit un objet (resp. un morphisme) de la catégorie des données de descente, et on utilise l'équivalence de catégories. 

On dit également qu'une condition sur les objets ou les morphismes \emph{descend} ou bien \emph{est locale} pour une certaine topologie, s'il suffit de la vérifier après changement de base à chacun des schémas d'un recouvrement pour cette topologie. 

\begin{exem} \label{champfaisceau_exem}
Le champ associé à un faisceau $X$ sur $S$, est la $S$-catégorie fibrée dont la fibre sur le $S$-schéma $T$ est la catégorie dont les objets sont les éléments de $X(T)$ avec pour morphismes $\setHom(x,x)=\{id_x\}$ et $\setHom(x,y)=\emptyset$ si $x \neq y$. C'est un exercice facile, qui utilise bien entendu que $X$ est un faisceau et pas un simple préfaisceau, dans la vérification des deux points de la définition. En particulier, le champ associé au faisceau final (constant, un point partout) n'a qu'un seul objet par fibre, et cet objet n'a qu'un seul morphisme, l'identité. Nous noterons ce champ $\Final$. 
\end{exem}

\begin{exem} \label{fibresdiscretes_exem}
En fait, on voit même facilement que si les fibres d'une catégorie fibrée sont discrètes et petites, c'est-à-dire que ce sont des catégories avec pour seuls morphismes les identités des objets, et ces objets forment un ensemble, comme dans l'exemple précédent, alors cette catégorie fibrée est un champ si et seulement si le foncteur de points qui à $T$ associe l'ensemble des objets de $\cC_T$ est un faisceau.
\end{exem}

\begin{exo}
La catégorie fibrée $\CatFais_{/S,\Ttau}$ est un champ. 
\end{exo}

\begin{prop}
Si $\cC$ est un champ, alors $\cC_{\grpd}$ est un champ.
\end{prop}
\begin{proof}
On peut construire l'inverse d'un morphisme localement par descente. 
\end{proof}

\begin{prop} \label{champstruc_prop}
Si $\struc$ est une structure dans $\cC_S$, où $\cC$ est un $S$-champ (muni de produits finis respectés par les changements de base), alors la catégorie fibrée $\cC^\struc$ est un champ. En particulier si $\struc$ est une structure sur les $S$-faisceaux, alors $\CatFais_{/S}^\struc$, au sens de la définition \ref{catfibreestruc_defi}, est un champ.
\end{prop}
\begin{proof}
La condition faisceautique des $\Hom$ se vérifie par inclusion des morphismes d'objets en $\struc$ dans les morphismes de $\cC$ (ou de puissances de $\cC$ lorsque $I$ a plusieurs éléments). Pour la descente des objets en $\struc$, on construit les objets de $\cC$ nécessaires par descente, puis les morphismes structuraux dans $\cC$ par la condition faisceautique sur les $\Hom$ dans $\cC$, et on montre qu'ils vérifient les relations par cette même condition. 
\end{proof}

\begin{nota} \label{isoaut_nota}
Lorsque $\cC$ est un $S$-champ, et $X$ et $Y$ sont des objets de $\cC_S$, on note $\faisHom^\cC_{X,Y}$ le $S$-faisceau dont les $T$-points sont $\setHom_{\cC_{T}}(X_T,Y_T)$. Lorsque $\cC^\struc$ est le champ défini à partir d'une structure sur $\cC_S$, comme en \ref{champstruc_prop}, on utilise également la notation $\faisHom^{\struc}_{X,Y}$ au lieu de $\faisHom^{\cC^\struc}_{X,Y}$. De même, on définit les faisceaux $\faisEnd^\cC_{X}=\faisHom^\cC_{X,X}$, $\faisIso^\cC_{X,Y}=\faisHom^{\cC_{\grpd}}_{X,Y}$ et $\faisAut^\cC_{X}=\faisIso^\cC_{X,X}$.
\end{nota}

\begin{rema} \label{Homchgmt_rema}
Dans cette situation, si $T$ est un $S$-schéma, on a immédiatement la commutation au changement de base $\faisHom^{\cC^T}_{X_T,Y_T}=(\faisHom^\cC_{X,Y})_T$ où $\cC^T$ est le $T$-champ obtenu de $\cC$ par restriction. 
\end{rema}

\begin{exo}[produit fibré de champs] \label{prodfibchamps_exo}
Le produit fibré de deux champs au-dessus d'un troisième, au sens de la définition \ref{produitfibre_defi}, est un champ.
\end{exo}

Cette propriété va nous permettre de construire de nombreux champs à partir d'autres.

\subsection{Faisceaux en groupes et torseurs} \label{torseurs_sec}

Expliquons maintenant comment on tord un faisceau $X$ muni d'une action d'un faisceau en groupes $H$ par un torseur $P$ sous $H$. L'espace ainsi obtenu $P \contr{H} X$ est appelé produit contracté (déf. \ref{prodcontract_defi}). Lorsque $\Xstruc$ est un objet en une structure $\struc$ (ex: un faisceau en groupes) et que $H$ la respecte, on peut définir le produit contracté $P \contr{H} \Xstruc$, qui est également un objet en $\struc$. Sous des hypothèses raisonnables (prop. \ref{tordusformes_prop}), on obtient ainsi toutes les formes de $\Xstruc$, c'est-à-dire les objets isomorphes à $\Xstruc$ localement pour la topologie considérée. C'est ce qui nous servira à obtenir tous les groupes réductifs d'un type déployé donné à partir du groupe de Chevalley correspondant. Les questions de représentabilité de ces différents faisceaux par des schémas sont regroupées dans la section \ref{representabilite_sec} et en particulier dans la proposition \ref{reprschemas_prop}.  

\subsubsection{Actions de groupes}

Soit $H$ un $S$-faisceau en groupes et soit $X$ un $S$-faisceau, muni d'une action à gauche (resp. à droite) de $H$, donc un objet en $\actg{H}$ (resp. en $\actd{H}$) comme dans l'exemple \ref{structures_exem} \eqref{groupeaction_item} ci-dessus. Ceci est équivalent à la donnée d'un morphisme de faisceaux en groupes $H \to \faisAut_X$ (resp. $H^\opp \to \faisAut_X$) (un faisceau à groupes d'opérateurs $H$ dans la terminologie de \SGAtrois). 

Considérons alors la catégorie fibrée $\CatFais_{/S}^{\actg{H}}$ (resp. $\CatFais_{/S}^{\actd{H}}$), qui est un champ par la proposition \ref{champstruc_prop}. Les objets de sa $T$-fibre sont donc les $T$-faisceaux munis d'une action à gauche de $H_T$ (resp. à droite). 

\subsubsection{Torseurs}

Soit $\Ttau$ une topologie de Grothendieck sur les $S$-schémas. Sauf mention contraire, nous supposerons toujours que les faisceaux mentionnés dans cette partie sont des faisceaux au sens de cette topologie. Notons que $\CatFais_{/S}$ étant un champ, un morphisme $f$ de faisceaux est un isomorphisme si et seulement s'il en est un localement. 

Dans ce qui suit, $P$ désigne toujours un $S$-faisceau avec action à droite de $H$, donc un objet de la catégorie $(\CatFais_{/S}^{\actd{H}})_S$. La notion symétrique existe bien entendu avec une action à gauche.

\begin{defi}
On dit que $P$ est un \emph{pseudo-torseur} ou \emph{formellement principal homogène} 
(resp. \emph{formellement homogène})
sous $H$ si l'application $P \times H \to P \times P$ dont les composantes sont l'action et la projection est un isomorphisme (resp. un épimorphisme) de faisceaux. Voir \cite[Exp. IV, 5.1.0, et 6.7.1]{sga3}.
 
\end{defi}

\begin{defi} \label{torseur_defi}
On dit que $P$ est un torseur (resp. est homogène) sous $H$ pour la topologie $\Ttau$, ou un $\Ttau$-torseur, s'il est un pseudo-torseur (resp. formellement homogène) et que $\Ttau$-localement, $P$ a un point, i.e. $P(S_i)\neq \emptyset$ pour tout $S_i$ dans un certain recouvrement de $S$. 
\end{defi}

On notera que contrairement à la notion de pseudo-torseur, qui n'est qu'une notion de foncteurs de points, la notion de torseur dépend bien de la topologie $\Ttau$. Le fait que $P$ est localement non vide revient à ce que son morphisme structural $P \to S$ soit un épimorphisme de faisceau.

\begin{exem}
Lorsqu'on fait agir $H$ sur lui-même à droite par translations, on obtient évidemment un torseur, appelé torseur trivial.
\end{exem}

\begin{prop} \label{torseurloc_prop}
Un faisceau $P$ muni d'une action à droite de $H$ est un torseur si et seulement s'il est $\Ttau$-localement isomorphe au torseur trivial $H$. Autrement dit, un torseur est un faisceau avec action qui est localement trivial pour la topologie $\Ttau$.
De plus, on a alors $P_T \simeq H_T$ si et seulement si $P(T)\neq \emptyset$. 
\end{prop}
\begin{proof}
Si $P(T)$ contient un point $p$, alors on construit immédiatement un isomorphisme $H_T \isoto P_T$ de faisceaux avec action à droite de $H$ en envoyant $h$ vers $p\cdot h$, et bien entendu, si $P_T \simeq H_T$, alors $P(T)\neq \emptyset$, puisque $H(T)$ contient toujours l'élément neutre. Réciproquement, si $P$ est muni d'une action à droite de $H$ et si $P_T \simeq H_T$, alors l'application $P_T \times H_T \to P_T \times P_T$ mentionnée plus haut est alors un isomorphisme. C'est donc un isomorphisme localement, donc un isomorphisme.
\end{proof}

\begin{defi} \label{deploye_defi}
Lorsqu'un torseur $P$ sous $H$ devient isomorphe au torseur trivial après extension à un schéma $T$ sur $S$, donc $P_T \simeq H_T$ ou de manière équivalente $P(T)\neq \emptyset$, on dit qu'il est \emph{déployé sur $T$}. On abrège déployé sur $S$ par \emph{déployé}.
\end{defi}

\begin{prop} \label{flecheIsoTors_prop}
Si $P_1$ et $P_2$ sont des torseurs sous $H$, et si $f: P_1 \to P_2$ est un morphisme de faisceaux qui est équivariant sous l'action de $H$, alors $f$ est un isomorphisme. 
\end{prop}
\begin{proof}
Si $P_1(T) \neq \emptyset$, alors $f:P_1(T)\to P_2(T)$ est une bijection. En effet, l'action de $H(T)$ est transitive des deux côtés et commute à $f$. Si $P_2(T)\ni s$, alors on utilise un recouvrement de $T$ par des $T_i$ tels que $(P_1)_{T_i}$ est trivial pour tout $i$ (et a donc un point). On a donc des bijection $f_{T_i}:P_1(T_i) \to P_2(T_i)$, ce qui permet de construire un point dans $P_1(T)$ par recollement des $f^{-1}(s_{T_i})$, ce qui ramène au cas précédent.
\end{proof}

\begin{prop} \label{stabilisateur_prop}
Soit $G$ un $S$-faisceau en groupes qui agit sur un $S$-faisceau $X$ de manière homogène. Soit $x \in X(S)$, et soit $H$ le sous-faisceau de $G$ des éléments fixant $x$. Alors $X$ muni du morphisme $G \to X$ d'action sur $x$ s'identifie au faisceau quotient $G/H$, muni du morphisme canonique $G \to G/H$. 
\end{prop}
\begin{proof}
On vérifie la propriété universelle du faisceautisé associé au préfaisceau $G/H$ donné par $(G/H)(T)=G(T)/H(T)$. Il y a bien entendu une flèche canonique de préfaisceaux $G/H \to X$. Pour tout faisceau $Y$ muni d'une flèche $G/H \to Y$, il faut construire l'unique flèche $X \to Y$ l'étendant. Or si un $T$-point de $X$ est dans l'image de $G \to X$, il est clair qu'il ne peut s'envoyer que sur un seul élément, obtenu en le relevant à $G(T)$, puis en appliquant $G(T) \to (G/H)(T) \to Y(T)$. Puisque l'action de $G$ est homogène, c'est vrai localement, et par descente dans le faisceau $Y$, on construit donc de manière unique l'image de tout point de $X$. 
\end{proof}

Si $P$ est un $S$-torseur sous $H$, alors pour tout $S'\to S$, $P_{S'}$ est un $H_{S'}$-torseur, puisqu'être un pseudo-torseur et le fait d'avoir un point localement sont deux notions stables par changement de base. On peut donc définir une sous-catégorie fibrée (pleine) de $\CatFais_{/S}^{\actd{H}}$, donc les objets de la fibre en $S'$ sont les faisceaux avec action de $H_{S'}$ qui sont des torseurs. Cette sous-catégorie fibrée est un champ. En effet, les morphismes de torseurs sont simplement des isomorphismes d'objets avec action à droite, donc forment des faisceaux puisque $\CatFais_{/S}^{\actd{H}}$ est un champ, et on a donc également la descente des objets. De plus, étant donné un morphisme $T \to S$, on peut considérer la sous-catégorie fibrée dont les objets de la fibre en $S'$ sont les torseurs déployés par $T\times_S S'\to T$. Elle forme également un sous-champ de manière évidente. 
\begin{defi} \label{torschamp_defi}
On note $\Tors{H}$ le champ des torseurs (à droite) sous $H$, et $\Tors{H,T/S}$ le sous-champ des torseurs déployés par $T\to S$.
\end{defi}

\begin{defi} \label{prodcontract_defi}
Soient $X$ et $P$ des faisceaux avec action de $H$ respectivement à gauche et à droite. On note $P \contr{H} X$ le faisceau conoyau\footnote{Pour l'existence de ce faisceau, on a fait abstraction des problèmes de théorie des ensembles sur l'existence des faisceautisés d'un préfaisceau pour une topologie de Grothendieck. Pour y remédier, il faut soit savoir que ce faisceau existe dans les cas qui nous intéressent, soit se placer dans un univers fixé comme dans \cite{dg}.} des deux flèches $H \times P \times X \to P \times X$ données respectivement sur les points par $(h,p,x) \mapsto (ph,x)$ et $(h,p,x) \mapsto (p,hx)$. On appelle $P \contr{H} X$ le \emph{produit contracté} de $P$ et $X$. 
\end{defi}
Autrement dit, $P \contr{H} X$ est le faisceautisé du préfaisceau des orbites de $H$ agissant sur $P \times X$ par $(h,(p,x)) \mapsto (ph^{-1},hx)$. Il est immédiat qu'on définit ainsi un foncteur de la catégorie des $S$-faisceaux avec action de $H$ (avec morphismes respectant l'action) vers les $S$-faisceaux.

\begin{lemm} \label{assoccontr_lemm}
Cette construction est associative: lorsque $Q$ est muni d'une action à gauche de $H$ et à droite de $H'$ et qu'elles commutent, on a 
$$(P \contr{H} Q) \contr{H'} X \cong P \contr{H} (Q \contr{H'} X).$$ 
D'autre part, on a évidemment $H \contr{H} X \cong X$, d'où pour tout morphisme de groupes $\phi:H' \to H$, un isomorphisme $P \contr{H'} H \contr{H} X \cong P \contr{H'} X$ où $H'$ agit sur $X$ à travers $\phi$. 
\end{lemm}
\begin{proof}
Laissée au lecteur.
\end{proof}

Dans ce qui suit, $P$ désigne toujours un pseudo-torseur.

\begin{lemm} \label{tordprod_lemm}
Si $X_1$ et $X_2$ sont munis d'une action de $H$, alors l'application naturelle $P \contr{H} (X_1 \times X_2) \to (P \contr{H} X_1) \times (P \contr{H} X_2)$, obtenue par covariance de $P \contr{H} (-)$ et par propriété universelle du produit, est un isomorphisme. N.B. C'est l'application faisceautisée de $(p,(x_1,x_2)) \mapsto ((p,x_1),(p,x_2))$.
\end{lemm}
\begin{proof}
L'application inverse se définit aisément en utilisant que $P$ est un pseudo-torseur pour identifier $P \times X_1 \times P \times X_2$ avec $P \times X_1 \times H \times X_2$ avant de quotienter par l'action de $H \times H$. 
\end{proof}

\begin{prop} \label{fonctorialitegroupe_prop}
Soit $\phi:H_1 \to H_2$ un morphisme de $S$-faisceaux en groupes. On peut alors munir $H_2$ d'une action à gauche de $H_1$ par $\phi$. L'application 
$$P \mapsto P \contr{H_1} H_2$$
définit un foncteur de la catégorie des $H_1$-torseurs (à droite) vers les $H_2$ torseurs (à droite), et il y a un morphisme de foncteurs d'associativité en cas de composition de deux morphismes.
\end{prop}
\begin{proof}
Le groupe $H_2$ agit bien à droite sur $P \contr{H_1} H_2$ car les actions de $H_1$ à gauche et de $H_2$ à droite sur $H_2$ commutent. On vérifie qu'il en fait un torseur. La fonctorialité en $P$ est évidente. L'associativité provient du lemme \ref{assoccontr_lemm}.
\end{proof}

Relions maintenant les produits fibrés de groupes et les produits fibrés de champs.
Soient $f_1:G_1 \to H$ et $f_2:G_2 \to H$ deux morphismes de $S$-faisceaux en groupes. Le faisceau en groupes $G=(G_1 \times_H G_2)$, défini par $(G_1 \times_H G_2)(T) =G_1(T)\times_{H(T)} G_2(T)$, est un faisceau par l'exemple \ref{fibresdiscretes_exem} et l'exercice \ref{prodfibchamps_exo}. Considérons le foncteur fibré
$$\begin{array}{lll}
\Tors{G} & \to &  \Tors{G_1} \times_{\Tors{H}} \Tors{G_2} \\
P & \mapsto & (P \contr{G} G_1, P\contr{G} G_2, i).
\end{array}
$$
où $i$ est la composée $(P \contr{G} G_1)\contr{G_1} H \isoto P \contr{G} H \isoto (P \contr{G} G_2)\contr{G_2} H$.
\begin{prop} \label{prodfibgroupechamp_prop}
Ce foncteur $\Tors{G}  \to   \Tors{G_1} \times_{\Tors{H}} \Tors{G_2}$
est une équivalence de catégories fibrées si l'une des des flèches $G_1 \to H$ ou $G_2 \to H$ est un épimorphisme de faisceaux.
\end{prop}
\begin{proof}
On construit un foncteur dans l'autre sens de la manière suivante. \'Etant donné un triplet $(P_1,P_2, \phi: P_1 \contr{G_1} H \isoto P_2 \contr{G_2} H)$, on fabrique le produit fibré $P_1 \times_{\phi} P_2$, à l'aide des morphismes $P_1 \to P_2 \contr{G_2} H$ et $P_2 \to P_2 \contr{G_2} H$ définis respectivement sur les points par $p \mapsto \phi((p,1))$ et $p \mapsto (p,1)$. On le munit de l'action évidente de $G$ à droite par propriété universelle du produit fibré. Enfin, on vérifie que c'est un pseudo-torseur, grâce aux mêmes propriétés de $P_1$ et $P_2$. Reste à vérifier que localement, le pseudo-torseur obtenu a un point. On peut donc se ramener au cas où $P_1=G_1$, $P_2=G_2$, et $G_1 \contr{G_1} H = G_2 \contr{G_2} H=H$. Le morphisme $\phi$ étant $H$ équivariant à droite, il correspond alors à la multiplication à gauche par un point $h$ de $H$. Par hypothèse, quitte à localiser encore, on peut supposer qu'il est l'image d'un point point $g$ de $G_1$ (resp. de $G_2$). Le point $(g,1)$ (resp. $(1,g^{-1})$) est alors dans dans le produit fibré, qui est non vide. Cette construction est bien sûr fonctorielle, et le fait que le foncteur obtenu donne un inverse à isomorphisme près du précédent est laissé au lecteur.
\end{proof}

Cette équivalence nous permettra de décrire les torseurs sous des noyaux, ou sous des produits cartésiens de groupes, par exemple les torseur sous $\faisSL_n$ ou $\faismu_n$.

\subsubsection{Torsion et structure}

Expliquons maintenant comment un objet en une structure est tordu en un objet en la même structure. Soit $\cC$ une catégorie, soit $H$ un objet en groupes de $\cC$, i.e. un objet de la catégorie $\cC^\gr$, et soit $\struc$ une structure élémentaire dans $\cC$. Soit $\Xstruc \in \cC^{\struc}$, de famille $(X_i)_{i\in I}$ d'objets de $\cC$. Supposons maintenant donnée une action de $H$ sur chacun des $X_i, i\in I$ (exemple \ref{groupeaction_item} de \ref{structures_exem}) et faisons-le agir trivialement sur tous les $Y_i$, puis diagonalement sur tout objet source ou but des morphismes de la structure.
\begin{defi} \label{actiongroupeelem_defi}
On dit que l'action de $H$ sur $\Xstruc$ \emph{respecte la structure} $\struc$ si ces actions ``diagonales'' commutent à tous les morphismes de la structure. On dit alors que $H$ agit sur $X$ par automorphismes de $\struc$. 
\end{defi}

Soit $\Xstruc$ un $S$-faisceau en $\struc$, d'objets structuraux $(X_i), i\in I$. Supposons que chaque $X_i$ soit muni d'une action de $H$, de manière à ce que les morphismes structurels soient équivariants quand on fait agir $H$ trivialement sur les objets constants $Y_j$ de $\struc$, puis diagonalement sur les produits. Alors les objets $P \contr{H} X_i$ et les morphismes obtenus en tordant les morphismes structuraux de $X$ par $P$ (le produit contracté commutant aux produits, par \ref{tordprod_lemm}) forment un nouvel objet en $\struc$, qu'on notera $P \contr{H} \mathfrak{X}$. De plus, cette construction est fonctorielle: un morphisme d'objets en $\struc$ se tord en un morphisme d'objets en $\struc$. 

\begin{exem}
Tout les exemples de \ref{structures_exem} conservent donc leur structure lorsqu'on les tord par un pseudo-torseur sous un groupe qui agit en respectant la structure. 
\end{exem}

\begin{exem}
Si $\Xstruc=G$ et $H=\faisAut^\gr_G$ est le faisceau d'automorphismes d'objet en groupes (voir notation \ref{isoaut_nota}), alors $P \contr{H} G$ est naturellement muni d'une structure de groupe.
Si plutôt $\Xstruc=H=G$ muni de son action sur lui-même par conjugaison, c'est bien entendu également le cas, puisque $H$ agit à travers $\faisAut^\gr_G$.
\end{exem}

\begin{prop} \label{torsionmorphisme_prop}
Si $f:\Xstruc \to \Ystruc$ est un morphisme de $S$-faisceaux en $\struc$, que $\phi: H \to G$ est un morphisme de $S$-faisceaux en groupes, que $H$ (resp. $G$) agit sur $\Xstruc$ (resp. $\Ystruc$) en respectant la structure, et que $f$ est $H$-équivariant quand $H$ agit sur $\Ystruc$ à travers $\phi$, alors pour tout $H$-torseur $P$, le morphisme $f$ se tord en un morphisme $P \contr{H} \Xstruc \to (P \contr{H} G)\contr{G} \Ystruc$ d'objets en $\struc$. 
\end{prop}
\begin{proof}
C'est une conséquence facile de la fonctorialité.
\end{proof}

Jusqu'ici, nous avons donc vu le produit contracté, qui permet de tordre par un torseur un $S$-faisceau en $\struc$, pour en obtenir un autre. Cette construction se généralise aux champs de la manière suivante.

Soit $\cC$ un $S$-champ, et soit $H$ un $S$-faisceau en groupes. Un action (à gauche) de groupe de $H$ sur un objet $X$ de $\cC_T$, est un morphisme de faisceaux en groupes $\phi:H_T\to \faisAut_X$. 
Une action à droite est un morphisme $H_T^\opp \to \faisAut_X$ où $H^\opp$ est le faisceau en groupes opposé à $H$. 

Les paires $(X,\phi)$ forment donc une catégorie fibrée dite des objet à groupe d'opérateurs $H$, notée $\Oper{H}{\cC}$ (resp. $\Oper{\cC}{H}$ lorsque l'action est à droite). On vérifie immédiatement que cette catégorie fibrée est un champ. Notons que lorsque $\cC=\CatFais_{/S}$, on a par définition $\Oper{H}{\cC}=\CatFais_{/S}^{\actd{H}}$.

Par \cite[Ch. III, prop. 2.3.1]{gir}, il existe un morphisme de champs (donc un foncteur de catégories fibrées) 
$$Tw: \Tors{H} \times \Oper{H}{\cC} \to \cC$$
qui envoie un couple $(P,X)$ vers un objet noté $P\contr{H} X$ et tel que si $P=H$ le torseur trivial, alors $H \contr{H} X$ est canoniquement isomorphe (donc par un isomorphisme de foncteurs, disons $i$) à $X$, de manière à ce que lorsque $P=H$, l'opération de $H$ sur $X$ se retrouve naturellement de la manière suivante: Notons pour un instant $H_d$ le $H$-torseur trivial. On impose alors que l'opération $H \to \faisAut_X$ soit égale à la composée
$$H \to \faisAut_{H_d} \tooby{Tw(-,X)} \faisAut_{H_d \contr{H} X} \tooby{i} \faisAut_X.$$
où la première flèche envoie un élément sur son action par translation à gauche sur $H$.
\emph{Cette paire $(Tw,i)$ est alors unique à isomorphisme unique près.} Bien entendu, lorsque $\cC=\CatFais^{\struc}$, on retrouve le produit contracté défini plus haut, par unicité. 

Ce foncteur $Tw$ n'est pas mystérieux. On peut construire $P\contr{H} X$ à la main par descente: on choisit un recouvrement $(T_i)$ de $T$ pour lequel $P_{T_i}$ est trivial, et on choisit un point $p_i$ de $P(T_i)$. Ensuite, puisque $P_{T_{ij}}$ reste trivial, il existe un élément $h_{ij}$ de $H(T_{ij})$ qui envoie $(p_i)_{T_{ij}}$ vers $(p_j)_{T_{ij}}$. On peut donc considérer la donnée de descente $X_i=X_{T_i}$ et $(X_i)_{T_{ij}} =X_{T_{ij}} \simeq (X_j)_{T_{ij}}$ donnée par l'action de $h_{ij}$. Il est immédiat que c'est bien une donnée de descente, et cela définit donc un objet de $\cC$, puisque $\cC$ est un champ. 
\medskip

Si $X\in \Oper{H}{\cC}$ alors pour tout $Y \in \cC$, le faisceau $\faisHom_{X,Y}$ est naturellement muni d'une action à droite de $H$ par précomposition par l'action de $H$. C'est donc un objet de $\Oper{\CatFais_{/S}}{H}$. 

Si $P$ est un $H$-torseur, et $\psi: P \to \faisHom_{X,Y}$ est un morphisme de $\Oper{\CatFais_{/S}}{H}$, alors on obtient canoniquement un morphisme $\phi: P\contr{H}X \to Y$, de manière fonctorielle en $X$ et $Y$. Lorsque $P$ a un point $p$ et s'identifie donc à $H$ au moyen de ce point, on construit le morphisme comme $P\contr{H} X \isoto H \contr{H} X \isoby{i} X \tooby{\psi(p)} Y$. Puis, on vérifie facilement que ce morphisme ne dépend pas du choix de $p$, et est donc suffisamment canonique pour descendre par le deuxième point de la définition des champs.

Dans l'autre sens, étant donné un morphisme dans $\cC$ de $\phi:P \contr{H} X \to Y$, on construit un morphisme de faisceaux $P \to \faisHom_{X,Y}$ en envoyant un point $p$ de $P$ sur la composition $X \isoby{i} H \contr{H} X \isoto P \contr{H} X \tooby{\phi} Y$, où le deuxième isomorphisme est l'identification de $H$ et $P$ induit par $p$. On retiendra donc une sorte d'adjonction (qui n'en est pas tout à fait une) qui permet souvent de faciliter les constructions:

\begin{lemm} \label{adjtorsmorph_lemm}
Si $P$ est un $H$-torseur et si $X$ est un objet de $\Oper{H}{\cC}$, alors par les deux constructions ci-dessus, il est équivalent de se donner un morphisme de $P \to \faisHom_{X,Y}$ dans $\Oper{\CatFais_{/S}}{H}$ ou un morphisme $P \contr{H} X \to Y$ dans $\cC$. 
\end{lemm}

On veut maintenant pouvoir comparer le faisceau en groupes d'automorphismes de $X=P\contr{H} X_0$ avec le groupe $P \contr{H} \faisAut_X$ où $H$ agit sur $\faisAut_X$ par automorphismes intérieurs. 
\begin{prop} \label{auttordus_prop}
Considérons le morphisme de faisceaux en groupes
$$P \contr{H} \faisAut_X \to \faisAut_{(P \contr{H}X)}$$
défini sur les points par 
$$(p,a) \mapsto \big((p',x) \mapsto (p,a h_{p,p'}(x))\big)$$
où $h_{p,p'} \in H$ est tel que $p h_{p,p'} = p'$.
C'est un isomorphisme. 
\end{prop}
\begin{proof}
Tout d'abord, il faut vérifier que ce morphisme est bien défini et respecte les structures de groupes. Ce point est laissé au lecteur. 
Il suffit ensuite de prouver que c'est un isomorphisme après trivialisation de $P$. Or dans ce cas, à isomorphismes canoniques près, c'est l'application identité de $\faisAut_X$.
\end{proof}

\begin{rema}
Le cas particulier où $H=\faisAut_X$ est souvent utile.
\end{rema}

\begin{prop} \label{fonctorialiteTors_prop}
Si $H_1\to H_2$ est un morphisme de faisceaux en groupes, et $P$ est un torseur sous $H_1$, alors on a un isomorphisme canonique, fonctoriel en l'objet $X$ dans $\Oper{H_2}{\cC}$,
$$(P \contr{H_1} H_2)\contr{H_2} X \isoto P \contr{H_1} X $$
où dans le second terme, $H_1$ agit sur $X$ à travers $H_2$.
\end{prop}
\begin{proof}
Lorsque $P \simeq H_1$, la formule est triviale et canonique. Le cas général s'obtient alors par descente.
\end{proof}

Soit $F:\cC\to \cD$ un foncteur de catégories fibrées. Par exemple, un $\faisO_S$-module quadratique est naturellement un $\faisO_S$-module tout court. Pour tout objet $X$ de $\cC$, on a donc un morphisme $\faisAut_X \to \faisAut_{F(X)}$.  
Dans cette situation, on a alors le résultat suivant.
\begin{prop} \label{foncttors_prop}
Soit $P$ un torseur sous $\faisAut_X$, et soit $P'$ le torseur sous $\faisAut_{F(X)}$ obtenu en poussant $P$ le long de $\faisAut_X \to \faisAut_{F(X)}$. 
\begin{enumerate}
\item \label{fonctobjet_item} On a un isomorphisme naturel  $P'\contr{\faisAut_{F(X)}} F(X) \isoto F(P \contr{\faisAut_{X}} X)$.
\item \label{fonctaut_item} Par la proposition \ref{torsionmorphisme_prop}, le morphisme $\faisAut_X \to \faisAut_{F(X)}$ se tord en le morphisme $\faisAut_{P\contr{\faisAut_X} X} \to \faisAut_{F(P\contr{\faisAut_X} X)}$, en utilisant les identifications de la proposition \ref{auttordus_prop} et du point précédent.
\end{enumerate}
\end{prop}
\begin{proof}
\ref{fonctobjet_item}. On part de l'identité $P \contr{\faisAut_X} X \to P \contr{\faisAut_X} X$, pour obtenir un morphisme $P \to \faisHom_{X,P \contr{\faisAut_X} X}$, qu'on compose avec le morphisme vers $\faisHom_{F(X),F(P\contr{\faisAut_X} X)}$, induit par $F$. De nouveau par le lemme \ref{adjtorsmorph_lemm}, on obtient un morphisme $P \contr{\faisAut_X} F(X) \to F(P \contr{\faisAut_X} X)$, or le premier terme est canoniquement isomorphe à $P' \contr{\faisAut_{F(X)}} F(X)$ par la proposition \ref{fonctorialiteTors_prop}. On a donc obtenu le morphisme recherché, fonctoriellement, et on vérifie sans peine que c'est un isomorphisme lorsque $P=H$, et donc en général par descente.

Le point \ref{fonctaut_item} est une conséquence simple de la fonctorialité du produit contracté, et est laissé au lecteur.
\end{proof}

\begin{exem}
Regardons le cas où le foncteur fibré est donné par un oubli de structure, par exemple des modules quadratiques (voir \ref{O_defi}) vers les modules localement libres (voir \ref{VecChamp_prop}). La proposition \ref{foncttors_prop} dit alors que lorsqu'on tord $(M,q)$ par un torseur sous ses automorphismes, autrement dit sous le groupe orthogonal de $\faisorthO_{M,q}$, et qu'on oublie la structure quadratique sur le résultat, on obtient le même $\faisO_S$-module (à isomorphisme canonique près) qu'en poussant ce torseur à $\faisGL_M$ et en tordant $M$. Cela n'a évidemment rien de surprenant.
\end{exem}

\subsubsection{Formes}

Soit $\cC$ un $S$-champ, et soit $X_0$ un objet de $\cC_T$, pour un certain $T\to S$. 
\begin{defi} \label{formes_defi} 
On dit qu'un objet $X$ de $\cC_T$ est une \emph{forme} de $X_0$ s'il est $\Ttau$-localement isomorphe à $X_0$. C'est-à-dire qu'il existe un recouvrement $(T_i)_{i\in I}$ de $T$ tel que $X_{T_i}\simeq (X_0)_{S_i}$. On dit que $X$ est \emph{déployé} par un morphisme $T' \to T$ si $X_{T'} \simeq (X_0)_{T'}$. 
\end{defi}

Soit $X_0$ un objet de $\cC_S$. Une forme de $(X_0)_T$ s'envoie par changement de base $\cC_T \to \cC_{T'}$ sur une forme de $(X_0)_{T'}$, on obtient donc une ainsi une catégorie fibrée $\Formes{X_0}$. De même, $T \to S$ étant donné, on peut considérer les formes de $(X_0)_{S'}$ déployées par un morphisme $T\times_S S' \to S$, et cela forme une sous-catégorie fibrée de la précédente, notée $\Formes{X_0,T/S}$. 

\begin{prop}
Les catégories fibrées $\Formes{X_0}$ et $\Formes{X_0,T/S}$ sont des champs.
\end{prop}
\begin{proof}
Exercice.
\end{proof}

\begin{prop} \label{formesstruc_defi}
Si $X$ est un objet de $\cC_T$ et si $P$ est un torseur sous $H$ agissant sur $X$, le produit contracté $P \contr{H} X$ est une forme de $X$.
\end{prop}
\begin{proof}
Par la proposition \ref{torseurloc_prop}, le torseur $P$ est $\Ttau$-localement isomorphe à $H$, et on a $H \contr{H} X \simeq X$. 
\end{proof}

\begin{prop} \label{structfonctIsotors_prop}
Soit $X$ une forme de $X_0$ (resp. déployée sur $T'$). Alors $\faisIso_{X_0,X}$ est un torseur sous $\faisAut_{X_0}$ (resp. trivialisé sur $T'$). 
\end{prop}
\begin{proof}
Clair.
\end{proof}

On considère le champ $\Tors{\faisAut_{X_0}}$. Notons que c'est un champ en groupoïdes par la proposition \ref{flecheIsoTors_prop}. On considère également le champ en groupoïdes $\Formes{X_0}_\grpd$. 
Les foncteurs 
$$\begin{array}{cll}
\Tors{\faisAut_{X_0}} & \to & \Formes{X_0} \\
P & \mapsto & P \contr{\faisAut_{X_0}} X_0
\end{array} 
\hspace{1ex}\text{et}\hspace{1ex}
\begin{array}{cll}
\Formes{X_0} & \to & \Tors{\faisAut_{X_0}} \\
X & \mapsto & \faisIso_{X_0,X}
\end{array}$$ 
sont adjoints en utilisant pour unité et coünité les applications évidentes 
$$\begin{array}{cll}
\faisIso_{X_0,X}\contr{\faisAut_{X_0}}X_0 & \to & X \\
(\psi,x) & \mapsto & \psi(x)
\end{array}
et 
\begin{array}{cll}
P & \to & \faisIso_{X_0,(P \contr{\faisAut_{X_0}} X_0)} \\
p & \mapsto & (x \mapsto (p,x))
\end{array}
$$ 
(données ici quand $\faisIso_{X_0,X}$ a un point, et on les construit donc par descente à partir de ce cas).

\begin{prop} \label{tordusformes_prop}
Ces foncteurs sont des équivalences de catégories fibrées inverses l'une de l'autre, entre $\Tors{\faisAut_{X_0}}$ et $\Formes{X_0}$. Ils se restreignent en une équivalence de catégorie fibrées entre $\Tors{\faisAut_{X_0},T/S}$ et $\Formes{X_0,T/S}$.
\end{prop}
\begin{proof}
Si $X$ est une forme de $X_0$ (resp. déployée par $T\to S$), le faisceau $\faisIso_{X_0,X}$ est un $\faisAut_{X_0}$-torseur (resp. trivialisé par $T\to S$). La counité de l'adjonction $\faisIso_{X_0,X}\contr{\faisAut_{X_0}}X_0 \to X$ est un isomorphisme localement (resp. après extension à $T$), donc c'est un isomorphisme.
On raisonne de même avec le morphisme naturel $P \to \faisIso_{X_0,(P \contr{\faisAut_{X_0}} X_0)}$.
\end{proof}

\begin{prop} \label{produitformes_prop}
Soient $F_1:\cC_1 \to \cD$ et $F_2:\cC_2\to \cD$ deux foncteurs de catégories fibrées entre $S$-champs. Soient $X_1 \in (\cC_1)_S$, $X_2 \in (\cC_2)_S$ et $\phi:F_1(X_1)\isoto F_2(X_2)$ un isomorphisme, donc un objet $X_0=(X_1,X_2,\phi)$ de $(\cC_1 \times_{\cD} \cC_2)_S$. Alors $\faisAut_{X_0} \simeq \faisAut_{X_1} \times_{\faisAut_{F_1(X_1)}} \faisAut_{X_2}$. Supposons de plus que 
$$\faisAut_{X_1} \times \faisAut_{X_2} \to \faisIso_{F(X_1),F(X_2)}$$
donné sur les points par $(f,g)\mapsto F_2(g) \circ \phi \circ F_2(f^{-1})$ est un épimorphisme de faisceaux (ce qui est en particulier le cas si 
 $\faisAut_{X_1} \to \faisAut_{F_1(X_1)}$ ou $\faisAut_{X_2}\to \faisAut_{F_2(X_2)}$ est un épimorphisme de faisceaux). Alors le foncteur fibré $\Formes{X_0} \to \Formes{X_1} \times_{\Formes{F_1(X_1)}} \Formes{X_2}$ obtenu par la propriété universelle du produit fibré est une équivalence de catégories.
\end{prop}
\begin{proof}
La propriété sur les groupes d'automorphismes est immédiate sur les points.
L'équivalence de catégories suit alors des propositions \ref{prodfibgroupechamp_prop} et \ref{tordusformes_prop}.
\end{proof}

\begin{prop} \label{isoFoncFib_prop}
Soient $F_1, F_2:\cC \to \cD$ deux foncteurs fibrés entre deux champs. Supposons que pour un objet $X_0 \in \cC$, on ait un isomorphisme $\phi:F_1(X_0)\isoto F_2(X_0)$ et que la conjugaison par cet isomorphisme $\int_\phi : \faisAut_{F_1(X_0)} \isoto \faisAut_{F_2(X_0)}$ fasse commuter le diagramme 
$$\xymatrix{
\faisAut_{X_0} \ar[r]^{F_1} \ar[dr]^{F_2} & \faisAut_{F_1(X_0)} \ar[d]^{\int_\phi} \\
 & \faisAut_{F_2(X_0)}.
}$$ 
Alors $F_1$ et $F_2$ se restreignent en des foncteurs $\Formes{X_0} \to \Formes{F_1(X_0)}=\Formes{F_2(X_0)}$, munis d'un isomorphisme de foncteurs canoniquement induit par $\phi$.
\end{prop}
\begin{proof}
Par la proposition \ref{tordusformes_prop}, il suffit de considérer les objets de la forme $P \contr{\faisAut_{X_0}} X_0$. Les foncteurs se restreignent par le point \ref{fonctobjet_item} de la proposition \ref{foncttors_prop} et l'isomorphisme de foncteurs est donné par 
\begin{multline*}
F_1(P \contr{\faisAut_{X_0}} X_0) \isoto P \contr{\faisAut_{F_1(X_0)}} F_1(X_0) \isoto P \contr{\faisAut_{F_2(X_0)}} F_2(X_0) \\
\isoto F_2(P \contr{\faisAut_{X_0}} X_0).\qedhere
\end{multline*}
\end{proof}

Si $\struc=\alg$ et si $A$ est un faisceau en $\faisO_S$-algèbres, on peut considérer son faisceau des éléments inversibles pour la multiplication, noté $\faisGL_{1,A}$, et qui est naturellement un faisceau en groupes. Soit maintenant $H$ un faisceau en groupes qui agit sur $A$ en respectant sa structure d'algèbre, il agit donc sur $\faisGL_{1,A}$ en respectant la structure $\gr$. Soit $P$ un torseur sous $H$. Dans le même esprit que la proposition \ref{auttordus_prop} et avec une preuve analogue, on obtient:
\begin{prop} \label{GLtordu_prop}
Le morphisme $P \contr{H} \faisGL_{1,A} \to \faisGL_{1,P \contr{H} A}$ est un isomorphisme.
\end{prop}

Si $G$ est un faisceau en groupes, on définit son centre $\faisCentre_G$ par 
$$\faisCentre_G(T)=\{g \in G(T),\ gg'=g'g,\ \forall T'\to T\text{ et } g' \in G(T')\}.$$
On vérifie facilement que c'est bien un faisceau. Pour tout faisceau en groupes $G$, on a des morphismes canoniques de faisceaux $G \to G/\faisCentre_G \to \faisAut^\gr_G$.

\begin{defi} \label{intfortint_defi}
Si $G_0$ est un faisceau en groupes, on dit que $G$ est une forme \emph{intérieure} (resp. \emph{strictement intérieure}) de $G_0$ s'il est obtenu en tordant $G_0$ par un torseur provenant d'un torseur sous $G_0/\faisCentre_{G_0}$ (resp. sous $G_0$). Si $G_0$ est semi-simple, on dit que $G$ est une forme \emph{fortement intérieure} si elle est obtenue en tordant $G_0$ par un torseur sous $(G_0)_{\simco}$, le revêtement simplement connexe de $G_0$.
\end{defi}
En particulier, lorsque $G_0$ est semi-simple et adjoint, il est de centre trivial, et ses formes intérieures sont strictement intérieures. Si au contraire il est simplement connexe, alors ses formes strictement intérieures sont fortement intérieures.

Soit $G'$ un sous-faisceau en groupes normal d'un faisceau en groupes $G$, et soit $H$ un faisceau en groupes agissant sur $G$ et dont l'action préserve $G'$. L'action de $H$ sur $G$ se restreint à $G'$ et descend au quotient $G/G'$ en rendant les morphismes $G' \to G$ et $G \to G/G'$ équivariants.

\begin{prop} \label{torsionsec_prop}
Soit $P$ un torseur sous un faisceau en groupes $H$ agissant sur $G$ en préservant $G'$, par exemple si $H$ agit par automorphismes intérieurs de $G$.
La suite exacte de faisceaux en groupes
$$1 \to G' \to G \to G/G'\to 1$$
se tord en une suite exacte de faisceaux en groupes
$$1 \to P \contr{H} G' \to P \contr{H} G \to P \contr{H} (G/G')  \to 1.$$
\end{prop}
\begin{proof}
Il est évident que les groupes et les morphismes se tordent, et la suite reste exacte: utiliser que l'application naturelle du tordu d'un noyau vers le noyau des tordus est un isomorphisme puisque c'en est un localement après trivialisation de $P$, de même pour les images.
\end{proof}

\begin{rema} \label{secautext_rema}
La proposition qui précède s'applique en particulier au cas où $G$ est le groupe des automorphismes de $G'$ et où $G'$ s'inclut dedans comme les automorphismes intérieurs. On notera qu'il agit alors sur $G'$ par un automorphisme, et sur lui-même par automorphismes intérieurs, et que l'inclusion de $G'$ dans $G$ est compatible à ces deux actions. En d'autres termes, on peut voir un automorphisme extérieur de $G'$ comme un automorphisme intérieur, mais sur le groupe plus gros $G$.
\end{rema}

\begin{exem}
Supposons qu'on a une suite exacte de faisceaux en groupes
$$1 \to G \to H \oto{\phi} E \to 1$$
muni d'une section $s: E \to H$ (ce qui est équivalent à se donner un isomorphisme $G \rtimes E \simeq H$). Faisons agir $E$ sur $G$ à travers $s$ par conjugaison (dans $H$) et supposons que $Q$ est un sous-faisceau en groupes de $G$ préservé par cette action de $E$. Considérons l'action de $H$ sur lui-même donnée par $(h,g)\mapsto h\cdot g \cdot (s\circ \phi(h^{-1}))$. Alors
\begin{itemize}
\item Cette action préserve $G$ et les orbites de $G$ sous $Q$ (agissant à droite). Elle induit donc une action de $H$ sur $G/Q$. 
\item Faisons agir $H$ sur $G$ par conjugaison sur $G$. Avec l'action précédente, cela fournit une action de $H$ sur $(G,G/Q)$. Cette action est compatible à la structure d'action de groupe $(G,G/Q)$.
\end{itemize}
Pour tout $H$-torseur, le tordu $(P \contr{H} G, P \contr{H} G/Q)$ est donc bien toujours muni d'une structure d'action de groupe. 

Plus généralement, on peut tordre de même en utilisant un torseur sous tout groupe $H'$ muni d'un morphisme $f:H' \to H$. 

Cet exemple s'applique en particulier lorsque $G$ est un $S$-groupe semi-simple adjoint, que $H=\faisAut_G$ et que $Q$ est un sous-groupe parabolique standard 
défini par un sous-ensemble du diagramme de Dynkin préservé par tout automorphisme de celui-ci. Ces paraboliques sont dits adaptés à un épinglage dans la terminologie de \SGAtrois, voir \cite[Exp. XXVI, déf. 1.11]{sga3}. Lorsqu'on utilise $H'=E$ et $f=s$, on dit qu'on est dans la situation \emph{quasi-déployée}.
\end{exem}

\subsubsection{Suites exactes de cohomologie}

En présence d'une suite exacte courte de $\Ttau$-faisceaux en groupes 
$$1 \to G_1 \to G_2 \to G_3 \to 1,$$ 
on obtient une suite exacte longue en cohomologie
\begin{equation} \label{selcohomologie_eq}
1 \to G_1(S) \to G_2(S) \to G_3(S) \to \Htau^1(S,G_1) \to \Htau^1(S,G_2) \to \Htau^1(S,G_3) 
\end{equation}
où $\Htau^1(S,G)$ désigne l'ensemble des classes d'isomorphismes de $S$-torseurs sous $G$, pointé par $G$ lui-même. Les morphismes entre les points des groupes sont les morphismes évidents, les morphismes entre $\Htau^1$ sont induits par la fonctorialité \ref{fonctorialitegroupe_prop} et le morphisme $G_3(S)\to \Htau^1(S,G_1)$ envoie un point $s$ de $G_3(S)$ sur la fibre image réciproque de $s$ dans $G_2$, qui est naturellement un $G_1$-torseur. L'exactitude en $G_1(S)$ et $G_2(S)$ est en tant que groupes, alors que c'est en tant qu'ensembles pointés pour le reste de la suite. Voir \cite[Ch. III, prop. 3.3.1]{gir} pour la preuve de ces faits. 

Dans certains cas, la suite exacte longue se prolonge à un ensemble $\Htau^2$, défini à l'aide de gerbes. Survolons donc la définition et les quelques propriétés des gerbes qui nous serviront à définir et à utiliser l'ensemble $\Htau^2(S,A)$ d'un schéma en groupes abéliens $A$ sur $S$. La référence de base \cite{gir} traite ces questions dans un grand degré de généralité. Nous avons au contraire cherché à traduire autant que possible cette généralité dans les cas particuliers qui nous intéressent, ce qui nous conduit à énoncer des définitions dans un cadre beaucoup plus restreint, mais probablement plus familier au lecteur peu au fait de la cohomologie non abélienne.

\begin{defi}
Une \emph{$S$-gerbe} est un $S$-champ $\cG$ pour lequel toutes les catégories $\cG_T$ sont des groupoïdes, et qui est: 
\begin{itemize}
\item localement non vide: pour tout $T$ sur $S$, il existe un recouvrement $(T_i)_{i\in I}$ de $T$ tel que pour tout $i$, la catégorie $\cG_{T_i}$ a au moins un objet (il suffit de l'imposer pour $T=S$);
\item localement connexe: pour tout couple d'objets $x$ et $y$ de $\cG_T$, il existe un recouvrement $(T_i)_{i\in I}$ de $T$ tel que $\setHom_{\cG_{T_i}}(x_{T_i},y_{T_i})$ est non vide pour tout $i$.
\end{itemize}
Un morphisme de $S$-gerbes est un morphisme de champs, donc un foncteur de catégories fibrées. 
\end{defi}

Notons que ces conditions sont conservées par restriction de la catégorie fibrée de base des $S$-schémas aux $T$-schémas, ce qui définit un foncteur de changement de base des $S$-gerbes vers les $T$-gerbes.

\begin{exem}
Soit $G$ un faisceau en groupes sur $S$. Le champ des $G$-torseurs $\Tors{G}$ est une gerbe. En effet, la première condition est triviale, puisque le torseur trivial $G$ est présent dans toute fibre. La deuxième condition provient de la proposition \ref{torseurloc_prop}. 
\end{exem}

\begin{defi}
Une $S$-gerbe isomorphe à la gerbe $\Tors{G}$ d'un $S$-faisceaux en groupes $G$ est dite \emph{neutre}.
\end{defi}

\begin{prop}
Toute gerbe est localement neutre.
\end{prop}
\begin{proof}
Localement, une gerbe $\cG$ est non vide et on prend pour $G$ le groupe des automorphismes d'un objet $x$. Remarquons que les objets de $\cG_T$ sont des formes de $x \cG_T$, par connexité locale de la gerbe, et que toutes les formes de $x$ sont bien des objets de $\cG_T$, à isomorphisme près, par descente. On utilise alors une équivalence de catégories analogue à celle de la proposition \ref{tordusformes_prop} entre $\cG_T$ et les et les $\faisAut_x$-torseurs. 
\end{proof}
Remarquons que le groupe $G$ ci-dessus n'est pas unique. Il dépend du choix de $x$: les formes d'un groupe ne lui sont pas toutes isomorphes.

\begin{defi} \label{gerbeliée_defi}
Une $S$-gerbe \emph{liée\footnote{La notion de gerbe liée est définie de manière beaucoup plus générale dans \cite{gir} en utilisant la notion de lien, mais nous nous contenterons de cette situation simplifiée, qui correspond chez Giraud au cas où le lien est représenté par un faisceaux en groupes abéliens $A$.}
par un faisceau en groupes abéliens $A$} est une $S$-gerbe $\cG$ munie, pour tout objet $x \in \cG_T$, d'un isomorphisme de $T$-faisceaux $\iota_x: A_T \to \faisAut_x$. Ces isomorphismes doivent commuter aux changements de base dans $\cG$ et à tout morphisme $\faisAut_x \to \faisAut_y$ induit par un isomorphisme $x \to y$. 
Un morphisme $(\cG,\iota) \to (\cG',\iota')$ de gerbes liées par $A$ est un morphisme de gerbes, donc un foncteur de $S$-catégories fibrées $F$, tel que pour tout $x$ le morphisme $\phi_{F,x}:\faisAut_x \to \faisAut_{F(x)}$ induit par $F$ vérifie $\phi_{F,x} \circ \iota_x = \iota'_{F(x)}$. 
\end{defi}

Il suit de la définition que pour deux objets $x,y \in \cG_T$, le faisceau $\faisHom_{x,y}=\faisIso_{x,y}$ est un $A_T$-torseur à travers $\iota_x$.

\begin{exem}
La gerbe $\Tors{A}$ est naturellement liée par $A$, car pour tout $A$-torseur $P$, on a un isomorphisme canonique $A \isoto \faisAut_P$. Bien entendu, le fait que $A$ est abélien est essentiel; cela implique qu'il ne se tord pas par automorphismes intérieurs. Dans le cas contraire, $\faisAut_x$ serait un tordu de $A$, par \ref{auttordus_prop}, non nécessairement isomorphe à $A$.
\end{exem}

\begin{lemm}
Tout morphisme de gerbes liées par $A$ est une équivalence de catégories. 
\end{lemm}
\begin{proof}
La pleine fidélité vient du fait que le foncteur $F$ induit un morphisme de $A$-torseurs entre $\faisHom_{x,y}=\faisAut_{x,y}$ et $\faisHom_{F_T(x),F_T(y)}=\faisAut_{F_T(x),F_T(y)}$. Or tout morphisme de $A$-torseurs est un isomorphisme, voir la proposition \ref{flecheIsoTors_prop}.
Le caractère essentiellement surjectif de $F$ se déduit alors de la condition de descente.
\end{proof}

\begin{lemm}
Une gerbe $(\cG,\iota)$ liée par $A$ est équivalente par un morphisme de gerbes liées par $A$ à la gerbe $\Tors{A}$ si et seulement si $\cG_S \neq \emptyset$.
\end{lemm}
\begin{proof}
Puisque $\Tors{A}_S$ contient l'objet $A$, un sens est évident. Pour l'autre, considérons $x \in \cG_S$. Alors le foncteur fibré $F$ donné par $F_T: y \mapsto \faisIso_{x_T,y}$ est bien un morphisme de gerbes liées par $A$.
\end{proof}

\begin{defi}
L'ensemble pointé $\Htau^2(S,A)$ est l'ensemble des classes d'équivalences (de catégories fibrées) de gerbes liées par $A$, pointé par la classe de la gerbe $\Tors{A}$.
\end{defi}

\begin{exem} \label{deuxbord_exem}
Soit $B \to C$ un morphisme de $S$-faisceaux en groupes de noyau abélien, et soit $Q$ un $C$-torseur. Un relèvement de $Q$ à $B$ est un couple $(P,p)$ où $P$ est un $B$-torseur, et $p$ est un morphisme de $C$-torseurs $P\contr{B} C \isoto Q$. Un morphisme de relèvements $(P_1,p_1) \to (P_2,p_2)$ est un morphisme de torseurs $\phi:P_1\to P_2$ tel que $p_2 \circ (\phi\contr{B} C) =p_1$. Le champ en groupoïdes dont la fibre sur $T$ est la catégorie des relèvements de $Q_T$ à $B_T$ est une gerbe liée par le noyau de $B \to C$. On la note $\deuxbord(P)$. 
\end{exem}

En présence d'une suite exacte de $S$-faisceaux en groupes
$$1 \to A \to B \to C\to 1$$
où $A$ est central dans $B$ (condition portant sur les points), on définit un foncteur de la catégorie des $C$-torseurs vers celle des gerbes liées par $A$, en envoyant  un $C$-torseur $P$ sur la gerbe $\deuxbord(P)$ des relèvements de $P$ à $B$.

\begin{defi} \label{deuxbord_defi}
Le foncteur précédent induit un morphisme d'ensembles pointés $\deuxbord: \Htau^1(G,C) \to \Htau^2(G,A)$ appelé \emph{second morphisme de bord}. 
\end{defi}

\begin{prop}
Si $A$ est central dans $B$, la suite exacte longue \eqref{selcohomologie_eq} associée à la suite exacte $1 \to A \to B \to C \to 1$ prolongée par le second morphisme de bord est encore exacte en $\Htau^1(C)$ en tant que suite d'ensembles pointés.
\end{prop}
\begin{proof}
C'est presque tautologique: Un torseur $P$ sur $C$ provient de $B$ si et seulement si $\deuxbord(P)_S$ est non vide, si et seulement si la gerbe $\deuxbord(P)$ est neutre.
\end{proof}
On vérifie immédiatement que cette construction de deuxième cobord est fonctorielle en la suite exacte courte.
\medskip

Soient $\Ttau_1$ et $\Ttau_2$ deux topologies de Grothendieck, avec $\Ttau_2$ plus fine que $\Ttau_1$, \ie\ tout $\Ttau_1$-recouvrement est un $\Ttau_2$ recouvrement. 

\begin{exo}
Soit $F$ un $\Ttau_1$-faisceau qui est $\Ttau_1$-localement un faisceau pour la topologie $\Ttau_2$. Alors $F$ est un $\Ttau_2$-faisceau.
\end{exo}
Il suit de cet exercice que si $G$ est un $\Ttau_2$-faisceau en groupes, alors tout $\Ttau_1$-torseur sous $G$ est bien un $\Ttau_2$-faisceau, puisque localement isomorphe à $G$, et donc un $\Ttau_2$-torseur. On obtient ainsi une application
$$\Hcoh^1_{\Ttau_1}(S,G) \to \Hcoh^1_{\Ttau_2}(S,G)$$ 
qui est injective, puisque l'existence d'un isomorphisme entre deux torseurs ne dépend pas de la topologie considérée. 
C'est en particulier le cas si $G$ est un foncteur de points représentable, et si $\Ttau_2$ est moins fine que la topologie \emph{canonique} (la topologie la plus fine pour laquelle tout foncteur représentable est un faisceau). S'il est de plus affine sur $S$, et si la topologie $\Ttau_1$ est plus fine que la topologie de Zariski, tout torseur sous $G$ est alors également représentable, par la proposition \ref{reprLocal_prop} ci-après. 

Le même type de raisonnement est valable pour les gerbes: si $A$ est un $\Ttau_2$-faisceau en groupes abéliens, les gerbes pour la topologie $\Ttau_1$ liées par $A$ vu comme $\Ttau_1$-faisceau sont également des gerbes pour la topologie $\Ttau_2$ liées par $A$ vu comme $\Ttau_2$-faisceau. L'application de changement de topologie
$$\Hcoh^2_{\Ttau_1}(S,A) \to \Hcoh^2_{\Ttau_2}(S,A)$$
envoie la classe d'une gerbe sur elle-même, et est injective car une gerbe $\cG$ est triviale si et seulement si $\cG_S\neq \emptyset$, ce qui ne dépend pas non plus de la topologie considérée.  

Pour les groupes réductifs sur une base $S$, on a alors les résultats suivants.

\begin{theo} \label{reductifetale_theo}
Tout groupe réductif sur une base $S$ est localement déployé pour la topologie étale. 
\end{theo}
\begin{proof}
Voir \cite[Exp. XXII, cor. 2.3]{sga3}.
\end{proof}

\begin{theo} \label{lisseetalefppf_theo}
Pour tout groupe algébrique $G$ lisse sur $S$ (ex: $G$ réductif), l'application $\Het^1(S,G) \to \Hfppf^1(S,G)$ est un isomorphisme. Il en est de même pour l'application $\Het^2(S,A) \to \Hfppf^2(S,A)$ si $A$ est abélien.
\end{theo}
\begin{proof}
Au niveau des $\mathrm{H}^1$, l'application est injective, voir ci-dessus. Il reste donc à voir qu'un pseudo-torseur déployé par un recouvrement \fppf\ est en fait déjà déployé par un recouvrement étale. C'est vrai lorsque la base $S$ est locale hensélienne par \cite[Exp. XXIV, prop. 81]{sga3}, puis pour tout anneau local et enfin pour un voisinage Zariski de tout point de $S$ par deux arguments de limite successifs.
Pour les $\mathrm{H}^2$, faute d'argument élémentaire, référons à \cite[11.7]{brauer3} pour la comparaison des groupes de cohomologie étale usuels, et à \cite[Ch. IV, §3.5]{gir} pour comparer ces groupes avec le $\mathrm{H}^2$ défini en terme de gerbes.
\end{proof}

\subsection{Représentabilité} \label{representabilite_sec}

Rappelons quelques résultats de représentabilité dont nous nous servirons constamment. 
\begin{prop}
La topologie \fpqc\ est moins fine que la topologie canonique.
\end{prop}
\begin{proof}
Voir \cite[Exp. IV, proposition 6.3.1]{sga3}.
\end{proof}
En d'autres termes, tout $S$-foncteur de points qui est représentable est un faisceau pour la topologie \fpqc, et donc également pour les topologies \fppf, étales et Zariski. 

\begin{prop} \label{reprLocal_prop}
Un faisceau pour la topologie de Zariski qui est représentable localement sur la base $S$ est représentable. Un faisceau pour la topologie étale ou \fppf\ qui est représentable localement par des schémas affines sur $S$ est représentable (par un schéma affine sur $S$).
\end{prop}
\begin{proof}
Voir \cite[Exp. VIII, lemme 1.7.2]{sga3}.
\end{proof}

Cette proposition nous sera souvent utile pour nous ramener au cas d'une base $S$ affine. Nous aurons également besoin du résultat suivant.

\begin{lemm}\label{produit_fibre_lemm}
Soient $F,G,H$ des $S$-foncteurs en groupes représentables par les schémas $X$, $Y$ et $Z$ respectivement. Soient encore $f:F\to H$ et $g:G\to H$ des morphismes de foncteurs en groupes. Alors le foncteur $F\times_H G$ défini par $(F\times_H G)(T):=F(T)\times_{H(T)}G(T)$ est également représentable par le schéma $X \times_Z Y$; il est muni naturellement d'une structure de schéma en groupes et il satisfait la propriété universelle du produit fibré de $F$ et $G$ au-dessus de $H$ dans la catégorie des schémas en groupes. Si $X$, $Y$ et $Z$ sont affines sur $S$, alors il en est de même de $X \times_Z Y$.
\end{lemm}
\begin{proof}
Voir \cite[I, (3.2.6) et (3.4.3.2)]{ega} pour l'existence du produit. Le reste est laissé au lecteur.
\end{proof}

\begin{rema} \label{imageinvnoyau_rema}
Un cas particulier du lemme précédent est le cas où $g:G \to H$ est l'inclusion d'un sous-groupe. On obtient alors le schéma en groupes \emph{image inverse de $G$ par $f$}. Si $G$ est le groupe trivial, on obtient le schéma noyau de $f$, noté $\ker(f)$.
\end{rema}

\begin{prop} \label{reprschemas_prop}
Si $T \to S$ est un morphisme fidèlement plat et quasi-compact et si $X_0$ est un $S$-schéma, alors toute forme de $X_0$ déployée sur $T$ est encore un $S$-schéma si $X_0$ est affine sur $S$.
\end{prop}
\begin{proof}
Voir \cite[Exp. VIII, Lemme 1.7.2 (ii)]{sga3} (toute forme est canoniquement munie d'une donnée de descente).
\end{proof}

\subsection{Quelques schémas fondamentaux} \label{schemasFond_sec}

Introduisons maintenant certains schémas qui nous serviront constamment par la suite. Rappelons que $\faisO_S$ est le $S$-schéma en anneaux tel que $\faisO_S(T)=\Gamma(T,\cO_{T})=\Gamma(T)$ et donc $\faisO_S(\Spec(R))=R$ pour tout schéma affine $\Spec(R)$ au dessus de $S$, avec addition et multiplication données par celles de $R$ (point \ref{anneaustruc_item} des exemples \ref{structures_exem}).  

\subsubsection{$\faisO_S$-modules localement libres de type fini}

Si $\cF$ est un $\cO_S$-module, et $T\to S$ un morphisme on note $\cF_{T}$ le $\cO_{T}$ module $\cF\otimes_{\cO_S}\cO_{T}$. Rappelons que le foncteur $\wW$ des $\cO_S$-modules vers les $\faisO_S$-modules est défini par
$$\wW(\cF)(T)=\Gamma(T,\cF_{T}).$$
\begin{prop} \label{Wprop_prop}
Le foncteur $\wW$ possède les propriétés suivantes:
\begin{enumerate}
\item \label{Wpf_item} Il est pleinement fidèle.
\item \label{Wchmtbase_item} Il est compatible au changement de base: $\wW(\cF\otimes_{\cO_S} \cO_{T})=\wW(\cF)_{T}$.
\item \label{Wloclib_item} Il envoie les faisceaux localement libres de type fini au sens de Zariski (resp. étale) vers les faisceaux localement libres de type fini pour la topologie de Zariski (resp. étale) sur $S$.
\end{enumerate}
\end{prop}
\begin{proof}
Pour les points \ref{Wpf_item} et \ref{Wchmtbase_item}, voir \cite[Exp. I, prop. 4.6.2]{sga3}. Le point \ref{Wloclib_item} est une conséquence immédiate du point \ref{Wchmtbase_item}.
\end{proof}

\begin{prop} \label{loclibre_prop}
Soit $M$ un foncteur de points en $\faisO_S$-modules. Les conditions suivantes sont équivalentes.
\begin{enumerate}
\item \label{zarzar_item} $M$ est un faisceau Zariski et il est Zariski-localement isomorphe à un $\faisO_S$-module libre de type fini. 
\item \label{Wzarlib_item} $M\simeq \wW(\cM)$ où $\cM$ est localement libre de type fini. 
\item \label{etzar_item} $M$ est un faisceau étale, et il est Zariski-localement isomorphe à un $\faisO_S$-module libre de type fini.
\item \label{etet_item} $M$ est un faisceau étale, et il est étale-localement isomorphe à un $\faisO_S$-module libre de type fini.
\item \label{Wetlib_item} $M\simeq \wW(\cM)$ où $\cM$ est étale-localement libre de type fini. 
\end{enumerate}
\end{prop}
\begin{proof}
Les implications \ref{etzar_item} $\implies$ \ref{zarzar_item} et \ref{etzar_item} $\implies$ \ref{etet_item} sont évidentes, ainsi que \ref{Wzarlib_item} $\implies$ \ref{zarzar_item} et \ref{Wetlib_item} $\implies$ \ref{etet_item} car $\wW$ commute au changement de base. Pour \ref{zarzar_item} $\implies$ \ref{Wzarlib_item}, il suffit de voir qu'un tel faisceau Zariski est représentable par la proposition \ref{reprLocal_prop} et parce que $\faisO_S^n$ l'est, et la structure de module est transportée car $\wW$ est pleinement fidèle. 
De même, \ref{etet_item} $\implies$ \ref{Wetlib_item} par la même proposition, car les schémas représentant étale-localement $M$ sont affines sur $S$. 
Toujours par représentabilité, on a \ref{Wzarlib_item} $\implies$ \ref{etzar_item} (tout foncteur représentable est un faisceau étale). 
Enfin, \ref{Wetlib_item} $\implies$ \ref{Wzarlib_item} par descente sur $\cM$ (voir \cite[IV 2.5.2]{ega}).
\end{proof}

\begin{defi} \label{loclibre_defi}
Nous appellerons simplement $\faisO_S$-module localement libre de type fini un foncteur de points en $\faisO_S$-modules qui vérifie les conditions équivalentes de la proposition précédente. Son rang est alors défini comme on le devine, et il est localement constant et fini.
\end{defi}
La proposition \ref{loclibre_prop} dit alors en particulier:
\begin{coro} \label{Wequivloclibres_coro}
Le foncteur $\wW$ définit une équivalence de catégories des $\cO_S$-modules localement libres de type fini vers les $\faisO_S$-modules localement libres de type fini.
\end{coro}

Si $S=\Spec(R)$, on a $\faisO_S(S)=R$ et donc la suite de foncteurs suivants 
$$\xymatrix{
R\text{-mod} \ar[r]^-{\tilde{(-)}}_-{\simeq} & \cO_S\text{-mod quasi-cohérents} \ar[r]^-{\wW} & \faisO_S-\text{mod}
}$$
où le premier est l'équivalence de catégories bien connue, et le second est pleinement fidèle. Il s'ensuit que si $R'$ est une extension de $R$ et qu'on pose $T=\Spec(R')$, on a alors pour tous modules $M$ et $N$ sur $R$, on a
\begin{equation} \label{Waffine_eq}
\begin{split}
\faisHom^{\faisO_S\mathrm{-mod}}_{\wW(\tilde{M}),\wW(\tilde{N})}(T) & = \Hom_{\faisO_{T}\mathrm{-mod}}(\wW(\tilde{M})_{T},\wW(\tilde{N})_{T}) \\
& = \Hom_{\faisO_{T}\mathrm{-mod}}(\wW(\tilde{M}_{T}),\wW(\tilde{N}_{T})) \\
& = \Hom_{\cO_{T}\mathrm{-mod}}(\tilde{M},\tilde{N}) \\ 
& = \Hom_{R'\mathrm{-mod}}(M_{R'},N_{R'})
\end{split}
\end{equation}
où le premier terme est défini en \ref{isoaut_nota}. Il en est de même pour $\faisEnd$, $\faisIso$ et $\faisAut$.

\begin{rema} \label{Walg_rema}
Le foncteur $\wW$ étant additif, il permet de transporter une structure de $\cO_S$-algèbre sur un module $\cA$ en une structure de $\faisO_S$-algèbre sur $\wW(\cA)$. Réciproquement, comme $\wW$ est pleinement fidèle, si un $\faisO_S$-module $\wW(\cA)$ est muni d'une structure de $\faisO_S$-algèbre, elle provient d'une unique structure de $\cO_S$-algèbre sur $\cA$. Il en est de même pour le foncteur $\tilde{(-)}$ dans le cas affine.
\end{rema}

\begin{lemm} \label{homMN_lemm}
La flèche naturelle 
$$\Hom_{\faisO_S\mathrm{-mod}}(M,N)\to \Hom_{\faisO_S(S)\mathrm{-mod}}(M(S),N(S))$$ 
est un isomorphisme lorsque $M$ et $N$ sont libres ou lorsqu'ils sont localement libres de type fini et $S$ est affine.
\end{lemm}
\begin{proof}
Lorsque $M$ et $N$ sont libres, on utilise $\Hom_{\faisO_S\mathrm{-mod}}(\faisO_S,\faisO_S)=\faisO_S(S)=\Gamma(S,\cO_S)$. Dans le cas affine, on utilise la suite d'égalités \eqref{Waffine_eq}.
\end{proof}

On pose que $\faisM_{n,S}$ est la $\faisO_S$-algèbre $\faisEnd^{\faisO_S\mathrm{-mod}}_{\faisO_S^n}$ et par ce qui précède, on a bien
$$\faisM_{n,S}(T) = \setM_n(\Gamma(T)).$$
Ce foncteur de points est représentable par un schéma affine sur $S$ comme un produit de $n^2$ copies de $\faisO$, et les morphismes définissant sa structure de $\faisO$-algèbre sont ceux qu'on imagine. Plus généralement, lorsque $M$ est un $\faisO_S$-module localement libre de type fini, le foncteur de points $\faisEnd^{\faisO_S\mathrm{-mod}}_{M}$ est représentable par un schéma affine sur $S$ par \ref{reprLocal_prop}, car c'est un faisceau Zariski puisque $M$ en est un, et il est localement isomorphe à $\faisM_{n,S}$.
\begin{defi} \label{matrices_defi}
Pour un $\faisO_S$-module localement libre de type fini $M$, on note $\faisEnd_{M}$ le schéma en $\faisO_S$-algèbres qui vérifie 
$$\faisEnd_{M}(T)=\setEnd_{\faisO_{T}\mathrm{-mod}}(M_{T})$$ 
pour tout $T\to S$.
En particulier, si $M=\faisO_S^n$, on a $\faisEnd_{M}=\faisM_{n,S}$ avec
$$\faisM_{n,S}(T)=\setM_n(\Gamma(T))$$
pour tout $T\to S$.
On supprime $S$ de toutes ces notations lorsqu'il n'y a pas d'ambiguïté. 
\end{defi}

Considérons la $S$-catégorie fibrée $\Vect$ (resp. $\Vecn{n}$) dont la fibre en $T$ est la catégorie des $\faisO_T$-modules localement libres de type fini (resp. de rang constant $n$). 
\begin{prop} \label{VecChamp_prop}
Les catégories fibrées $\Vect$ et $\Vecn{n}$ sont des champs (étales ou \fppf).
\end{prop}
\begin{proof}
Sans la restriction ``localement libre de type fini'', la catégorie fibrée obtenue est un champ par la proposition \ref{champstruc_prop} appliquée à l'exemple \ref{structures_exem} \eqref{OSmod_item}. Il suffit alors de voir que la propriété d'être localement libre de type fini descend ainsi que d'être de rang constant $n$, ce qui est immédiat. 
\end{proof}

\subsubsection{Groupe linéaire}

\begin{prop} \label{GLrepr_prop}
Soit $A$ une $\faisO_S$-algèbre unitaire et associative (non nécessairement commutative) qui est localement libre de type fini en tant que $\faisO_S$-module. Considérons le foncteur en groupes ``éléments inversibles''
$$T\mapsto A(T)^\times.$$
Ce foncteur est représentable par un schéma affine sur $S$. 
\end{prop}
\begin{proof}
Il est facile de voir que ce foncteur est un faisceau pour la topologie de Zariski, en utilisant que $A$ en est un. Par \ref{reprLocal_prop}, on peut donc supposer que $S=\Spec(R)$, et par la remarque \ref{Walg_rema} que $A=\wW(\tilde{B})$ où $B$ est une $R$-algèbre libre de rang fini comme $R$-module. On peut donc définir une norme $N:B \to R$ en associant à un élément le déterminant de la multiplication à gauche par cet élément. Le foncteur de points de l'énoncé est alors représentable par le schéma affine $\Spec(\Symalg(B^\dual)[N^{-1}])$ où $\Symalg(B^\dual)$ est l'algèbre symétrique sur $B^\dual=\Hom_R(B,R)$ et $N$ est l'élément de $\Symalg^n(B^\vee)$ correspondant à la norme définie ci-dessus. En effet
$$\Hom_{R-\text{alg}}(\Symalg(B^\dual),R')=\Hom_{R-\text{mod}}(B^\dual,R')=\Hom_{R-\text{mod}}(B^\dual,R)\otimes_R R'= B_{R'}$$
où on a utilisé le fait que $B$ est libre pour les deux dernières égalités. Lorsqu'on remplace $\Symalg(B^\dual)$ par $\Symalg(B^\dual)[N^{-1}]$ à gauche, on obtient le sous-ensemble des morphismes qui envoie $N$ dans $R^\times$, et donc à droite $B_{R'}^\times$.
\footnote{Si la multiplication à gauche par $a$ est inversible, il a un inverse à droite, donc la multiplication à droite par $a$ a un déterminant inversible, et est donc inversible, donc $a$ a un inverse à gauche.}
\end{proof}

\begin{defi} \label{GLn_defi}
Soit $\faisGL_{1,A}$ le schéma en groupes défini par la proposition précédente. Lorsque $A=\faisM_{n,S}$, on utilise la notation $\faisGL_{n,S}=\faisGL_{1,\faisM_{n,S}}$, et on supprime $S$ de la notation lorsqu'il n'y a pas d'ambiguïté. Enfin, lorsque $M$ est un $\faisO_S$-module localement libre de type fini, on utilise également la notation $\faisGL_{M}$ pour $\faisGL_{1,\faisEnd(M)}$. 
\end{defi}

\begin{defi} \label{Gm_defi}
Le schéma en groupes $\faisGmS{S}$, ou $\faisGm$, est par définition $\faisGL_{1,S}$.
\end{defi}
On a donc $\faisGm(T)=\Gamma(T)^\times$.

\begin{defi} \label{mun_defi}
Soit $\faismu_{n,S}$, ou $\faismu_n$, le noyau de l'application élévation à la puissance $n$ du schéma en groupes $\faisGmS{S}$ vers lui-même. 
\end{defi}
Pour tout schéma $T$ sur $S$, on a donc $\faismu_n(T)=\{u\in \Gamma(T)^\times \text{ t.q. } u^n=1\}$. De plus, il est facile de voir que sur une base affine $S=\Spec(R)$, le schéma en groupes $\faismu_n$ est représenté par $\Spec(R[x]/(x^n-1))$. Il est donc représentable par la proposition \ref{reprLocal_prop}.
C'est un cas particulier de groupe diagonalisable, voir \cite[Exp. VIII]{sga3} pour plus de détails, et il est lisse si et seulement si $n$ est premier aux caractéristiques résiduelles de $S$, voir \cite[Exp. VIII, prop. 2.1]{sga3}.

\subsubsection{Torseurs sous $\faisGL_n$, $\faisGm$ et $\faismu_n$.} \label{GLGmmunTors_sec}

Notons que par définition, $\faisGL_{n,S}=\faisAut_{\faisO_S^n}$ et $\faisGL_{M}=\faisAut_M$, où $\faisO_S^n$ et $M$ sont des objets du champ $\Vect$.

Les propositions \ref{tordusformes_prop} et \ref{loclibre_prop} et impliquent immédiatement: 
\begin{prop} \label{GLntors_prop}
Le foncteur $M \mapsto \faisIso_{\faisO_S^n,M}$ définit une équivalence de catégories fibrées 
$$(\Vecn{n})_{\grpd} \isoto \Tors{\faisGL_n}$$
du champ en groupoïdes des modules localement libres de rang constant $n$ vers le champ en groupoïdes des $\faisGL_n$-torseurs pour la topologie Zariski ou étale. En particulier, lorsque $n=1$, cela définit une équivalence de catégories fibrées 
$$(\Vecn{1})_\grpd \isoto \Tors{\faisGm}$$
des $\faisO_S$-modules inversibles vers les $\faisGm$-torseurs. 
\end{prop}

Notons $\Det{M}$ le déterminant du $\faisO_S$-module (localement libre de type fini) $M$. Sur une composante connexe de $S$, il s'agit du module puissance extérieure maximale $\Lambda^{r}(M)$ où $r$ est le rang (constant) de $M$, avec la convention que si $r=0$, on a $\Lambda^{0}(M)=\faisO_S$. Il s'agit donc toujours d'un $\faisO_S$-module localement libre de rang $1$. Cela définit, pour chaque rang, non nécessairement constant, un foncteur $\Det{}$ des $\faisO_S$-modules localement libres de ce rang vers les $\faisO_S$-modules localement libre de rang constant $1$. 

\begin{defi}\label{Detchamps_defi}
Soit $\Det{}$ le foncteur de catégories fibrées $\Det{}: \Vect \to \Vecn{1}$ défini sur les fibres comme ci-dessus.
\end{defi}

Le morphisme déterminant $\faisGL_n \to \faisGm$ de groupes algébriques est défini sur les points par la formule habituelle.
\begin{lemm} \label{torsDetFonct_lemm}
Le long de ce morphisme déterminant, le torseur $\faisIso_{\faisO_S^n,M}$ se pousse en le torseur $\faisIso_{\faisO_S,\Det{M}}$.
\end{lemm}
\begin{proof}
Par la proposition \ref{flecheIsoTors_prop}, il suffit de donner un morphisme $\faisIso^{\faisO_S-\text{mod}}_{\faisO_S^n,M}\contr{\faisGL_n} \faisGm \to \faisIso^{\faisO_S-\text{mod}}_{\faisO_S,\Det{M}}$. On vérifie immédiatement que sur les points, $(f,u) \mapsto \det(f)\cdot u$ est bien définie, où $\det(f):\faisO_S \to \Det{M}$ est la puissance extérieure $n$-ième de $f: \faisO_S^n \to M$, ce qui définit le morphisme par descente. 
\end{proof}

Lorsque le schéma $\faismu_{n,S}$ n'est pas lisse, i.e. lorsque $S$ a un point géométrique de caractéristique non première à $n$, ses torseurs pour la topologie étale ne coïncident pas avec ses torseurs pour la topologie \fppf. Ces derniers sont plus faciles à décrire, et suffisants pour les applications qui nous intéressent; en effet, $\faismu_n$ apparaît principalement dans des suite exactes courtes de groupes algébriques où les autres groupes sont lisses, et pour lesquels il est donc indifférent de considérer les torseurs étales ou \fppf. Décrivons donc les $\faismu_n$-torseurs pour la topologie \fppf.

Considérons le foncteur fibré $(\Vecn{1})_{\grpd}\to (\Vecn{1})_{\grpd}$ qui envoie un module de rang $1$ sur sa puissance $n$-ième, ainsi que le foncteur fibré $\Final \to \Vect{1}$ qui envoie le seul objet de la fibre sur $T$ vers $\faisO_T$. 
\begin{defi}
Le champ (\cf exercice \ref{prodfibchamps_exo}) des modules inversibles $n$-trivialisés $\nTriv{n}$ est le produit fibré de $(\Vecn{1})_{\grpd}\times_{\Vecn{1}} \Final$ à l'aide de ces deux foncteurs.
\end{defi}
Explicitement, un objet de $\nTriv{n}_T$ est donc donné par un $\faisO_T$-module localement libre $L$ de rang constant $1$ muni d'un isomorphisme $\phi: L^{\otimes n} \to \faisO_S$, et un morphisme de $\faisO_S$-modules trivialisés $(L_1,\phi_1)\to (L_2,\phi_2)$ est donc un morphisme $r: L_1 \to L_2$ tel que $\phi_2 \circ r^{\otimes n} = \phi_1$. Un tel morphisme est automatiquement un isomorphisme.

\begin{defi}
Le module $n$-trivialisé \emph{trivial} est le couple $(\faisO_S,m)$ où $m:\faisO_S^{\otimes n} \to \faisO_S$ est la multiplication. 
\end{defi}

\begin{lemm}
Pour toute topologie entre la topologie de Zariski et la topologie canonique telle que $\phi$ est localement une puissance $n$-ième, \ie $\phi:L^{\otimes n}\to \faisO_S$ factorise localement par $m:\faisO_S^{\otimes n} \to \faisO_S$, un module $n$-trivialisé $(L,\phi)$ est localement trivial, i.e. est une forme de $(\faisO_S,m)$, C'est toujours le cas pour la topologie \fppf. 
\end{lemm}
\begin{proof}
La seule chose qui n'est pas évidente est l'affirmation sur la topologie \fppf: on commence par se restreindre à un ouvert Zariski sur lequel $L$ est trivial et la base est affine et vaut $\Spec(R)$, et il faut alors montrer que tout élément $\lambda\in R$ devient une puissance $n$-ième après une extension \fppf. On prend l'extension évidente $R[x]/(x^n-\lambda)$. 
\end{proof}

On suppose maintenant que la topologie vérifie les conditions du lemme précédent.

Il est immédiat qu'on a $\faisAut^{\nTriv{n}}_{(\faisO_S,m)}=\faismu_n$ où les objets sont dans le champ $\nTriv{n}$. On considère le faisceau $\faisIso_{(\faisO_S,m),(L,\phi)}$, torseur sous le précédent. Explicitement: 
$$\faisIso_{(\faisO_S,m),(L,\phi)}(T) =\{r: \faisO_{T} \to L_{T},\ \phi_{T'} \circ r_{T'}^{\otimes n} = m_{T'},\ \forall T'\to T\}.$$ 
(Les $r$ considérés sont des morphismes de $\faisO_S$-modules; ils sont automatiquement inversibles.) L'action à droite de $\faismu_n=\faisAut_{(\faisO_S, m)}$ revient à la multiplication d'un morphisme par un scalaire. Remarquons au passage que le faisceau $\faisIso_{(\faisO_S,m),(L,\phi)}$ est représentable par un schéma affine sur $S$ car c'est la fibre au dessus de $\phi$ de l'application $(-)^{\otimes n}:\faisHom_{L,\faisO_S} \to \faisHom_{L^{\otimes n},\faisO_S}$ entre faisceaux d'ensembles représentables par des schémas affines sur $S$.
La proposition \ref{tordusformes_prop} donne alors immédiatement:
\begin{prop} \label{torseursfppfmun_prop}
Le foncteur $(L,\phi) \mapsto \faisIso_{(\faisO_S,m),(L,\phi)}$ est une équivalence de catégories fibrées entre champs en groupoïdes
$$\Formes{(\faisO_S,m)} \isoto \Tors{\faismu_n}.$$ 
En particulier, si la topologie est \fppf, tous les modules $n$-trivialisés sont des formes de $(\faisO_S,m)$ et on a 
$$\nTriv{n} \isoto \Tors{\faismu_n}.$$ 
\end{prop}
Enfin, par la proposition \ref{foncttors_prop}, on a:
\begin{lemm} \label{pousseisontriv_lemm}
Le torseur $\faisIso_{(\faisO_S,m),(L,\phi)}$ se pousse le long de $\faismu_n \to \faisGm$ en un torseur isomorphe à $\faisIso_{\faisO_S,L}$ par la fonctorialité définie en \ref{fonctorialitegroupe_prop}.
\end{lemm}
Remarquons que l'isomorphisme $\faisIso_{(\faisO_S,m),(L,\phi)}\contr{\faismu_n} \faisGm \isoto \faisIso_{\faisO_S,L}$ peut s'expliciter. Sur les points, il envoie un couple $(r,u)$, où $r$ est un isomorphisme tel que $\phi\circ r^{\otimes n}=m$ et $u$ un point de $\faisGm$, donc un automorphisme de $\faisO_S$, vers l'élément $r \circ u$. 

\subsubsection{Groupe projectif linéaire}

\begin{prop} \label{autrepr_prop}
Soit $A$ une $\faisO_S$-algèbre localement libre de type fini.
Le foncteur en groupes $\faisAut^\alg_{A}$ est représentable par un schéma en groupes affine sur $S$ et de type fini, qui est un sous-groupe fermé de $\faisGL_{A}$ ($A$ est vu comme $\faisO_S$-module, ici).
\end{prop}
\begin{proof}
C'est un faisceau Zariski puisque $A$ en est un. On se ramène donc au cas $S=\Spec(R)$ affine et $A$ libre comme $\faisO_S$-module, avec $A=\wW(\tilde{B})$.
Considérons alors le $R$-module $M=\Hom_R(B \otimes_R B, B)$ et l'élément $m$ qui y représente la multiplication de $B$. Le stabilisateur de $m$ est représentable par un sous-schéma en groupes fermé de $\faisGL_{\wW(\tilde{M})}$, par \cite[Exp. I, \S 6.7]{sga3}. Définissons un morphisme $\faisGL_{A} \to \faisGL_{\wW(\tilde{M})}$ en envoyant un élément $\alpha$ des points de $\faisGL_{A}(T)$ sur l'endomorphisme de $M_T$ qui à un morphisme $f:B_T \otimes B_T \to B_T$ associe le morphisme $\alpha \circ f \circ (\alpha^{-1} \otimes \alpha^{-1})$. L'image réciproque du stabilisateur de $m$ est le foncteur en groupes de l'énoncé, et il est représentable par la remarque \ref{imageinvnoyau_rema}.
\end{proof}
En fait, dans cette proposition, le fait que $A$ soit associative ou unitaire n'est pas utilisé.

\begin{defi} \label{PGLA_defi}
On note $\faisPGL_A$ le foncteur de points représentable de la proposition précédente.
\end{defi}
Pour tout schéma $T$ au dessus de $S$, on a bien entendu $(\faisPGL_A)_{T} = \faisPGL_{A_{T}}$ par la remarque \ref{Homchgmt_rema}.

\begin{prop} \label{PGLApoints_prop}
Si $T$ est affine et égal à $\Spec(R')$, alors
$$\faisPGL_A (T)= \setAut_{R'\mathrm{-alg}}(A(T)).$$
\end{prop}
\begin{proof}
Lorsque $T$ est affine, alors $A_{T}=\wW(\tilde{A(T)})$ puisque $A_{T}$ vérifie la condition \ref{Wzarlib_item} de la proposition \ref{loclibre_prop}. On conclue par pleine fidélité de $\wW$ et $\tilde{(-)}$.
\end{proof}

Nous aurons enfin besoin des deux foncteurs suivants.
\begin{prop}\label{stab_prop}
Soit $N$ un $\faisO_S$-module localement libre de type fini et $v\in N(S)$. Alors $\faisStab_v$, défini par
$$\faisStab_v(T)=\{\alpha \in \faisGL_N(T),\ \alpha(v_{T'})=v_{T'},\ \forall T'\to T\}$$ 
est un sous-groupe représentable et fermé de $\faisGL_N$.
\end{prop}
\begin{proof}
On vérifie aisément que ce foncteur est un faisceau Zariski. On peut donc se ramener au cas où $N$ est libre sur une base affine $S=\Spec(R)$ par la proposition \ref{reprLocal_prop}. Choisissons alors une base $e_1,\ldots,e_n$ de $N(R)$, qui permet d'identifier $\faisGL_N\simeq \faisGL_n$. Le terme de droite est représenté par l'anneau $A=R[x_{ij},\mathrm{det}^{-1}]$. \'Ecrivant $v=(v_1,\ldots,v_n)$ dans la base donnée, on considère l'idéal $I$ de $A$ défini par les équations
$$\sum_{k=1}^n x_{jk}v_k=v_j$$
pour $j=1,\ldots,n$. On voit que $A/I$ représente $\faisStab_v$ et que ce foncteur est fermé dans $\faisGL_N$. 
\end{proof}

\begin{prop} \label{normalisateur_prop}
Soit $G$ un $S$-groupe algébrique muni d'une représentation $\rho:G \to \faisGL_N$, $N$ un $\faisO_S$-module localement libre, et $M$ un sous-$\faisO_S$-module localement facteur direct de $N$. Alors le foncteur normalisateur $\faisNorm_{\rho,M}$, défini par 
$$\faisNorm_{\rho,M}(T)=\{g \in G(T),\ \rho(g)(M_T)=M_T,\}$$ 
est un sous-groupe fermé représentable de $G$.
\end{prop}
\begin{proof}
Comme dans la proposition précédente, on se ramène au cas où $M$ et $N$ sont libres avec $N \simeq M \oplus M'$ et sur une base affine $\Spec(R)$. On choisit alors des bases $\{e_1,\ldots,e_m\}$ de $M$ et $\{e_{m+1},\ldots,e_{m+n}\}$ de $M'$, ce qui permet d'identifier $\faisGL_N$ avec $\faisGL_{m+n}$, représenté par $R[x_{ij},1 \leq i,j\leq m+n,\mathrm{det}^{-1}]$. Le groupe recherché est alors l'image réciproque par $\rho$ du sous-groupe fermé de $\faisGL_N$ défini par 
l'idéal engendré par les $x_{ij}$ avec $1 \leq j < n+1 \leq i \leq n+m$. On conclut donc par la remarque \ref{imageinvnoyau_rema}.
\end{proof}

\subsection{Algèbres séparables, étales et d'Azumaya} \label{AlgSepEtAz_sec}

Rappelons dans cette partie la définition et quelques propriétés des algèbres séparables, puis étales et enfin d'Azumaya sur un anneau, ou plus généralement des faisceaux en algèbres de ces types. Le livre de Knus et Ojanguren \cite{kn-oj} est une bonne référence pour ces trois notions, ainsi que l'article de Grothendieck \cite{brauer1}.

Dans cette section, sauf mention contraire, la topologie de Grothendieck considérée est la topologie étale, et toutes les formes sont donc implicitement des formes étales.

\subsubsection{Algèbres séparables}

\begin{defi} 
Une algèbre $A$ localement libre de type fini sur $R$ est dite \emph{séparable} si elle est projective comme module sur $A \otimes_R A^{\opp}$ (agissant par $(x\otimes y).a = xay$). 
\end{defi}

\begin{prop} \label{separableequiv_prop}
Soit $A$ un faisceau en $\faisO_S$-algèbres localement libres de type fini. Les conditions suivantes sont équivalentes.
\begin{enumerate}
\item \label{loczarsep_item} Localement pour la topologie de Zariski, l'algèbre $A$ est l'image par le foncteur $\wW$ d'une algèbre séparable sur $\Gamma(S)$.
\item \label{locaffinesep_item} Pour tout schéma affine $T$ au-dessus de $S$, l'algèbre $A(T)$ est séparable sur $\Gamma(T)$.
\item \label{locetsep_item} Localement pour la topologie étale, l'algèbre $A$ est associée par le foncteur $\wW$ à une algèbre séparable sur $\Gamma(S)$.
\end{enumerate}
\end{prop}
\begin{proof}
Le point \ref{loczarsep_item} implique évidemment les points \ref{locaffinesep_item} et \ref{locetsep_item}. Pour \ref{locaffinesep_item} $\Rightarrow$ \ref{loczarsep_item}, voir \cite[Ch. III, prop. 2.5]{kn-oj}. Enfin, pour \ref{locetsep_item} $\Rightarrow$ \ref{loczarsep_item}, voir \cite[Ch. III, prop. 2.2, (b)]{kn-oj}.  
\end{proof}

\begin{defi}
Un faisceau en $\faisO_S$-algèbres localement libres de type fini qui satisfait aux conditions équivalentes de la proposition précédente est appelé \emph{séparable}.
\end{defi}

\begin{exem}
Un produit fini de copies de $R$ est séparable. Une algèbre de matrices $\setM_n(R)$ est séparable sur $R$.
\end{exem}

\begin{prop} \label{sepsurcorps_prop}
Lorsque $R$ est un corps, une algèbre de type fini est séparable si et seulement si elle est isomorphe à un produit fini d'algèbres de matrices sur des algèbres à division (corps gauches) de dimension finie sur $R$ et dont les centres sont des extensions de corps (finies) séparables de $R$. En particulier, si une telle algèbre est commutative, c'est un produit fini d'extensions de corps finies séparables de $R$. Voir \cite[Ch. III, th. 3.1]{kn-oj}.
\end{prop}

\subsubsection{Algèbres étales}

\begin{defi} \label{algetale_defi}
Une algèbre commutative sur un anneau $R$ est dite \emph{étale} si elle est séparable, plate et de présentation finie. Elle est alors dite \emph{finie} si elle est de type fini comme $R$-module. 
\end{defi}

Une algèbre étale finie est automatiquement localement libre de rang fini comme module sur $R$ (voir \cite[IV$_4$, 18.2.3]{ega}). 

\begin{prop}
Soit $Z$ un faisceau en $\faisO_S$-algèbres. Les conditions suivantes sont équivalentes.
\begin{enumerate}
\item Localement pour la topologie de Zariski, l'algèbre $A$ est l'image par le foncteur $\wW\circ \tilde{(-)}$ d'une algèbre étale finie sur $\Gamma(S)$.
\item Pour tout schéma affine $T$ au-dessus de $S$, l'algèbre $Z(T)$ est étale finie sur $\Gamma(T)$. 
\item Localement pour la topologie étale, l'algèbre $A$ est l'image par le foncteur $\wW\circ \tilde{(-)}$ d'une algèbre étale finie sur $\Gamma(S)$.
\end{enumerate}
\end{prop}
\begin{proof}
Les conditions d'être séparable, plate ou de présentation finie peuvent se tester localement pour les topologies Zariski ou étale, respectivement par la proposition \ref{separableequiv_prop}, puis par les propositions 2.5.1 et 2.7.1 et le corollaire 17.7.3 de \cite[IV]{ega}.
\end{proof}

\begin{defi} \label{faiscalget_defi}
Nous appellerons $\faisO_S$-algèbre étale finie un faisceau $Z$ en $\faisO_S$-algèbres qui satisfait aux conditions équivalentes de la proposition précédente. On dit qu'elle est \emph{de rang $n$} si son rang comme module localement libre est constant sur $S$ et égal à $n$. 
\end{defi}

\begin{prop} \label{algetfinies_prop}
Les formes étales de la $\faisO_S$-algèbre $\faisO_S^n$ sont les $\faisO_S$-algèbres étales finies de rang $n$. Pour une telle forme $Z$, le faisceau $\faisAut^\alg_Z$ est un forme étale du faisceau en groupes constant $(\mathcal{S}_n)_S$ (groupe symétrique).
\end{prop}
\begin{proof}
On dispose d'un morphisme évident $\mathcal{S}_n\to \faisAut^\alg_{\faisO_S^n}$ dont on vérifie facilement qu'il est un isomorphisme localement. La deuxième affirmation de l'énoncé découle alors de la première par la proposition \ref{auttordus_prop}. 

Les propriétés d'être étale, finie et localement libre de rang $n$ sont locales pour la topologie étale, donc toute forme étale de $\faisO_S^n$ est encore étale finie de rang $n$.

Réciproquement, montrons que toute $\faisO_S$-algèbre étale finie de rang $n$ est localement isomorphe à $\faisO_S^n$ pour la topologie étale. On peut supposer $S$ affine et l'algèbre libre sur $\faisO_S$ et la base locale, par un argument de limite. On est alors ramené au cas d'une algèbre étale finie $A$ sur un anneau local $R$. Par changement de base au corps résiduel, on obtient une extension de degré $n$ de corps, étale donc séparable. On prend un générateur de cette extension, et on le relève en un élément $x$ de $A$. Par le lemme de Nakayama, on montre facilement que les éléments $1,x,\ldots,x^{n-1}$ engendrent $A$, qui s'écrit donc $R[x]/P(x)$ avec $P$ unitaire de degré $n$, et $P'(x)$ inversible dans $A$, par \cite[IV, cor. 18.4.3]{ega}. Montrons alors par récurrence sur $n$ que $A$ est déployée par une extension étale finie. C'est évident si $n=1$. Utilisons l'extension $R \to A$ elle-même, et considérons donc $A \otimes_R A \simeq A[y]/P(y)$. Le polynôme unitaire $P(y)$ se décompose en $(y-x)Q(y)$ où $Q$ est un polynôme unitaire en $y$ à coefficients dans $A$. Puisque $P'(y)=(y-x)Q'(y)+Q(y)$ est inversible, les éléments $y-x$ et $Q(y)$ engendrent l'idéal unité. Par le théorème chinois, on a donc $A \otimes_R A \simeq A[y]/(y-x) \times A[y]/Q(y) \simeq A \times A[y]/Q(y)$. De plus $Q'(y)$ est inversible dans $A[y]/Q(y)$ par la même équation, ce qui nous ramène au cas $n-1$ de la récurrence. 
\end{proof}

\begin{rema}
Le fait que toute $\faisO_S$-algèbre étale finie de degré $n$ est localement $\faisO_S^n$ pour la topologie étale peut également se voir comme conséquence de la théorie des revêtements étales (\ie des morphismes finis étales). En effet, c'est exactement dire que le revêtement correspondant est localement trivial pour la topologie étale. C'est une conséquence de deux faits: l'hensélisé strict d'un anneau local est obtenu comme limite inductive d'algèbres étales, et tout revêtement d'un hensélisé strict est trivial (voir \cite[IV, déf. 18.8.7 et prop. 18.8.8, (i)]{ega}).
\end{rema}

\begin{coro} \label{etaledegre2_coro}
Si $Z$ est une $\faisO_S$-algèbre étale finie de rang $n=2$, $\faisAut^\alg_Z=(\ZZ/2)_S$ canoniquement. 
\end{coro}
\begin{proof}
Le groupe symétrique $\mathcal{S}_2 = \ZZ/2$ étant commutatif, ses tordus par des torseurs sous un groupe agissant par automorphismes intérieurs lui sont isomorphes.
\end{proof}

\begin{rema}
On peut construire à la main l'automorphisme différent de l'identité lorsque $n=2$. Si $Z(T)$ est libre sur $R'=\faisO_S(T)$ pour $T$ affine, on définit $\trace(z)$ la trace de $z\in Z(T)$ comme la trace de la multiplication par $z$ exprimée sur un base quelconque. Le morphisme $z \mapsto \trace(z)-z$ est bien un automorphisme de $R'$-algèbres et c'est le seul différent de l'identité. Comme il est canonique, il permet de définir par recollement un automorphisme d'algèbres de $Z$ au-dessus de $S$.
\end{rema}

\begin{defi} \label{Etn_defi}
La catégorie fibrée de $T$-fibre les algèbres étales finies de rang $n$ est notée $\Etn{n}$. C'est un champ pour les topologies Zariski, étale ou \fppf.
\end{defi}
\begin{proof}
Les $\faisO_S$-algèbres forment un champ par \ref{champstruc_prop}, et le fait d'être étale fini de degré $n$ descend pour la topologie \fppf, comme vu plus haut.
\end{proof}

Les propositions \ref{tordusformes_prop} et \ref{algetfinies_prop} impliquent immédiatement:
\begin{prop} \label{SntorsEtn_prop}
Le foncteur $Z \mapsto \faisIso_{Z,\faisO_S^n}$ définit une équivalence du champ en groupoïdes $(\Etn{n})_{\grpd}$ des algèbres étales finies de rang $n$ vers $\Tors{\mathcal{S}_n}$ le champ des torseurs étales ou \fppf\ sous $\mathcal{S}_n$.
\end{prop}

Attardons-nous enfin sur le cas $n=2$, qui nous sera plus utile que les autres. \'Etant donnée une algèbre étale finie $E$ de rang $2$, on lui associe fonctoriellement un module localement libre de rang $1$, le noyau de la trace, qu'on note $\chi_E$. De plus, la multiplication induit un isomorphisme $\mu_E:\chi_E\otimes \chi_E \to \faisO_S$ et cela fait de $(\chi_E,\mu_E)$ un module déterminant au sens de la définition \ref{moddet_defi} plus bas. Ces faits peuvent se vérifier localement, auquel cas, voir \cite[Ch. III, (4.2.1) à (4.2.3)]{knus}. 

Cela permet de définir un foncteur fibré $\XiFonc:(\Etn{2})_\grpd \to \ModDet$, qui envoie une algèbre $E$ sur $(\chi_E,\mu_E)$, qui induit sur les automorphismes $(\ZZ/2)_S=\faisAut_E \to \faisAut_{\XiFonc(E)}=(\faismu_2)_S$ le morphisme canonique $\ZZ/2 \to \faismu_2$ envoyant $i$ sur $(-1)^i$ (localement). 
\begin{rema} \label{Zmu2_rema}
Les points de $\ZZ/2$ peuvent se décrire comme des projecteurs: $(\ZZ/2)(T) \cong \{ p \in \Gamma(T) \text{ t.q. } p^2=p\}$, où une fonction localement constante à valeurs dans $\{0,1\}$ est identifiée au projecteur qui vaut l'identité là où elle vaut $1$ et $0$ là où elle vaut $0$. Le morphisme $\ZZ/2 \to \faismu_2$ est alors donné sur les points par $p \mapsto 1-2p$. C'est donc un isomorphisme lorsque $2\in \Gamma(S)^\times$, alors que c'est un morphisme constant de valeur $1$ lorsque $2=0$ dans $\Gamma(S)$. 
\end{rema}

\subsubsection{Algèbres d'Azumaya} \label{AlgAzu_sec}

Examinons maintenant les algèbres d'Azumaya.

\begin{defi}
Une algèbre séparable $A$ sur un anneau $R$ est appelée \emph{$R$-algèbre d'Azumaya} si $R$ s'injecte dans $A$ et si le centre de $A$ est réduit à $R$. 
\end{defi}

\begin{exem}
Une algèbre d'Azumaya sur un corps est une algèbre centrale simple sur ce corps par la proposition \ref{sepsurcorps_prop}.
\end{exem}

\begin{prop} \label{azudefi_prop}
Soit $A$ une $\faisO_S$-algèbre localement libre de type fini. Les conditions suivantes sont équivalentes.
\begin{enumerate}
\item \label{azuloczar_item} Localement pour la topologie de Zariski, l'algèbre $A$ est l'image par le foncteur $\wW\circ \tilde{(-)}$ d'une algèbre d'Azumaya sur $\Gamma(S)$.
\item \label{azuaffine_item} Pour tout schéma affine $T$ au-dessus de $S$, l'algèbre $A(T)$ est une algèbre d'Azumaya sur $\Gamma(T)$.
\item \label{azulocet_item} Localement pour la topologie étale, l'algèbre $A$ est l'image par le foncteur $\wW\circ \tilde{(-)}$ d'une algèbre d'Azumaya sur $\Gamma(S)$.
\end{enumerate}
\end{prop}
\begin{proof}
On a trivialement \ref{azuaffine_item} $\Rightarrow$ \ref{azuloczar_item} $\Rightarrow$ \ref{azulocet_item}.  
Pour \ref{azulocet_item} $\Rightarrow$ \ref{azuaffine_item}, se ramener à une base affine, puis utiliser l'équivalence entre les points 4 et 1 de \cite[Ch. III, th. 6.6]{kn-oj}.
\end{proof}

\begin{defi}
Un \emph{$\faisO_S$-faisceau en algèbres d'Azumaya}, ou plus simplement une $\faisO_S$-algèbre d'Azumaya, est un faisceaux en $\faisO_S$-algèbres $A$ qui satisfait aux conditions équivalentes de la proposition précédente.
\end{defi}

\begin{exem} \label{EndMAzumaya_exem}
Pour tout $n$, la $\faisO_S$-algèbre $\faisM_{n,S}$ est une $\faisO_S$-algèbre d'Azumaya; on vérifie que son centre est bien $\faisO_S$ en utilisant les matrices élémentaires. De même, pour tout $\faisO_S$-module localement libre $M$, le faisceau en $\faisO_S$-algèbres $\faisEnd_M$ est une $\faisO_S$-algèbre d'Azumaya: on se ramène au cas précédent localement.
\end{exem}

\begin{prop}
Le rang comme $\faisO_S$-module d'une $\faisO_S$-algèbre d'Azumaya est un carré.
\end{prop}
\begin{proof}
Voir \cite[Ch. III, prop. 6.1 et th. 6.4]{kn-oj}.
\end{proof}

\begin{defi} \label{degreAzumaya_defi}
Le degré d'une $\faisO_S$-algèbre d'Azumaya est la racine carré de son rang. C'est une fonction localement constante à valeurs entières par la proposition précédente.
\end{defi}

\begin{prop} \label{Azumayaforme_prop}
Un $\faisO_S$-faisceau en algèbres d'Azumaya de degré $n$ sur une base $S$ est un faisceau en $\faisO_S$-algèbres qui est une forme (étale) de $\faisM_{n,S}$. 
\end{prop}
\begin{proof}
Puisque $\faisM_{n,S}$ est une algèbre d'Azumaya, toute forme étale de $\faisM_{n,S}$ en est également une par le point \ref{azulocet_item} de \ref{azudefi_prop}.

Réciproquement, on se ramène au cas où la base est affine puis on utilise le point 4 de \cite[Ch. III, th. 6.6]{kn-oj}.
\end{proof}

\begin{prop}
Le centre d'une algèbre d'Azumaya est $\faisO_S$.
\end{prop}
\begin{proof}
Comme tout recouvrement étale est épimorphique (universel) par \cite[Exp. IV, prop. 4.4.3]{sga3}, pour tester si un élément central est dans $\faisO_S(T)$, on peut supposer que l'algèbre est $\faisM_{n,S}$, dont le centre est bien $\faisO_S$ par un calcul direct de commutation aux matrices élémentaires.  
\end{proof}

\begin{defi} \label{Azchamp_defi}
Lorsque la topologie est \fppf\ ou bien étale, on définit le champ des algèbres d'Azumaya de degré $n$, noté $\Azumayan{n}$ en posant que $(\Azumayan{n})_T$ est la catégorie des $\faisO_T$-algèbres d'Azumaya de degré $n$.
\end{defi}
Notons que cette catégorie fibrée est bien un champ pour la topologie étale ou \fppf: en effet, les $\faisO_S$-algèbres forment un champ par la proposition \ref{champstruc_prop}, et la propriété d'être d'Azumaya de degré $n$ est locale par définition pour la topologie étale. C'est de plus un champ en groupoïdes, puisque tout morphisme d'algèbres entre deux algèbres d'Azumaya de même degré est un isomorphisme: on le vérifie localement pour la topologie étale, auquel cas, il suffit de montrer que tout endomorphisme $f:\setM_n(R) \to \setM_n(R)$ est un isomorphisme, pour un anneau $R$. Voir alors \cite[Ch. III, Cor. 5.4]{kn-oj}. 
 
Décrivons maintenant les torseurs sous $\faisGL_{1,A}$ lorsque $A$ est une $\faisO_S$-algèbre d'Azumaya.

Soit $M$ un $A$-module, donc donné par une structure analogue à celle des $\faisO_S$-modules, en remplaçant $\faisO_S$ par $A$. C'est automatiquement un $\faisO_S$-module, puisque $\faisO_S$ est inclus dans $A$ (comme son centre). 

\begin{lemm}
Si un tel $M$ est localement libre type fini sur $A$, alors il est localement libre de type fini sur $\faisO_S$. De plus, si $A$ est de degré constant $n$, alors $M$ est de rang $m$ sur $A$ si et seulement s'il est de rang $mn^2$ sur $\faisO_S$.
\end{lemm}
\begin{proof}
Localement, $A \simeq \faisM_n$.
\end{proof}

Pour toute $\faisO_S$-algèbre d'Azumaya $A$ de degré constant $n$, considérons la catégorie fibrée $\AVecn{A}{m}$ dont les objets de la fibre sur $T$ sont des $A_T$-modules à gauche, qui sont localement libres de rang $m$ comme $A$-modules. Cette catégorie fibrée est un champ: sans la condition ``localement libre de rang $m$'', c'est un champ (pour les topologies Zariski, étale ou \fppf) par la proposition \ref{champstruc_prop}. Il faut ensuite voir que cette condition descend, ce qui est clair.

Définissons le foncteur fibré ``Endomorphismes''
$$\EndFonc:(\AVecn{A^{\opp}}{1})_\grpd\to \Azumayan{n}$$
qui envoie un objet $M$ sur $\faisEnd_{M}$ (endomorphismes de $A^{\opp}$-modules) et un (iso)morphisme $f$ sur $\int_f$. 

Le $A^{\opp}$-module $A_d$ qui est $A$ vu comme $A$-module à droite, et donc comme $A^{\opp}$-module à gauche est évidemment libre de rang $1$.
Alors, dans le champ $\AVecn{A^{\opp}}{1}$ le faisceau d'endomorphismes $\faisEnd_{A_d}$ est canoniquement isomorphe à $A$ (comme $\faisO_S$-algèbre) par $f \mapsto f(1)$. L'objet $A_d$ s'envoie donc sur $A$ par le foncteur $\EndFonc$.
De plus, $\faisAut_{A_d}=\faisGL_{1,A}$, et le foncteur $\EndFonc$ induit la flèche canonique
$$\faisGL_{1,A}=\faisAut^{\AVecn{A^{\opp}}{1}}_{A_d}\to \faisAut^{\Azumayan{n}}_A = \faisPGL_A.$$
(les premiers automorphismes sont de $A^{\opp}$-modules, les seconds de $\faisO_S$-algèbres).

Puisque tout objet de $\AVecn{A^\opp}{1}$ est localement isomorphe à $A_d$, par la proposition \ref{tordusformes_prop}, on obtient:
\begin{prop} \label{GL1Atorseurs_prop}
Le foncteur fibré $\faisIso_{A_d,-}$ définit une équivalence de champs en groupoïdes 
$$(\AVecn{A^{\opp}}{1})_\grpd\isoto\Tors{\faisGL_{1,A}}.$$
\end{prop}
De même, puisque toute algèbre d'Azumaya est localement isomorphe à $A$, on obtient également:
\begin{prop} \label{PGLAtorseurs_prop}
Le foncteur fibré $\faisIso_{A,-}$ définit une équivalence de champs en groupoïdes 
$$\Azumayan{n}\isoto\Tors{\faisPGL_{A}}.$$
\end{prop}
De plus, par la proposition \ref{foncttors_prop}, on a enfin:
\begin{lemm}
A travers ces équivalences, le foncteur $\EndFonc$ induit la fonctorialité $\Tors{\faisGL_{1,A}}\to \Tors{\faisPGL_{A}}$ le long du morphisme canonique $\faisGL_{1,A} \to \faisPGL_A$.
\end{lemm}

Intéressons-nous au cas particulier où $A=\faisEnd_M$ où $M$ est un $\faisO_S$-module localement libre de rang $n$. On a alors l'équivalence de Morita:
$$\AVecn{\End_M^\opp}{1} \isoto \Vecn{n}$$
qui est donnée par le foncteur $M \otimes_{\faisEnd_M^{\opp}} - $ et par le foncteur $M^{\dual} \otimes -$ dans l'autre sens. Cela utilise les isomorphismes canoniques 
$$M \otimes_{\faisEnd_M^\opp} M^\dual \isoto \faisO_S\quad\text{et}\quad M^\dual \otimes M \isoto \faisEnd_M^\opp$$ dont le premier envoie $m \otimes f$ sur $f(m)$. 
Par cette équivalence, $M \in \Vecn{n}$ est donc envoyé sur $\faisEnd_M^{\opp}$ dans $\AVecn{\End_M^\opp}{1}$ et on retrouve la description de la proposition \ref{GLntors_prop} dans le cas où $M =\faisO_S^n$.

Le foncteur fibré $\Vecn{1} \to \AVecn{A^{\opp}}{1}$ envoyant un fibré en droites $L$ sur le $A^{\opp}$-module $A_d \otimes L$ (l'action de $A^{\opp}$ se faisant sur $A_d$) envoie l'objet $\faisO_S$ sur $A_d$ et induit entre leurs automorphismes l'inclusion $\faisGm \to \faisGL_{1,A}$.

Introduisons enfin la norme et la trace réduite.

On considère la catégorie fibrée $\TrAz$ dont la $T$-fibre a pour objets les paires $(A,t)$ où $A$ est une $\faisO_T$-algèbre d'Azumaya et où $t:A \to \faisO_T$ est un morphisme de $\faisO_T$-modules qui vérifie $t(ab)=t(ba)$ sur les points. Les morphismes $(A,t)\to (A',t')$ sont les morphismes d'algèbres $f:A \to A'$ tels que $t'\circ f=t$. Cette catégorie fibrée est un champ par \ref{champstruc_prop}, et il y a un foncteur fibré oubli évident $\TrAz \to \Azumaya$, qui induit donc pour toute paire $(A,t)$ un morphisme de faisceaux en groupes $\faisAut_{A,t} \to \faisAut_A$. Ce morphisme est en fait un isomorphisme. En effet, tout morphisme d'algèbres $f:A \to A$ vérifie $t \circ f = t$: comme les morphismes sont à valeurs dans $\faisO_S$, un faisceau, on peut vérifier l'égalité localement, et donc supposer que $f$ est intérieur, auquel cas cela découle immédiatement de la propriété imposée sur $t$. 
Partant donc d'un torseur sous $\faisPGL_n=\faisAut_{\faisM_n}$ correspondant à $A$ par \ref{PGLAtorseurs_prop}, on peut donc tordre $(\faisM_n,\trace)$, où $\trace$ est la trace, en une paire $(A,\trd_A)$. 
\begin{defi} \label{trd_defi}
Pour toute algèbre d'Azumaya $A$, le morphisme de $\faisO_S$-modules $\trd_A: A \to \faisO_S$ défini ci-dessus est appelé \emph{trace réduite}.
\end{defi}
Le foncteur fibré $\Azumaya \to \TrAz$ qui envoie $A$ sur $(A,\trd_A)$ est une section de l'oubli $\TrAz \to \Azumaya$.

\begin{lemm} \label{trdreg_lemm}
Si $A$ est de degré constant pair, le module bilinéaire $(A,t)$ où $t:A \to A^\dual$ est donné par $t(a)=\trd(a-)$ est régulier.
\end{lemm}
\begin{proof}
C'est une propriété locale pour la topologie étale, on peut donc supposer $A=\faisM_n$ et $\trd$ est alors la trace usuelle, auquel cas c'est un exercice facile d'algèbre linéaire.
\end{proof}

De manière analogue, on introduit la norme réduite en considérant le champ $\NrAzn{n}$ des $(A,d)$, avec $A$ d'Azumaya de degré $n$ et $d:A \to \faisO_S$ satisfaisant $d(ab)=d(a)d(b)$ sur les points. Comme plus haut, on a $\faisAut_{A,d} \to \faisAut_A$ est un isomorphisme, et on obtient un morphisme $\nrd$ en tordant la paire $(\faisM_n,\det)$.
\begin{defi} \label{normereduite_defi}
Ce morphisme $\nrd: A \to \faisO_S$ est appelé \emph{norme réduite}. Il induit sur les éléments inversibles un morphisme $\faisGL_{1,A} \to \faisGm$.
\end{defi} 
 
\begin{rema} \label{detGLM_rema}
Lorsque $A=\faisEnd_M$ pour un $\faisO_S$-module localement de type fini $M$, cette norme réduite est bien le déterminant usuel qui, sur les points, à un endomorphisme de $M$ associe son endomorphisme déterminant, point de $\faisEnd_{\Det{M}}=\faisO_S$. Par construction, c'est vrai localement et donc en général puisque $\faisO_S$ est un faisceau. 
\end{rema}

\subsection{Modules quadratiques} \label{modulesquadratiques_sec}

Rappelons qu'un $\faisO_S$-module bilinéaire (resp. quadratique) est un objet en $\faisO_S$-modules bilinéaires (resp. quadratiques), au sens de l'exemple \ref{bilineaire_item} de \ref{structures_exem} (resp. \ref{quadratique_item}). Nous éviterons d'utiliser la terminologie ``forme bilinéaire'' car le mot ``forme'' est déjà pris (déf. \ref{formes_defi}). De plus, pour alléger, nous désignerons souvent un module bilinéaire (resp. quadratique) par la seule mention de l'application $b:M \times M \to \faisO_S$ (resp. $q: M \to \faisO_S$), et nous utiliserons parfois le terme ``module'' à la place de $\faisO_S$-module. 

La notation $M^\dual$ désigne le module dual donné par 
$$M^\dual(T)=\Hom_{\faisO_{T}-\text{mod}}(M_{T},\faisO_{T}).$$ 
Pour un $\faisO_S$-module localement libre de type fini, le morphisme canonique $\bid_M: M \to (M^\dual)^\dual$ est un isomorphisme. Nous supposerons dans cette partie et dans toute la suite du texte que, sauf mention explicite, tous les modules quadratiques ou bilinéaires sont localement libres de type fini. 

La somme orthogonale de deux modules bilinéaires $b_1$ et $b_2$ est notée $b_1\orth b_2$. Idem pour les modules quadratiques.

On montre facilement que se donner un $\faisO_S$-module bilinéaire est équivalent à se donner un $\faisO_S$-module $M$ (loc. libre de type fini) et un morphisme $\phi:M \to M^\dual$. Ce $\faisO_S$-module sera symétrique si $(\phi)^\dual \circ \bid_M=\phi$ et antisymétrique si $(\phi)^\dual \circ \bid_M=-\phi$. Par ailleurs un module bilinéaire est dit alterné si la composition 
\[
\xymatrix{M\ar[r]^-\Delta & M\times M\ar[r]^-b & \faisO_S}
\]
est triviale, où $\Delta$ est l'application diagonale. Il est clair qu'un module alterné est antisymétrique.

Lorsque la base $S$ est égale à $\Spec(R)$, le foncteur $\tilde{(-)}$ induit une équivalence de catégorie entre les modules bilinéaires (resp. quadratiques) usuels sur $R$ et les $\faisO_S$-modules bilinéaires (resp. quadratiques). 

Un module quadratique $q$ vient par définition avec un \emph{module polaire associé} $b_q$, qui est un module bilinéaire symétrique. Sur les points, il s'agit de $b_q(x,y)=q(x+y)-q(x)-q(y)$.

De même, partant d'un module module bilinéaire $b$, on obtient un module quadratique associé $q_b$ en précomposant le morphisme bilinéaire $M \times M \to M$ par l'application diagonale $M \to M \times M$. On a $q_{b_q} = 2 q$ et si $b$ est symétrique, $b_{q_b}=2b$. Ces deux constructions sont clairement fonctorielles.

\begin{defi}
Nous utiliserons la notation usuelle $\fq{a_1,a_2,\ldots,a_n}$ pour désigner le module quadratique diagonal d'équation $a_1 x_1^2 + \cdots + a_n x_n^2$ sur $\faisO_S^n$. Pour les modules bilinéaires, nous utiliserons la même notation pour désigner le module d'équation $a_1 x_1 y_1 +a_2x_2 y_2+\dots a_n x_n y_n$ sur $\faisO_S^n$.
\end{defi}
On a donc $b_{\fq{a_1,\ldots,a_n}}=\fq{2a_1,\ldots,2a_n}$ et $q_{\fq{a_1,\ldots,a_n}} =\fq{a_1,\ldots,a_n} $. 

Si $(M,b)$ est un module bilinéaire et si $N$ est un sous-module de $M$, l'orthogonal de $N$, noté $N^{\orth}$ désigne le sous-module de $M$ noyau de l'application correspondante $M \to M^{\dual}$ restreinte à $N$, autrement dit
$$N^{\orth}(T)=\{m \in M(T) \text{ t.q. $b(m_{T'},n)=0$ $\forall T'\to T$ et $\forall n \in N(T')$}\}.$$
Si $(M,q)$ est un module quadratique, alors l'orthogonal d'un sous-module est toujours pris au sens de la forme polaire $b_q$ associée.

\begin{defi} \label{radical_defi}
Le \emph{radical} $\rad_b$ d'un module bilinéaire $(M,b)$ est le noyau de l'application $M \to M^\dual$. \'Etant donné un module quadratique $(M,q)$, nous appellerons \emph{radical polaire}, noté $\rad_q$, le radical de son module polaire $b_q$ et nous appellerons \emph{radical quadratique}, noté $\qrad_q$ le sous-module de son radical polaire constitué des éléments s'envoyant sur $0$ par $q$. 
\end{defi}

\begin{rema}
On vérifie aisément que le radical quadratique d'un module quadratique est bien un sous-module. Par contre, en général, il n'est pas localement libre. Si $2$ est inversible dans $\Gamma(S)$, alors le radical quadratique et le radical polaire de $q$ coïncident.
Sur un anneau, ni le radical polaire, ni le radical quadratique ne commutent en général à l'extension des scalaires, au sens que $\qrad_q(R) \otimes R'$ peut-être strictement contenu dans $\qrad_{q_{R'}}(R')$. Il suffit de prendre l'exemple des formes $\fq{1}$ et $\fq{2}$ sur $\ZZ$ puis sur les corps de car. $2$ ou pas, pour avoir tous les cas possibles. Le radical polaire commute aux extensions plates de la base, car c'est le noyau de $M \to M^\dual$ et donc en particulier, sur un corps de base, comme tout est plat, il y a commutation à l'extension des scalaires. Par contre, le radical quadratique ne commute même pas à ces extensions plates: sur un corps $k$ de caractéristique $2$, on peut prendre $R'=k[t]/t^2$ (qui est plate). Mais le module quadratique quasi-linéaire $\fq{1}$ vérifie $\qrad_{\fq{1}}(k)=\{0\}$ et pourtant $\qrad_{\fq{1}_{R'}}(R')=\{tx,\ x\in k\}$. 
\end{rema}

La situation en rang un est la suivante.
\begin{prop} \label{rangun_prop}
Le foncteur $b \mapsto q_b$ induit une équivalence de catégories des modules bilinéaires de rang $1$ vers les modules quadratiques de rang $1$.
\end{prop}
\begin{proof}
\'Etant donné un module module quadratique $q$ sur un $\faisO_S$-module localement libre de rang $1$ , construisons une forme bilinéaire $b$ telle que $q_b\simeq q$. Lorsque $L$ est libre de rang $1$ et que $e$ est une base de $\faisO_S(S)$, l'application $q$ est entièrement déterminée par l'image $q(e)$, et on choisit la forme bilinéaire donnée sur les points par $(xe,ye)\mapsto xyq(e)$; cette forme est indépendante du choix de $e$. Pour un $L$ quelconque, on prend un recouvrement ouvert Zariski qui trivialise $L$, et on construit $b$ par recollement à partir du cas précédent; l'isomorphisme $q\simeq q_b$ peut se tester localement. Cela montre que le foncteur est essentiellement surjectif. Il est fidèle puisque les modules sous-jacents ne changent pas, et on vérifie qu'il est plein en localisant à nouveau à des ouverts Zariski. 
\end{proof}

\begin{rema}
\'Etant donné $b_{q_b}=2b$, il est clair que si $2$ est inversible dans $\faisO_S(S)$, alors $q \mapsto b_q$ et $b \mapsto q_b$ sont des équivalences de catégories. Mais hors de cette hypothèse, dès le rang $2$, il y a des modules quadratiques qui ne sont pas de la forme $q_b$ avec $b$ bilinéaire. Il suffit de prendre une base $S$ en caractéristique $2$ (\ie où $2=0$ dans $\faisO_S(S)$). On a alors $b_{q_b}=2b=0$, donc un module quadratique $q$ tel que $b_q$ est non nul fournit un exemple. C'est le cas de $q$ sur $M=\faisO_S^2$ donné par $q(a e_1 + be_2)=ab$ (qui est hyperbolique, voir ci-dessous).
\end{rema}

\begin{defi} \label{hyperbolique_defi}
Le $\faisO_S$-module bilinéaire sur $M \oplus M^\dual$ donné par la matrice $\left(\begin{smallmatrix} 0 & 1 \\ \bid_M & 0 \end{smallmatrix} \right)$ est symétrique et sera noté $\hyps_M$. De même, le module bilinéaire donnée par la matrice $\left(\begin{smallmatrix} 0 & 1 \\ -\bid_M & 0 \end{smallmatrix}\right)$ est alterné et sera noté $\hypa_M$. On les appelle respectivement module hyperbolique bilinéaire symétrique et alterné. Le module quadratique hyperbolique sur $M \oplus M^\dual$, noté $\hypq_M$, est donné par l'évaluation $\hypq_M(m,e)=e(m)$ sur les points.
Lorsque $M=\faisO_S^n$, on utilise les notations $\hyps_{2n}$, $\hypa_{2n}$ et $\hypq_{2n}$. 
\end{defi}
Attention, on a $b_{\hypq_M}=\hyps_M$ mais $q_{\hyps_M}=2 \hypq_M$.
\begin{defi} \label{qhyperbolique_defi}
En rang impair, on note $\hyps_{2n+1}$ et $\hypq_{2n+1}$ les sommes orthogonale de $\hyps_{2n}$ et $\hypq_{2n}$ avec le module $\fq{1}$, bilinéaire ou quadratique respectivement.
\end{defi}
Attention, on a alors $b_{\hypq_{2n+1}}=\fq{2} \orth \hyps_{2n}$ et $q_{\hyps_{2n+1}} = \fq{1} \orth 2\hypq_{2n}$. 

\begin{defi} \label{regulier_defi}
Un $\faisO_S$-module bilinéaire est dit \emph{régulier} si son radical est nul. Un $\faisO_S$-module quadratique est dit régulier si son module bilinéaire polaire associé est régulier. 
\end{defi}

\begin{exem}
Les modules hyperboliques $\hyps_M$, $\hypa_M$ et donc $\hypq_M$ sont réguliers.
\end{exem}

\begin{prop} \label{regulierloca_prop}
Un $\faisO_S$-module bilinéaire (resp. quadratique) est régulier si et seulement s'il l'est localement pour la topologie étale ou même \fppf. 
\end{prop}
\begin{proof}
Le fait d'être un isomorphisme se teste localement pour la topologie \fppf.
\end{proof}

\subsubsection{Déterminant}

Définissons maintenant le déterminant d'un module bilinéaire. Rappelons que $\Det{M}$ est le déterminant du $\faisO_S$-module localement libre de type fini $M$. 
\begin{defi}
Le \emph{morphisme déterminant} d'un module bilinéaire $b:M \to M^\dual$, est le morphisme de $\faisO_S$-modules $(\Det{M})^{\otimes 2} \to \faisO_S$ donné par la composition du morphisme $id\otimes \Det{b}: \Det{M}^{\otimes 2} \to \Det{M} \otimes \Det{M^\dual}$ suivi des isomorphismes canoniques
$$\Det{M} \otimes \Det{M^\dual} \simeq \Det{M} \otimes \Det{M}^{\dual} \simeq \faisO_S.$$
On le note $\bDetMorph{b}$.
\end{defi}

\begin{rema}
Lorsque $M=\faisO_S^n$, on a canoniquement $\faisO_S \simeq \faisO_S^{\otimes 2} \simeq \Det{M}^{\otimes 2}$ et on obtient bien le déterminant de la matrice qui décrit l'application bilinéaire sur la base canonique, en regardant l'élément de $\faisO_S(S)$ qui est l'image de $1$ par ce morphisme.  
\end{rema}

\begin{prop} \label{detregiso_prop} 
Un module bilinéaire $M$ est régulier si et seulement si son morphisme déterminant $\Det{M}^{\otimes 2} \to \faisO_S$ est un isomorphisme. Un module quadratique est régulier si et seulement si le déterminant de son module bilinéaire associé est inversible.
\end{prop}
\begin{proof}
La deuxième affirmation est une conséquence triviale de la première, qui se vérifie localement pour la topologie de Zariski, et lorsque le module est $\faisO_S^n$, l'affirmation est classique: une matrice est inversible si et seulement si son déterminant est inversible. 
\end{proof}

Un $\faisO_S$-\emph{module déterminant} est un module $n$-trivialisé avec $n=2$, voir la partie \ref{GLGmmunTors_sec}; Explicitement, c'est un couple $(L,\phi)$ où $L$ est un $\faisO_S$-module inversible de rang constant $1$, et $\phi:L^{\otimes 2} \to \faisO_S$ est un isomorphisme. 

\begin{defi} \label{moddet_defi}
On note $\ModDet=\nTriv{2}$  le champ des modules déterminants.
\end{defi}

Notons que pour tout module déterminant $(L,\phi)$, on a $\faisAut_{(L,\phi)}=\faismu_2$.
 
\begin{defi}
Soit $(M,b)$ un module bilinéaire (resp. quadratique) régulier. Son $\emph{module déterminant}$ est le module déterminant $(\Det{M},\bDetMorph{b})$ (resp. $(\Det{M},\bDetMorph{b_q})$).  
\end{defi}

Par compatibilité aux formes hyperboliques, il est parfois nécessaire d'ajouter un signe au module déterminant trivial.
\begin{defi} \label{moddettriv_defi}
Soit $\signe=1$ ou $-1$.
Le $\faisO_S$-module déterminant $(\faisO_S,\signe\cdot m)$ où $m:\faisO_S^{\otimes 2} \to \faisO_S$ est la multiplication est appelé le $\faisO_S$-module déterminant \emph{trivial pair} si $\signe =1$ et \emph{trivial impair} si $\signe =-1$. 
\end{defi}

\begin{lemm}
Le module déterminant du module hyperbolique $\hypq_{2n}$ est le module déterminant trivial de la parité de $n$.
\end{lemm}
\begin{proof}
Le déterminant de la matrice associée à la forme bilinéaire sur la base canonique est $(-1)^n$. 
\end{proof}

La proposition \ref{torseursfppfmun_prop} avec $n=2$ se décline pour $\signe=\pm 1$. 
\begin{prop} \label{champmu2_prop}
Le foncteur $(L,\phi) \mapsto \faisIso_{(\faisO_S,\signe\cdot  m),(L,\phi)}$ est une équivalence de catégories fibrées entre champs en groupoïdes 
$$\Formes{(\faisO_S,\signe\cdot m)} \isoto \Tors{\faismu_2}.$$ 
des formes de $(\faisO_S,\signe \cdot m)$ vers les $\mu_2$-torseurs.
En particulier, si la topologie est \fppf, tous les modules $2$-trivialisés sont des formes de $(\faisO_S,m)$ puisque $(-1)$ est localement un carré et on a 
$$\nTriv{2}_{\simeq} \isoto \Tors{\faismu_2}.$$ 
\end{prop}
Notons enfin:
\begin{lemm}
Deux $(L_1,\phi_1)$ et $(L_2,\phi_2)$ sont isomorphes si et seulement s'il existe un isomorphisme $\psi: L_1 \simeq L_2$ et si $\phi_1^{-1} \psi^{\otimes 2} \phi_2(1)$ est un carré de $\faisGm(S)$. 
\end{lemm}
\begin{proof}
Un sens est évident, et dans l'autre, s'il existe un tel $\phi$, on peut le modifier par une racine carrée de l'élément de $\faisGm(S)$ pour le corriger en un morphisme de module déterminants.
\end{proof}

Lorsque $(M,b)$ est un module bilinéaire régulier, son module déterminant $(\Det{M},\bDetMorph{b})$ définit donc un $\faismu_2$-torseur $\faisIso_{(\faisO_S,\signe\cdot m),(\Det{M},\bDetMorph{b})}$. 
A priori, c'est un torseur \fppf, mais c'est en fait un torseur étale quand 
$q$ est une forme du module hyperbolique de rang pair, par la proposition \ref{formesetaleshyppair_prop}.
\begin{defi} \label{torseurdet_defi}
On note $\bDetMod{b}$ le module déterminant $(\Det{M},\bDetMorph{b})$ associé à un module bilinéaire $(b,M)$ régulier. Lorsque $q$ est un module quadratique régulier, on définit son module déterminant $\qDetMod{q}=\bDetMod{b_q}$.
Cela définit un foncteur fibré
$$\qDetFonc:\Formes{\hypq_{2n}} \to \ModDet.$$
qui à $q$ associe $\qDetMod{q}$.
\end{defi} 

Sur le module $\faisO_S^2$, de base canonique $(e_1,e_2)$, étant donné deux éléments $a$ et $b$ de $\Gamma(S)$, on note $[a,b]$ le module quadratique donné sur les points par $xe_1+ye_2 \mapsto ax^2+xy+by^2$. 
\begin{lemm} \label{ab_lemm}
Le module $[a,b]$ est régulier si et seulement si $1-4ab \in \Gamma(S)^\times$ et le schéma $T=\Spec(\cO_S[t]/(t^2-t+ab))$ est alors étale sur $S$. Si de plus $a,b\in \Gamma(S)^\times$, le module $[a,b]_{T}$ est isomorphe à $(\hypq_2)_{T}$.
\end{lemm}
\begin{proof}
On vérifie immédiatement que le déterminant de $[a,b]$ est la multiplication par $4ab-1$ de $\faisO_S=\faisO_S^{\otimes 2} \to \faisO_S$. C'est donc un isomorphisme si et seulement si $1-4ab$ est inversible, ce qui prouve l'affirmation sur la régularité.

On peut vérifier que l'extension est étale localement, lorsque $S=\Spec(R)$, en calculant la dérivée de $t^2-t+ab$ qui vaut $2t-1$, qui est donc inversible dans $R'=R[t]/(t^2-t+ab)$ puisque $(2t-1)^2=1-4ab$. 

Lorsque $a$ et $b$ sont inversibles, on peut alors vérifier que sur $T$, $(t e_1 -ae_2,(1-t)e_1-ae_2)$ est une nouvelle base du module, sur laquelle il est clairement isomorphe à $\hypq_{2}$.
\end{proof}

\begin{prop}
Sur un anneau local, un module quadratique de rang pair est régulier si et seulement s'il est isomorphe à une somme directe de modules de la forme $[a,b]$ ci-dessus avec $1-4ab$, $a$ et $b$ dans $\Gamma(S)^\times$. 
\end{prop}
\begin{proof}
Voir \cite[Ch. IV, lemme 2.2.2]{knus}. La preuve est essentiellement qu'on peut prouver ce type de décomposition sur un corps, puis la relever du corps résiduel à l'anneau local par le lemme de Nakayama.
\end{proof}

\begin{prop} \label{formesetaleshyppair_prop}
Les formes étales du module quadratique hyperbolique $\hypq_{2n}$ (de rang $2n$ sur $\faisO_S^{2n}$) sont les modules quadratiques réguliers de rang constant $2n$. 
\end{prop}
\begin{proof}
Puisque $\hypq_{2n}$ est régulier, toute forme étale est régulière par \ref{regulierloca_prop}. Réciproquement, étant donné un module quadratique régulier, on peut supposer que son module sous-jacent est libre. Comme il s'écrit comme somme de modules de la forme $[a,b]$ avec $a,b,1-4ab\in R^\times$ sur les anneaux locaux et qu'il n'y a alors qu'un nombre fini de coefficients, c'est encore le cas sur un voisinage Zariski et affine de chaque point. On conclut par le lemme \ref{ab_lemm}.
\end{proof}
\medskip

On considère alors la catégorie fibrée de tous les modules quadratiques (à module sous-jacent localement libre), ainsi que sa sous-catégorie fibrée des modules quadratiques réguliers de rang $n$. Pour les topologies Zariski, étale ou \fppf, la première est un champ par \ref{champstruc_prop}, et par conséquent la seconde aussi parce que la propriété d'être régulier se teste localement, par \ref{regulierloca_prop}.
\begin{defi} \label{champquadreg_defi}
On appelle $\Quad$ le champ de tous les module quadratiques (à module sous-jacent localement libre), et $\RegQuadn{n}$ celui des modules quadratiques réguliers de rang constant $n$. 
\end{defi}
La proposition \ref{formesetaleshyppair_prop} montre que le sous-champ $\Formes{\hypq_{2n}}$ est égal à tout $\RegQuadn{2n}$.

En rang impair, la notion adéquate pour remplacer la régularité qui n'arrive jamais en caractéristique $2$ est la semi-régularité. Pour éviter de développer toute la notion, nous renvoyons à \cite[Ch. IV, §3]{knus}. C'est une notion locale pour la topologie \fppf\ (cf \cite[Ch. IV, (3.1.5)]{knus}). On en tire immédiatement, comme dans le cas pair:

\begin{prop} \label{semiregformeshyp_prop}
Les formes \fppf\ du module quadratique hyperbolique $\hypq_{2n+1}$ sont les modules quadratiques semi-réguliers de rang constant $2n+1$.
\end{prop}

\begin{defi} \label{O_defi}
Partant d'un module quadratique $(M,q)$, considérons $\faisAut_{M,q}$, que nous noterons $\faisorthO_{M,q}$ ou même $\faisorthO_q$ (voir les parties \ref{Groupesquadratiques_sec} et \ref{Bn_sec} pour plus de détails).  
\end{defi}

Nous aurons besoin d'un raffinement dû à Kneser du théorème de Cartan-Dieudonné. 
\begin{defi} \label{reflexion_defi}
Soit $(M,q)$ un module quadratique quelconque. Soit $v$ un élément de $M(S)$ tel que $q(v) \in \Gamma(S)^\times$. La \emph{réflexion orthogonale} associée à $v$ est l'élément de $\faisorthO_{M,q}(S)$ de la forme $\tau_v=x \mapsto x -\frac{b_q(x,v)}{q(v)}v$.
\end{defi}
Lorsqu'on parle de réflexion orthogonale sans préciser le vecteur $v$, il s'agit toujours d'une réflexion orthogonale associée à un vecteur $v$ comme dans la définition précédente (donc avec $q(v) \in \Gamma(S)^\times$).

\begin{lemm} \label{detrefl_lemm}
Le déterminant d'une réflexion orthogonale est $-1 \in \Gamma(S)^\times$.
\end{lemm}
\begin{proof}
Ceci peut se vérifier localement pour la topologie de Zariski. On peut donc supposer que $S$ est le spectre d'un anneau local et que $M$ est libre. On prend alors une base $e_1,\ldots,e_n$ de $M$ avec $v=\sum v_i e_i$, et on calcule 
\begin{align*}
\Det{}(\tau_v)(e_1\wedge \ldots \wedge e_n) & = \tau_v(e_1)\wedge \ldots \wedge \tau_v(e_n) \\
& = (e_1 - \frac{b_q(v,e_1)}{q(v)} v) \wedge \ldots \wedge (e_n-\frac{b_q(v,e_n)}{q(v)} v) \\
& = \big(1 - \sum v_i \frac{b_q(v,e_i)}{q(v)}\big) (e_1\wedge \ldots\wedge e_n)
\end{align*}
or $2q(v)=b_q(v,v)=\sum v_i b_q(v,e_i)$ donc le facteur entre parenthèses vaut bien $(-1)$.
\end{proof}

\begin{theo}[Kneser] \label{Kneser_theo}
Soit $(q,M)$ un module quadratique sur un anneau local $R$, et soit $N$ et $N'$ deux sous-modules libres de $M$ tels que $q_{|N}$ soit régulière, et munis d'un isomorphisme $t: N \to N'$ de modules quadratiques $(q_{|N},N) \to (q_{|N'},N')$. Alors $t$ peut se prolonger en un élément de $\faisorthO_{M,q}(R)$ et de plus, cet élément peut-être choisi comme un produit de réflexions orthogonales si l'une des conditions suivantes est satisfaite: 1. le corps résiduel $k$ de $R$ est différent de $\mathbb{F}_2$, et la forme $q_{k}$ n'est pas identiquement nulle. 2. le corps résiduel de $R$ est $\mathbb{F}_2$, et la forme $q_{\mathbb{F}_2}$ n'est pas identiquement nulle sur $M^\orth$. 
\end{theo}
\begin{proof}
Voir \cite{Kneser02}, Ch. I, (4.4) et (4.5) dans le cas où $E=H$ avec la notation de \loccit.
\end{proof}
Attention, la condition 2 doit être comprise comme non satisfaite si $M^\orth=\{0\}$, donc en particulier, ce théorème ne dit rien sur la décomposition en réflexions lorsque le corps résiduel est $\mathbb{F}_2$ et le module est régulier.

\begin{theo}[Cartan-Dieudonné] \label{Cartan-Dieudonne_theo}
Si $(M,q)$ est un module quadratique régulier sur un anneau local, alors tout élément de $\faisorthO_{M,q}(R)$ se décompose en produit de réflexions orthogonales, sauf peut-être si le corps résiduel est $\mathbb{F}_2$ et si le rang de $M$ est inférieur où égal à $4$.
\end{theo}
\begin{proof}
On prend $N=N'=M$ dans le théorème précédent. Si le corps résiduel n'est pas $\mathbb{F}_2$, alors la condition 1 est vérifiée. Si le corps résiduel est $\mathbb{F}_2$, alors voir \cite[Ch. I, (4.6)]{Kneser02}.
\end{proof}

Pour un module semi-régulier, on a également:
\begin{theo}\label{Cartan-Dieudsemi_theo}
Si $(M,q)$ est un module semi-régulier sur un corps $k$, alors tout élément de $\faisorthO_{q}(k)$ est un produit de réflexions orthogonales.
\end{theo}
\begin{proof}
Voir \cite[Ch. I, (3.5)]{Kneser02}.
\end{proof}

\begin{rema} \label{contreExCartanDieudonne_rema}
Pour une module semi-régulier sur un anneau local, tout élément du groupe orthogonal n'est pas forcément un produit de réflexions. C'est déjà faux pour le module hyperbolique $\hypq_{2n+1}$, $n\geq 1$ sur $\mathbb{F}_2[x]/x^2$: si $e$ est le vecteur de base portant $\fq{1}$, toute réflexion orthogonale $\tau_v$ préserve le sous-espace engendré par $e$, puisque dans ce cas de caractéristique $2$, c'est le radical polaire. Donc, $\tau_v(e)=\lambda e$. Il suit que $\lambda e=\tau_v(e)=e-\frac{b_q(e,v)}{q(v)}v$, donc soit $b_q(e,v)=0$ et $\tau_v(e)=e$, soit $b_q(e,v)\neq 0$, et $v$ est proportionnel à $e$, donc $\tau_v(e)=-e=e$. Dans les deux cas, $\tau_v(e)=e$ pour tout vecteur $v$ et donc $f(e)=e$ pour tout produit de réflexions. Ainsi, l'application égale à l'identité sur le sous-espace sous-jacent à $\hypq_{2n}$ et envoyant $e$ sur $(1+x)e$ est bien dans le groupe orthogonal puisque $(1+x)^2=1$, mais ne peut être un produit de réflexions. 
\end{rema}

\subsection{Algèbres d'Azumaya à involution} \label{AlgInv_sec}

Passons maintenant en revue différents types d'involutions sur des algèbres d'Azumaya. Notre but n'est pas d'être exhaustif, mais de préparer le matériel qui interviendra dans la partie sur les groupes réductifs. Rappelons qu'une involution sur un anneau $A$ est un anti-auto\-mor\-phisme d'ordre $2$, c'est-à-dire un endomorphisme de $A$ qui vérifie $\sigma(xy)=\sigma(y)\sigma(x)$ pour $x,y \in A$ et $\sigma^2=\id$.

\begin{defi}
Une $\faisO_S$-algèbre d'Azumaya \emph{à involution de première espèce} est un faisceau en $\faisO_S$-algèbres à involution (structure $\alginv$ des exemples \ref{structures_exem}), dont le faisceau en algèbres sous-jacent est un faisceau en $\faisO_S$-algèbres d'Azumaya et dont l'involution est $\faisO_S$-linéaire. Il est immédiat que sur une base affine, cette structure est l'image par le foncteur $\wW \circ \tilde{(-)}$ d'une algèbre à involution de première espèce.
\end{defi}
Lorsque la base $S$ est affine et égale à $\Spec(R)$, on retrouve la définition classique en prenant les sections globales (\cf \ref{Walg_rema}): une algèbre d'Azumaya de première espèce sur un anneau $R$ est une algèbre d'Azumaya $A$ sur $R$ munie d'une involution $\sigma$ qui est $R$-linéaire. 

Si $M$ est un $\faisO_S$-module localement libre de type fini et $b:M \to M^\dual$ est un isomorphisme $\delta$-symétrique ($b=\delta b^\dual \circ \bid$ où $\bid: M \to (M^\dual)^\dual$ est l'identification canonique, et $\delta \in \faismu_2(S)$), autrement dit $(M,b)$ est un $\faisO_S$-module bilinéaire $\delta$-symétrique non dégénéré, alors on peut lui associer une involution $\invadj_b$ sur $A=\faisEnd_M$, donnée sur les points par $\invadj_b(x) = b^{-1} x^\dual b$, où $\faisEnd_M^\opp$ est naturellement identifiée avec $\faisEnd_{M^\dual}$, par l'application $x \mapsto x^\dual$. De manière équivalente, si on identifie $M \otimes M \isoto M \otimes M^\dual \isoto \faisEnd_M$ par $m_1 \otimes m_2 \mapsto m_1 b(m_2)$, alors $\sigma_b$ correspond à l'échange des facteurs sur $M \otimes M$. 
\begin{defi} \label{invadj_defi}
Si $(M,b)$ est un module bilinéaire symétrique non dégénéré (resp. $(M,q)$ un module quadratique régulier), on appelle \emph{involution adjointe à $b$} (resp. à $q$) l'involution $\invadj_b$ (resp. $\invadj_{b_q}$ définie ci-dessus).
\end{defi}
Notons que si $\invadj_b = \invadj_{b'}$, alors $\int_b = \int_{b'}$ et donc $b=\lambda b'$ pour un certain $\lambda$ inversible, puisque $A$ est de centre $\faisO_S$. 

\begin{prop} \label{invlocale_prop}
Localement pour la topologie étale, toute $\faisO_S$-algèbre à involution $(A,\sigma)$ de première espèce est isomorphe à une involution $(\faisEnd_V,\invadj_b)$ pour un certain $\faisO_S$-module bilinéaire $b$, qui est $\delta$-symétrique pour un certain $\delta \in \faismu_2(S)$, qui ne dépend que de l'involution de départ.
\end{prop}
\begin{proof}
Puisque c'est une propriété locale, on se ramène au cas où $S=\Spec(R)$ et où $A$ est $\faisM_n$ par \ref{Azumayaforme_prop}. La composition $\sigma \circ \transp(-)$ est donc un automorphisme de $\setM_n(R)$. Quitte à localiser encore $R$, il est donc intérieur (voir \ref{secPGLn_prop}), égal à $\int_b$ pour un certain $b \in \setGL_n(R)$. Or comme $\sigma^2=\id$, on doit avoir $\transp b b^{-1} \in R^\times$. Or $b=\delta \transp b =\delta^2 b$ donc $\delta\in\faismu_2(R)$. L'élément $b$ est celui cherché, à l'identification canonique de $R^n$ avec $(R^n)^\dual$ près. Il est fixé à un élément du centre près, le produit $\transp b b^{-1}$ est donc canonique. Par recollement ($\faismu_2$ est un faisceau), on construit ainsi l'élément $\delta \in \faismu_2(S)$ recherché, qui ne dépend que de $\sigma$. 
\end{proof}

\begin{defi}\label{involutionsymplectique_defi}
Le type d'une $\faisO_S$-algèbre à involution $A$ de première espèce est l'élément de $\faismu_2(S)$ obtenu par la construction précédente. Si cet élément est $1$, on dit que l'involution est \emph{orthogonale}, et s'il est $-1$, on dit que l'involution est \emph{faiblement symplectique}.\footnote{\label{orthsymp_foot}Attention, notre terminologie diffère de celle de \cite{bookinv}, dans lequel il n'y a pas d'involution orthogonale sur un corps de caractéristique $2$, où elles sont toutes dénommées symplectiques. Nous choisissons au contraire de dire qu'une telle involution est à la fois orthogonale et (faiblement) symplectique.} Si de plus $(A,\sigma)$ est étale localement de la forme $(\faisEnd_V,\invadj_b)$ pour un module $b$ alterné, alors on dit que l'involution est \emph{symplectique}.
\end{defi}

\begin{rema}
Notons que tout $\delta \in \faismu_2(S)$ est effectivement le type d'une involution de première espèce. Par exemple, sur $A=\setM_2$, alors l'application $a \mapsto u\cdot \transp a\cdot u^{-1}$ avec $u=\left(\begin{smallmatrix} 0 & \delta \\ 1 & 0\end{smallmatrix}\right)$ est de type $\delta$. 

Lorsque $2 \in \Gamma(S)^\times$ et que $S$ est connexe, une involution est donc soit orthogonale soit faiblement symplectique et ces deux cas s'excluent. De plus les involutions faiblement symplectiques et symplectiques coïncident. Lorsque $2=0 \in \Gamma(S)$, au contraire, les types orthogonal et faiblement symplectique se confondent$^{\eqref{orthsymp_foot}}$, et lorsque $S$ n'est pas un corps, il y a donc d'autres types, même localement, si $\mu_2(S)$ est strictement plus gros que $(\ZZ/2)(S)$. 
\end{rema}

\begin{lemm} \label{TrdNrdSigma_lemm}
Soit $(A,\sigma)$ une $\faisO_S$-involution de première espèce. Alors on a $\trd \circ \sigma = \trd$ et $\nrd \circ \sigma = \nrd$.
\end{lemm}
\begin{proof}
Il suffit de vérifier ces égalités étale localement. On peut donc supposer que l'algèbre est déployée, et par un argument limite, qu'on est sur un anneau local. Dans ce cas, puisque tout automorphisme est intérieur, l'involution $\sigma$ est de la forme $\int_g \circ \transp$, la trace réduite coïncide avec la trace et la norme réduite avec le déterminant (par construction). La formule est alors classique.
\end{proof}

Si $(A,\sigma)$ est une $\faisO_S$-algèbre à involution de première espèce, on définit le sous-$\faisO_S$-module des éléments symétriques de $(A,\sigma)$ par
$$\faisSym_{A,\sigma}(T) =\{a \in A(T)\text{ t.q. }\sigma(a)=a\}.$$
On définit également le sous-modules des éléments alternés 
$$\faisAlt_{A,\sigma}=\Im (\sigma -\id) \subseteq A$$
comme le faisceau Zariski image de l'application $\sigma -\id$. En d'autres termes, pour tout élément $x$ de $\faisAlt_{A,\sigma}(T)$, il existe un ouvert Zariski $U$ de $T$ tel que $x_{|U}$ soit de la forme $\sigma(a)-a$, avec $a \in A(U)$. 
\begin{lemm} \label{symAsigma_lemm}
Soit $(M,q)$ un $\faisO_S$-module quadratique régulier de rang constant $n$.
Si $(A,\sigma)$ est une forme \fppf\ de $(\faisEnd_{M},\invadj_q)$, alors 
\begin{enumerate}
\item \label{SymLoc_item} $\faisSym_{A,\sigma}$ est un $\faisO_S$-module localement libre de rang constant $n(n+1)/2$.
\item \label{AltLoc_item} $\faisAlt_{A,\sigma}$ est un $\faisO_S$-module localement libre de rang constant $n(n-1)/2$.
\item \label{SymAltorth_item} L'orthogonal de $\faisSym_{A,\sigma}$ pour la $\faisO_S$-module bilinéaire $A$ muni de $(x,y)\mapsto \trd(xy)$ est $\faisAlt_{A,\sigma}$. 
\end{enumerate}
\end{lemm}
\begin{proof}
Les points \ref{SymLoc_item} et \ref{AltLoc_item} peuvent se vérifier \fppf-localement. On peut donc supposer $(A,\sigma)=(\faisEnd_M,\invadj_q)$.
On définit un morphisme $(M \otimes M) \isoto \faisEnd_M$ en envoyant $m_1 \otimes m_2$ sur $m_1 b_q (m_2,-)$. Ce morphisme est un isomorphisme: cela se vérifie localement, lorsque $M$ est libre. Par cet isomorphisme, $\invadj_q$ correspond donc à l'échange des facteurs sur $M \otimes M$, et toujours localement, en prenant une base, on vérifie alors que les éléments symétriques de $M \otimes M$ sont bien un sous-module libre de la dimension requise, et de même pour les éléments alternés.

Pour l'orthogonalité, on utilise le lemme \ref{TrdNrdSigma_lemm} pour vérifier l'inclusion $\faisAlt_{A,\sigma}\subseteq \faisSym_{A,\sigma}^\perp$, puis le fait que sur tous les corps résiduels, on doit avoir égalité pour des raisons de dimension (la forme est régulière par \ref{trdreg_lemm}), ce qui implique l'égalité sur les anneau locaux des points par le lemme de Nakayama, et donc globalement.
\end{proof}
\medskip

Introduisons maintenant la notion de paire quadratique, Nous nous limiterons au cas de degré pair. Ce qui suit est essentiellement le contenu de \cite[Ch 1, §5, B]{bookinv}, légèrement adapté pour fonctionner sur une base quelconque.$^{\eqref{orthsymp_foot}}$ 

\begin{defi} \label{pairesquad_defi}
Une paire quadratique est un triplet $(A,\sigma,f)$ où $A$ est une algèbre à involution orthogonale, de degré pair (sur chaque composante connexe), et $f:\faisSym_{A,\sigma} \to \faisO_S$ est un morphisme de $\faisO_S$-modules qui satisfait à $f(a+\sigma(a))=\trd(a)$. On vérifie facilement que cette structure définit un champ, par la proposition \ref{champstruc_prop}, qu'on note $\PairesQuadn{2n}$.
\end{defi}

Si $(M,q)$ est un module quadratique régulier de rang $2n$, notons $\phi_q$ l'isomorphisme $M \otimes M \to \faisEnd_M$ décrit dans la preuve de \ref{symAsigma_lemm}, et $\invadj_q=\invadj_{b_q}$ l'involution décrite avant la proposition \ref{invlocale_prop}.

\begin{prop} \label{deffq_prop}
Il existe un unique morphisme $\faisO_S$-linéaire $f_q: \faisSym_{\faisEnd_M,\invadj_q}\to \faisO_S$ tel que $f_q \circ \phi_q (m \otimes m) = q(m)$ pour tout $T$-point $m$ de $M$, pour tout $T$ sur $S$, et $(\faisEnd_M, \invadj_q, f_q)$ est alors une paire quadratique. Cette construction est fonctorielle en $q$. Réciproquement, étant donné une paire quadratique $(\faisEnd_M,\sigma,f)$, il existe un module quadratique $(M,q)$ tel que $\invadj_q=\sigma$ et $f_q=f$, et $q$ est déterminé à un élément de $\Gamma(S)^\times$ près.
\end{prop}
\begin{proof}
Une preuve est donnée dans \cite[Ch. I, (5.11)]{bookinv} pour les corps, et elle est valable sans changement sur n'importe quelle base lorsque $M$ est libre. Par unicité de $f_q$ dans une telle situation, le cas général suit par descente. 
\end{proof}

Nous noterons $\invadj_{2n}$ au lieu de $\invadj_{\hypq_{2n}}$ et $f_{2n}$ au lieu de $f_{\hypq_{2n}}$. 

On vérifie que sur la base de $\faisO_S^{2n}$ telle que $\hypq_{2n}$ est donnée par $\hypq_{2n}(x_1,\ldots,x_{2n})=\sum_1^n x_{2i-1} x_{2i}$, alors $\invadj_{2n}=\int_H\circ \transp(-)$ où $H=\sum_1^n (E_{2i-1,2i}+E_{2i,2i-1})$ (N.B. $H=H^{-1}=\transp H$). Le module $\faisSym_{\faisM_{2n},\invadj_{2n}}$ admet alors pour base les $E_{2i-1,2j-1}+E_{2i,2j}$, $1\leq i,j\leq n$ ajoutés aux $E_{2i-1,2j}+E_{2i,2j-1}$, $1\leq i\neq j \leq n$ et aux $E_{2i,2i-1}$ et $E_{2i-1,2i}$, $1\leq i \leq n$. Le morphisme $f_{2n}$ est alors donné sur cette base par $f_{2n}(E_{i,i}+E_{i+1,i+1})=1$ alors que tous les autres sont envoyés sur $0$. En d'autres termes, c'est la somme d'un terme diagonal sur deux (sachant que les autres sont identiques).

\begin{coro} \label{pairequadformes_coro}
Toute paire quadratique est une forme étale de la paire $(\faisM_{2n},\invadj_{2n},f_{2n})$.
\end{coro}

Rappelons que les modules quadratiques réguliers sont les formes du module hyperbolique de rang $2n$ et qu'ils sont les objets du champ $\RegQuadn{2n}$ défini en \ref{champquadreg_defi}. La proposition précédente permet ainsi de définir un foncteur fibré $\RegQuadn{2n}\to \PairesQuadn{2n}$, qui envoie l'objet $(M,q)$ sur $(\faisEnd_M,\invadj_q,f_q)$.
\medskip

Passons maintenant aux algèbres à involution de deuxième espèce.

\begin{defi}
Une algèbre à involution de deuxième espèce sur un anneau $R$ est une algèbre d'Azumaya $A$ sur $Z$, qui est une algèbre étale finie de degré $2$ sur $R$, munie d'une involution $R$-linéaire $\sigma$ qui se restreint à $Z$ en son unique automorphisme de $R$-algèbres partout localement non trivial (voir cor. \ref{etaledegre2_coro}). 
\end{defi}

\begin{defi} \label{alginvdeuxesp_defi}
Un faisceau en algèbres d'Azumaya \emph{avec involution de deuxième espèce} est un faisceau en $\faisO_S$-algèbres d'Azumaya $A$ sur une $\faisO_S$-algèbre étale finie $Z$ de degré $2$ sur $\faisO_S$, muni d'une involution qui se restreint à $Z$ en son unique automorphisme de $\faisO_S$-algèbre partout localement non trivial. Clairement, dans le cas où $S$ est affine, c'est l'image par le foncteur $\wW\circ \tilde{(-)}$ de la structure de la définition précédente.
\end{defi}

Soit $\exch$ l'involution sur $\setM_n(\ZZ)\times \setM_n^\opp(\ZZ)$ définie par $(a,b) \mapsto (b,a)$. En considérant le morphisme de faisceau associé à $\exch$, et en étendant la base de $\ZZ$ à $S$, on obtient donc un faisceau en $\faisO_S$-algèbres à involution $(\faisM_{n,S} \times \faisM_{n,S}^\opp, \exch_S)$ de deuxième espèce, que nous appellerons \emph{déployée}.

\begin{prop} \label{alginvdeuxesp_prop}
Les $S$-formes (au sens de la structure $\faisO_S$-$\alginv$) de la $\faisO_S$-algèbre à involution déployée $(\faisM_{n,S} \times \faisM_{n,S}^\opp, \exch_S)$ sont les $\faisO_S$-algèbres d'Azumaya de degré $n$ avec involution de deuxième espèce.
\end{prop}
\begin{proof}
Par les propositions \ref{algetfinies_prop} et \ref{Azumayaforme_prop}, toute forme de l'algèbre à involution déployée est bien une $\faisO_S$-algèbre à involution de deuxième espèce. 

Réciproquement, étant donné une $\faisO_S$-algèbre à involution de deuxième espèces $(A,Z)$, par les mêmes propositions, localement pour la topologie étale, on peut supposer que $Z=\faisO_S\times \faisO_S$ et $A=\faisM_{n,S}\times \faisM_{n,S}^\opp$. Un court calcul utilisant les deux projecteur canoniques du centre $\faisO_S \times \faisO_S$ montre alors que l'involution est de la forme $(\sigma \times (\sigma^{-1})^\opp) \circ \exch$ où $\sigma$ est un automorphisme de $\faisM_n$. Voir par exemple \cite[2.14]{bookinv}. Or dans ce cas, le morphisme $(\sigma \times \id)$ définit un isomorphismes de $\faisO_S$-algèbres à involution de $(\faisM_n \times \faisM_n^\opp,\exch)$ vers $(\faisM_n \times \faisM_n^\opp,(\sigma \times (\sigma^{-1})^\opp) \circ \exch)$. On a donc montré que localement pour la topologie étale, toute $\faisO_S$-algèbre à involution de deuxième espèce est déployée.
\end{proof}

\subsection{Algèbres de Lie}

Soit $\faisDual_S$ le schéma des nombres duaux sur $S$, défini comme le spectre de la $\cO_S$-algèbre $\cO_S[t]/t^2$. Son morphisme structural $\faisDual_S \to S$ admet une section canonique correspondant à l'application envoyant $t$ sur $0$. Voir \cite[Exp. II, \S 2]{sga3} pour les détails. On a $\faisO_S(\faisDual_{T})=\faisO_S(T)[t]/t^2$ pour tout $T\to S$. 

Dans \cite[Exp. II]{sga3}, l'algèbre de Lie $\faisLie_G$ d'un schéma en groupes $G$ sur une base $S$ est défini (en 3.9.0.1) comme le foncteur de points tiré le long de la section unité $S \to G$ du fibré tangent $\faisTan_G$ à $G$, lui-même défini (en 3.1) comme le foncteur de points des $T\mapsto \faisTan_G(T)=G(\faisDual_{T})$. Ce fibré tangent et l'algèbre de Lie sont munis d'une structure de foncteur en groupes, et on a ainsi pour tout $T\to S$ une suite exacte scindée de groupes
\begin{equation} \label{seLie_eq}
\xymatrix{
\faisLie_G(T) \ar[r]^-{i} & G(\faisDual_{T}) \ar[r]^-{p} & G(T) \ar@/^1ex/[l]^{s}
}
\end{equation}
qui identifie $\faisLie_G(T)$ au noyau du morphisme $p$, et où $p$ et $s$ sont induits par les deux morphismes mentionnés plus haut entre $S$ et $\faisDual_S$.
Le foncteur $\faisLie_G$ s'identifie (par 3.3) à la fibration vectorielle $\vV(\omega^1_{G/S})$, au sens de \cite[Exp. I, déf. 4.6.1]{sga3}, associée au $\cO_S$-module cohérent $\omega^1_{G/S}$ des différentielles relatives de $G$ par rapport à $S$ tiré à $S$ par la section unité. En particulier, ce foncteur 
est muni d'une structure de $\faisO_S$-module. Son crochet de Lie est défini (en 4.7.2) par l'intermédiaire de la représentation adjointe, et $G\mapsto \faisLie_G$ définit naturellement un foncteur des schémas en groupes vers les $\faisO_S$-modules munis d'une loi de composition bilinéaire (voir \cite[Exp. II, (i) avant la prop. 4.8]{sga3}). 

\begin{prop} \label{LieOS_prop}
Dans le cas du groupe additif sous-jacent au $\faisO_S$-module $\faisO_S$ lui-même, on a les identifications $\faisTan_{\faisO_S}= \faisO_S[t]/t^2$ et $\faisLie_{\faisO_S}=\faisO_S$, muni du crochet nul.
\end{prop}
\begin{proof}
C'est immédiat.
\end{proof}

\begin{prop}
Pour toute extension $T \to S$, on a $(\faisLie_G)_{T} = \faisLie_{G_{T}}$.
\end{prop}
\begin{proof}
Voir \cite[Exp. II]{sga3}, proposition 3.4 et 4.1.2.0.
\end{proof}
Posons $\mLie_{G}=(\omega^1_{G/S})^\dual$.
\begin{prop}
Si $G$ est lisse sur $S$, alors $\omega^1_{G/S}$ est localement libre de type fini et $\faisLie_G = \wW(\mLie_{G})$. 
Si de plus $S$ est affine égal à $\Spec(R)$, on a donc $\mLie_G=\tilde{A}$ où $A=\mLie_G(S)=\faisLie_G(S)$ est un module localement libre sur $R$ et pour schéma affine $T=\Spec(R')$ au-dessus de $S$, on a
$$\faisLie_G(T)=A_{R'}.$$
\end{prop}
\begin{proof}
Le faisceau $\omega^1_{G/S}$ est localement libre de type fini par \cite[Exp. II, th. 4.3]{sga1}. On a $\faisLie_G=\vV(\omega^1_{G/S})$ et $\vV(\omega^1_{G/S})=\wW((\omega^1_{G/S})^\dual)$ par \cite[Exp. I, cor. 4.6.5.1]{sga3}.
Le cas affine s'obtient par la suite d'identifications
$$\faisLie_G(T) = (\faisLie_G)_{T}(T) = \wW(\tilde A)_{T}(T)=\wW((\tilde A)_{T})(T)=\wW(\tilde{A_{R'}})(T)=A_{R'}.$$ 
\end{proof}
\begin{rema}
Dans le cas affine et lisse de la proposition précédente, le module $A$ est muni d'une structure d'algèbre de Lie sur $R$ par pleine fidélité de $\wW$ et comme cette algèbre permet de retrouver complètement le foncteur $\faisLie_G$, nous appellerons parfois $A$ l'algèbre de Lie de $G$. 
\end{rema}

Le lemme suivant justifie les calculs des algèbres de Lie de nombreux groupes ``classiques''.

\begin{lemm}\label{sous_algebre_lemm}
Soit $G$ un foncteur en groupes représentable sur une base $S$ et soit $f:F\to G$ un sous-groupe représentable. Alors $\faisLie_f:\faisLie_F\to \faisLie_G$ identifie $\faisLie_F$ à une sous-$\faisO_S$-algèbre de Lie de $\faisLie_G$. 
\end{lemm}
\begin{proof}
Le fait que c'est un sous-$\faisO_S$-module est clair par la définition utilisant les nombres duaux.
\end{proof}

Soit $A$ une $\faisO_S$-algèbre associative et unitaire (mais non nécessairement commutative), localement libre comme $\faisO_S$-module et telle que $\faisO_S$ est central dans $A$. On note $\alL_A$ le foncteur en algèbres de Lie de $\faisO_S$-module sous-jacent $A$ et de crochet $[a,b]=ab-ba$. 

\begin{prop} \label{LieGLn_prop}
L'algèbre de Lie $\faisLie_{\faisGL_{1,A}}$ est isomorphe à $\alL_A$, et par cet isomorphisme, la représentation adjointe correspond à la conjugaison.
\end{prop}
\begin{proof}
Le morphisme naturel $A\otimes_{\faisO_S} \faisTan_{\faisO_S} \to \faisTan_A$ est un isomorphisme de $\faisO_S$-modules car $A$ est localement libre (voir \cite[Exp. II, 4.4.2]{sga3}). Cela permet d'identifier l'espace tangent $\faisTan_A(T)=A(I_{T})=A(T)[t]/t^2$, et la suite \eqref{seLie_eq} pour $G=\faisGL_{1,A}$ devient alors
\begin{equation} \label{LieGLn_eq}
\xymatrix{
A(T) \ar[r]^-{i} & A(T)^\times \times A(T) \ar[r]^-{p} & A(T)^\times \ar@/^1ex/[l]^-{s}
}
\end{equation}
avec la loi de groupe additive sur $A(T)$. La représentation adjointe est donc immédiate, et on vérifie facilement que le crochet de Lie est bien celui de $\alL_A$ en utilisant la méthode de \cite[Exp. II, 4.7.3]{sga3}, ce qui revient à calculer $(1+xt)(1+yt')(1-xt)(1-yt')=1+(xy-yx)tt'$ où $t$ et $t'$ sont deux variables de carré nul, intervenant dans $\faisDual_S \times_S \faisDual_S$
\end{proof}

Si $B$ est une $R$-algèbre, soit $\Der_R(B,B)$ le module des différentielles de $B$ dans $B$, c'est-à-dire des applications $R$-linéaires $d:B \to B$ satisfaisant à $d(ab)=d(a)b+ad(b)$. Elles sont automatiquement nulles sur $R$. Muni du crochet $[d,e]=d\circ e-e\circ d$, c'est une algèbre de Lie. Pour tout $\faisO_S$-algèbre $A$, on définit alors $\Der_{\faisO_S-\mathrm{alg}}(A,A)$ comme le sous-ensemble des endomorphismes de $\faisO_S$-module de $A$ dont les éléments vérifient la condition ci-dessus sur les $T$-points pour tout $T\to S$. 
 
Le $R$-module sous-jacent à $B$ muni du crochet $[a,b]=ab-ba$ est également une algèbre de Lie. Ce crochet est nul sur $R$, et définit donc une structure d'algèbre de Lie sur le module quotient $B/R$. L'application $\phi: B \to \Der_R(B,B)$ qui à $a$ associe la dérivation $(x \mapsto ax-xa)$ est un morphisme d'algèbres de Lie (\ie\ respecte le crochet). Son noyau est le centre (multiplicatif) de $B$; elle se factorise donc en une application d'algèbres de Lie $\bar \phi: B/R \to \Der_R(B,B)$. Pour toute $\faisO_S$-algèbre $A$, on définit ainsi une application $\bar \phi: A/\faisO_S \to \Der_{A,A}$.
\begin{defi}
Soit $A$ une $\faisO_S$-algèbre localement libre. Notons $\faisDer_{A,A}$ le foncteur de points
$T \mapsto \Der_{\faisO_{T}-\mathrm{alg}}(A_{T},A_{T})$.
\end{defi}

\begin{prop} \label{algLieAut_prop}
On a un isomorphisme d'algèbres de Lie 
$$\faisLie_{\faisAut^\alg_A} \simeq \faisDer_{A,A}.$$ 
La représentation adjointe est donnée par $\sigma(d)=\sigma \circ d \circ \sigma^{-1}$ pour $\sigma \in \setAut_{\faisO_{T}-\mathrm{alg}}(A_{T})$ et $d \in \Der_{\faisO_{T}}(A_{T},A_{T})$. Lorsque l'algèbre $A$ est séparable\footnote{C'est-à-dire que localement pour la topologie étale, l'algèbre $A(T)$ est séparable.} sur $\faisO_S$ et de centre $\faisO_S$, le morphisme $\bar\phi$ est un isomorphisme, et à travers cet isomorphisme, la représentation adjointe est donnée par l'action usuelle des $\faisO_S$-automorphismes de $A$ sur $A/\faisO_S$, le quotient par le centre. 
\end{prop}
\begin{proof}
En utilisant à nouveau $A_{\faisDual_S}(T)\simeq A(T)[t]/t^2$, on calcule directement que 
$$\faisTan_{\faisAut^\alg_A}(T)=\faisAut^\alg_A(\faisDual_{T})=\setAut_{\faisO_{T}[t]/t^2-\mathrm{alg}}(A_{T}[t]/t^2)$$ 
s'identifie à l'ensemble des couples $(\phi_1,\phi_2)$ où $\phi_1 \in \setAut_{\faisO_{T}-\mathrm{alg}}(A_{T})$, $\phi_2\in \setEnd_{\faisO_{T}-\mathrm{mod}}(A_{T})$ avec $\phi_2(ab)=\phi_1(a)\phi_2(b)+\phi_2(a)\phi_1(b)$ sur les points, et l'algèbre de Lie est le sous-groupe des éléments tels que $\phi_1$ est l'identité, donc $\faisDer_{A,A}$. 
Pour une algèbre séparable, l'application $\phi$ est un isomorphisme localement par \cite[III, th. 1.4 p. 73]{kn-oj}. Pour tout $\sigma \in \setAut(A_{T})$, on a $\phi(\sigma(a))=\sigma \circ \phi(a) \circ \sigma^{-1}$, d'où l'action adjointe annoncée.
\end{proof}

\subsection{Lissité} \label{Lisse_sec}

Pour montrer qu'un groupe algébrique sur $S$ est réductif, sa lissité sur $S$ est un point clé. Le critère de lissité par fibres \cite[IV-4, prop. 17.8.2]{ega} permet de déduire la lissité sur $S$ de la lissité des fibres au-dessus de tous les points de $S$ à condition d'avoir la platitude sur $S$, ce qui n'est pas évident pour des groupes définis comme des foncteurs de points. On rappelle donc un critère issu de \SGAtrois\ qui nous permettra d'obtenir la lissité sans peine dans les cas que nous considérerons. 

\begin{prop} \label{sgcorpslisse_prop}
Soit $G$ un schéma en groupes de type fini sur un corps $k$. Alors $G$ est lisse sur $k$ si et seulement si son algèbre de Lie est de dimension égale à la dimension (de Krull) de $G$.
\end{prop}
\begin{proof}
Soit $g\in G$. Par descente plate de la lissité \cite[Exp. II, cor. 4.13]{sga1}, on peut supposer que $g$ est $k$-rationnel, et par changement de base le long de la multiplication par $g$, on peut supposer que $g$ est l'unité de $G$. Le résultat découle alors de \cite[Exp. II, th. 5.5]{sga1}.
\end{proof}

\begin{prop} \label{lissite_prop}
Soit $G$ un groupe algébrique localement de présentation finie sur une base $S$ noethérienne réduite, dont les fibres sont lisses, connexes et de dimension localement constante. Alors $G$ est lisse sur $S$.
\end{prop}
\begin{proof}
Si $S$ n'est pas connexe, $G$ se décompose en une union disjointe de groupes sur les composantes connexes de $S$. On peut donc supposer que $S$ est connexe. On remarque ensuite que le morphisme structurel $G\to S$ est scindé par la section unité. On en déduit que $G$ est connexe puisque $S$ est connexe et les fibres sont connexes. Le résultat suit alors de \cite[exp. VI, cor. 4.4]{sga3}.
\end{proof}

\begin{exem} \label{munlisse_exem}
Le schéma en groupes $\faismu_{n,S}$ est lisse si et seulement si $n$ est premier aux caractéristiques résiduelles de $S$. Voir \cite[Exp. VIII, prop. 2.1]{sga3}.
\end{exem}
Signalons enfin un résultat bien utile pour calculer la dimension des fibres géométriques de schémas en groupes. 

\begin{prop} \label{dimsec_prop}
Soient $G$ et $H$ deux schémas en groupes localement de type fini sur un corps $k$, et $f:G \to H$ un morphisme quasi-compact surjectif de $k$-schémas en groupes. Alors $\dim G = \dim H +\dim \ker(f)$.  
\end{prop}
\begin{proof}
Voir \cite[Exp. VI, 2.5.2]{sga3}.
\end{proof}
Remarquons que les hypothèses sont vérifiées si $G$ et $H$ sont affines de type fini sur $k$ et si $G \to H$ est un épimorphisme de faisceaux \fppf. En effet, tout morphisme provenant de $G$ est alors quasi-compact, et un épimorphisme de faisceaux \fppf\ est automatiquement surjectif comme morphisme de schémas.

\subsection{Données radicielles}

Rappelons brièvement la notion de donnée radicielle, qui est la donnée combinatoire qui permet la classification des groupes réductifs déployés, et comment un groupe réductif fournit une donnée radicielle, afin de pouvoir par la suite faire correspondre les groupes classiques avec la classification de \SGAtrois.

Une donnée radicielle est la donnée d'un $\ZZ$-module libre de rang fini $M$, d'un sous-ensemble fini $\Roots$ de $M$ dont les éléments sont appelés \emph{racines}, d'une injection $\Roots \to M^\cdual$, habituellement notée $\alpha \mapsto \alpha^\cdual$, dont les images sont appelées \emph{coracines}, et qui satisfont à deux axiomes rappelés en \cite[Exp. XXI, déf. 1.1.1]{sga3}. Cette donnée est plus générale que celle d'un système de racines pour deux raisons: premièrement, il peut y avoir une partie de $M$ qui n'est pas engendrée par les racines (même après tensorisation par $\QQ$) et qui correspond au fait que le groupe peut avoir un radical non trivial s'il n'est pas semi-simple. Deuxièmement, par rapport à un système de racines, on conserve la donnée du réseau des caractères du tore entre le réseau des racines et celui des poids. 

Soit $G$ un groupe algébrique réductif sur $S$ contenant un tore maximal $T$ qu'on suppose déployé, c'est-à-dire isomorphe à $D_S(M)$ pour un groupe abélien libre $M$, qui s'identifie donc aux caractères de $T$. Les racines de $G$ sont alors les éléments $\alpha \in M$ pour lesquels, sur toutes les fibres, $\alpha$ reste non nul et l'espace propre correspondant dans la représentation adjointe de $G$ restreinte à $T$ est également non nul \cite[Exp. XIX, déf. 3.2]{sga3}. La coracine $\alpha^\cdual \in M^\cdual$ correspondant à $\alpha$ est obtenue par unicité dans \cite[th. 2.4.1]{D}, qui reprend \cite[Exp. XX, th. 2.1]{sga3}. Tout cela forme bien une donnée radicielle par \cite[Exp. XXII prop. 1.14]{sga3}. Pour identifier les données radicielles de groupes réductifs donnés, nous utiliserons donc \cite[th. 2.4.1]{D}, qui fait intervenir les morphismes exponentiels, qu'il faudra donc avoir identifiés avant, toujours par unicité, voir \cite[th. 2.3.4]{D} ou \cite[Exp. XX, th. 1.5]{sga3}. 

\begin{rema} \label{signeAccouplement_rema}
Attention, les énoncés de \cite[th. 2.4.1]{D} diffèrent d'un signe dans l'accouplement entre espaces propres de racines opposées de l'algèbre de Lie (désigné par $(X,Y)\mapsto XY$). La raison de cette différence de signe, qui est un choix à faire, est expliquée par la note de Demazure dans le présent volume, et il en ressort que le choix de \cite[th. 2.4.1]{D} est meilleur. Nous nous y tiendrons donc, et le lecteur devra donc multiplier l'accouplement par $(-1)$ s'il veut utiliser les formules de \cite{sga3}.
\end{rema}

\subsection{Groupes d'automorphismes}

Rappelons quelques résultats sur les groupes d'automorphismes de groupes réductifs, qui nous seront utiles par la suite.

Un schéma en groupes réductif $G$ est en particulier un faisceau étale en groupes, et a donc à ce titre un faisceau en groupes d'automorphismes $\faisAut_G=\faisAut_G^\gr$, défini par \ref{isoaut_nota}. 

Le quotient (en tant que faisceaux \fpqc) d'un groupe réductif par son centre est représentable, affine sur la base, et est un groupe réductif semi-simple adjoint, appelé groupe adjoint de $G$ et qu'on notera ici $G_\adjoint$.

\begin{theo} \label{automorphismes_theo}
Pour tout groupe réductif sur $S$, on a une suite exacte de $S$-faisceaux étales en groupes
$$1 \to G_\adjoint \to \faisAut_G \to \faisAutExt_G \to 1$$
où la première flèche envoie un élément sur l'automorphisme intérieur associé, et où $\faisAutExt_G$ est défini comme le quotient.
\begin{enumerate}
\item Lorsque $G$ est déployé et donc muni d'un épinglage, on a un isomorphisme canonique entre $\faisAutExt_G$ et le groupe des automorphismes de la donnée radicielle épinglée, et la suite est scindée.
\item Lorsque $G$ est semi-simple simplement connexe ou adjoint, ce dernier groupe coïncide avec celui des automorphismes du diagramme de Dynkin.
\end{enumerate}
\end{theo}
\begin{proof}
Voir \cite[Exp. XXIV, th. 1.3]{sga3} pour les deux premiers points, suivi de \cite[Exp. XXI, prop. 6.7.2 et cor. 7.4.5]{sga3} pour le dernier.
\end{proof}

Il se trouve que $\faisAut_G^\gr$ est alors représentable comme conséquence de ce théorème, mais nous n'en aurons pas besoin. Nous utiliserons uniquement ce faisceau en groupes pour tordre un groupe déployé donné, et obtenir tous les groupes réductifs de même donnée radicielle déployée. 

\section{Groupes de type $A_n$} \label{An_sec}

\subsection{Groupe linéaire}

Le groupe linéaire est probablement le groupe le plus classique qui soit, et comme la plupart des autres groupes d'automorphismes sont définis comme des sous-groupes du groupe linéaire, il est naturel de le traiter en premier. C'est de plus un bon exemple pour se faire la main. Il a été défini en \ref{GLn_defi}, et sa représentation adjointe dans son algèbre de Lie est décrite en \ref{LieGLn_prop}. C'est le prototype du groupe réductif qui n'est pas semi-simple.

\begin{prop} \label{GL1red_prop} 
Pour toute $\faisO_S$-algèbre $A$ localement libre de type fini séparable, le schéma en groupes $\faisGL_{1,A}$ est réductif. 
\end{prop}
\begin{proof}
Comme cela se vérifie localement sur la base, on se ramène au cas $S$ affine et $A$ libre. Par construction (voir prop. \ref{GLrepr_prop}), le schéma $\faisGL_{1,A}$ est alors affine et ouvert dans $A$ qui est un fibré vectoriel sur $S$, donc lisse. Il suffit alors de vérifier que les fibres géométriques sont réductives et connexes. On est donc ramené à la théorie sur les corps algébriquement clos. Or sur un tel corps, une algèbre séparable est un produit d'algèbres de matrices par la proposition \ref{sepsurcorps_prop}, et son groupe linéaire est donc bien réductif par la preuve classique.
\end{proof}

Soit $\faisDiag_{n}$ le sous-$S$-faisceau en groupes de $\faisGL_{n}$ formé des matrices diagonales (\ie\ le foncteur de points de ces matrices). C'est évidemment un sous-faisceau en groupes de $\faisGL_{n}$ représentable puisqu'isomorphe à $\faisGm^n$, qui est un tore déployé. Les caractères de ce tore sont donnés par le $\ZZ$-module libre de base $t_1,\ldots,t_n$ où $t_i$ donne la $i$-ème coordonnée d'un élément diagonal.

\begin{prop} \label{toreGLn_prop}
Le sous-groupe $\faisDiag_{n}$ est un tore maximal de $\faisGL_{n}$.
\end{prop}
\begin{proof}
La théorie sur les corps étant supposée connue, on voit immédiatement que c'est un tore maximal sur toute fibre géométrique.  
\end{proof}

Soit $\faisTriSup_{n}$ le sous-$S$-faisceau en groupes de $\faisGL_{n,S}$ formé des matrices triangulaires supérieures. On montre qu'il est représentable comme un sous-schéma fermé de $\faisGL_{n}$ défini par les équations qui fixent la partie sous la diagonale à zéro. 

\begin{prop} \label{borelGLn_prop}
Le sous-groupe $\faisTriSup_{n}$ de $\faisGL_{n}$ est un sous-groupe de Borel au sens de \cite[exp. XIV, déf. 4.5]{sga3}.
\end{prop}
\begin{proof}
La situation sur un corps $k$ étant connue, $(\faisTriSup_{n})_k$ est un sous-groupe de Borel de $\faisGL_{n,k}$ et il est lisse.
Le groupe $\faisTriSup_{n,\ZZ}$ est alors lisse par la proposition \ref{lissite_prop}, ainsi que $\faisTriSup_{n,S}$ par changement de base. Le fait d'être un sous-groupe de Borel se vérifie alors par définition sur les fibres.
\end{proof}

Considérons l'action adjointe de $\faisGL_n$ sur son algèbre de Lie. Par la proposition \ref{LieGLn_prop}, cette algèbre de Lie est $\faisM_n$ muni du crochet classique, et l'action adjointe est donnée sur les points par la conjugaison de matrices. On restreint alors cette action au tore maximal $\faisDiag_n$. Soit $E_{i,j}$ la matrice élémentaire avec un $1$ en ligne $i$ colonne $j$ et zéro ailleurs. Un calcul immédiat montre que $\faisDiag_n$ agit sur $\faisO_S\cdot E_{i,j}$ par $t_i-t_j$. Cela donne donc les racines de $\faisGL_n$ par rapport au tore $\faisDiag_n$.

Par \cite[exp. XXII, th. 1.1]{sga3}, pour chaque racine $\alpha$ il existe un unique morphisme $\exp_\alpha$ du groupe algébrique (additif) correspondant à l'espace propre de $\alpha$ de l'algèbre de Lie vers $\faisGL_n$, qui induise entre algèbre de Lie l'inclusion canonique de cet espace propre dans l'algèbre de Lie de $\faisGL_n$. Un calcul immédiat montre que le morphisme
$$\begin{array}{llll}
\exp_{\alpha_{i,j}}: & \faisO_S\cdot E_{i,j} & \to & \faisGL_{n} \\
& \lambda\cdot E_{i,j} & \mapsto & Id +\lambda\cdot E_{i,j}
\end{array}$$
(décrit sur les points) fournit le candidat.

\begin{prop} \label{donneeradicielleGLn_prop}
La donnée radicielle du $S$-schéma en groupes $\faisGL_n$ relativement au tore maximal déployé $\faisDiag_n$ est 
\begin{itemize}
\item Le $\ZZ$-module libre $M\simeq \ZZ^n$ des caractères de $\faisDiag_n$ (dont $t_1,\ldots,t_n$ forment une base); 
\item Le sous-ensemble fini de celui-ci formé des racines $\alpha_{i,j}=t_i-t_j$ pour $1\leq i,j\leq n$ et $i\neq j$;
\item Le $\ZZ$-module libre $M^\dual\simeq\ZZ^n$ des cocaractères (de base duale $t_1^\cdual,\ldots,t_n^\cdual$);
\item Le sous-ensemble fini de celui-ci formé des coracines $\alpha_{i,j}^\cdual=t_i^\dual-t_j^\dual$ pour $1\leq i,j\leq n$ et $i\neq j$.
\end{itemize}
\end{prop} 
\begin{proof}
On lit les racines dans la décomposition précédente de l'algèbre de Lie, et on vérifie immédiatement que les coracines associées sont bien celles proposées en appliquant le critère-définition \cite[th. 2.4.1]{D} (voir la remarque \ref{signeAccouplement_rema}): l'accouplement entre sous-espaces propres de racines opposées est $\lambda\cdot E_{i,j}\otimes \mu\cdot E_{j,i} \mapsto -\lambda \mu$ et le morphisme de groupes $\alpha_{i,j}^\cdual: \faisGm \to \faisDiag_n$ envoie $u \in \faisGm(S)$ sur la matrice diagonale avec $u$ (resp. $u^{-1}$) en position $i$ (resp. $j$) et $1$ ailleurs. C'est bien ce qui correspond à $t_i^\dual -t_j^\dual$ par l'isomorphisme $\faisDiag_n \simeq D(M)$ (dual de Cartier). 
On vérifie à la main la formule (F) de \loccit, les morphismes exponentiels ayant déjà été décrits.
\end{proof}

\begin{exo} \label{autoDonneeRadGL_exo}
Montrer par un calcul direct que pour $n\geq 2$, si $\phi$ est un automorphisme de $M$ qui permute les racines simples $(t_i-t_{i+1})_{1\leq i\leq n-1}$ et tel que $\phi^\dual$ permute les coracines simples $(t_i^\dual-t_{i+1}^\dual)_{1\leq i\leq n-1}$, alors $\phi$ est l'identité ou la matrice avec des $-1$ sur l'antidiagonale, et $0$ partout ailleurs.
\end{exo}

\begin{prop} \label{centreGLn_prop}
Le centre de $\faisGL_n$ est $\faisGm$ envoyé diagonalement dedans.
\end{prop}
\begin{proof}
Les matrices inversibles qui commutent à toutes les autres sont diagonales, par un calcul classique utilisant les matrices élémentaires. 
\end{proof}

\subsection{Groupe déployé adjoint}

\begin{defi} \label{PGLn_defi}
Soit $\faisPGL_{n,S}$ le groupe algébrique sur $S$ défini par le foncteur de points $\faisAut^{\faisO_S-\alg}_{\faisM_{n,S}}$ (représentable, affine et de type fini sur $S$ par la proposition \ref{autrepr_prop}). Lorsque la base $S$ est implicite, on ne la mentionne plus dans la notation.
\end{defi}
On a bien entendu $\faisPGL_{n,S}=(\faisPGL_{n,\ZZ})_{S}$ par la remarque \ref{Homchgmt_rema}.
\begin{prop}
Pour tout $S$-schéma $T$,
$$\faisPGL_{n,S}(T)=\setAut_{\Gamma(T)\mathrm{-alg}}(\setM_n(\Gamma(T)).$$ 
Pour tout anneau $R$, on a donc $\faisPGL_n(R)=\setAut_{R\mathrm{-alg}}(\setM_n(R))$. 
\end{prop}
\begin{proof}
Il suffit de voir que la flèche naturelle 
$$\faisAut^{\faisO_{S}-\alg}_{\faisM_{n,S}}(T)=\faisAut_{\faisO_{T}\mathrm{-alg}}(\faisM_{n,T})\to \faisAut_{\Gamma(T)\mathrm{-alg}}(\setM_{n,S}(\Gamma(T)))$$ 
est un isomorphisme. Cela provient du même isomorphisme pour $\faisEnd$ car $\faisM_n$ est un $\faisO_S$-module libre et de ce que les conditions à imposer pour être un automorphisme d'algèbre sont évidemment les mêmes des deux côtés de ce dernier isomorphisme.
\end{proof}

On définit alors le morphisme canonique $\faisGL_n \to \faisPGL_n$ en envoyant un élément inversible de $\faisM_n(T)$ vers l'automorphisme conjugaison par cet élément. On obtient ainsi:
 
\begin{prop} \label{secPGLn_prop}
Sur toute base $S$, la suite de faisceaux Zariski 
$$1 \to \faisGm \to \faisGL_n \to \faisPGL_n \to 1$$
est exacte ($\faisGm$ est envoyé diagonalement dans $\faisGL_n$). Il en est de même pour les topologies entre Zariski et canonique, comme étale ou \fppf.
\end{prop}
\begin{proof}
Une suite exacte de faisceaux reste exacte par changement de base, voir \cite[Exp. IV, 4.5.7]{sga3}. Il suffit donc de faire le cas $S=\Spec(\ZZ)$. Le noyau du morphisme $\faisGL_n \to \faisPGL_n$ est bien $\faisGm$, par la proposition \ref{centreGLn_prop}. Il suffit de vérifier la surjectivité de faisceaux sur leurs germes, qui sont des anneaux locaux. Or, sur un anneau local, le théorème de Skolem-Noether vaut: tous les automorphismes des algèbres de matrices sont intérieurs (voir \cite[th. 3.6]{ausgol}). Le foncteur de points $\faisPGL_n$ étant représentable, c'est directement un faisceau pour toute topologie entre la topologie Zariski et canonique, la surjectivité comme faisceau Zariski montre la surjectivité pour les autres topologies. 
\end{proof}

En d'autres termes, le faisceau étale $\faisPGL_n$ est le faisceau associé au préfaisceau qui à $Y$ associe $\faisGL_n(Y)/\faisGm(Y)$.

\begin{coro} \label{PGLncorps_coro}
Pour tout schéma $S$ tel que $\Pic(S)=0$, on a $\faisPGL_n(S) \simeq \faisGL_n(S)/\faisGm(S)$. En particulier, c'est le cas lorsque $S$ est le spectre d'un corps, ou même d'un anneau local.
\end{coro}
\begin{proof}
On considère la suite exacte en cohomologie associée à la suite exacte courte de faisceaux étales de la proposition \ref{secPGLn_prop}, avec l'identification $\Het^1(S,\faisGm)=\Pic(S)$.
\end{proof}

Considérons le sous-groupe de $\faisPGL_n$ image du tore $\faisDiag_n$ par le morphisme $\faisGL_n \to \faisPGL_n$. C'est le quotient de $\faisDiag_n$ par $\faisGm$, c'est donc le dual de Cartier du noyau de l'application somme des coordonnées $\ZZ^n \to \ZZ$. Ce noyau est isomorphe à $\ZZ^{n-1}$ donc le sous-groupe considéré est un tore déployé de rang $n-1$. Appelons-le $\faisPDiag_n$.
 
\begin{coro} \label{PGLToreMax_coro}
Le sous-groupe $\faisPDiag_n$ de $\faisPGL_n$ décrit ci-dessus est un tore maximal (déployé) au sens de \SGAtrois.
\end{coro}
\begin{proof}
Par le corollaire \ref{PGLncorps_coro}, ce tore donne le tore maximal connu de $\faisPGL_n$ sur chaque fibre géométrique (étant donné connue la théorie sur les corps). 
\end{proof}

\begin{prop} \label{algLiePGLn_prop}
L'algèbre de Lie de $\faisPGL_n$ est le $\faisO_S$-module $\faisM_n/\faisO_S$, muni du crochet $[\bar m, \bar n] = \overline{mn-nm}$ sur les points. La représentation adjointe est l'action des automorphismes sur le quotient par le centre. 
\end{prop}
\begin{proof}
C'est une application immédiate de la proposition \ref{algLieAut_prop}.
\end{proof}

\begin{prop}
Le groupe $\faisPGL_n$ est réductif.
\end{prop}
\begin{proof}
Le groupe $\faisPGL_n$ est de type fini et affine par construction (voir la prop. \ref{autrepr_prop}) . Sur un corps $k$, on sait que $\faisPGL_{n,k}$ est lisse, connexe et réductif. La lissité de $\faisPGL_{n,\ZZ}$ résulte donc de la proposition \ref{lissite_prop}, puis celle de $\faisPGL_{n,S}$ vient par changement de base.

Alternativement, cela découle de \cite[Exp. XXII, prop. 4.3.1]{sga3} puisque $\faisGL_n$ est réductif et déployé par les propositions \ref{GL1red_prop} et \ref{toreGLn_prop} et que $\faisPGL_n$ en est le quotient par son centre, par les propositions \ref{centreGLn_prop} et \ref{secPGLn_prop}.
\end{proof}

Comme dans le cas de $\faisGL_{n,R}$, on identifie les applications exponentielles par unicité. L'application $\exp_{\bar\alpha_{i,j}}$ envoie $\lambda\cdot \bar E_{i,j}$ vers l'image de la matrice $1+\lambda E_{i,j}$ dans $\faisGL_n(S)$ (donc vers l'automorphisme de $\faisM_n(S)$ de conjugaison par cette matrice). 

\begin{prop} \label{donneeradiciellePGLn_prop}
La donnée radicielle du $S$-schéma en groupes $\faisPGL_n$ relativement au tore maximal déployé $\faisPDiag_n$ est 
\begin{itemize}
\item Le $\ZZ$-module libre $\ker(\ZZ^n\to \ZZ)$ des caractères de $\faisDiag_n$ dont la somme des coordonnées est nulle (dont $t_1-t_2,\ldots,t_{n-1}-t_n$ forment une base); 
\item Le sous-ensemble fini de celui-ci formé des racines $\alpha_{i,j}=t_i-t_j$ pour $1\leq i,j\leq n$ et $i\neq j$;
\item Le $\ZZ$-module libre dual $\coker(\ZZ \to \ZZ^n)$, quotient des cocaractères de $\faisDiag_n$ par l'application diagonale (de base duale $\overline{t_1^\cdual},\overline{t_1^\dual+t_2^\dual},\ldots,\overline{t_1^\cdual+\cdots +t_{n-1}^\dual}$);
\item Le sous-ensemble fini de celui-ci formé des coracines $\alpha_{i,j}^\cdual=\overline{t_i^\dual-t_j^\dual}$ pour $1\leq i,j\leq n$ et $i\neq j$.
\end{itemize}
\end{prop}
\begin{proof}
Par la description du tore maximal $\faisPDiag_n$ comme quotient de $\faisDiag_n$, on a immédiatement qu'il agit sur le sous-espace $\faisO_S\cdot \bar E_{i,j}$ par $\bar\alpha_{i,j}$, ce qui donne les racines. Les coracines sont alors immédiatement obtenues par unicité en utilisant comme pour $\faisGL_n$ la formule (F) de \cite[th. 2.4.1]{D} (il n'y a rien à vérifier, cela a déjà été fait pour $\faisGL_n$, et on ne fait que regarder les images dans $\faisPGL_n$).
\end{proof}

\begin{theo}
Pour tout $n\geq 1$, le groupe (semi-simple) adjoint déployé de type $A_n$ sur une base $S$ est $\faisPGL_{n+1}$.
\end{theo}
\begin{proof}
Un tel groupe est unique à isomorphisme près.
Le groupe $\faisPGL_n$ étant réductif et déployé, l'énoncé ne dépend que de sa donnée radicielle (voir \cite[exp. XXIII, cor. 5.4]{sga3}), on est donc ramené à la situation sur un corps, supposée connue. On peut en tout cas facilement lire directement sur la donnée radicielle que le groupe est semi-simple adjoint: les racines engendrent les caractères du tore. Alternativement, c'est une conséquence de \cite[prop. 4.3.5]{sga3}.
\end{proof}

\begin{prop}
Le sous-ensemble de racines $\Delta=\{\alpha_{i,i+1},\ 1\leq i\leq n-1\}$ forme un système de racines simples, et assorti des éléments $u_{i,i+1}=\overline{1+E_{i,i+1}}=\bar{E_{i,i+1}} \in \faisM_n(S)/\faisO_S(S)$, ils munissent le groupe $\faisPGL_n$, le tore $\faisPDiag_n$ et son système de racines d'un épinglage. 
\end{prop}
\begin{proof}
L'ensemble de racines simples est classique. Il est immédiat de vérifier que les $u_{i,i+1}$ satisfont à la condition de \cite[exp. XXIII, déf. 1.1]{sga3}.
\end{proof}

Profitons-en pour signaler que le morphisme $\faisGL_n \to \faisPGL_n$ induit un morphisme de données radicielles (voir \cite[exp. XXI, déf. 6.1]{sga3}). Il est donné par la projection de $\faisDiag_n$ vers $\faisPDiag_n$ et des morphismes associés respectivement entre caractères et cocaractères. Ces morphismes induisent bien des bijections entre ensemble de racines, ainsi qu'entre ensembles de coracines. Il s'agit du morphisme de $\faisGL_n$ vers son groupe adjoint associé, qui correspond au morphisme de la donnée radicielle adjointe vers la donné radicielle de $\faisGL_n$ (voir \cite[prop. 6.5.5]{sga3}).

\subsection{Groupe déployé simplement connexe}

\begin{defi} \label{SLn_defi}
Soit $n\geq 1$. Soit $\faisSL_n$ le schéma en groupe affine de type fini sur $S$ noyau de l'application déterminant $\faisGL_n \to \faisGm$. 
\end{defi}
\begin{proof}
Ce noyau est clairement représentable par un schéma affine de type fini sur $S$, car localement sur un anneau $R$, il est le sous-schéma fermé de $\faisGL_n$ donné par l'équation du déterminant.
\end{proof}

Pour tout $S$-schéma $T$, le groupe $\faisSL_{n,S}(T)$ est l'ensemble des matrices de $\faisM_n(T)=\setM_n(\faisO_S(T))$ de déterminant $1$.
En particulier, on a
$$\faisSL_{n,T}=(\faisSL_{n,S})_{T} =(\faisSL_{n,\ZZ})_{T}.$$ 

\begin{prop}
L'algèbre de Lie de $\faisSL_n$ est isomorphe au sous $\faisO_S$-module libre de $\faisM_n$ constitué des matrices de trace nulle. L'action adjointe de $\faisSL_n$ dessus est donnée par conjugaison.
\end{prop}
\begin{proof}
C'est immédiat à partir de l'algèbre de Lie de $\faisGL_n$ en utilisant le lemme \ref{sous_algebre_lemm}.
\end{proof}

\begin{prop}
Le groupe $\faisSL_n$ est réductif. 
\end{prop}
\begin{proof}
Il ne manque que la lissité, qui se déduit sur $\ZZ$ de la situation connue sur les corps par la proposition \ref{lissite_prop}, puis par changement de base pour $\faisSL_{n,S}$.
\end{proof}

\begin{prop}
L'intersection du tore maximal $\faisDiag_n$ de $\faisGL_n$ avec $\faisSL_n$ est un tore maximal (déployé) de $\faisSL_n$. Notons-le $\faisSDiag_n$. Il est de rang $n-1$.
\end{prop}
\begin{proof}
Il est clair que c'est un tore déployé de rang $n-1$, dual de Cartier du conoyau de l'application diagonale $\ZZ \to \ZZ^n$ duale du déterminant sur $\faisDiag_n$. La théorie sur les corps étant connue, il est maximal dans chaque fibre géométrique. 
\end{proof}

\begin{prop}
L'intersection de $\faisTriSup_n$ avec $\faisSL_n$ est un sous-groupe de Borel de $\faisSL_n$ contenant le tore précédent.
\end{prop}
\begin{proof}
Immédiat par fibres.
\end{proof}

\begin{prop} \label{centreSL_prop}
Le centre de $\faisSL_n$ est $\faismu_n$ inclus diagonalement.
\end{prop}
\begin{proof}
C'est clair sur tous les $T$-points en utilisant les matrices élémentaires.
\end{proof}

La décomposition de l'algèbre de Lie de $\faisSL_n$ en caractères se lit immédiatement, (comme celle de $\faisGL_n$ qu'on a déjà traitée). On obtient que les racines sont les caractères $\alpha_{i,j}=\overline{t_i-t_j}$ avec $i\neq j$, d'espace propre $\faisO_S\cdot E_{i,j}$. Les morphismes exponentiels sont les mêmes que pour $\faisGL_n$ (ils se factorisent par $\faisSL_n$).

\begin{prop} \label{donneeradicielleSLn_prop} 
La donnée radicielle de $\faisSL_n$ relativement au tore maximal déployé $\faisSDiag_n$ est 
\begin{itemize}
\item Le $\ZZ$-module libre $\coker(\ZZ \to \ZZ^n)$ (diagonale), quotient des caractères de $\faisDiag_n$ (dont $\overline{t_1},\overline{t_1+t_2},\ldots,\overline{t_1+\cdots +t_{n-1}}$ est une base). 
\item Le sous-ensemble fini de celui-ci formé des racines $\alpha_{i,j}=\bar t_i-\bar t_j$ pour $1\leq i,j\leq n$ et $i\neq j$;
\item Le $\ZZ$-module libre dual $\ker(\ZZ^n \to \ZZ)$ (somme des coordonnées), sous groupe des cocaractères de $\faisDiag_n$ (de base duale $t_1^\dual-t_n^\dual,\cdots,t_{n-1}^\dual-t_n^\dual$);
\item Le sous-ensemble fini de celui-ci formé des coracines $\alpha_{i,j}^\cdual=t_i^\dual-t_j^\dual$ pour $1\leq i,j\leq n$ et $i\neq j$.
\end{itemize}
\end{prop}
\begin{proof}
La preuve est la même que pour $\faisGL_n$.
\end{proof}

\begin{theo}
Pour $n\geq 1$, le $S$-schéma en groupes (semi-simple) simplement connexe déployé de type $A_n$ est $\faisSL_{n+1}$.
\end{theo}
\begin{proof}
Un tel groupe est unique à isomorphisme près. Le groupe $\faisSL_{n+1}$ est semi-simple et simplement connexe de manière évidente puisque les coracines engendrent les cocaractères. Son type de Dynkin se lit sur une fibre géométrique où la théorie est connue. Alternativement, il se voit sur la donnée radicielle.
\end{proof}

On a donc le diagramme de morphismes de groupes algébriques correspondant au diagramme de données radicielles de \cite[Exp. XXI, prop. 6.5.5]{sga3}:

$$\xymatrix{
\faisSL_n =(\faisGL_n)_\simco= (\faisGL_n)_\derive \ar[rr] \ar[dr] & & (\faisGL_n)_\semisim = (\faisGL_n)_\adjoint = \faisPGL_n  \\
& \faisGL_n \ar[ur] &  \\
}$$

Décrivons maintenant les $\faisSL_n$-torseurs.
Considérons les foncteurs fibrés $\det: (\Vecn{n})_{\grpd} \to \Vecn{1}$, défini en \ref{Detchamps_defi}, et $\Final \to \Vecn{1}$ qui envoie le seul objet de $\Final_T$ sur $\faisO_T$. Définissons alors le produit fibré $\Vecn{n} \times_{\Vecn{1}} \Final$, qui est un champ par l'exercice \ref{prodfibchamps_exo}.
\begin{defi}
On appelle $\Vectrivdetn{n}$ le produit fibré précédent, et on le nomme \emph{champ des $\faisO$-modules de déterminant trivialisé}.
\end{defi}
Explicitement, un objet de $(\Vectrivdetn{n})_T$ est donc une paire $(M,\phi)$ où $M$ est un $\faisO_T$-module localement libre de rang constant $n$, et $\phi: \Det{M} \isoto \faisO_S$ est un isomorphisme. Un morphisme $(M_1,\phi_1)\to (M_2,\phi_2$ est un isomorphisme de $\faisO_S$-modules $f:M_1 \to M_2$ tel que $\det{f}:\Det{M_1} \isoto \Det{M_2}$ satisfait à $\phi_2 \circ \det{f}=\phi_1$. 

Considérons alors dans $(\Vectrivdetn{n})_S$ l'objet $(\faisO_S^n,d)$ où $d$ est l'isomorphisme ``déterminant'' qui envoie un produit extérieur $v_1\wedge\ldots\wedge v_n$ sur le déterminant des vecteurs $v_1,\ldots,v_n$. On a alors l'égalité de faisceaux $\faisAut_{(\faisO_S^n,d)} = \faisSL_n$, puisque par la description concrète des morphismes ci-dessus, un isomorphisme $f$ doit vérifier $d \circ \det{f}=d$, donc $\det{f}=1$.

Localement, pour la topologie Zariski (donc également étale, etc.), tout objet $(M,\phi)$ de $\Vectrivdetn{n}$ est isomorphe à $(\faisO_S^n,d)$. En effet, localement $M \simeq \faisO_S^n$ par la définition \ref{loclibre_defi}, puis $(\faisO_S^n,\phi)\isoto(\faisO_S^n,d)$ localement dans le cas affine, car il est facile de construire dans le cas affine un isomorphisme $\faisO_n \isoto \faisO_n$ de déterminant donné, avec une matrice diagonale. 

La proposition \ref{tordusformes_prop} fournit alors immédiatement:
\begin{prop}
Le foncteur fibré qui envoie un objet $(M,\phi)$ de $\Vectrivdetn{n}$ vers $\faisIso_{(O_S^n,d),(M,\phi)}$ définit une équivalence de catégories fibrées du champ en groupoïdes $\Vectrivdetn{n}$ vers le champ $\Tors{\faisSL_n}$ (la topologie sous-jacente étant toujours la topologie étale, mais c'est vrai pour une topologie plus fine que Zariski).
\end{prop}

\subsection{Automorphismes} \label{automorphismes_sec}

Intéressons-nous maintenant aux automorphismes de $\faisPGL_n$, et donnons-en une description en termes de groupes classiques.

\'Etant donnée un faisceau en algèbres $A$ sur $S$ localement libre de type fini comme $\faisO_S$-module, munie d'une involution $\faisO_S$-linéaire $\sigma$, on considère $\faisAut_{A,\sigma}$ son faisceau des automorphismes. 
\begin{prop}\label{Autalginv_prop}
Pour tout schéma $T$ au-dessus de $S$, on a
$$(\faisAut^\alginv_{A,\sigma})_{T}=\faisAut^\alginv_{A_{T},\sigma_{T}}.$$
Lorsque $T$ est un schéma affine, alors $\faisAut^\alginv_{A,\sigma}(T)$ est l'ensemble des automorphismes d'algèbre de $A(T)$ commutant à l'involution.
De plus, ce faisceau est représentable par un schéma affine sur $S$.
\end{prop}
\begin{proof}
En dehors de la représentabilité, la preuve est identique à celle de la proposition \ref{PGLApoints_prop}, mais avec une involution en plus.
Le faisceau est ensuite représentable comme sous-schéma fermé de $\faisAut_A$. En effet, on se réduit au cas affine $S=\Spec(R)$ par \ref{reprLocal_prop}. L'algèbre $A(S)$ est alors de type fini sur $R$, la commutation à $\sigma$ s'exprime par un certain nombre (fini) d'équations sur des générateurs. 
\end{proof}

\begin{defi}\label{Autalginv_defi}
On note $\faisAut^\alginv_{A,\sigma}$ le schéma affine sur $S$ ainsi obtenu.
\end{defi}

L'opposé d'une algèbre $A$, notée $A^{\opp}$ a la même structure de $R$-module sous-jacente que $A$ et la multiplication y est définie à partir de celle de $a$ en inversant l'ordre des deux facteurs.
Rappelons que $\exch$ est l'involution sur $\faisM_n\times \faisM_n^{\opp}$, qui à $(a,b)$ associe $(b,a)$. Considérons un automorphisme d'algèbre commutant à $\exch$. Sa restriction au centre $\faisO_S \times \faisO_S$ est l'identité ou l'échange des facteurs (sur chaque composante connexe). Si sa restriction au centre est l'identité, il est facile de voir qu'il est de la forme $(\epsilon,\epsilon^{\opp})$ avec $\epsilon \in \faisAut_{\faisM_n}$. On obtient donc que $\faisPGL_n$ s'identifie à un sous-groupe normal de $\faisAut^\alginv_{\setM_n \times \setM_n^{\opp},\exch}$, noyau de l'application de restriction au centre 
$$\faisAut^\alginv_{\faisM_n \times \faisM_n^{\opp},\exch} \to \faisAut^\alginv_{\faisO_S\times \faisO_S,\exch}\simeq (\ZZ/2)_S.$$
De plus la conjugaison par tout $\sigma \in \setAut(\faisM_n(S)\times \setM_n^{\opp}(S),\exch)$ préserve ce noyau et définit donc une application 
$$\phi:\faisAut^\alginv_{\faisM_n \times \faisM_n^{\opp},\exch} \to \faisAut^\gr_{\faisPGL_n}$$ 
(toutes ces constructions commutent à l'extension de la base $S$ de manière évidente). 

\begin{theo} \label{PGLAut_theo}
Pour $n\geq 3$, on a un isomorphisme de suites exactes courtes scindées
$$\xymatrix{
1 \ar[r] & \faisAut^\alg_{\faisM_n} \ar[r] \ar@{=}[d] & \faisAut^\alginv_{\faisM_n\times \faisM^\opp_n,\exch} \ar[r] \ar[d]^{\simeq}_{\phi} & \faisAut^\alginv_{\faisO_S\times \faisO_S,\exch} \ar[r] \ar[d]^{\simeq} & 1 \\
1 \ar[r] & \faisPGL_n \ar[r] & \faisAut^\gr_{\faisPGL_n} \ar[r] & \faisAutExt_{\faisPGL_n} \ar[r] & 1 
}$$
où les flèches de la ligne du haut sont décrites ci-dessus, les flèches de la ligne du bas sont les flèches canoniques (conjugaison et quotient).
Les deux groupes de droite sont isomorphes au groupe (localement) constant $(\ZZ/2)_S$. Soit $d$ la matrice antidiagonale et $\transp : \setM_n \to \setM^{\opp}_n$ la transposition. L'application qui, sur chaque composante connexe de $S$, envoie l'élément $\exch\neq id$ sur l'automorphisme $ex:(a,b)\mapsto (d(\transp b)d^{-1}, d(\transp a)d^{-1})$ est un scindage de la surjection d'en haut. 

Pour $n\leq 2$, le groupe $\faisPGL_n$ n'a pas d'automorphisme extérieur. On a le même diagramme, mais la flèche $\phi$ n'est pas un isomorphisme, et la flèche de droite est la projection vers le groupe trivial. 
\end{theo}
\begin{proof}
Toutes les affirmations étant Zariski locales sur la base $S$, on peut supposer $S=\Spec(R)$ où $R$ est un anneau local.
Le carré de gauche commute par construction, et l'affirmation sur le scindage se vérifie immédiatement par calcul. 
Le groupe $\faisPGL_1$ est trivial et $\faisPGL_2$ n'a pas d'automorphisme extérieur car le diagramme de Dynkin est un point. Il n'y a donc rien de plus à justifier pour $n=1,2$. 
Supposons donc $n\geq 3$. Les deux groupes de droite sont bien isomorphes au groupe constant $(\ZZ/2)_R$. C'est évident pour $\faisAut_{R \times R,\exch}$, et cela découle de l'interprétation de $\faisAutExt$ en termes d'automorphismes du diagramme de Dynkin, voir \cite[Exp. XXIV, th. 1.3 (iii)]{sga3}.
Pour vérifier que le carré de droite commute, il suffit donc de vérifier que $\phi(ex)$ s'envoie bien sur l'élément non trivial de $\faisAutExt$. Or, on voit comment $\phi(ex)$ agit sur $\faisPGL_{n,R}$ (identifié aux éléments de la forme $(\epsilon,\epsilon^{\opp})$) et en particulier sur l'image de $\faisGL_{n,R}$, où les élément sont de la forme $(\int_g,\int_g^{\opp})$, $g \in \faisGL_{n,R}$, qu'il envoie sur $(\int_{d(\transp g^{-1})d^{-1}},\int_{d(\transp g^{-1})d^{-1}}^{\opp})$. Il préserve donc le tore $\faisPDiag$ image de $\faisDiag$, en agissant dessus comme le dual de Cartier de l'automorphisme des caractères qui envoie $\alpha_{i,i+1}=t_i-t_{i+1}$ sur $-t_{n+1-i}+t_{n-i}=\alpha_{n-i,n-i+1}$. Il préserve donc aussi l'ensemble des racines simples $\Delta$, sur lequel il agit comme l'automorphisme non trivial du diagramme de Dynkin. De plus, il envoie l'élément $u_{i,i+1}$ sur $u_{n-i,n-i+1}$. Il est ainsi un morphisme de groupes épinglés selon \cite[exp. XXIII, déf. 1.3]{sga3}, et fournit l'élément non trivial de $\faisAutExt$.

\end{proof}

\begin{rema}
Dans $\faisPGL_2$, le candidat à être un automorphisme extérieur n'en est pas un: pour tout automorphisme $\epsilon\in \setAut(\setM_n(R))$, on a $\transp{}\circ \epsilon^{\opp} \circ \transp{} = \int_{e} \circ \epsilon \circ \int_{e}^{-1}$ où $e=\left( \begin{smallmatrix} 0&1\\ -1&0 \end{smallmatrix} \right)$. C'est un calcul immédiat quand l'automorphisme $\epsilon$ est lui-même intérieur, ce qui suffit localement.
\end{rema}

Par le théorème \ref{automorphismes_theo}, le groupe $\faisSL_n$ a les mêmes automorphismes que $\faisPGL_n$. On veut toutefois comprendre comment ils agissent sur $\faisSL_n$ lorsqu'on les identifie à $\faisAut^\alginv_{\faisM_n \times \faisM_n^{\opp},\exch}$. Considérons le monomorphisme $\faisGL_n \to \faisM_n\times \faisM_n^{\opp}$ qui envoie un élément $g\in \faisGL_n(S)$ vers l'élément $(g,g^{-1})\in \faisM_n(S) \times \faisM_n^{\opp}(S)$. L'image de cette application est préservée par $\faisAut^\alginv_{\faisM_n \times \faisM_n^{\opp},\exch}$. Par la suite exacte scindée du théorème précédent, il suffit de le vérifier pour l'élément $ex$ et pour tout $(\epsilon, \epsilon^{\opp})$, ce qui est immédiat. Par identification de $\faisGL_n$ à l'image, cela définit donc une action de $\faisAut^\alginv_{\faisM_n \times \faisM_n^{\opp},\exch}$ sur $\faisGL_n$, pour laquelle $ex(g)=\transp{g}^{-1}$ et $(\epsilon,\epsilon^{\opp})(g)=\epsilon(g)$. 

\begin{lemm} \label{SLnactionaut_lemm}
Le sous-groupe algébrique $\faisSL_n$ est préservé par cette action. 
\end{lemm}
\begin{proof}
Comme plus haut, il suffit de le vérifier pour l'élément $ex$, pour lequel c'est évident, et pour tout élément de la forme $(\epsilon,\epsilon^\opp)$. On a affaire à des faisceaux, on peut vérifier l'énoncé localement pour la topologie étale. Par \ref{secPGLn_prop}, on peut alors supposer que $\epsilon$ est la conjugaison par un élément inversible, et préserve donc le déterminant.
\end{proof}

\begin{theo} \label{autGLSL_theo}
L'action de $\faisAut^\alginv_{\faisM_n \times \faisM_n^{\opp},\exch}$ sur $\faisGL_n$ (resp. $\faisSL_n$) définie ci-dessus et celle de $\faisAut^\alginv_{\faisM_n \times \faisM_n^{\opp},\exch}$ sur $\faisPGL_n$ du théorème \ref{PGLAut_theo} commutent au morphisme naturel $\faisGL_n \to \faisPGL_n$ (resp. $\faisSL_n \to \faisPGL_n$). Pour $n \geq 3$, il identifie ainsi les automorphismes de $\faisGL_n$ (resp. $\faisSL_n$) à $\faisAut^\alginv_{\faisM_n \times \faisM_n^{\opp},\exch}$. Pour $n=2$, ce dernier groupe coïncide toujours avec les automorphismes de $\faisGL_2$ (mais pas avec ceux de $\faisSL_2$ ni $\faisPGL_2$ comme vu plus haut).
\end{theo}
\begin{proof}
La compatibilité avec la projection $\pi:\faisGL_n \to \faisPGL_n$ se teste sur $ex$ et les $(\epsilon,\epsilon^{\opp})$. Pour $n\geq 3$, on a bien les automorphismes de $\faisSL_n$ et ceux de $\faisGL_n$ par le théorème \ref{automorphismes_theo} et l'exercice \ref{autoDonneeRadGL_exo}. Pour $n=2$, l'automorphisme extérieur de $\faisGL_2$ est bien fourni par $g \mapsto \transp{g}^{-1}$, le raisonnement de la fin de la preuve du théorème \ref{PGLAut_theo} s'applique.
\end{proof}

\subsection{Groupes tordus}

Commençons par le cas semi-simple adjoint et donc de donnée radicielle déployée isomorphe à celle de $\faisPGL_{n,S}$ sur une base quelconque $S$, omise ci-après des notations pour alléger. 

Par \ref{tordusformes_prop}, toute forme de $\faisPGL_n$ est obtenue par torsion par un torseur $P$ sous son groupe d'automorphismes, qui est $\faisAut_{\faisM_n\times \faisM^\opp_n,\exch}$ (voir \ref{PGLAut_theo}). On tord donc la suite du théorème \ref{PGLAut_theo} à l'aide de la proposition \ref{torsionsec_prop} dans le cas de la remarque \ref{secautext_rema}. On obtient une suite de groupes d'automorphismes tordus
$$1 \to P \contr{H} \faisPGL_n \to  P \contr{H} H \to P \contr{H} \faisAut^\alginv_{S \times S,\exch} \to 1$$
où $H:=\faisAut^\alginv_{\faisM_n\times \faisM^\opp_n, \exch}$.
En utilisant \ref{auttordus_prop}, on identifie le terme du milieu 
$$P \contr{H} H = \faisAut^\alginv_{P \contr{H} (\faisM_n \times \faisM^\opp_n, \exch)} = \faisAut_{A,\nu}$$ 
pour un certain faisceau en algèbres d'Azumaya $(A,\nu)$ de deuxième espèce sur $S$.
De même, le terme de droite s'identifie à
$$P \contr{H} \faisAut_{S \times S,\exch}= \faisAut_{Z(A),\nu}$$
où $Z(A)$ est le centre de $A$, un faisceau en algèbres étales finies de degré $2$ sur $S$ et $\nu$ est le seul $S$-automorphisme non trivial de $Z(A)$ sur toutes les composantes connexes de $S$ (Voir \ref{algetale_defi} et \ref{etaledegre2_coro}). Notons que $\nu$ est une involution et un automorphisme puisque $Z(A)$ est commutatif.
Dans le cas d'une forme intérieure, le torseur $P$ est de la forme $Q \contr{\faisPGL_n} H$ où $Q$ est un torseur sous $\faisPGL_n$, et on a donc 
$$P \contr{H} \faisPGL_n = Q \contr{\faisPGL_n} \faisPGL_n = Q \contr{\faisAut^\alg_{\faisM_n}} \faisAut^\alg_{\faisM_n} = \faisAut^\alg_{Q \contr{\faisAut^\alg_{\faisM_n}} \faisM_n} \hspace{-1.5ex}= \faisAut^\alg_{A}$$ 
pour un certain faisceau en algèbres d'Azumaya $A$ de degré $n$ sur $S$. Puisque toutes les formes s'obtiennent par torsion (voir \ref{tordusformes_prop}) et que tout groupe réductif est localement déployé pour la topologie étale (voir \ref{reductifetale_theo}), on a donc montré:

\begin{prop} \label{adjointtordu_prop}
\`A isomorphisme près, les groupes réductifs de donnée radicielle déployée identique à celle de $\faisPGL_n$ ($n\geq 3$), c'est-à-dire de type déployé $A_{n-1}$ et adjoint, sont les noyaux de $\faisAut_{A,\nu} \to \faisAut_{Z(A),\nu}$ où $(A,\nu)$ est un faisceau d'algèbre d'Azumaya sur $S$ de deuxième espèce, de degré $n$ et de centre $Z(A)$ étale fini de degré $2$ sur $S$. Pour $n\geq 2$, les formes intérieures sont les $\faisPGL_A$ où $A$ est un faisceau en algèbre d'Azumaya de degré $n$ (voir déf. \ref{PGLA_defi}). Le groupe $\faisPGL_2$ n'a que des formes intérieures. 
\end{prop}

Pour les formes fortement intérieures de $\faisPGL_n$, voir \ref{adjointAnfortint_coro}.
\medskip

Passons maintenant au cas simplement connexe, c'est-à-dire les formes (étales) de $\faisSL_n$. Nous obtiendrons également celles de $\faisGL_n$. D'après \ref{reductifetale_theo} et \ref{tordusformes_prop}, un tel groupe est obtenu par torsion de $\faisGL_n$ par un torseur étale sous $\faisAut_{\faisSL_n}=\faisAut_{\faisPGL_n}$.

Soit $(A,\nu)$ un faisceau en algèbres à involution de deuxième espèce sur $S$ (voir déf. \ref{alginvdeuxesp_defi}). Considérons le $S$-foncteur de points défini par
$$\faisU_{A,\nu}(T)=\left\{a \in A(T),\ \nu(a)a=1\right\}.$$ 
\begin{prop} \label{Urepr_prop}
Pour tout schéma $T$ au-dessus de $S$, on a 
$$(\faisU_{A,\nu})_{T} = \faisU_{A_{T},\nu_{T}}.$$
De plus, ce foncteur est représentable.
\end{prop}
\begin{proof} 
La première égalité est claire. Pour la représentabilité, on montre que le foncteur est un faisceau Zariski en utilisant que c'est un sous-foncteur de $\faisGL_{1,A}$ et en utilisant que la condition $\nu(a)a=1$ est de nature locale. On se ramène alors au cas $S$ affine et $A$ libre par la proposition \ref{reprLocal_prop}. Dans ce cas, $A=\wW(\tilde{A(S)})$ et la condition $\nu(a)a=1$ est donnée par des équations sur des générateurs de $A(S)$. 
\end{proof}

\begin{defi} \label{U_defi}
On note $\faisU_{A,\nu}$ le $S$-groupe algébrique représentant le foncteur décrit ci-dessus, et on l'appelle le schéma en groupes \emph{unitaire} de $(A,\nu)$.
\end{defi}

\begin{lemm} \label{UGL_lemm}
Le morphisme définit sur les points par $g \mapsto (g,g^{-1})$ définit un isomorphisme de schémas en groupes
$\faisGL_n \simeq \faisU_{\faisM_n\times \faisM_n^\opp,\exch}$. 
\end{lemm}
\begin{proof}
Clair.
\end{proof}

Tout faisceau en groupes $H$ agissant sur $(A,\nu)$ en respectant la structure de $\alginv$ agit également sur $\faisU_{A,\nu}$ de la manière évidente. 

\begin{prop} \label{Uforme_prop}
Soit $P$ un torseur sous un tel $H$. On a un isomorphisme canonique
$$P \contr{H} \faisU_{\faisM_n\times \faisM_n^\opp,\exch} \simeq \faisU_{P \contr{H} (\faisM_n\times \faisM_n^\opp,\exch)}$$
\end{prop}
\begin{proof}
On procède comme pour la proposition \ref{GLtordu_prop}.
\end{proof}

\begin{theo} \label{formesGLn_theo}
Tout groupe réductif dont la donnée radicielle déployée est celle de $\faisGL_n$ est isomorphe à $\faisU_{A,\nu}$ où $(A,\nu)$ est un faisceaux en algèbres d'Azumaya à involution de deuxième espèce de degré $n$ sur $S$. Les formes intérieures sont les $\faisGL_{1,A}$ où $A$ est un faisceau en algèbres d'Azumaya de degré $n$. Les formes strictement intérieures sont les $\faisGL_{M}$ où $M$ est un $S$-faisceau localement libre de rang $n$ sur $S$. 
\end{theo}
\begin{proof}
Il suffit d'utiliser la proposition \ref{Uforme_prop} dans le cas $H=\faisAut_{\faisM_n\times\faisM_n^\opp}$ et de voir que l'action qui y est considérée correspond bien à celle du théorème \ref{autGLSL_theo} par l'identification du lemme \ref{UGL_lemm}. 
Pour les formes intérieures, il suffit d'utiliser la proposition \ref{GLtordu_prop}. 
Enfin, pour les formes strictement intérieures, on utilise l'équivalence \ref{GLntors_prop} entre les $\faisGL_n$-torseurs (étales ou Zariski) et les faisceaux localement libre de type fini sur $S$. Lorsqu'on tord $\faisM_n$ par un tel torseur, on obtient justement $\faisEnd_M$ où $M$ est le faisceau localement libre associé.
\end{proof}

Tout faisceau en algèbres à involution $(A,\nu)$ de deuxième espèce est obtenu par définition comme $P \contr{H}(\faisM_n \times \faisM_n^\opp,\exch)$ où $P$ est un torseur sous $\faisAut^\alginv_{\faisM_n \times \faisM_n^\opp,\exch}$. 
Le morphisme $\faisGL_{1,\faisM_n\times \faisM_n^\opp}=\faisGL_n \times \faisGL_n^\opp \to \faisGm$ donné par le déterminant sur la première composante se restreint en un morphisme $\faisU_{\faisM_n \times \faisM_n^\opp,\exch} \to \faisGm$.

\begin{lemm} 
Ce morphisme est équivariant si $\faisAut^\alginv_{\faisM_n \times \faisM_n^\opp,\exch}$ agit sur $\faisGm$ trivialement. Son noyau s'identifie à $\faisSL_n$ par l'isomorphisme du lemme \ref{UGL_lemm}, qui est donc stable par l'action.
\end{lemm}
\begin{proof}
Cela découle immédiatement du lemme \ref{SLnactionaut_lemm}. 
\end{proof} 

Il se tord donc en un morphisme $\faisU_{A,\nu} \to \faisGm$.
\begin{defi} \label{SU_defi}
On définit le faisceau en groupes $\faisSU_{A,\nu}$ comme le noyau du morphisme $\faisU_{A,\nu} \to \faisGm$ ci-dessus. C'est donc bien le tordu de $\faisSL_n$ par le torseur donnant $(A,\nu)$. 
\end{defi}
En particulier, à travers l'identification $\faisGL_n \simeq \faisU_{\faisM_n \times \faisM_n^\opp,\exch}$, le sous groupe $\faisSL_n$ s'identifie à $\faisSU_{\faisM_n \times \faisM_n^\opp,\exch}$.

On fait agir $\faisAut_{\faisM_n}$ sur $\faisGL_n=\faisGL_{\faisM_n}$ de manière évidente et sur $\faisGm$ trivialement. 
\begin{lemm}
Le morphisme déterminant $\faisGL_n \to \faisGm$ est équivariant. 
\end{lemm}
\begin{proof}
Il suffit de le vérifier localement pour la topologie étale, et on peut donc supposer par \ref{secPGLn_prop} qu'on fait agir un élément de $\faisAut_{\faisM_n}(T)$ qui la conjugaison par un élément de $\faisGL_n(T)$, auquel cas c'est évident. 
\end{proof}

\begin{defi} \label{SLA_defi}
Pour tout $S$-faisceau en algèbres d'Azumaya $A$ de degré $n$, on note $\faisSL_{1,A}$ le noyau du morphisme norme réduite défini en \ref{normereduite_defi}; C'est bien le tordu de $\faisSL_n$ par \ref{torsionsec_prop}.
\end{defi}

Pour tout $S$-schéma $T$, on a clairement $(\faisSL_{1,A})_{T} = \faisSL_{1,A_{T}}$.

\begin{theo} \label{formesSLn_theo}
Les groupes réductifs dont la donnée radicielle déployée est celle de $\faisSL_n$, c'est-à-dire de type déployé $A_{n-1}$ et simplement connexe, sont les groupes isomorphe à $\faisSU_{A,\nu}$ où $(A,\nu)$ est un faisceau en algèbres d'Azumaya à involution de deuxième espèce de degré $n$ sur $S$. Les formes intérieures sont les $\faisSL_{1,A}$ où $A$ est un faisceau en algèbres d'Azumaya de degré $n$. Les formes strictement intérieures sont les $\faisSL_M$ où $M$ est un faisceau localement libre de rang $n$ sur $S$ et de déterminant trivial. Lorsque $n\leq 2$, toutes les formes sont intérieures.
\end{theo}
\begin{proof}
Par la définition \ref{SU_defi}, si $P$ est un torseur sous $H=\faisAut_{\faisM_n\times \faisM_n^\opp,\exch}$, et si $(A,\nu)=P \contr{H} (\faisM_n\times \faisM_n^\opp,\exch)$, alors $P \contr{H} \faisSL_n=\faisSU_{A,\nu}$. Les formes intérieures s'obtiennent par la définition \ref{SLA_defi}. 
Pour les formes strictement intérieures, on utilise comme dans la preuve du théorème \ref{formesGLn_theo}, l'équivalence de catégories \ref{GLntors_prop} entre les torseurs sous $\faisGL_n$ et les fibrés vectoriels sur $S$ de rang $n$ et on constate par la suite exacte scindée $1 \to \faisSL_n \to \faisGL_n \to \faisGm \to 1$ que les torseurs provenant de $\faisSL_n$ correspondent aux fibrés vectoriels de déterminant trivial, grâce au lemme \ref{torsDetFonct_lemm}.

Le groupe $\faisSL_1$ est trivial. Toutes les formes de $\faisSL_2$ sont intérieures, l'absence d'automorphisme extérieur se lit sur le diagramme de Dynkin. Le groupe $\faisAut_{\faisM_2 \times \faisM_2^\opp}$ est plus gros que le groupe d'automorphismes $\faisAut_{\faisM_2}$ et son action se factorise par celui-ci.
\end{proof}

\begin{coro} \label{adjointAnfortint_coro}
Les formes fortement intérieures de $\faisPGL_{n}$ sont les $\faisPGL_{M}=\faisAut^\alg_{\faisEnd_M}$ où $M$ est un faisceau localement libre de rang $n$ sur $S$, et de déterminant trivial.
\end{coro}

\begin{rema}
La suite exacte
$$1 \to \faisGm \to \faisGL_n \to \faisPGL_n \to 1$$
ou encore en identifiant $\faisPGL_n=\ker(\faisAut_{\faisM_n\times \faisM_n^\opp} \to \faisAut_{\faisO_S \times \faisO_S,\exch})$,
$$1 \to \faisGm \to \faisGL_n \to \faisAut_{\faisM_n\times \faisM_n^\opp} \to \faisAut_{\faisO_S \times \faisO_S,\exch} \to 1$$
se tord en
$$1 \to \faisGm \to \faisU_{A,\nu} \to \faisAut_{A,\nu} \to \faisAut_{Z(A),\nu} \to 1.$$
(Le premier morphisme est bien équivariant pour l'action triviale sur $\faisGm$ et passe donc à la torsion par le même type de raisonnement local que dans le lemme \ref{SLnactionaut_lemm}.) 
Dans le cas d'un $\faisGL_n$-torseur, on a simplement
\begin{equation} \label{secPGLA_eq}
1 \to \faisGm \to \faisGL_{1,A} \to \faisAut_{A} \to 1.
\end{equation}
De même, la suite exacte
$$1 \to \faismu_n \to \faisSL_n \to \faisPGL_n \to 1$$
donne après torsion
$$1 \to \faismu_n \to \faisSU_{A,\nu} \to \faisAut_{A,\nu} \to \faisAut_{Z(A),\nu} \to 1$$ 
ou dans le cas d'un $\faisGL_n$-torseur
$$1 \to \faismu_n \to \faisSL_{1,A} \to \faisAut_{A} \to 1$$ 
et la suite exacte
$$1 \to \faisSL_n \to \faisGL_n \to \faisGm \to 1$$
donne
$$1 \to \faisSU_{A,\nu} \to \faisU_{A,\nu} \to \faisGm \to 1.$$
ou dans le cas d'un $\faisGL_n$-torseur
$$1 \to \faisSL_{1,A} \to \faisGL_{1,A} \to \faisGm \to 1.$$
\end{rema}

\subsection{Suites exactes longues de cohomologie}

Détaillons quelques suites exactes longues de cohomologie associées à des groupes de type $A_n$, en termes de torseurs et de groupes de Brauer d'algèbres d'Azumaya. 

\subsubsection{Gerbes liées par $\faismu_n$ et $\faisGm$ à la Brauer}

Commençons par rappeler quelques faits bien connus à propos du groupe de Brauer. Nous distinguerons le groupe de Brauer-Azumaya, défini en termes de classes d'algèbres d'Azumaya, du groupe de Brauer cohomologique, défini en termes de cohomologie étale.

\begin{defi} \label{BrauerAzumaya_defi}
Le \emph{Groupe de Brauer-Azumaya} $\Brauer(S)$ est l'ensemble des algèbres d'Azumaya sur $S$, quotienté par la relation $A \sim B$ s'il existe deux $\faisO_S$-modules localement libres de type fini $M$ et $N$ un isomorphisme de $\faisO_S$-algèbres $A\otimes_{\faisO_S} \faisEnd_M \simeq B \otimes_{\faisO_S} \faisEnd_N$ et muni de la loi de groupe induite par le produit tensoriel. On note $[A]$ la classe de $A$ dans $\Brauer(S)$.
\end{defi}
L'élément neutre du groupe est la classe de $\faisO_S$ elle-même, et l'inverse de $[A]$ est $[A^\opp]$ la classe de l'algèbre opposée: on dispose en effet de l'isomorphisme bien connu $A \otimes_{\faisO_S} A^\opp \simeq \faisEnd_A$ (ce dernier $A$ est vu comme $\faisO_S$-module) envoyant $a\otimes b$ sur $x \mapsto axb$. Dans le cas d'une base $S$ sur laquelle tous les $\faisO_S$-modules localement libres sont libres, par exemple un anneau local, voire un corps, on retrouve la relation d'équivalence plus connue, pour laquelle $A \sim B$ s'il existe $m,n \in \NN \setminus \{0\}$ et un isomorphisme $\faisM_m(A) \simeq \faisM_n(B)$.

Le groupe $\Het^2(S,\faisGm)$ coïncide avec $\Hfppf^2(S,\faisGm)$ car $\faisGm$ est lisse, par \ref{lisseetalefppf_theo}. Le \emph{groupe de Brauer cohomologique} en est la partie de torsion.
\medskip
 
La construction suivante est une légère généralisation de \cite[Ch. V, \S 4.2]{gir}. Soit $\AVec{A}$ le champ des $A$-modules qui sont localement libres comme $\faisO_S$-modules. \'Etant donnée une $\faisO_S$-algèbre d'Azumaya $A$ fixée, pour toute $\faisO_S$-algèbre d'Azumaya $B$, on considère le foncteur $\Final \to \Azumayan{}$ qui envoie le seul objet d'une fibre sur $B$, et le foncteur $\faisEnd:\AVec{A}_\grpd\to \Azumayan{}$. Le produit fibré $\Final \times_{\Azumaya} \AVec{A}_\grpd$ est un champ par l'exercice \ref{prodfibchamps_exo}, et un objet de ce champ est une paire $(M,\phi)$ où $M$ est un $A$-module localement libre de type fini (comme $A$-module) et $\phi:B \isoto \faisEnd_M$ est un isomorphisme de $\faisO_S$-algèbres (d'Azumaya). Un morphisme $(M,\phi) \to (N,\psi)$ est un isomorphisme $f:M \isoto N$ de $A$-modules tel que $\int_f \circ \phi = \psi$. Ce champ est donc une gerbe, puisque toute algèbre d'Azumaya est localement triviale.
\begin{defi} \label{banalisation_defi}
On appelle \emph{gerbe des $A$-banalisations de $B$} la gerbe ci-dessus. On la note $\Agerban{A}(B)$.
\end{defi}
Le morphisme de faisceaux $\faisGm  \to \faisAut^{\Agerban{A}(B)}_{M,\phi}$ envoyant un scalaire sur l'homothétie correspondante de $M$ est un isomorphisme, par l'exemple \ref{EndMAzumaya_exem}, on en déduit que cette gerbe est liée par $\faisGm$.
\begin{prop} \label{BrauerInject_defi}
L'application qui à une algèbre d'Azumaya $B$ associe la gerbe $\Agerban{A}(B)$ passe aux classes d'équivalences, et définit un morphisme injectif
$$\Brauer(S) \to \Het^2(S,\faisGm).$$
\end{prop}
\begin{proof}
Le morphisme précédent est en fait un morphisme de groupes lorsque $A$ est déployée, voir \loccit. L'injectivité vient alors du fait qu'une gerbe est neutre si et seulement si sa fibre sur $S$ est non vide. Pour le passage aux classes, voir \cite[Ch. V, Lemme 4.3]{gir}. Plus généralement, lorsque $A$ n'est pas déployée, on vérifie que la gerbe $\Agerban{\faisO_S}(A)\times \Agerban{\faisO_S}(B)$ est en fait équivalente à la gerbe $\Agerban{A}(B)$ par produit tensoriel sur les objets (comme dans \loccit, Ch. V, Lemme 4.3). On est donc ramené au cas $A$ déployé, à un décalage près. (Attention, lorsque $A$ n'est pas déployée, ce n'est donc plus un morphisme de groupes).
\end{proof}

Définissons maintenant un analogue du groupe de Brauer-Azumaya lorsqu'on remplace $\faisGm$ par $\faismu_n$. On considère les triplets $(B,P,\phi)$ où $B$ est une $\faisO_S$-algèbre d'Azumaya, $P$ est un $\faisO_S$-module, et $\phi:B^{\otimes n} \isoto \faisEnd_P$ est un isomorphisme de $\faisO_S$-algèbres. C'est donc une banalisation de $B^{\otimes n}$, mais en conservant la donnée de $B$. Deux tels triplets $(B_1,P_1,\phi_1)$ et $(B_2,P_2,\phi_2)$ sont dits Brauer-équivalents s'il existe 
\begin{itemize}
\item deux $\faisO_S$-modules localement libres $V_1$ et $V_2$
\item un isomorphisme de $\faisO_S$-modules $x:P_1 \otimes V_1^{\otimes n} \simeq P_2 \otimes V_2^{\otimes n}$
\item un isomorphisme de $\faisO_S$-algèbres $y: B_1 \otimes \faisEnd_{V_1} \simeq B_2 \otimes \faisEnd_{V_2}$  
\end{itemize}
qui satisfont à
$$ 
\xymatrix{
(B_1 \otimes \faisEnd_{V_1})^{\otimes n} \ar[r]^{y} \ar[d]_{\phi_1} & (B_2 \otimes \faisEnd_{V_2})^{\otimes n} \ar[d]_{\phi_2} \\ 
\faisEnd_{P_1} \otimes \faisEnd_{V_1}^{\otimes n} \ar[r]^{\int_x} & \faisEnd_{P_2} \otimes \faisEnd_{V_2}^{\otimes n} 
}
$$
(où on a supprimé les identifications canoniques $\faisEnd_{V\otimes W}\simeq \faisEnd_{V}\otimes\faisEnd_W$ et les permutations de facteurs des produits tensoriels pour alléger.)
Il est facile de voir que cela définit bien une relation d'équivalence. 
De plus, il y a un produit tensoriel évident sur les triplets: $(B_1,P_1,\phi_1)\otimes (B_2,P_2,\phi_2)=(B_1 \otimes B_2, P_1\otimes P_2, \phi)$ où la définition de $\phi$ se devine facilement. Ce produit tensoriel passe aux classes d'équivalences et définit une loi de groupe sur le quotient, la classe neutre étant $(\faisO_S,\faisO_S,\faisO_S^{\otimes n}\simeq \faisO_S =\faisEnd_{\faisO_S})$ avec la multiplication $\faisO_S^{\otimes n}\to \faisO_S$. L'inverse de la classe de $(B,P,\phi)$ est $(B^{\opp},P^{\dual},\phi^{\opp})$.
\begin{defi}
On note $\muBrauer{n}(S)$ le groupe ainsi obtenu. 
\end{defi}
Il est clair que $\muBrauer{n}(S)$ est fonctoriel en $S$. De plus, il s'envoie dans $\Brauer(S)$ en ne conservant que l'algèbre dans un triplet.

Définissons maintenant le champ $\nABan{n}{A}$ des algèbres d'Azumaya $n$-$A$-banalisées. Il s'agit du produit fibré $\Azumaya \times_{\Azumaya}\AVec{A}_\grpd$ obtenu en considérant les foncteurs $\Azumaya \to \Azumaya$ envoyant $B$ sur $B^{\otimes n}$ et $\AVec{A}_\grpd \to \Azumaya$ envoyant $P$ sur $\faisEnd^{\AVec{A}}_P$. Explicitement, un objet de ce champ est un $(B,P,\psi)$ avec $B$ une algèbre d'Azumaya, $P$ un $\faisO_S$-module localement libre, et $\psi$ un isomorphisme $B^{\otimes n} \simeq \faisEnd_P$.

\'Etant donné un tel objet $(B,P,\psi)$, considérons maintenant la catégorie fibrée 
$$\germuban{A}{n}(B,P,\psi)=\Final \times_{\nABan{n}{A}} \AVec{A}_\grpd$$ 
où le foncteur $\Final\to \nABan{n}{A}$ envoie le seul objet d'une fibre sur $(B,P,\psi)$, et où celui $\AVec{A}_\grpd \to \nABan{n}{A}$ envoie $N$ sur $(\End_N,N^{\otimes n}, \End_{N^{\otimes n}} \simeq \faisEnd_{N}^{\otimes n})$. C'est bien un champ, comme produit fibré de champs, et c'est une gerbe pour la topologie \fppf, essentiellement parce que toute algèbre d'Azumaya est localement triviale, et qu'on peut localement extraire les racines $n$-ièmes. 

De manière analogue à la proposition \ref{BrauerInject_defi}, on prouve:
\begin{prop}
L'application qui à une algèbre d'Azumaya trivialisée $(B,P,\psi)$ associe la gerbe $\germuban{A}{n}(B)$ passe aux classes d'équivalences, et définit un morphisme injectif
$$\muBrauer{n}(S) \to \Hfppf^2(S,\faisGm).$$
\end{prop} 

\subsubsection{Suites exactes}

Signalons enfin quelques identifications dans les suites exactes longues en cohomologie étale associées à certaines suites exactes courtes rappelées à la fin de la section \ref{torseurs_sec}. La suite exacte de faisceaux
$$1 \to \faisGmS{S} \to \faisGL_{1,A} \to \faisPGL_A \to 1$$ 
fournit une suite exacte longue
\begin{multline*}
1 \to \Gamma(S)^\times \to A(S)^\times \to \setAut_A(S) \\
\to \Het^1(S,\faisGm) \to \Het^1(S,\faisGL_{1,A}) \to \Het^1(S,\faisPGL_A) \\
\to \Het^2(S,\faisGm)
\end{multline*}
dans laquelle $\Het^1(S,\faisGm)$ et $\Het^1(S,\faisGL_{1,A})$ s'identifient par la proposition \ref{GLntors_prop} respectivement à $\Pic(S)$ et à l'ensemble des classes d'isomorphismes de $A^{\opp}$-modules localement libres de rang $1$. De même, $\Het^1(S,\faisPGL_A)$ s'identifie aux classes d'isomorphismes d'algèbres d'Azumaya (pointé par la classe de $A$). Modulo ces identifications, il est facile de comprendre les foncteurs entre champs qui induisent les flèches de cette suite exacte: l'homomorphisme de connexion $\setAut_A(S)\to \Pic(S)$ est donné par le foncteur qui envoie un automorphisme $\sigma$ sur le fibré en droites $L_\sigma$ tel que $L_\sigma(T) = \{ a \in A(T),\ a_T m =\sigma(m)a_T,\ \forall m \in A(T'),\ \forall T'\to T\}$. On vérifie immédiatement que c'est un fibré en droite en utilisant que localement, $\sigma$ est intérieur.

Le morphisme $\Het^1(S,\faisGm)\to \Het^1(S,\faisGL_{1,A})$ est induit par le foncteur $\Vecn{1}\to \AVecn{1}{A^{\opp}}$ qui envoie un fibré en droites $L$ sur $L \otimes A_d$, vu comme $A^{\opp}$-module par l'action à droite de $A$ sur $A_d$. 

Le morphisme $\Het^1(S,\faisGL_{1,A})\to \Het^1(S,\faisPGL_A)$ est induit par le foncteur $\EndFonc$, décrit juste avant la proposition \ref{GL1Atorseurs_prop}. 

Enfin, le morphisme $\Het^1(S,\faisPGL_A) \to \Het^2(S,\faisGm)$ est induit par le foncteur ``gerbe des relèvements'', défini en \ref{deuxbord_exem}. On obtient donc qu'une $\faisO_S$-algèbre d'Azumaya $B$ est envoyée sur la gerbe dont la fibre en $T$ est constituée des relèvements $(M,\phi)$, où $M$ est un $A^{\opp}$-module, et $\phi:\faisEnd^{A^{\opp}\text{-mod}}_M \isoto B$ est un isomorphisme de $\faisO_S$-algèbres d'Azumaya. On reconnaît donc la gerbe $\Agerban{A^{\opp}}(B)$ des $A^{\opp}$-banalisations de $B$ définie en \ref{banalisation_defi}. Cela prouve que ce dernier morphisme se factorise en fait par $\Brauer(S)$. 
\medskip

Passons maintenant à la suite exacte courte
$$1 \to \faisSL_n \to \faisGL_n \to \faisGm \to 1$$
qui donne la suite exacte longue
\begin{multline*}
1 \to \setSL_n(\Gamma(S)) \to \setGL_n(\Gamma(S)) \to \Gamma(S)^\times \\
\to \Het^1(S,\faisSL_n) \to \Het^1(S,\faisGL_n) \to \Het^1(S,\faisGm)
\end{multline*}
Outre les précédentes identifications, on peut également voir $\Het^1(S,\faisSL_n)$ comme l'ensemble des classes d'isomorphismes de couples $(M,d)$ où $M$ est un fibré de rang $M$ et $d:\Det{M} \to \faisO_S$ est un isomorphisme. Le morphisme de connexion envoie alors un élément $d \in \Gamma(S)^\times$ sur le couple $(\faisO_S^{\oplus n},d)$, et il est donc facile de voir que cet morphisme est trivial (mais en général le morphisme $\Het^1(S,\faisSL_n) \to \Het^1(S,\faisGL_n)$ n'est pas injectif).  Le morphisme suivant oublie $d$, et le dernier envoie un fibré sur son déterminant. 
\medskip

\section{Groupes de nature quadratique} \label{Groupesquadratiques_sec}

Introduisons maintenant un certain nombre de groupes classiques en rapport avec les modules quadratiques, comme le groupe orthogonal, le groupe spin, etc. et essayons de donner une description concrète des torseurs sous ces groupes. Les questions de lissité sont laissées de côté pour le moment, et seront traitées dans les parties de classification. La topologie de Grothendieck considérée est, sauf mention contraire, la topologie étale. 

\subsection{Groupe orthogonal}

Soit $(M,q)$ un module quadratique (voir partie \ref{modulesquadratiques_sec}). Rappelons en particulier que $M$ est un $\faisO_S$-module localement libre de type fini, et qu'on l'oublie souvent dans la notation.

Soit $\Quad$ la catégorie fibrée telle que $\Quad_T$ soit la catégorie des $\faisO_T$-modules quadratiques. C'est un champ par la proposition \ref{champstruc_prop}. 

\begin{prop} \label{Oqrepr_prop}
Le foncteur de points $\faisAut_{M,q}$ est représentable par un sous-schéma en groupes fermé de $\faisGL_M$, donc affine et de type fini sur $S$.
\end{prop}
\begin{proof}
On voit $q$ comme un élément de $\Symalg^2(M^\vee)$ 
et on vérifie que $\faisAut_q$ est l'image inverse de $\faisStab_q$ par le morphisme $\faisGL_M\to \faisGL_{\Symalg^2(M^\vee)}$ qui à $\alpha:M\to M$ associe l'automorphisme de $\Symalg^2(M^\vee)$ induit par $\alpha^\vee$. Il est donc représentable par la proposition \ref{stab_prop} et la remarque \ref{imageinvnoyau_rema}, comme un sous-schéma fermé de $\faisGL_M$, qui est de type fini sur $S$. Il est donc également affine de type fini sur $S$. 
Dans le cas affine, on identifie $\faisGL_M(S)=\setGL(M_R)$ grâce au lemme \ref{homMN_lemm}. 
\end{proof}
\begin{defi}
On note $\faisorthO_{q}$ le schéma en groupes défini par la proposition précédente et également $\faisorthO_n=\faisorthO_{\hypq_{n}}$ pour le module hyperbolique $\hypq_n$ de rang $n$.
\end{defi}
Notons que par définition,
$$\faisorthO_q(T)=\{\alpha \in \faisGL_M(T),\ q_{T'}\circ \alpha_{T'}(x) = q_{T'}(x),\ \forall T' \to T,\ \forall x \in M(T')\}.$$
Lorsque $S=\Spec(R)$ et que $(M,q)$ est le $\faisO_S$-module quadratique associé à un $R$-module quadratique $(M_R,q_R)$, on a 
$$\faisorthO_{q}(S)=\setorthO(q_R)=\{\alpha \in \setGL(M_R),\ q_R\circ \alpha = q_R\}$$
en utilisant l'équivalence $\wW$ et le lemme \ref{homMN_lemm}.
Ceci garantit que lorsque la base est un corps, le schéma en groupes défini ci-dessus est bien le schéma en groupes classique, utilisé par exemple dans \cite{borelag} ou \cite{bookinv}. 
Donnons tout de suite une autre caractérisation de $\faisorthO_q$ en rang pair qui servira par la suite, en termes de la paire quadratique associée, définie en \ref{deffq_prop}.
\begin{lemm} \label{Oqinvolution_lemm}
Si $q$ est un module quadratique régulier de rang constant $2n$, alors pour tout schéma $T$ sur $S$,
$$\faisorthO_q(T)=\{g \in \faisGL_M(T),\ \invadj_q(g)g =1,\ f_q \circ \int_g = f_q\}.$$
\end{lemm}
\begin{proof}
Notons $\phi_q:M \otimes M \isoto \faisEnd_M$ l'isomorphisme $m_1 \otimes m_2 \mapsto m_1 b_q(m_2,-)$. Rappelons qu'à travers cet isomorphisme, l'involution adjointe $\invadj_q$ correspond à l'échange des facteurs, et qu'on a évidemment $b_q(\invadj_q(x),y)=b_q(x,\invadj_q(y))$. On en tire que si $g \in \faisorthO_q(T)$, alors pour tous les points $x,y$ au-dessus de $T$, on a $b_q(x,\invadj_q(g)(y))=b_q(x,g^{-1}(y))$, d'où $\invadj_q(g)=g^{-1}$ car $b_q$ est non dégénérée par hypothèse. On vérifie également immédiatement que pour tout $g \in \faisGL_M(T)$, $\int_g \circ \phi_q(m_1 \otimes m_2)=\phi_q(g(m_1)\otimes \invadj_q(g^{-1})(m_2)$.
D'où, 
\begin{equation} \label{fqintg_eq}
f_q \circ \int_g \circ \phi_q (m\otimes m)= f_q \circ \phi_q(g(m)\otimes g(m))=q(g(m))=q(m)
\end{equation}
si $g \in \faisorthO_q(T)$. Mais par la propriété d'unicité de $f_q$ de \ref{deffq_prop}, on a alors $f_q\circ \int_g=f_q$.   
Réciproquement, si $g \in \faisGL_M(T)$ vérifie les deux propriétés, on obtient $q(m)=f_q \circ \phi_q (m\otimes m)=f_q \circ \int_g \circ \phi_q (m\otimes m) = q(g(m))$ pour tous les points $m$ au-dessus de $T$, donc $g \in \faisorthO_q(T)$. 
\end{proof}

\begin{prop} \label{formesmodulesquad_prop}
Le foncteur qui à un module quadratique $q$ associe le torseur $\faisIso_{\hypq_{2n},q}$ (resp. $\faisIso_{\hypq_{2n+1},q}$) définit une équivalence de catégories fibrées  
\begin{itemize}
\item du champ $\Formes{\hypq_{2n}}$ vers le champ $\Tors{\faisorthO_{2n}}$;
\item du champ $\Formes{\hypq_{2n+1}}$ vers le champ $\Tors{\faisorthO_{2n+1}}$. 
\end{itemize}
Ceci est valable pour la topologie étale ou \fppf.
\end{prop}
\begin{proof}
C'est une application immédiate de \ref{tordusformes_prop}.
\end{proof}

Soit $q_0$ un module quadratique. Considérons le foncteur fibré $\Quad \to \Quad$ qui à un module quadratique $q$ associe le module quadratique $q \orth q_0$. 
\'Etant donné une forme $q$, ce foncteur induit donc l'inclusion $\faisorthO_q \to \faisorthO_{q\orth q_0}$ prolongeant un automorphisme par l'identité sur le sous-espace correspondant à $q_0$. 
\begin{lemm} \label{Osousforme_lemm}
Si $q'$ est une forme de $q$, correspondant à un torseur sous $\faisorthO_q$, alors le morphisme $\faisorthO_q \to \faisorthO_{q \orth q_0}$ se tord par la procédure de la proposition \ref{foncttors_prop} en le morphisme $\faisorthO_{q'} \to \faisorthO_{q' \orth q_0}$.
\end{lemm}
\begin{proof}
C'est une application directe de la proposition \ref{foncttors_prop}.
\end{proof}

\subsection{Algèbres de Clifford}

Rappelons maintenant la construction et quelques propriétés de l'algèbre de Clifford d'un module quadratique $(M,q)$, toujours supposé localement libre de type fini. Par le corollaire \ref{Wequivloclibres_coro}, on peut considérer que $M=\wW(\cM)$ et que la structure de module quadratique sur $M$ provient d'une structure de même nature sur $\cM$, également notée $q$.

On considère l'algèbre tensorielle $\bigoplus_i \cM^{\otimes i}$, construite à l'aide du produit tensoriel des $\cO_S$-modules. 
On considère alors le quotient, en tant que préfaisceau, de ce produit tensoriel par l'idéal bilatère engendré par les sections de la forme $m^{\otimes 2}-q(m)1$, puis on le faisceautise (Zariski). On applique $\wW$ au faisceau en $\cO_S$-algèbres obtenu, qui hérite d'une $\ZZ/2$-graduation de l'algèbre tensorielle. 

\begin{defi} \label{algCliff_defi}
Pour tout module quadratique $(M,q)$ sur la base $S$, on note $\faisCliff_{M,q}$ ou même $\faisCliff_q$ le faisceau en $\faisO_S$-algèbres $\ZZ/2$-graduées construit ci-dessus. On l'appelle \emph{algèbre de Clifford} de $(M,q)$. On note $\faisCliff_{0,q}$ et $\faisCliff_{1,q}$ ses parties respectivement paire et impaire. 
\end{defi}

Dans le cas affine, le $\cO_S$-module $\cM$ étant cohérent, 
on a:
\begin{prop}
Si $T=\Spec(R')$, alors $\faisCliff_{M,q}(R')=\setCliff(M(R'),q_{R'})$, où ce dernier est l'algèbre de Clifford usuelle d'un module quadratique sur un anneau (voir \cite[Ch. IV, \S 1]{knus}).
\end{prop}

\begin{prop} \label{chgmtBaseCliff_prop}
Pour tout morphisme $T \to S$, on a $\faisCliff_{q_{T}} \simeq (\faisCliff_q)_{T}$.
\end{prop}
\begin{proof}
Voir \cite[Ch. IV, Prop. 1.2.3]{knus} pour le cas affine, avant d'appliquer le foncteur $\wW$, qui commute au changement de base par la proposition \ref{Wprop_prop}. Ce raisonnement s'adapte directement au cas général (non affine) ou bien, au choix, on utilise le cas affine et la naturalité de l'isomorphisme pour globaliser le morphisme et vérifier que c'est un isomorphisme.
\end{proof}

L'algèbre de Clifford d'un module quadratique $(M,q)$ satisfait à la propriété universelle suivante. C'est l'objet initial de la catégorie des paires $(A,j)$ où $A$ est une $\faisO_S$-algèbre localement libre de type fini, et où $j:M \to A$ est un morphisme de $\faisO_S$-module, qui satisfait à $j(m)^2=q(m)\cdot 1$ dans l'algèbre $A$ (ce qui peut s'exprimer comme un diagramme).
Vérifier cette propriété universelle revient à construire un morphisme, évident sur le préfaisceau, puis qu'on faisceautise.
Cela permet par exemple de construire deux involutions:
soit $\invstd: \faisCliff_q \to \faisCliff_q^\opp$ l'unique morphisme de $\faisO_S$-algèbres qui vérifie $\invstd(m)=m$ (resp. $\invstd(m)=-m$) pour tout $m \in M(T)$. Il est automatiquement d'ordre $2$ ($\invstd^2=\id$) par la propriété universelle. 
\begin{defi}[involutions] \label{involutionstd_defi}
L'involution $\invstd$ de $\faisCliff_q$ telle que $\invstd_{|M}=\id_M$ est appelée l'involution canonique. Celle telle que $\invstd_{|M}=-\id_M$ est appelée l'involution standard, et on note cette dernière $\invstd_q$.
\end{defi}

Considérons la forme hyperbolique $\hypq_M$ sur le module $M \oplus M^\dual$ (avec la notation de \ref{hyperbolique_defi}), et soit $\totExt M = \bigoplus_i \Lambda^i M$. On définit une application $j: M \oplus M^\dual \to \faisEnd_{\totExt M}$ sur les points en envoyant $m \in M(T)$ sur l'endomorphisme $m\wedge(-)$ et $f \in M^\vee(T)$ sur $d_f: m_1\wedge \ldots \wedge m_r \mapsto \sum_{i=1}^r (-1)^{i-1} f(m_i) m_1\wedge \ldots \check{m}_i \ldots \wedge m_r$. On vérifie qu'elle satisfait bien à $j(m)^2=q(m)\cdot Id_{\totExt M_T}$ (voir par exemple \cite[Ch. IV, Prop. 2.1.1]{knus}. Cela induit donc par propriété universelle un morphisme
$$\faisCliff_{\hypq_M} \to \faisEnd_{\totExt M}.$$
\begin{prop} \label{CliffHypPaire_prop}
Le morphisme ci-dessus est un isomorphisme de $\faisO_S$-algèbres graduées si on met la graduation sur $\faisEnd_{\totExt M}$ induite par la décomposition $\totExt M \simeq \totExtPair M \oplus \totExtImpair M$ (puissances extérieures paires et impaires).
\end{prop}
\begin{proof}
Cela se vérifie localement, auquel cas voir \cite[Ch. IV, Prop. 2.1.1]{knus}.
\end{proof}

Puisque $\faisEnd_{\totExt M}$ est une $\faisO_S$-algèbre d'Azumaya, il en est de même pour $\faisCliff_q$ pour toute forme $q$ de $\hypq_{2n}$ par \ref{chgmtBaseCliff_prop}. En particulier, le centre de $\faisCliff_q$ est alors $\faisO_S$. On obtient donc un foncteur de catégories fibrées:
$$\CliffFonc:\RegQuadn{2n}\to \Azumayan{2^n}.$$
Nous considérerons plus tard également un foncteur similaire et du même nom à valeurs dans les algèbres d'Azumaya graduées.

\begin{theo}[Structure des algèbres de Clifford en rang pair] \label{strucalgcliffpaires_theo}
Si $(M,q)$ est un module quadratique régulier de rang $2n$, alors son algèbre de Clifford est une $\faisO_S$-algèbre d'Azumaya. Le centre $Z_0$ de l'algèbre de Clifford paire $\faisCliff_{0,q}$ est une $\faisO_S$-algèbre étale finie de rang $2$, et $\faisCliff_{0,q}$ est d'Azumaya dessus.
\end{theo}
\begin{proof}
Toutes ces propriétés peuvent se vérifier localement, auquel cas on est ramené à \cite[Ch. IV, (2.2.3)]{knus}.
\end{proof}
\medskip

Avant de passer aux modules quadratiques de rang impair, introduisons l'algèbre de Clifford paire d'une paire quadratique de rang pair. Nous nous bornons à vérifier que la construction donnée dans \cite[Ch. II, § 8B]{bookinv} sur un corps s'étend sans réelle difficulté à une base quelconque. 

Soit $(A,\sigma,f)$ une paire quadratique. 
\begin{lemm} \label{lsigma_lemm}
Localement pour la topologie de Zariski, il existe un élément $l$ tel que $\trd(ls)=f(s)$ pour tout $s\in \faisSym_{A,\sigma}(S)$, et cet élément vérifie $l+\sigma(l)=1$. Deux tels éléments diffèrent localement d'un élément alterné, c'est-à-dire de la forme $\sigma(a)-a$, $a \in A(S)$.
\end{lemm}
\begin{proof}
Sur un corps, ceci est prouvé en \cite[Ch. I, (5.7)]{bookinv}. Répétons-donc cette preuve avec les légères modifications nécessaires. On commence par se placer sur un ouvert sur lequel le module $\faisSym_{A,\sigma}$ est facteur direct de $A$. Cela permet d'étendre la forme linéaire $f$ à tout $A$ (n'importe comment). Puis, par régularité de la forme trace réduite (lemme \ref{trdreg_lemm}), on trouve alors $l$ tel que $f(a)=\trd(la)$ sur tout $A$. On a alors $\trd(\sigma(l)a)=\trd(\sigma(a)l)=\trd(l\sigma(a))$, d'où
$$\trd\big((l+\sigma(l))a\big)=\trd(l(a+\sigma(a))=f(a+\sigma(a))=\trd(a)$$ 
et $l+\sigma(l)=1$ par régularité de la forme trace.
Par ailleurs, par le lemme \ref{symAsigma_lemm}, point \ref{SymAltorth_item}, il est clair que deux tels éléments $l$ diffèrent localement d'un élément alterné.
\end{proof}

Par l'équivalence de catégories $\wW$, $A$ provient d'une algèbre d'Azumaya $\cA$ sur $\cO_S$. On considère l'algèbre tensorielle $T(\cA)$, puis le faisceau Zariski quotient $T(\cA)/(J_1(\sigma,f) + J_2(\sigma,f))$ où $J_1(\sigma, f)$ est le faisceau d'idéaux engendré par les éléments de la forme $s - f(s)$ avec $s \in \faisSym_{A,\sigma}$ et $J_2(\sigma, f)$ est le faisceaux d'idéaux engendré par les élément de la forme $\iota(u)$, où $\iota$ se définit de la façon suivante (voir \cite[Ch I, 8B]{bookinv}). 

Soit $\sand: \cA\otimes \cA \to \faisEnd_{\cA}$ l'application ``sandwich'', qui envoie $a \otimes b$ sur $x \mapsto a x b$. Premièrement, on définit $\sigma_2:\cA \otimes \cA \to \cA \otimes \cA$ par la condition $\sand(u)(\sigma(a))=\sand(\sigma_2(u))(a)$ pour des points $u$ de $\cA \otimes \cA$ et $a$ de $\cA$. Puis, on vérifie que si $u$ est tel que $\sigma_2(u)=u$, et si $l$ satisfait aux conditions du lemme \ref{lsigma_lemm} sur un ouvert, alors $\sigma(u)(l)$, ne dépend pas du choix d'un tel $l$. Cela permet par recollement (descente) de construire l'application $\iota=\sigma(-)(l)-\id$, de source le sous-faisceau de $\cA$ des éléments tels que $\sigma_2(a)=a$. Enfin, on définit l'algèbre de Clifford paire de $(\cA,\sigma,f)$ comme le quotient Zariski de $T(\cA)/(J_1(\sigma,f) + J_2(\sigma,f))$. 
\begin{defi} \label{AlgCliffPaireQuad_defi}
L'\emph{algèbre de Clifford} d'une paire quadratique $(A,\sigma,f)$, notée $\faisCliff_{0,A,\sigma,f}$, est définie comme la $\faisO_S$-algèbre image par le foncteur $\wW$ de la $\cO_S$-algèbre $(\cA,\sigma,f)$, définie ci-dessus. 
\end{defi}
Elle est naturellement munie d'une involution standard $\underline{\sigma}$ induite par l'inversion de l'ordre des facteurs sur l'algèbre tensorielle de $A$, et par l'involution $\sigma$ sur chaque facteur $A$ (voir \cite[Ch. II (8.11)]{bookinv}).

\begin{theo} \label{CompareClifford_theo}
Si $(M,q)$ est un module quadratique régulier de rang $2n$, alors on a un isomorphisme canonique
$$\faisCliff_{0,q} \simeq \faisCliff_{0,\faisEnd_M,\invadj_q,f_q}$$
où la paire quadratique $(\faisEnd_M,\invadj_q,f_q)$ a été définie en \ref{deffq_prop}. 
Par cet isomorphisme, l'involution standard $\invstd_q$ restreinte à $\faisCliff_{0,q}$ correspond à l'involution standard $\underline{\sigma}$.
\end{theo}
\begin{proof}
Le morphisme canonique est construit en passant au quotient de la partie paire de l'algèbre tensorielle, au moyen de l'identification canonique $M \otimes M \to \faisEnd_M$ envoyant $m_1 \otimes m_2$ sur $m_1 b_q (m_2,-)$. Ensuite, il suffit de vérifier que c'est un isomorphisme sur tous les corps résiduels des points fermés. On est donc ramené au cas des corps, qui est traité dans \cite[Ch. II, Prop. (8.8)]{bookinv}.
\end{proof}

\begin{theo}[Structure des algèbres de Clifford paires des paires quadratiques] \label{CliffodPaireQuad_theo}
Soit $(A,\sigma,f)$ une paire quadratique. Son algèbre de Clifford paire $\faisCliff_{0,A,\sigma,f}$ a pour centre une $\faisO_S$-algèbre étale finie de degré $2$, notée ici $Z_{0,A,\sigma,f}$. Elle est d'Azumaya sur ce centre.
\end{theo}
\begin{proof}
Tout cela peut se vérifier étale localement, et découle alors immédiatement des théorèmes \ref{strucalgcliffpaires_theo} et \ref{CompareClifford_theo} puisque localement $(A,\sigma,f)\simeq (\faisM_n,\invadj_{2n},f_{2n})$ par le théorème \ref{pairequadformes_coro}. 
\end{proof}

Ceci permet de définir un foncteur fibré 
$$\ArfFonc:\PairesQuadn{2n} \to \Etn{2},$$
du champ des paires quadratiques vers celui des algèbres étales finies de rang $2$, qui envoie une paire quadratique sur le centre de son algèbre de Clifford paire. 
\medskip

Considérons maintenant le module hyperbolique de rang impair $\hypq_{2n+1}$ et soit $e$ un vecteur de base du module sous-jacent à $\fq{1}$ tel que $q(e)=1$. 
\begin{theo}[Structure des algèbres de Clifford en rang impair, cas déployé] \label{CliffordSplit_theo}
\begin{enumerate}
\item \label{centreCliffhyp_item} Le centre de $\faisCliff_{\hypq_{2n+1}}$, noté ici $Z$, est une $\faisO_S$-algèbre $\ZZ/2$-graduée de degré $2$ telle que $Z_0=\faisO_S$ (comme $\faisO_S$-algèbre) et $Z_1\simeq \faisO_S$ (comme $\faisO_S$-module) avec pour multiplication $Z_1 \otimes Z_1 \to Z_0$ donné par la multiplication usuelle de $\faisO_S \otimes \faisO_S \to \faisO_S$ (qui est donc un isomorphisme). 
\item \label{actionpardet_item} L'action d'un élément $f \in \faisorthO_{\hypq_{2n+1}}(S)$ par $\CliffFonc(f)$ sur $\faisCliff_{\hypq_{2n+1}}$ se restreint à $Z_1$ et agit dessus par le déterminant de $f$. 
\item \label{C0impair_item} L'inclusion $i:\hypq_{2n} \to \hypq_{2n+1}$ induit par la décomposition $\hypq_{2n+1}=\hypq_{2n} \orth \fq{1}$ induit un isomorphisme $\faisCliff_{\hypq_{2n}} \simeq \faisCliff_{0,\hypq_{2n+1}}$ qui envoie un élément $(m,g)$ de $\faisO_S^n \oplus (\faisO_S^{n})^\dual$ sur $i(m,-g)\cdot e$.
\item \label{multCliff_item} La multiplication induit un isomorphisme $Z \otimes \faisCliff_{0,\hypq_{2n+1}} \to \faisCliff_{\hypq_{2n+1}}$ d'algèbres graduées.
\end{enumerate}
\end{theo}
\begin{proof} 
Tout les points peuvent se tester localement pour la topologie Zariski. 
On se ramène donc au cas où la base $S$ est affine. Pour les points \ref{centreCliffhyp_item}, \ref{C0impair_item} et \ref{multCliff_item}, voir alors \cite[Ch. IV, Prop. 3.2.2]{knus}. Pour prouver le point \ref{actionpardet_item}, notons que (3.2.3) p. 210 de \loccit fournit un générateur de $Z_1$, de la forme $w=x e$ où $e$ est comme ci-dessus et $x$ est un élément de $i(\faisCliff_{0,\hypq_{2n}})$. En particulier, $ex=xe$. On a de plus $w^2=e^2 x^2=1$. Toute réflexion orthogonale $\tau_v$ par rapport à un vecteur $v$ tel que $q(v)$ est inversible est de déterminant $(-1)$ par \ref{detrefl_lemm} et agit bien par son déterminant, car $\CliffFonc(\tau_v)=-\int_v$ et $w$ est central. L'affirmation est donc vraie pour les produits de réflexions. On peut par ailleurs se ramener au cas du spectre d'un anneau local, qui est une limite de voisinages Zariski, puisque tout élément vit à un niveau fini dans la limite. Malheureusement, même sur un anneau local, tout élément de $\faisorthO_{\hypq_{2n+1}}(S)$ n'est pas un produit de réflexions (voir \ref{contreExCartanDieudonne_rema}). Soit $F$ le module sous-jacent au sous-module quadratique $\hypq_{2n}$, qui est régulier. Un élément $f\in\faisorthO_{\hypq_{2n+1}}(S)$ envoie le plan $F$ sur son image $f(F)$, et par le théorème de Kneser \ref{Kneser_theo}, on peut donc trouver un produit de réflexions $g$ qui coïncide avec $f$ sur $F$. L'élément $f\circ g^{-1}$ est donc l'identité sur $F$, et préserve donc son orthogonal, qui est engendré par $e$. On a donc $f\circ g^{-1}(e)=\lambda e$ pour un certain $\lambda \in \faismu_2(S)$, et $\CliffFonc(f\circ g^{-1})(x)=x$, d'où $\CliffFonc(f\circ g^{-1})(w)=\lambda w$. Ainsi, $f\circ g^{-1}$ agit bien par son déterminant $\lambda$ et comme $g$ est un produit de réflexions, l'affirmation est vraie pour $f$. 
\end{proof}

\begin{theo}[Structure des algèbres de Clifford en rang impair, cas général] \label{CliffStruc_theo}
Soit $(M,q)$ une forme \fppf\ de $\hypq_{2n+1}$.
\begin{enumerate}
\item \label{centredet_item} Le centre de $\faisCliff_{q}$, noté ici $Z$, est une $\faisO_S$-algèbre $\ZZ/2$-graduée de degré $2$ telle que $Z_1=\faisO_S$ (comme $\faisO_S$-algèbre) et $Z_1\simeq \Det M$ (comme $\faisO_S$-module) par un isomorphisme canonique. 
\item \label{demi-det_item} La multiplication $Z_1 \otimes Z_1 \to Z_0$, canoniquement isomorphe à $\Det M \otimes \Det M \to \faisO_S$, est un isomorphisme.
\item Un élément $f \in \faisorthO_{q}(S)$ agit par son déterminant sur $Z_1$.
\item Lorsque $q=\hypq_M \orth \fq{1}$, alors l'inclusion $i:\hypq_M \to q$ induit un isomorphisme de $\faisO_S$-algèbres $\faisCliff_{\hypq_M} \simeq \faisCliff_{0,q}$ en envoyant un élément $(m,g)$ de $M \oplus M^\dual$ vers $(m,-g)e$. En particulier, $\faisCliff_{0,q}$ est une $\faisO_S$-algèbre d'Azumaya.
\item La multiplication induit un isomorphisme $Z \otimes \faisCliff_{0,q} \to \faisCliff_q$ d'algèbres graduées, autrement dit, $\faisCliff_{1,q}\simeq \faisCliff_{0,q}\otimes \Det M$ comme $\faisCliff_{0,q}$-module et la multiplication $\faisCliff_{1,q}^{\otimes 2} \to \faisCliff_{0,q}$ est induite par celle sur $Z_1\otimes Z_1 \to \faisO_S$.
\end{enumerate}
\end{theo}
\begin{proof}
Toutes ces affirmations se prouvent localement pour la topologie \fppf, on est donc ramené à la proposition précédente, à l'exception de l'identification de $Z_1 \simeq \Det M$. Pour cela, on utilise que tout morphisme dans $\Formes{\hypq_{2n+1}}$ agit par le foncteur $\CliffFonc$ en préservant le centre et la graduation, donc en préservant $Z_1$, et définit donc un foncteur fibré $\Formes{\hypq_{2n+1}} \to \Vecn{1}$. Par le point \ref{actionpardet_item} de la proposition \ref{CliffordSplit_theo}, ce foncteur induit le morphisme déterminant $\faisorthO_{\hypq_{2n+1}}=\faisAut_{\hypq_{2n+1}} \to \faisAut_{\faisO_S}=\faisGm$. Par la proposition \ref{isoFoncFib_prop}, il est donc canoniquement isomorphe au foncteur déterminant $\Formes{\hypq_{2n+1}} \to \Vecn{1}$, qui envoie le module $(q,M)$ sur le fibré en droites déterminant $\Det{M}$ de $M$..
\end{proof}

\begin{rema} \label{Cliffhyp_rema}
En particulier, lorsque $q=\hypq_M \orth \fq{1}$, alors $\faisCliff_{0,q} \simeq \faisEnd_{\totExt M}$ comme $\faisO_S$-algèbre et $\faisCliff_{1,q} \simeq \faisEnd_{\totExt M}$ comme $\faisCliff_{0,q}$-module. 
\end{rema}

\begin{defi} \label{demi-det_defi}
Le foncteur fibré \emph{demi-déterminant} 
$$\qDemiDetFonc:\Formes{\hypq_{2n+1}}\to \ModDet$$ 
associe à un module quadratique $q$ le module déterminant $(\Det{M}, \Det{M} \otimes \Det{M} \isoto \faisO_S)$ défini par le point \ref{demi-det_item} de la proposition \ref{CliffStruc_theo}. La définition du foncteur sur les morphismes est induit par la fonctorialité de l'algèbre de Clifford. 
\end{defi}

La composée des deux foncteurs fibrés
$$\Formes{\hypq_{2n+1}} \to \ModDet \to \Vecn{1}$$
est donc par construction canoniquement isomorphe au foncteur oubli $\Formes{\hypq_{2n+1}}\to \Vecn{2n+1}$ suivi du déterminant $\DetFonc: \Vecn{2n+1} \to \Vecn{1}$. 

\begin{rema} 
La notion de demi-déterminant est due à Kneser, dans le cadre des modules quadratiques quelconques sur un anneau.
Il peut se calculer de manière explicite sur une base, voir \cite[Ch. IV, §3]{knus}.
\end{rema}

\subsection{Groupe spécial orthogonal}

Considérons maintenant les $S$-foncteurs fibrés (pour la topologie étale ou \fppf) définis en \ref{torseurdet_defi} et \ref{demi-det_defi}
$$
\hfil
\begin{array}{clc}
\qDetFonc:\Formes{\hypq_{2n}} & \to & \ModDet \\
q & \mapsto & \qDetMod{q}
\end{array}
\qquad
\text{et} 
\qquad
\begin{array}{clc}
\qDemiDetFonc:\Formes{\hypq_{2n+1}} & \to & \ModDet \\
q & \mapsto & \qDemiDetMod{q}.
\end{array}
\hfil
$$

Définissons maintenant le groupe spécial orthogonal. Les carrés commutatifs de foncteurs fibrés\footnote{Le carré de droite n'est commutatif qu'à un isomorphisme canonique près, comme expliqué après la définition \ref{demi-det_defi}.} 
$$
\xymatrix{
\Formes{\hypq_{2n}} \ar[r]^-{\qDetFonc{}} \ar[d] & \ModDet \ar[d] \\
\Vecn{2n} \ar[r]^{\DetFonc} & \Vecn{1} 
}
\hspace{5ex}
\xymatrix{
\Formes{\hypq_{2n+1}} \ar[r]^-{\qDemiDetFonc{}} \ar[d] & \ModDet \ar[d] \\
\Vecn{2n+1} \ar[r]^{\DetFonc} & \Vecn{1} 
}
$$
où les foncteurs verticaux ne conservent que le module sous-jacent, induisent sur les automorphismes de $\hypq_{2n}$ (resp. $\hypq_{2n+1}$) et de ses images par les différents foncteurs les carrés commutatifs de morphismes de faisceaux\footnote{Cette fois-ci le carré de droite est bien strictement commutatif.}
$$
\xymatrix{
\faisorthO_{2n} \ar[r]^{\det} \ar@{^(->}[d] & \faismu_2 \ar@{^(->}[d] \\
\faisGL_{2n} \ar[r]^{\det} & \faisGm
}
\hspace{8ex}
\xymatrix{
\faisorthO_{2n+1} \ar[r]^{\det} \ar@{^(->}[d] & \faismu_2 \ar@{^(->}[d] \\
\faisGL_{2n+1} \ar[r]^{\det} & \faisGm
}
$$
La proposition \ref{foncttors_prop} donne alors immédiatement:
\begin{lemm} \label{detOq_lemm}
Soit $h$ l'un des modules quadratiques $\hypq_{2n}$ ou $\hypq_{2n+1}$ et soit $q$ une forme étale (resp. \fppf) de $h$, obtenue en tordant $h$ par le $\faisO_h$-torseur $\faisIso_{h,q}$. 
Alors si $M$ est le module sous-jacent à $q$, les carrés commutatifs ci-dessus se tordent par $\faisIso_{h,q}$ et ses poussés $\faisIso_{\faisO_S^{2n},M}$ (resp. $\faisIso_{\faisO_S^{2n+1},M}$), $\faisIso_{\qDetMod{h},\qDetMod{q}}$ (resp. $\faisIso_{\qDemiDetMod{h},\qDemiDetMod{q}}$) et $\faisIso_{\faisO_S,\Det{M}}$ en
$$
\xymatrix{
\faisorthO_{q} \ar[r]^{\det} \ar@{^(->}[d] & \faismu_2 \ar@{^(->}[d] \\
\faisGL_{M} \ar[r]^{\det} & \faisGm.
}
$$
\end{lemm}

Soit $\faisSO_{M,q}$ le noyau du morphisme déterminant $\faisorthO_{M,q}\to \faismu_2$.
Ses points sont donc donnés par
$$\faisSO_{M,q}(S) =\setSO(q_R)$$
lorsque $S=\Spec(R)$ et $(M,q)$ est le $\faisO_S$-module quadratique associé à un $R$-module quadratique $(M_R,q_R)$. 
Le faisceau en groupes $\faisSO_{M,q}$ est un sous-groupe fermé de $\faisorthO_{M,q}$ représenté par un schéma affine de type fini sur $\faisO_S$ par le lemme \ref{produit_fibre_lemm}.
 
\begin{defi} \label{SO_defi}
On note $\faisSO_{M,q}$, ou même simplement $\faisSO_q$, le schéma en groupes ainsi obtenu, et également $\faisSO_{n}$ lorsque $q = \hypq_n$.
\end{defi}

\begin{rema}
Si $s\in \Gamma(S)^\times$, on a les isomorphismes canoniques $\faisorthO_{M,sq}\simeq \faisorthO_{M,q}$ et $\faisSO_{M,sq}\simeq \faisSO_{M,q}$ comme sous-groupes de $\faisGL_{M}$.
\end{rema}

\begin{prop} \label{seSOOimpair_prop}
La suite
$$1 \to \faisSO_{2n+1} \to \faisorthO_{2n+1} \oto{\det} \faismu_2 \to 1$$
est une suite exacte de faisceaux étales.
Soit $q$ une forme \fppf\ de $\hypq_{2n+1}$. Par le torseur $\faisIso_{\hypq_{2n+1},q}$, la suite exacte précédente se tord en
$$1 \to \faisSO_q \to \faisorthO_q \oto{\det} \faismu_2 \to 1$$
qui est donc également une suite exacte de faisceaux étales ou \fppf.
\end{prop}
\begin{proof}
Pour avoir l'exactitude de la première suite, il suffit, par définition de $\faisSO_{2n+1}$, de montrer que le morphisme déterminant est surjectif. Or la suite est trivialement scindée en rang impair.
La suite exacte tordue est alors une conséquence directe du lemme précédent et de la proposition \ref{torsionsec_prop}.
\end{proof}

\begin{prop} \label{SOCartanDieudonne_prop}
Lorsque la base $S$ est le spectre d'un anneau local, et lorsque $q$ est de la forme $q=\fq{1}\orth q'$ avec $q'$ régulière (en particulier $q$ est une forme \fppf\ de $\hypq_{2n+1}$), alors tout élément de $\faisSO_{q}(S)$ est un produit de réflexions orthogonales. 
\end{prop}
\begin{proof}
Par le même argument que dans la fin de la preuve du théorème \ref{CliffordSplit_theo}, en utilisant le théorème de Kneser \ref{Kneser_theo}, on voit que tout élément de $\faisorthO_q$ est un produit de réflexions orthogonales et d'une application, notée ici $u$, envoyant $e$ (base du sous-module sous-jacent à $\fq{1}$) sur $\lambda e$, $\lambda \in \faismu_2$ et se restreignant à l'identité sur le sous-module $F$ sous-jacent à $q'$. Or le déterminant d'un tel produit étant $\lambda \cdot (-1)^n$, il ne sera dans $\faisSO_q$ que si $\lambda = \pm 1$. Si $\lambda = 1$, alors $u=\id$, et on peut la supprimer du produit, et si $\lambda=-1$, alors $u$ est la réflexion par rapport au vecteur $e$. 
\end{proof}

En rang pair, la situation est un peu différente: le déterminant $\faisorthO_{2n} \to \faismu_2$ n'est pas surjectif, ce n'est même pas un épimorphisme de faisceaux en général, comme nous allons le voir. Considérons le morphisme du faisceau (localement) constant $\ZZ/2$ vers $\faismu_2$, qui sur les points, envoie $0$ sur $1$ et $1$ sur $-1$. C'est bien entendu un morphisme de faisceaux en groupes abéliens. Son image sera notée provisoirement $\Im(\ZZ/2)$. Remarquons que si $2=0$ sur la base, cette image est le faisceau en groupes trivial, alors que si $2$ est inversible, cette image est $\faismu_2$ tout entier.

\begin{prop} \label{seSOOpair_prop}
Le déterminant $\faisorthO_{2n} \to \faismu_2$ a pour image $\Im(\ZZ/2) \subseteq \faismu_2$ et on a donc une suite exacte de faisceaux 
$$1 \to \faisSO_{2n} \to \faisorthO_{2n} \oto{\det} \Im(\ZZ/2) \to 1$$
qui est scindée.
Soit $q$ une forme de $\hypq_{2n}$. Par le torseur $\faisIso_{\hypq_{2n},q}$, la suite exacte précédente se tord en
$$1 \to \faisSO_q \to \faisorthO_q \oto{\det} \Im(\ZZ/2) \to 1$$
qui est donc également une suite exacte de faisceaux étales ou \fppf.
\end{prop}
\begin{proof}
On peut vérifier la factorisation du déterminant localement pour la topologie de Zariski. Par un argument de limite, on peut donc supposer l'anneau local. De plus, quitte à injecter cet anneau dans un anneau plus gros (par exemple par hensélisation stricte), on peut supposer que le corps résiduel n'est pas $\FF_2$. Alors, par le théorème \ref{Cartan-Dieudonne_theo}, tout élément de $\faisorthO_{2n}$ est un produit de réflexions, et a donc un déterminant égal à $1$ ou $-1$, par \ref{detrefl_lemm}. De plus, $-1$ est atteint par n'importe quelle réflexion orthogonale $\tau_v$ (déf. \ref{reflexion_defi}), et il en existe bien car il y a des vecteurs $v$ tels que $q(v)$ est inversible, par exemple le vecteur $(1,1,0,\ldots,0)$. On obtient donc bien une surjection sur les points, et elle est scindée par l'application définie sur les points par $p \mapsto (1-p)\id + p \tau_v$ (à l'aide de la description de $\ZZ/2$ de la remarque \ref{Zmu2_rema}). On a donc bien un épimorphisme de faisceaux $\det: \faisorthO_{2n} \to \Im(\ZZ/2)$.

La suite tordue s'obtient alors par la proposition \ref{torsionsec_prop}, après identification du tordu de $\faisorthO_{2n}$ avec $\faisorthO_q$ par la proposition \ref{auttordus_prop}. Notons que $\Im(\ZZ/2)$ n'est pas tordu parce qu'il est abélien et toute action par automorphismes intérieurs est donc triviale.
\end{proof}

Dans le cas de rang pair, un autre groupe nous sera plus utile que $\faisSO$. Il s'agit du groupe $\faisOplus$, que nous introduisons maintenant.

\begin{defi} \label{Oplus_defi}
Pour tout module quadratique régulier $(M,q)$ de rang constant $2n$, on note $\faisOplus_{M,q}$ ou $\faisOplus_{q}$ le noyau de l'application $\faisorthO_q =\faisAut_{q}\to \faisAut_{Z_{0,q}}=\ZZ/2$ induite par le foncteur $\ArfFonc:\RegQuadn{2n}\to \Etn{2}$ qui envoie un module quadratique sur le centre de son algèbre de Clifford paire (et qui se factorise donc par $\PairesQuadn{2n}$)
\end{defi}
Notons que $\faisOplus_q$ est représentable et affine sur $S$ par la remarque \ref{imageinvnoyau_rema}.
\begin{prop} \label{secOplus_prop}
La suite 
$$1 \to \faisOplus_q \to \faisorthO_q \tooby{\ArfFonc} \ZZ/2 \to 0$$
est une suite exacte de faisceaux étales ou \fppf. Lorsque $q=\hypq_{2n}$, le morphisme $\faisorthO_q(T) \to \ZZ/2(T)$ est même surjectif pour tout schéma $T$ sur $S$.
\end{prop}
\begin{proof}
Par définition de $\faisOplus$, la seule chose à montrer est que $\faisorthO_q \to \ZZ/2$ est un épimorphisme de faisceaux. Cela découle de la surjectivité sur les points lorsque la forme est hyperbolique, qui est traitée en \cite[Ch. IV, Prop. (5.2.2)]{knus}.\footnote{Attention, dans \cite{knus}, en rang pair, le groupe $\faisOplus$ est noté $\faisSO$, et le noyau du déterminant n'est pas considéré du tout.}
\end{proof}

\begin{prop} \label{ArfDet_prop}
La composition $\RegQuadn{2n} \tooby{\ArfFonc} \Etn{2} \tooby{\XiFonc} \ModDet$ induit le déterminant $\faisorthO_q \to \faismu_2$ (Le foncteur $\XiFonc$ a été défini juste avant la remarque \ref{Zmu2_rema}).
\end{prop}
\begin{proof}
Cela peut se vérifier localement pour la topologie \fppf, on est alors ramené au cas de la forme hyperbolique. Il faut alors vérifier l'action d'un élément de $\faisorthO_q$ sur le noyau de la forme trace du centre de l'algèbre de Clifford paire. C'est ce qui est fait en \cite[Ch. IV, (5.1.2)]{knus}.
\end{proof}

Cela implique que lorsque $2$ est inversible dans $\Gamma(S)$, puisque $(\ZZ/2)_S=(\faismu_2)_S$, la suite exacte de la proposition \ref{seSOOpair_prop} et celle de la proposition \ref{secOplus_prop} coïncident, et $\faisOplus_q=\faisSO_q$.
\medskip

On obtient maintenant une description des torseurs sous $\faisOplus_{2n}=\faisOplus_{\hypq_{2n}}$. Soit $\RegQuadArfTrivn{2n}$ le produit fibré de champs $\RegQuadn{2n} \times_{\Etn{2}} \Final$. Le centre de l'algèbre de Clifford paire du module hyperbolique $\hypq_{2n}$ étant isomorphe à $\faisO_S\times \faisO_S$ par l'isomorphisme de \ref{CliffHypPaire_prop} qu'on notera $\zeta_{2n}$, on obtient un objet trivial $(\hypq_{2n},\zeta_{2n})$. La proposition \ref{produitformes_prop} implique alors à l'aide de la surjectivité de la suite de la proposition \ref{secOplus_prop}:
\begin{prop} \label{Oplustorseurs_prop}
Le foncteur fibré $(q,\zeta) \mapsto \faisIso_{(\hypq_{2n},\zeta_{2n}),(q,\zeta)}$ définit une équivalence de champs $\RegQuadArfTrivn{2n}\to \Tors{\faisOplus_{2n}}$.
\end{prop}

Donnons maintenant une description concrète des torseurs sous $\faisSO_{2n+1}$. Considérons le foncteurs fibré 
$\qDemiDetFonc: \Formes{\hypq_{2n+1}} \to \ModDet$ 
ainsi que le foncteur fibré $\Final \to \ModDet$ qui envoie l'unique objet sur le module déterminant trivial $(\faisO_S, \faisO_S \otimes \faisO_S \simeq \faisO_S)$.
Formons alors le produit fibré de champs $\QuadDetTrivn{h}=\Formes{\hypq_{2n+1}} \times_{\ModDet} \Final$. Explicitement, les objets de sa fibre sur $T$ sont donc des paires
$(q,\phi)$ où $q$ est une forme de $\hypq_{2n+1}$ et où $\phi$ est un isomorphisme du module demi-déterminant de $q$ vers le module déterminant trivial, autrement dit, $M$ étant le $\faisO_S$-module sous-jacent à $q$, c'est un isomorphisme $\phi:\Lambda M \simeq \faisO_S$ tel que $\phi^{\otimes 2}=\qDemiDetMorph{q}$. L'objet trivial est $(\hypq_{2n+1},\id_{\faisO_S})$.
\begin{prop} \label{SOtorseursimpairs_prop}
Le foncteur fibré qui envoie un objet $(q,\phi)$ de $\QuadDetTrivn{\hypq_{2n+1}}$ sur $\faisIso_{(\hypq_{2n+1},\id_{\faisO_S}),(q,\phi)}$ définit une équivalence de champs 
$$\QuadDetTrivn{\hypq_{2n+1}} \simeq \Tors{\faisSO_{2n+1}}.$$
\end{prop}
\begin{proof}
On applique la proposition \ref{produitformes_prop}:
Par définition, le groupe $\faisSO_{\hypq_{2n+1}}$ est le produit fibré
$$\xymatrix{
\faisSO_{2n+1} \ar[r] \ar[d] \ar@{}[dr]|(.40){\ulcorner} & \faisUn \ar[d] \\
\faisO_{2n+1} \ar[r]^{\det} & \faismu_2
}$$ 
et le déterminant $\faisO_{2n+1} \to \faismu_2$ est surjectif sur les points (donc bien entendu un épimorphisme de faisceaux) par \ref{seSOOimpair_prop}.
\end{proof} 

\subsection{Groupes $\faisPGO$, $\faisPGOplus$, $\faisGO$, $\faisorthO$ et $\faisOplus$ d'une paire quadratique}

Ce qui suit est essentiellement une rapide adaptation des notions introduites dans \cite[Ch. VI, § 23.B]{bookinv} au cas d'une base quelconque.

Dans cette section $(A,\sigma,f)$ désigne toujours une paire quadratique de degré $2n$, donc un objet du champ $\PairesQuadn{2n}$, au sens de la définition \ref{pairesquad_defi}. Rappelons que toutes les paires quadratiques sont des formes étales de la paire hyperbolique $(\faisM_{2n},\invadj_{2n},f_{2n})$ par le corollaire \ref{pairequadformes_coro}.

\begin{defi} \label{PGO_defi}
Soit $\faisPGO_{A,\sigma,f}$ le faisceau en groupes des automorphismes $\faisAut_{A,\sigma,f}$. Lorsque $(M,q)$ est un module quadratique régulier, on note $\faisPGO_{M,q}$ ou $\faisPGO_q$ au lieu de $\faisPGO_{\faisEnd_M,\invadj_q,f_q}$.
\end{defi}

\begin{prop} \label{PGOrepresentable_prop}
Le faisceau $\faisPGO_{A,\sigma,f}$ est représentable par un schéma affine sur $S$. 
\end{prop}
\begin{proof}
Par la proposition \ref{reprLocal_prop}, on se réduit au cas de $\faisAut_{\faisM_{2n},\invadj_{2n},f_{2n}}$ et on peut supposer $S=\Spec(R)$. Le faisceau $\faisAut_{\faisM_{2n},\invadj_{2n}}$ est représentable par un schéma affine sur $S$ par \ref{Autalginv_prop}. On considère alors la représentation $\faisAut_{\faisM_{2n},\invadj_{2n}}\to \faisGL_{\faisSym_{\faisM_{2n},\invadj_{2n}}^\dual}$ induite par le fait que tout élément symétrique est préservé par un automorphisme d'algèbre à involution. Le stabilisateur de $f$ dans $\faisGL_{\faisSym_{\faisM_{2n},\invadj_{2n}}^\dual}$ est représentable par la proposition \ref{stab_prop}, et donc également $\faisPGO_{\faisM_{2n},\invadj_{2n},f_{2n}}$ comme son image réciproque, grâce au lemme \ref{produit_fibre_lemm}.
\end{proof}

Les torseurs sous $\faisPGO_{A,\sigma,f}$ se décrivent donc immédiatement.

\begin{prop} \label{PGOtorseurs_prop}
Le foncteur $(B,\tau,g)\mapsto \faisIso_{(A,\sigma,f),(B,\tau,g)}$ définit une équivalence de catégories fibrées $\PairesQuadn{2n} \isoto \Tors{\faisPGO_{A,\sigma,f}}$
\end{prop}
\begin{proof}
C'est une application immédiate de la proposition \ref{tordusformes_prop}.
\end{proof}
 
\begin{defi} \label{PGOplus_defi}
Soit $\faisPGOplus_{A,\sigma,f}$ le noyau de l'application 
$$\faisPGO_{(A,\sigma,f)}=\faisAut_{(A,\sigma,f)} \to \faisAut_{Z_{0,A,\sigma,f}} = \ZZ/2$$ 
induite par le foncteur $\ArfFonc:\PairesQuadn{2n} \to \Etn{2}$. 
Lorsque $(M,q)$ est un module quadratique régulier, on note $\faisPGOplus_{M,q}$ ou $\faisPGOplus_q$ au lieu de $\faisPGOplus_{\faisEnd_M,\invadj_q,f_q}$.
\end{defi}
Comme à l'accoutumée, on note $\faisPGO_{2n}$ et $\faisPGOplus_{2n}$ les schémas en groupes correspondants à l'involution hyperbolique $(\faisM_{2n},\invadj_{2n},f_{2n})$ 

\begin{prop} \label{PGOplusrepresentable_prop}
Le faisceau en groupes $\faisPGOplus_{A,\sigma,f}$ est représentable par un schéma affine sur $S$.
\end{prop}
\begin{proof}
C'est un noyau, on applique le lemme \ref{produit_fibre_lemm}.
\end{proof}

\begin{prop} \label{secPGOplus_prop}
Pour toute paire quadratique $(A,\sigma,f)$, la suite
$$1 \to \faisPGOplus_{A,\sigma,f} \to \faisPGO_{A,\sigma,f} \to \ZZ/2 \to 0.$$
est une suite exacte de faisceaux étales. 
En particulier, pour tout module quadratique régulier $q$, cette suite s'écrit
$$1 \to \faisPGOplus_{q} \to \faisPGO_{q} \to \ZZ/2 \to 0.$$
Soit $P$ le $\faisPGO_{2n}$-torseur correspondant à $(A,\sigma,f)$. Alors la première suite est obtenue par torsion sous $P$ de la même suite pour $(A,\sigma,f)=(\faisM_{2n},\invadj_{2n},f_{2n})$.
\end{prop}
\begin{proof}
Par définition de $\faisPGOplus$, il suffit de montrer l'exactitude à droite.
On peut supposer que $(A,\sigma,f)$ est $(\faisM_{2n},\invadj_{2n},f_{2n})$, et que la base est affine et même locale par un argument de limite. On calcule alors à la main que $\int_a$ est dans $\faisPGO$ et s'envoie sur $1 \in \ZZ/2$ si $a$ est la matrice de l'application qui permute les deux premiers vecteurs de base. 

L'affirmation sur la torsion découle immédiatement de la proposition \ref{auttordus_prop} et de la proposition \ref{foncttors_prop}, point \ref{fonctaut_item} appliquée au foncteur $\ArfFonc$, puisque $\ZZ/2$ ne se tord pas.
\end{proof}

Puisque $\faisPGOplus$ est un noyau, ses torseurs se décrivent comme à l'accoutumée à partir de ceux de $\faisPGO$. Soit $\PairesQuadArfTrivn{2n}$ le produit fibré des champs $\PairesQuadn{2n}\times_{\Etn{2}} \Final$, où $\Final\to \Etn{2}$ est le foncteur fibré envoyant le seul objet de chaque fibre sur l'algèbre étale de degré $2$ triviale $\faisO_S \times \faisO_S$. Un objet de ce produit fibré est donc de la forme $(A,\sigma,f,\zeta)$ où $\zeta:Z_{0,A,\sigma,f} \isoto \faisO_S \otimes \faisO_S$ est un isomorphisme de $\faisO_S$-algèbres. En particulier, par l'isomorphisme $\zeta_{2n}:Z_{0,\hypq_{2n}}\simeq \faisO_S\times \faisO_S$, nous avons donc un objet $(\faisM_n,\invadj_{2n},f_{2n},\zeta_{2n})$. 

\begin{prop} \label{PGOplustorseurs_prop}
Le foncteur 
$$(A,\sigma,f,\zeta)\mapsto \faisIso_{(\faisM_n,\invadj_{2n},f_{2n},\zeta_{2n}),(A,\sigma,f,\zeta)}$$
définit une équivalence de catégories fibrées 
$$\PairesQuadArfTrivn{2n} \isoto \Tors{\faisPGOplus_{\faisM_n,\invadj_{2n},f_{2n},\zeta_{2n}}}.$$
\end{prop}
\begin{proof}
Cela découle de \ref{produitformes_prop}, grâce à la suite exacte de la proposition \ref{secPGOplus_prop}, qui donne la condition d'épimorphisme requise.
\end{proof}

Par le foncteur oubli de la paire quadratique $\PairesQuadn{2n}\to \Azumaya{2n}$, on obtient sur les automorphisme un morphisme naturel $\faisPGO_{A,\sigma,f} \to \faisPGL_A$ qui fait de $\faisPGO_{A,\sigma,f}$ un sous-groupe fermé de $\faisPGL_A$ (le stabilisateur de $\sigma$ et $f$).
\begin{defi} \label{GO_defi}
Le faisceau en groupes $\faisGO_{A,\sigma,f}$ est défini comme le produit fibré
$$\xymatrix{
\faisGO_{A,\sigma,f} \ar[r] \ar@{}[dr]|(.40){\ulcorner} \ar[d] & \faisPGO_{A,\sigma,f} \ar[d] \\
\faisGL_{1,A} \ar[r] & \faisPGL_{A}
}
$$
Pour tout module quadratique régulier $(M,q)$, on note $\faisGO_q$ le groupe $\faisGO_{\faisEnd_M,\invadj_q,f_q}$.
\end{defi}
C'est donc un sous-groupe fermé de $\faisGL_{1,A}$, représentable par le lemme \ref{produit_fibre_lemm}. 
On vérifie sans peine que les points de $\faisGO_{A,\sigma,f}$ sont donnés par
$$\faisGO_{A,\sigma,f}(T)=\{a \in \faisGL_{1,A}(T),\ a\sigma(a) \in \faisGm(T),\ f\circ \int_a=f\}.$$
Cela revient à montrer que $\int_a \in \faisPGO_{A,\sigma,f}=\faisAut_{A,\sigma,f}$ si et seulement si les deux conditions définissant l'ensemble ci-dessus sont satisfaites, or $f\circ \int_a =\int_a$ est satisfaite par définition, et il est alors aisé de montrer que $\int_a \circ \sigma = \sigma \circ \int_a$ si et seulement si $a \sigma(a)$ est central, donc dans $\faisGm(T)$.

\begin{prop} \label{secGOPGO_prop}
La suite
$$1 \to \faisGm \to \faisGO_{A,\sigma,f} \to \faisPGO_{A,\sigma,f} \to 1$$
est une suite exacte de faisceaux étales.
En particulier, pour tout module quadratique $q$ régulier, cette suite est
$$1 \to \faisGm \to \faisGO_{q} \to \faisPGO_{q} \to 1.$$
\end{prop}
\begin{proof}
C'est une chasse au diagramme à partir de la définition de $\faisGO$ et de la suite exacte \eqref{secPGLA_eq}.
\end{proof}

\begin{prop}
Pour tout $T$ sur $S$ et pour tout point $a$ de $\faisGO_{A,\sigma,f}(T)$ le produit $\sigma(a)a$ est dans $\faisGm(T)$ inclus dans $\faisGL_{1,A}(T)$.
\end{prop}
\begin{proof}
Cela se vérifie localement pour la topologie étale. On peut donc supposer que $S=\Spec(R)$ et que la paire quadratique est $(\faisM_{2n},\invadj_{2n},f_{2n})$. Le calcul est alors fait dans la partie \ref{dn_adjoint} lors du calcul de l'algèbre de Lie.
\end{proof}

On obtient ainsi un morphisme $\faisGO_{A,\sigma,f} \to \faisGm$ en envoyant $a$ sur $\sigma(a)a$.

\begin{defi} \label{Opaire_defi}
Le faisceau en groupes $\faisorthO_{A,\sigma,f}$ est défini comme le noyau du morphisme $\faisGO_{A,\sigma,f} \to \faisGm$.
\end{defi}
Il est donc représentable par le lemme \ref{produit_fibre_lemm}. 

\begin{defi} \label{Opluspaire_defi}
Le faisceau en groupes $\faisOplus_{A,\sigma,f}$ est défini comme le noyau du morphisme composé 
$$\faisorthO_{A,\sigma,f}\to \faisGO_{A,\sigma,f}\to \faisPGO_{A,\sigma,f}\to \ZZ/2$$
où tous les morphismes sont ceux définis précédemment et le dernier est induit par le foncteur $\ArfFonc$.
\end{defi}

\begin{prop}
Lorsque $(M,q)$ est un module régulier de rang $2n$, on a un isomorphisme canonique $\faisorthO_q \isoto \faisorthO_{\faisEnd_M,\invadj_q,f_q}$, et donc par conséquent $\faisOplus_q \isoto \faisOplus_{\faisEnd_M,\invadj_q,f_q}$.
\end{prop}
\begin{proof}
Cela découle directement de la description de $\faisorthO_q$ de la proposition \ref{Oqinvolution_lemm} et des points de $\faisGO$ ci-dessus.
\end{proof}

\begin{prop} \label{secOGO_prop}
La suite
$$1 \to \faisorthO_{A,\sigma,f} \to \faisGO_{A,\sigma,f} \to \faisGm \to 1.$$
est une suite exacte de faisceaux \fppf.
\end{prop}
\begin{proof}
Par définition de $\faisorthO_{A,\sigma,f}$, seule l'épimorphie à droite est à montrer, ce qui est une propriété locale pour la topologie \fppf. On peut donc supposer que l'élément de $\faisGm$ à atteindre est un carré $\lambda^2$, et il est alors atteint par $\lambda \id_A$. 
\end{proof}

La composée $\faisorthO_{A,\sigma,f} \to \faisGO_{A,\sigma,f} \to \faisPGO_{A,\sigma,f}$ et l'inclusion $\faismu_2$ dans $\faisorthO_{A,\sigma,f}$ induite par l'inclusion de $\faismu_2$ dans $\faisGL_{1,A}$ (un élément de $\faismu_2$ agit par multiplication sur $A$) fournissent: 
\begin{prop} \label{secOPGO_prop}
La suite 
$$1 \to \faismu_2 \to \faisorthO_{A,\sigma,f} \to \faisPGO_{A,\sigma,f}\to 1$$
est une suite exacte de faisceaux étales. Elle en induit une autre: 
$$1 \to \faismu_2 \to \faisOplus_{A,\sigma,f} \to \faisPGOplus_{A,\sigma,f} \to 1.$$
\end{prop}
\begin{proof}
L'exactitude de la première suite s'obtient par chasse au diagramme, en croisant les suites exactes des propositions \ref{secGOPGO_prop} et \ref{secOGO_prop} et en remarquant que la composée $\faisGm \to \faisGO_{A,\sigma,f} \to \faisGm$ est l'élévation au carré. La seconde en est une conséquence immédiate.
\end{proof}

Le groupe $\faisPGO_{A,\sigma,f}=\faisAut_{A,\sigma,f}$ agit sur les différents groupes $\faisPGO$, $\faisPGOplus$, $\faisGO$, $\faisorthO$ et $\faisOplus$ fabriqués à partir de l'algèbre de Clifford paire, car celle-ci est fonctorielle en la paire quadratique, et les structures définissant ces groupes sont préservés par les morphismes de paires quadratiques. On obtient alors immédiatement:
\begin{prop} \label{PGOetctordus_prop}
Par l'action ci-dessus, le $\faisPGO_{\faisM_n,\invadj_{2n},f_{2n}}$-torseur $\faisIso_{(\faisM_n,\invadj_{2n},f_{2n}),(A,\sigma,f)}$ tord $\faisPGO_{\faisM_n,\invadj_{2n},f_{2n}}$  en $\faisPGO_{A,\sigma,f}$, et de même pour $\faisPGOplus$, $\faisorthO$, $\faisOplus$ et $\faisGO$. Les différents morphismes reliant ces groupes sont également tordus.
\end{prop} 
\begin{proof}
Voir la démonstration de \ref{sntordu_prop} un peu plus loin, qui fonctionne sur le même principe. 
\end{proof}

\subsection{Groupes de Clifford et Spin}

\subsubsection{Cas d'un module quadratique}

Soit $(M,q)$ un module quadratique, $M$ toujours supposé localement libre, et où $q$ est une forme \fppf\ de $\hypq_{2n}$ ou bien de $\hypq_{2n+1}$. Le foncteur fibré algèbre de Clifford $\CliffFonc$ à valeur dans le champ des algèbres $\ZZ/2$-graduées induit un morphisme de faisceaux en groupes $\faisorthO_q \to \faisAut_{\faisCliff_q}$. Ce morphisme est injectif, car $\faisCliff_q$ contient $M$ comme sous-module et il est par définition préservé par tout $\CliffFonc(f)$ avec $f\in \faisorthO_q$. 
On peut également considérer les deux foncteurs qu'il induit
$$\CliffFonc:\Formes{\hypq_{2n}} \to \Azumayan{2^{2n}} \hspace{3ex} \text{et}\hspace{3ex} \CliffFonc_{0}:\Formes{\hypq_{2n+1}} \to \Azumayan{2^{2n}}.$$

Par ailleurs, afin de tenir compte de la graduation dans l'algèbre de Clifford, nous aurons besoin de manipuler des algèbres d'Azumaya et des $\faisO_S$-modules $\ZZ/2$-gradués. On définit donc plusieurs champs: 
\begin{itemize}
\item $\GrAzumayan{2n}$, dont les objets sont des faisceaux en algèbres d'Azumaya de degré constant $2n$, munis en plus d'une $\ZZ/2$-graduation que les morphismes respectent.  
\item $\GrVecn{2n}$, dont les objets sont des $\faisO_S$-modules de rang $2n$ munis d'une $\ZZ/2$-graduation telle que les sous-modules de degré $0$ et $1$ sont tout deux des $\faisO_S$-modules localement libre de rang constant $n$, et les morphismes sont localement homogènes. 
\end{itemize}
Il est immédiat que ce sont des champs en combinant les procédures maintenant habituelles.

Soit $\ZeroUn$ le champ associé au faisceau $(\ZZ/2)_S$, au moyen de l'exemple \ref{champfaisceau_exem}. Chaque fibre n'a donc qu'un objet, et les automorphismes de cet objet sont le faisceau $(\ZZ/2)_S$. Nous disposons d'un foncteur fibré $\DegLocFonc:\GrVecn{2n} \to \ZeroUn$ qui envoie un objet sur le seul disponible, et un morphisme sur son degré local. 
Notons $\GrZeroVecn{2n}$ le produit fibré $\GrVecn{2n} \times_{\ZeroUn} \Final$, qui a donc les mêmes objets que $\GrVecn{2n}$, mais dont les morphismes sont homogènes de degré $0$, \ie préservent la graduation.

Pour toute $\faisO_S$-algèbre graduée, notons $\faisGLh_{1,A}$ le sous-faisceau en groupes de $\faisGL_{1,A}$ formé des éléments localement homogènes.

\begin{defi} \label{groupeClifford_defi}
Le faisceau en groupes de Clifford $\faisGamma_{q}$ est défini comme le produit fibré
$$\xymatrix{
\faisGamma_q \ar[r] \ar[d] \ar@{}[dr]|(.40){\ulcorner} & \faisGLh_{1,\faisCliff_q} \ar[d] \\
\faisO_q \ar@{^(->}[r]^-{\CliffFonc} & \faisAut_{\faisCliff_q} 
}$$ 
où la flèche verticale de droite est l'action par conjugaison.
\end{defi}
\'Etant donné $T$ sur $S$, $g\in \faisGL_{1,\faisCliff_q}(T)$ et $f \in \faisorthO_q(T)$, tels que $\CliffFonc(f)=\int_g$, on a alors que $\int_g$ préserve $M(T)$ dans $\faisCliff_q(T)$. Par ailleurs, si $g \in \faisGL_{1,\faisCliff_q}(T)$ est tel que $\int_g$ préserve $M(T)$, alors $q(g m g^{-1})=(g m g^{-1})^2=g m^{2}g^{-1} = g q(m) g^{-1}=q(m)$. 
Explicitement, sur les points, on a donc
$$\faisGamma_q(S)=\{\alpha \in \faisGLh_{1,\faisCliff_q}(S)\text{ t.q. $\alpha_T\cdot m\cdot \alpha_T^{-1}\in M(T)$ $\forall T$ sur $S$ et $\forall m \in M(T)$} \}.$$ 
\begin{defi} \label{groupeCliffordpair_defi}
Le faisceau en groupes de Clifford pair $\faisSGamma_{q}$ est défini comme l'intersection de $\faisGamma_{q}$ et de $\faisGL_{1,\faisCliff_{0,q}}$ dans $\faisGLh_{1,\faisCliff_q}$.
\end{defi}
Lorsque $q$ est $\hypq_{2n}$ (resp. $\hypq_{2n+1}$), on utilise la notation $\faisGamma_{2n}$ et $\faisSGamma_{2n}$ (resp. $\faisGamma_{2n+1}$ et $\faisSGamma_{2n+1}$).  
Notons que $\faisGamma_q$ et $\faisSGamma_q$ sont tout deux représentables par des schémas affines sur $S$ par le Lemme \ref{produit_fibre_lemm}. 

\begin{prop} \label{SGamOplus_prop}
Lorsque $q$ est une forme étale ou \fppf\ de $\hypq_{2n}$, on a les suites exactes de faisceaux en groupes suivantes, pour les topologies de Zariski, étale, ou \fppf:
\begin{equation}
1 \to \faisGm \to \faisGamma_q \to \faisorthO_q \to 1
\end{equation}
qui induit
\begin{equation} \label{SGamOplus_eq}
1 \to \faisGm \to \faisSGamma_q \to \faisOplus_q \to 1.
\end{equation}
Le faisceau en groupes $\faisSGamma_q$ s'identifie donc au produit fibré
$$\xymatrix{
\faisSGamma_q \ar[r] \ar[d] \ar@{}[dr]|(.40){\ulcorner} & \faisGamma_{q} \ar[d] \\
\faisOplus_q \ar@{^(->}[r] & \faisO_q.  
}$$ 
\end{prop}
\begin{proof}
On peut raisonner Zariski localement sur $S$, et donc supposer $S$ affine. 
Cela découle alors de \cite[Ch. IV, (6.2.2) et (6.2.3)]{knus}, car toute classe dans le groupe de Picard est localement nulle pour la topologie de Zariski (en prenant garde une fois de plus que $\faisOplus$ est noté $\faisSO$ par Knus). 
\end{proof}
Le cas du groupe de Clifford spécial en rang impair sera traité en \ref{secSpinSO_prop}.
\medskip

Décrivons maintenant les torseurs sous le groupe de Clifford $\faisGamma_{2n}$ de la forme hyperbolique de rang pair $2n$. Soit $\EndFonc:(\GrVecn{2n})_\grpd\to \GrAzumayan{2n}$ le foncteur fibré ``endomorphismes", qui à un module gradué $N_{\bullet}$ associe $\faisEnd_{N_\bullet}$ muni de la graduation du degré des éléments homogènes, \ie $\faisEnd_{N_0}$ et $\faisEnd_{N_1}$ sont en degré $0$ alors que $\faisHom_{N_0,N_1}$ et $\faisHom_{N_1,N_0}$ sont en degré $1$. \`A un isomorphisme $f:N_\bullet \to N'_\bullet$, on associe $\int_f: \faisEnd_{N_\bullet} \to \faisEnd_{N'_\bullet}$. Notons que ce dernier respecte automatiquement la graduation si $f$ est localement homogène. 
Soit $\GrDeplCliff{2n}$ le produit fibré de champs
$$
\xymatrix{
\GrDeplCliff{2n} \ar[r] \ar[d] \ar@{}[dr]|(.40){\ulcorner} & (\GrVecn{2^n})_\grpd \ar[d] \\
\Formes{\hypq_{2n}} \ar[r]^{\CliffFonc} & \GrAzumayan{2^n}.
}$$ 
Explicitement, un objet de $(\GrDeplCliff{2n})_T$ est donc un triplet $(q,N_0\oplus N_1,\phi)$ où $q$ est une forme de $(\hypq_{2n})_T$, $N_0$ et $N_1$ sont des $\faisO_T$-module localement libre de rang $2^{n-1}$, et $\phi: \faisEnd_N \isoto \faisCliff_q$ est un isomorphisme de $\faisO_T$-algèbres graduées.
Nous avons donné en \ref{CliffHypPaire_prop} un isomorphisme d'algèbres graduées $\faisCliff_{\hypq_{2n}}\simeq \faisEnd_{(\totExt \faisO_S^n)_\bullet}$, où $(\totExt \faisO_S^n)_0 = \totExtPair \faisO_S^n$ et $\totExt \faisO_S^n)_1=\totExtImpair \faisO_S^n$.  Notons-le $\phi_h$. Cela fournit donc un objet $(\hypq_{2n},(\totExt \faisO_S^n)_\bullet, \phi_h)$ de $(\GrDeplCliff{2n})_S$. 
Un automorphisme de cet objet est une paire $(f,g)$ où $f \in \faisorthO_{\hypq_{2n}}(T)$, $g \in \faisGLh_{(\totExt \faisO_S^{n})_\bullet}(T)$ et $\CliffFonc(f)=\int_g$. D'où $\faisAut_{(\hypq_{2n},(\totExt \faisO_S^n)_\bullet,\phi_h)}=\faisGamma_{2n}$. 
Puisque $\faisGLh_{(\totExt \faisO_S^n)_\bullet} \to \faisAut_{\faisEnd_{(\totExt \faisO_S^n})_\bullet}$ est un épimorphisme de faisceaux (Skolem-Noether gradué dans le cas local \cite[Ch. III, (6.5.1)]{knus}), la proposition \ref{produitformes_prop} donne:

\begin{prop} \label{Gamma2ntors_prop}
Le foncteur fibré qui envoie un objet $(q,N_\bullet,\phi)$ sur $\faisIso_{(\hypq_{2n},(\totExt\faisO_S^n)_\bullet,\phi_h),(q,N_\bullet,\phi)}$ définit une équivalence de champs 
$$\GrDeplCliff{2n}\simeq \Tors{\faisGamma_{2n}}$$
(pour les topologies étales ou \fppf).
\end{prop}

\begin{prop} \label{secSGamGam_prop}
Pour tout module régulier $q$ de rang $2n$, la suite de faisceaux en groupes
$$1 \to \faisSGamma_q \to \faisGamma_q \to \ZZ/2 \to 0$$
est une suite exacte de faisceaux étales, où la flèche $\faisGamma_q \to \ZZ/2$ est le degré local.
\end{prop}
\begin{proof}
Par définition de $\faisSGamma$, seule l'exactitude à droite est à montrer, et on peut supposer le module $q$ hyperbolique. Tout élément $m\in M$ tel que $q(m)$ est inversible est alors en degré $1$.
\end{proof}

Pour décrire les torseurs sous $\faisSGamma_{2n}$, considérons le produit fibré de champs 
$$\DeplCliff{2n}:=\GrDeplCliff{2n} \times_{\ZeroUn} \Final$$
où le foncteur $\GrDeplCliff{2n} \to \ZeroUn$ est la composée de l'oubli $\GrDeplCliff{2n} \to \GrVecn{2n}$ et du foncteur degré local $\DegLocFonc:\GrVecn{2n}\to \ZeroUn$  défini plus haut. Il est facile de voir que $\DeplCliff{2n}$ est équivalent au champ dont les objets sont les mêmes que ceux de $\GrDeplCliff{2n}$, mais dont les morphismes sont des paires $(f,g)$ avec $g$ de degré constant $0$. Il est donc immédiat que les automorphismes de l'objet $(\hypq_{2n},(\totExt \faisO_S^n)_\bullet,\phi_h)$ sont le groupe $\faisSGamma_{2n}$.
\begin{prop} \label{SGammapairtorseurs_prop}
Le foncteur fibré qui envoie un objet $(q,N_\bullet,\phi)$ sur $\faisIso_{(\hypq_{2n},(\totExt \faisO_S^n)_\bullet,\phi_h),(q,N_\bullet,\phi)}$ définit une équivalence de champs 
$$\DeplCliff{2n}\simeq \Tors{\faisSGamma_{2n}}$$
(pour les topologies étales ou \fppf).
\end{prop}
\begin{proof}
C'est encore une fois une application directe de la proposition \ref{produitformes_prop}, à l'aide de la suite exacte de la proposition \ref{secSGamGam_prop}.
\end{proof}

En rang impair, nous obtenons directement les $\faisSGamma$-torseurs. Notons $\hypq_{2n+1}$ la forme $\hypq_{2n}\orth \fq{1}$. On définit le champ $\ZGrAzumaya$ dont les objets sont des $\faisO_S$-algèbres $\ZZ/2$-graduées, de la forme $Z \otimes A$ où $A$ est une $\faisO_S$-algèbre d'Azumaya, et $Z$ est une algèbre graduée de degré $2$, avec $Z_0 =\faisO_S$ comme $\faisO_S$-algèbre, $Z_1$ localement libre comme $\faisO_S$-module et la multiplication $Z_1 \otimes Z_1 \to Z_0=\faisO_S$ étant un isomorphisme de $\faisO_S$-module. La graduation n'est fournie que par $Z$, et $A$ est homogène de degré $0$. Les morphismes sont les morphismes de $\faisO_S$-algèbre $\ZZ/2$-graduée. Il n'est pas difficile de vérifier que c'est un champ (pour la topologie étale ou \fppf), par la proposition \ref{champstruc_prop} puis par descente des propriétés comme ``localement libre''.  
Par le point \ref{multCliff_item} du théorème \ref{CliffordSplit_theo}, on peut voir le foncteur $\CliffFonc$ comme $\Formes{\hypq_{2n+1}} \to \ZGrAzumaya$. De plus, on peut définir un foncteur fibré $(\Vecn{2^n})_{\grpd} \to \ZGrAzumaya$ en envoyant un $\faisO_S$-module localement libre $N$ sur l'algèbre $\ZZ/2$-graduée donnée par $Z^{triv}\otimes\faisEnd_N$, où $Z^{triv}$ est telle que $Z^{triv}_1=\faisO_S$ avec la multiplication usuelle $\faisO_S \otimes \faisO_S \to \faisO_S$; un isomorphisme de $N$ s'envoie sur la conjugaison par cet isomorphisme sur $\faisEnd_N$.
On considère alors le champ $\DeplCliff{2n+1}$ défini comme le produit fibré
$$
\xymatrix{
\DeplCliff{2n+1} \ar[r] \ar[d] \ar@{}[dr]|(.40){\ulcorner} & (\Vecn{2^n})_\grpd \ar[d] \\
\Formes{\hypq_{2n+1}} \ar[r]^{\CliffFonc} & \ZGrAzumaya.
}$$ 
Explicitement, un objet de $(\DeplCliff{2n+1})_T$ est donc un triplet $(q,N,\phi)$ où $q$ est une forme de $(\hypq_{2n+1})_T$, $N$ est un $\faisO_T$-module localement libre de rang $2^n$, et $\phi: Z^{triv}\otimes \faisEnd_N \isoto \faisCliff_{q}$ est un isomorphisme de $\faisO_T$-algèbres graduées.
Par la remarque \ref{Cliffhyp_rema}, on obtient un isomorphisme de $\faisO_S$-algèbres $\phi_h:\faisCliff_{\hypq_{2n+1}}\simeq Z^{triv}\otimes \faisEnd_{\totExt \faisO_S^n}$ qui fournit donc un objet $(\hypq_{2n+1},\totExt \faisO_S^n, \phi_h)$ de $(\DeplCliff{2n+1})_S$, dont on vérifie immédiatement que les automorphismes sont $\faisSGamma_{2n+1}$. 
On obtient donc encore une fois par la proposition \ref{produitformes_prop}:
\begin{prop} \label{Sgammatorseursimpair_prop}
Le foncteur fibré qui envoie un objet $(q,N,\phi)$ sur $\faisIso_{(\hypq_{2n+1},\totExt \faisO_S^n,\phi_h),(q,N,\phi)}$ définit une équivalence de champs 
$$\DeplCliff{2n+1}\simeq \Tors{\faisSGamma_{2n+1}}$$
(pour les topologies étales ou \fppf).
\end{prop}
\begin{proof}
La seule chose à vérifier est la condition de surjectivité de la proposition \ref{produitformes_prop}. Lorsque $N=\totExt \faisO_S^n$ tout automorphisme de la $\faisO_S$-algèbre $\ZZ/2$-graduée $Z^{triv}\otimes \faisEnd_{N}$ peut se décomposer (par son action sur le centre) en un automorphisme d'algèbre de $\faisEnd_N$ et un automorphisme de $\faisO_S$-module $Z_1=\faisO_S$ de carré trivial (car il doit être compatible à la multiplication $Z_1 \otimes Z_1 \to Z_0$), donc un élément de $\faismu_2$. Or, Zariski localement, tout automorphisme de $\faisEnd_N$ est intérieur (Skolem-Noether, voir \cite[th. 3.6]{ausgol}). De plus, si $\lambda \in \faismu_2(T)$, l'automorphisme $e \mapsto \lambda e$ et l'identité sur le sous-espace sous-jacent au sous-module $\hypq_{2n}$ est bien dans $\faisAut_{\hypq_{2n+1}}$ et induit $\lambda$ sur $Z_1$, par la description du générateur de $Z_1$ fourni dans la preuve du théorème \ref{CliffordSplit_theo}. Donc, tout automorphisme de $Z^{triv} \otimes \faisEnd_N$ provient bien localement de $\faisAut_{\totExt \faisO_S^n} \times \faisAut_{\hypq_{2n+1}}$. 
\end{proof}

Remarquons que la preuve de la proposition précédente est un cas d'application de la proposition \ref{produitformes_prop} où aucun des deux foncteurs définissant le produit fibré n'induit individuellement un épimorphisme de faisceaux sur les automorphismes, il faut utiliser la condition plus générale portant sur les deux simultanément.
\medskip

Passons maintenant à la définition du groupe Spin, et à une description de ses torseurs. Pour cela nous aurons besoin de la norme spinorielle. 

On considère l'involution standard $\invstd_q$ du faisceau en algèbres de Clifford $\faisCliff_q$ d'un module quadratique $q$, forme (étale ou \fppf) de $\hypq_{2n}$ ou $\hypq_{2n+1}$. Alors l'application $\faisGamma_q \to \faisGL_{1,\faisCliff_q}$ définie sur les points par $x \mapsto \invstd_q(x)x$ arrive en fait dans $\faisGm \subset \faisGL_{1,\faisCliff_q}$, et définit un morphisme de faisceaux en groupes. Toutes ces affirmations peuvent se vérifier localement, et sont prouvées dans \cite[Ch. IV, Lemme 6.1.1]{knus} (on vérifie en fait que l'image est dans le centre de $\faisCliff_q$ ainsi que dans $\faisCliff_{0,q}$).

\begin{defi} \label{sn_defi}
Le morphisme $\sn :\faisGamma_q \to \faisGm$ défini ci-dessus est appelé \emph{norme spinorielle}.
\end{defi}

\begin{rema} \label{snsurj_rema}
Pour les formes $\hypq_{2n}$ et $\hypq_{2n+1}$, avec $n\geq 1$, la norme spinorielle est surjective sur les points (donc évidemment un épimorphisme de faisceaux Zariski, étale, \fppf). En effet, un élément $m\in M(T) \subset \faisCliff_q(T)$ est dans $\faisGamma_q(T)$ si et seulement si $q(m)$ est inversible, et auquel cas $\sn(m)=-q(m)$. Les formes $\hypq_{2n}$ et $\hypq_{2n+1}$ représentant trivialement tout élément de $\Gamma(T)^*$, on a le résultat. De plus, la restriction de $\sn$ à $\faisSGamma$ est également surjective: on utilise la classe de $m_1\otimes m_2$ avec $q(m_1)=1$ et $q(m_2)$ l'élément recherché.
\end{rema}

\begin{defi} \label{PinSpin_defi}
Soit $q$ une forme (étale ou \fppf) de $\hypq_{2n}$ ou $\hypq_{2n+1}$.
Le schéma en groupes $\faisPin_q$ est le noyau de la norme spinorielle $\faisGamma_q \to \faisGm$ et le schéma en groupes $\faisSpin_q$ est le noyau de la norme spinorielle $\faisSGamma_q \to \faisGm$. Donc, $\faisSpin_q= \faisPin_q \cap \faisSGamma_q$. (Ils sont bien représentables par des schémas affines sur $S$ par le lemme \ref{produit_fibre_lemm}.) 
\end{defi}
Les points de $\faisSpin_q$ sont donc donnés, pour tout $S$-schéma $T$, par
$$\faisSpin_q(T)=\{\alpha \in \faisSGamma_q(T)\text{ t.q. $\sn(\alpha)=1$}\}.$$
Comme d'habitude, on note $\faisSpin_{2n}$ (resp. $\faisSpin_{2n+1}$) lorsque le module quadratique est hyperbolique et de même pour $\faisPin$.
\medskip

Pour toutes paires de module quadratique $(q,M)$ et $(q',M')$, formes \fppf\ de $h=\hypq_{2n}$ ou bien $h=\hypq_{2n+1}$, on peut considérer l'action $\faisIso_{q,q'} \times \faisCliff_q \to \faisCliff_{q'}$ induite par le foncteur $\faisCliff$. Par construction, pour tout élément $f$ de $\faisIso_{q,q'}$, on a $\CliffFonc(f)(M)\subseteq M'$. On en tire que pour un élément $g$ de $\faisSGamma_{q}$, on a $\CliffFonc(f)(g)$ dans $\faisSGamma_{q'}$.
De plus, $\faisCliff(f)(\invstd_q(x))=\invstd_{q'}(\faisCliff(f)(x))$, où $\invstd_q$ et $\invstd_{q'}$ sont les involutions standard respectives, donc $\sn(\CliffFonc(f)(g))=\sn(g)$ et $\faisSpin_q$ est envoyé dans $\faisSpin_{q'}$ par $\CliffFonc(f)$. Le foncteur $\CliffFonc$ induit donc des action de $\faisorthO_q$ sur $\faisCliff_{q}$, $\faisCliff_{0,q}$, $\faisGamma_q$, $\faisSGamma_q$, $\faisSpin_q$, compatibles aux différentes inclusions les reliant.

\begin{prop} \label{sntordu_prop}
Par l'action ci-dessus, le $\faisorthO_{h}$-torseur $\faisIso_{h,q}$ tord $\faisSGamma_h$ en $\faisSGamma_q$, $\faisSpin_h$ en $\faisSpin_q$, le morphisme $\sn: \faisSGamma_h \to \faisGm$ en $\faisSGamma_q \to \faisGm$ (l'action de $\faisorthO_h$ sur $\faisGm$ est triviale), et les inclusions $\faisSpin_h \to \faisSGamma_h \to \faisGamma_h$ en les inclusions correspondantes pour $q$.  
\end{prop}
\begin{proof}
Le procédé est le même dans tous les cas. Par exemple pour $\faisSGamma$, on définit sur les points un morphisme $\faisIso_{h,q} \contr{\faisorthO_h} \faisSGamma_h \to \faisSGamma_q$ en envoyant $(f,g)$ sur $\CliffFonc(f)(g)$, puis on faisceautise. Par les propriétés citées ci-dessus, il est clair que c'est bien défini, et que localement, c'est un isomorphisme. Ensuite, la torsion des différents morphismes se fait par \ref{torsionmorphisme_prop}. 
\end{proof}

Utilisons le foncteur fibré $\faisEnd: (\Vecn{n})_\grpd \to \Azumayan{n}$ pour construire le produit fibré $(\Vecn{n})_\grpd\times_{\Azumayan{n}} (\Vecn{n})_\grpd$ du champ $(\Vecn{n})_\grpd$ avec lui-même.  
On définit un foncteur fibré 
$$(\Vecn{n})_\grpd\times_{\Azumayan{n}} (\Vecn{n})_\grpd \to (\Vecn{n})_\grpd \times (\Vecn{1})_\grpd$$
en envoyant un triplet $(F,G,\psi: \faisEnd_F \isoto \faisEnd_G)$ sur $(G, G^\dual \otimes_{\faisEnd_F} F)$, où $\faisEnd_F$ agit à gauche sur $G$ par $\psi$, et donc à droite sur $G^\dual$. Au niveau des morphismes, cela envoie un couple $(a,b)$ tel que $\phi\circ \int_a =\int_b \circ \phi$, \ie $\int_{\phi(a)}=\int_b$, \ie $b^{-1}\circ \phi(a)$ central, sur $(b, (b^{-1})^\dual \otimes a)$, ce dernier étant bien défini car $b^{-1} \phi(a)$ est central.
Dans l'autre sens, on définit également un foncteur fibré en envoyant un objet $(G,L)$ sur l'objet $(G \otimes L,G, \phi)$ où $\phi: \faisEnd_{F} \simeq \faisEnd_{F \otimes L}$ est l'isomorphisme canonique $f \mapsto f\otimes \id$.
\begin{lemm} \label{endendeq_lemm}
Ces deux foncteurs sont des équivalences de catégories inverses l'une de l'autre.
\end{lemm}
\begin{proof}
Laissant de côté la partie de l'objet qui ne bouge pas (le $G$), il suffit d'utiliser les isomorphismes canoniques de Morita $G^{\dual} \otimes_{\faisEnd_G} G \simeq \faisO_S$ et $G \otimes G^{\dual} \simeq \faisEnd_G$ qui peuvent d'ailleurs se vérifier localement lorsque $F=\faisO_S^n$. 
\end{proof}
\begin{lemm} \label{GLGLGLGm_lemm}
Pour tout triplet $(F,G,\phi: \faisEnd_F \isoto G)$, l'équivalence de catégories précédente induit sur les groupes d'automorphismes un isomorphisme
$$\faisGL_F \times_{\faisPGL_{\faisEnd_F}} \faisGL_G \isoto \faisGL_G \times \faisGm.$$
Sur les points, il se décrit ainsi: il envoie un couple $(a,b)$ tel que $\phi \circ \int_a = \int_b \circ \phi$, \ie $\int_{\phi(a)}=\int_b$, \ie $b^{-1}\circ \phi(a)$ central, sur le couple $(b,b^{-1}\circ \phi(a))$.
\end{lemm}
\begin{proof}
Tout d'abord, le morphisme $\faisGL_F \to \faisPGL_{\faisEnd_F}$ est bien un épimorphisme de faisceaux par \ref{secPGLn_prop}, donc $\faisAut_{(F,G,\phi)} =\faisGL_F \times_{\faisPGL_{\faisEnd_F}} \faisGL_G$. Ensuite, il suffit de suivre la manière l'équivalence de catégories sur les morphismes.
\end{proof}
En rang pair, définissons un foncteur fibré
$$\DeplCliff{2n} \to \Vecn{2^{2n}} \times_{\Azumayan{2^{2n}}} \Vecn{2^{2n}}$$
par
$$(q,N_\bullet,\phi: \faisCliff_{q}\simeq \faisEnd_{N_\bullet})\mapsto (N_\bullet^{\otimes 2},\faisCliff_{q}, \psi: \faisEnd_{N_\bullet^{\otimes 2}} \simeq \faisEnd_{\faisCliff_{q}})$$
où $\psi$ est donné par la composition $\faisEnd_{N_\bullet^{\otimes 2}} \simeq \faisEnd_{N_\bullet^{\otimes 2}} \simeq \faisCliff_q^{\otimes 2} \simeq \faisEnd_{\faisCliff_q}$, le dernier isomorphisme étant induit par l'involution standard $\invstd_q$: il envoie $a\otimes b$ sur $x \mapsto a\cdot x\cdot \invstd_q(b)$. 
En rang impair, on définit de manière analogue
$$\DeplCliff{2n+1}\to \Vecn{2^{2n}} \times_{\Azumayan{2^{2n}}} \Vecn{2^{2n}}$$ 
par 
$$(q,N,\phi: \faisCliff_{q}\simeq Z\otimes \faisEnd_N)\mapsto (N^{\otimes 2},\faisCliff_{0,q}, \psi: \faisEnd_{N^{\otimes 2}} \simeq \faisEnd_{\faisCliff_{0,q}}),$$ 
où $\psi$ est donné par la composition $\faisEnd_{N^{\otimes 2}} \simeq \faisEnd_{N}^{\otimes 2} \simeq \faisCliff_{0,q}^{\otimes 2} \simeq \faisEnd_{\faisCliff_{0,q}}$, le dernier isomorphisme étant toujours induit par l'involution standard $\invstd_q$.
Considérons alors les foncteurs fibrés 
$$\SnFonc_{2n}:\DeplCliff{2n} \to \Vecn{1} \hspace{3ex}\text{et}\hspace{3ex} \SnFonc_{2n+1}:\DeplCliff{2n+1} \to \Vecn{1}$$
obtenus comme 
$$\DeplCliff{2n} \to (\Vecn{2^{2n}})_\grpd\times_{\Azumayan{2^{2n}}} (\Vecn{2^{2n}})_\grpd \simeq (\Vecn{2^{2n}})_\grpd \times (\Vecn{1})_\grpd \to \Vecn{1}$$
\ie le précédent, suivi de l'équivalence du lemme \ref{endendeq_lemm} et de la projection sur le facteur $\Vecn{1}$. On fait de même avec $\DeplCliff{2n+1}$ au lieu de $\DeplCliff{2n}$.
Explicitement, $\SnFonc_{2n}(q,N_\bullet,\phi)$ est donc un fibré en droites $L$ (explicite) tel que $N_\bullet^{\otimes 2} \simeq \faisCliff_{q}\otimes L$. De même, $\SnFonc_{2n+1}(q,N,\phi)$ est un fibré en droites $L$ tel que $N^{\otimes 2} \simeq \faisCliff_{0,q} \otimes L$. Les objets triviaux $(\hypq_{2n},(\totExt \faisO_S^{n})_\bullet,\phi_h)$ et $(\hypq_{2n+1},\totExt \faisO_S^{n},\phi_h)$ sont envoyés sur le fibré trivial $\faisO_S$. Les foncteurs fibrés $\SnFonc_{2n}$ et $\SnFonc_{2n+1}$ induisent alors des morphismes de faisceaux en groupes 
$$\faisSGamma_{2n} = \faisAut_{(\hypq_{2n},(\totExt \faisO_S^{n})_\bullet,\phi_h)} \to \faisAut_{\faisO_S} = \faisGm.$$
et
$$\faisSGamma_{2n+1} = \faisAut_{(\hypq_{2n+1},\totExt \faisO_S^{n},\phi_h)} \to \faisAut_{\faisO_S} = \faisGm.$$
\begin{prop} \label{normeSpinId_prop}
Ces deux morphismes sont la norme spinorielle $\sn:\faisSGamma_{2n} \to \faisGm$ (resp. $\sn:\faisSGamma_{2n+1} \to \faisGm$) définie en \ref{sn_defi}. 
\end{prop}
\begin{proof}
La preuve des deux cas est identique. Faisons le cas impair pour fixer les idées.
On part d'un point de $\faisSGamma_{2n+1}(T)$, qu'on peut décrire comme une paire $(f,g)$ avec $f \in \faisorthO_q(T)$ et $g \in \faisGLh_{\totExt \faisO_S^n}(T)$ tels que $\CliffFonc_0(f) = \phi_h \circ \int_g \circ \phi_h^{-1}=\int_{\phi^{-1}_h(g)}$. Sa norme spinorielle est par définition $\invstd_h(\phi_h^{-1}(g))\phi_h^{-1}(g)$ où $\invstd_h$ est l'involution standard. Par le premier foncteur, on obtient la paire $(g^{\otimes 2}, \CliffFonc_0(f))$. Si $\psi: \faisEnd_{N^{\otimes 2}} \isoto \faisEnd_{\faisCliff_{0,q}}$ est obtenu à partir de $\phi_h$ comme ci-dessus, alors pour obtenir l'élément de $\faisGm$ recherché, on doit par le lemme \ref{GLGLGLGm_lemm} calculer $\CliffFonc_0(f)^{-1}\circ \psi(g^{\otimes 2})$. Or $\psi(g^{\otimes 2})=\big(x \mapsto \phi^{-1}(g)\cdot x\cdot \invstd_h (\phi^{-1}(g))\big)$. En composant par $\CliffFonc_0(f)^{-1}=\int_{\phi^{-1}(g^{-1})}$, on obtient donc
$$ x \mapsto \phi^{-1}(g^{-1})\cdot \phi^{-1}(g) \cdot x\cdot \invstd_h(\phi^{-1}(g))\cdot \phi^{-1}(g)=x\cdot \invstd_h(\phi^{-1}(g))\cdot \phi^{-1}(g)$$
ce qui est bien la multiplication par la norme spinorielle.
\end{proof}

Les torseurs sous les groupes $\faisSGamma_{2n}$ et $\faisSGamma_{2n+1}$ étant compris, ainsi que la norme spinorielle, on obtient facilement une description des torseurs sous les groupes $\faisSpin_{2n}$ et $\faisSpin_{2n+1}$. 

Considérons-donc les champs $\DeplCliffSn{2n}$ et $\DeplCliffSn{2n+1}$ définis comme les produits fibrés
$$\xymatrix{
\DeplCliffSn{2n} \ar[r] \ar[d] & \Final \ar[d] \\
\DeplCliff{2n} \ar[r]^-{\SnFonc_{2n}} & \Vecn{1} & 
}\hspace{7ex}\xymatrix{
\DeplCliffSn{2n+1} \ar[r] \ar[d] & \Final \ar[d] \\
\DeplCliff{2n+1} \ar[r]^-{\SnFonc_{2n+1}} & \Vecn{1} & 
}$$
où le foncteur $\Final \to \Vecn{1}$ qui envoie l'unique objet de chaque fibre sur le fibré trivial $\faisO_S$. 
Explicitement, un objet de $(\DeplCliffSn{2n})_T$ est donc de la forme $(q,N_\bullet,\phi,\psi)$ où 
\begin{itemize}
\item $q$ est un $\faisO_T$-module quadratique régulier de rang $2n$ (une forme de $\hypq_{2n}$); 
\item $N_\bullet$ est un $\faisO_T$-module localement libre dont les deux composantes sont localement libres de rang constant $2^{n-1}$; 
\item $\phi: \CliffFonc_q \isoto \faisEnd_{N_\bullet}$ est un isomorphisme de $\faisO_S$-algèbres graduées;
\item $\psi: L \isoto \faisO_S$ est un isomorphisme de $\faisO_S$-modules, où $L= \SnFonc(q,N_\bullet,\phi)=\faisCliff_{q}^{\dual}\otimes_{\faisEnd_{N_\bullet^{\otimes 2}}} N_\bullet^{\otimes 2}$ est un fibré en droites, qui satisfait à $N_\bullet^{\otimes 2} \simeq \faisCliff_{q} \otimes L$.
\end{itemize}
Notons que les autres données étant fixées la donnée de $\psi$ équivaut à la donnée d'un isomorphisme $\lambda:N_\bullet^{\otimes 2} \simeq \faisCliff_{q}$ tel que $\int_\lambda:\faisEnd_{N_\bullet^{\otimes 2}}\simeq \faisEnd_{\faisCliff_{q}}$ coïncide avec $\faisEnd_{N_\bullet^{\otimes 2}}\simeq \faisEnd_{N_\bullet}^{\otimes 2}\simeq \faisCliff_{q}^{\otimes 2} \simeq \faisEnd_{\faisCliff_{q}}$, ces deux derniers isomorphismes induit par $\phi$ et $\invstd_q$.
Un objet de $(\DeplCliffSn{2n+1})_T$ est lui de la forme $(q,N,\phi,\psi)$ où 
\begin{itemize}
\item $q$ est un $\faisO_T$-module quadratique, forme de $\hypq_{2n+1}$, 
\item $N$ est un $\faisO_T$-module localement libre de rang constant $2^n$, 
\item $\phi: \CliffFonc_q \isoto Z^{triv} \otimes \faisEnd_N$ est un isomorphisme de $\faisO_S$-algèbres graduées
\item $\psi: L \isoto \faisO_S$ est un isomorphisme de $\faisO_S$-modules, où $L= \SnFonc(q,N,\phi)=\faisCliff_{0,q}^{\dual}\otimes_{\faisEnd_{N^{\otimes 2}}} N^{\otimes 2}$ est un fibré en droites, qui satisfait à $N^{\otimes 2} \simeq \faisCliff_{0,q} \otimes L$.
\end{itemize}
De même que pour le rang pair, les autres données étant fixées la donnée de $\psi$ équivaut à la donnée d'un isomorphisme $\lambda:N^{\otimes 2} \simeq \faisCliff_{0,q}$ tel que $\int_\lambda:\faisEnd_{N^{\otimes 2}}\simeq \faisEnd_{\faisCliff_{0,q}}$ coïncide avec $\faisEnd_{N^{\otimes 2}}\simeq \faisEnd_{N}^{\otimes 2}\simeq \faisCliff_{0,q}^{\otimes 2} \simeq \faisEnd_{\faisCliff_{0,q}}$, ces deux derniers isomorphismes induit par $\phi$ et $\invstd_q$.
\medskip

Les foncteurs $\SnFonc_{2n}$ et $\SnFonc_{2n+1}$ induisent la norme spinorielle sur les groupes d'automorphismes des objets triviaux par la proposition \ref{normeSpinId_prop}, et elle est de plus surjective pour ces objets par la remarque \ref{snsurj_rema}. 
Par la proposition \ref{produitformes_prop}, on obtient:
\begin{prop} \label{SpinTorsPair_prop}
Le foncteur fibré qui envoie un objet $(q,N_\bullet,\phi,\psi)$ sur $\faisIso_{(\hypq_{2n},(\totExt \faisO_S^n)_\bullet,\phi_h,\psi_h),(q,N_\bullet,\phi,\psi)}$ définit une équivalence de champs 
$$\DeplCliffSn{2n}\simeq \Tors{\faisSpin_{2n}}$$
(pour les topologies étales ou \fppf).
\end{prop} 
ainsi que:
\begin{prop} \label{SpinTorsImpai_prop}
Le foncteur fibré qui envoie un objet $(q,N,\phi,\psi)$ sur $\faisIso_{(\hypq_{2n+1},\totExt \faisO_S^n,\phi_h,\psi_h),(q,N,\phi,\psi)}$ définit une équivalence de champs 
$$\DeplCliffSn{2n+1}\simeq \Tors{\faisSpin_{2n+1}}$$
(pour les topologies étales ou \fppf).
\end{prop} 

\subsubsection{Cas d'une paire quadratique}

Nous introduisons maintenant le groupe de Clifford (spécial) et le groupe Spin d'une paire quadratique $(A,\sigma,f)$. Encore une fois, ces notions sont développées sur un corps dans \cite[Ch. VI, § 23.2]{bookinv}, et nous nous bornons à vérifier qu'elles passent à une base quelconque.

Pour cela, il faut commencer par définir le bimodule de Clifford d'une paire quadratique. Nous ne reprenons pas la longue construction de \loccit, qui fonctionne essentiellement sans changement sur toute base, à ceci-près que le quotient qui intervient dans la définition doit être considéré comme un faisceau quotient pour la topologie de Zariski (et non uniquement un quotient de préfaisceaux). 

Ce bimodule de Clifford est noté $\faisCBim_{A,\sigma,f}$. Résumons ses propriétés en une proposition.

\begin{prop}
Si $(A,\sigma,f)$ est une paire quadratique, alors
\begin{enumerate}
\item $\faisCBim_{A,\sigma,f}$ est un $\faisO_S$-module localement libre. 
\item Il est muni d'une action à gauche de $A$, d'une action de bimodule (donc à gauche et à droite) de $\faisCliff_{0,A,\sigma,f}$, commutant l'une avec l'autre.
\item Il est muni d'un morphisme injectif de $\faisO_S$-modules $A \to \faisCBim_{A,\sigma,f}$.
\item Lorsque $q$ est un module quadratique régulier de rang $2n$ et $(A,\sigma,f)=(\faisEnd_M,\invadj_q,f_q)$, alors on a un isomorphisme canonique et fonctoriel $\faisCBim_{(\faisEnd_M,\invadj_q,f_q)} \simeq M \otimes \faisCliff_{1,q}$ qui identifie les actions précédentes avec celles que l'on imagine, et le morphisme $\faisEnd_{M} \to \faisCBim_{(\faisEnd_M,\invadj_q,f_q)}$ avec $(\id \otimes i) \circ \phi_q^{-1}: \faisEnd_M \simeq M \otimes M \to M \otimes \faisCliff_{1,q}$ où $\phi_q$ est l'identification $M\otimes M \simeq \faisEnd_M$ définie en section \ref{AlgInv_sec}, et $i:M \to \faisCliff_{1,q}$ est l'inclusion usuelle.
\end{enumerate}
\end{prop}

\begin{defi}
Le groupe de Clifford spécial\footnote{appelé groupe de Clifford tout court et noté $\Gamma_{A,\sigma,f}$ dans \cite{bookinv}.} $\faisSGamma_{A,\sigma,f}$ est le sous-groupe de $\faisGL_{1,\faisCliff_{0,A,\sigma,f}}$ qui normalise l'image de $A$ dans $\faisCBim_{A,\sigma,f}$. 
\end{defi}
Il est représentable par un schéma affine sur $S$ par \ref{normalisateur_prop}.

\begin{prop}
Lorsque $q$ est un module quadratique régulier de rang $2n$ et $(A,\sigma,f)=(\faisEnd_M,\invadj_q,f_q)$, alors on a un isomorphisme  
$$\faisSGamma_{\faisEnd_M,\invadj_q,f_q} \simeq \faisSGamma_q$$
induit par l'isomorphisme $\faisCliff_{0,\faisEnd_M,\invadj_q,f_q} \simeq \faisCliff_{0,q}$.
\end{prop}
\begin{proof}
Cela revient à vérifier que dans ce cas déployé, le normalisateur est bien le groupe de Clifford spécial. Comme le bimodule de Clifford est alors isomorphe à $M \otimes \faisCliff_{1,q}$, avec action à droite et à gauche de $\faisCliff_{0,q}$ sur le deuxième facteur, il y a une inclusion évidente, dont on vérifie facilement qu'elle est un isomorphisme localement lorsque $M$ est libre. 
\end{proof}

Rappelons que $\underline{\sigma}$ est l'involution standard (ou canonique) sur $\faisCliff_{0,A,\sigma,f}$ définie en \ref{AlgCliffPaireQuad_defi}, et qu'elle correspond à l'involution standard $\invstd_q$ sur $\faisCliff_{0,q}$ par l'isomorphisme $\faisCliff_{0,\faisEnd_M,\invadj_q,f_q} \simeq \faisCliff_{0,q}$. 
\begin{lemm}
Alors, pour tout point $g$ de $\faisSGamma_{A,\sigma,f}(T)$ l'élément $\sigma(g)g$ est dans $\faisGm(T)$.
\end{lemm}
\begin{proof}
Cela se vérifie localement pour la topologie étale, on peut donc supposer que $\faisSGamma_{A,\sigma,f}$ est celui d'un module quadratique régulier, auquel cas cela découle de l'explication précédant \ref{sn_defi}. 
\end{proof}

\begin{defi} \label{normespinpairequad_defi}
On appelle \emph{norme spinorielle} $\sn$ le morphisme $\faisSGamma_{A,\sigma,f} \to \faisGm$ ainsi défini.
\end{defi}

\begin{defi}
Le faisceau en groupes $\faisSpin_{A,\sigma,f}$ est défini comme le noyau de la norme spinorielle $\faisSGamma_{A,\sigma,f}\to \faisGm$.
\end{defi}
Il est représentable par un schéma affine sur $S$ par la remarque \ref{imageinvnoyau_rema}.

De même qu'en \ref{PGOetctordus_prop}, le groupe algébrique $\faisPGO_{A,\sigma,f}$ agit sur $\faisSGamma_{A,\sigma,f}$ et sur $\faisSpin_{A,\sigma,f}$ en préservant les morphismes d'inclusion et la norme spinorielle. On obtient donc immédiatement:
\begin{prop} \label{torsionSpin_prop}
Par l'action du groupe algébrique $\faisPGO_{\faisM_{2n},\invadj_{2n},f_{2n}}$, le torseur $\faisIso_{(M_n,\invadj_{2n},f_{2n}),(A,\sigma,f)}$ tord $\faisSGamma_{2n}$ en $\faisSGamma_{A,\sigma,f}$ et de même pour $\faisSpin$.
\end{prop}

\subsection{Suites exactes longues de cohomologie}\label{suitecohom_sec}

Notons que le foncteur $\DeplCliff{2n+1} \to \Formes{\hypq_{2n+1}}$ qui oublie tout sauf le module quadratique se factorise par un foncteur $\DeplCliff{2n+1} \to \QuadDetTrivn{2n+1}$, en conservant l'isomorphisme entre la partie impaire $Z_1$ du centre de son algèbre de Clifford (canoniquement isomorphe à $\Det{M}$ par \ref{CliffStruc_theo}, point \ref{centredet_item}) et $\faisO_S$. En d'autres termes, le morphisme naturel $\faisSGamma_{2n+1} \to \faisorthO_{2n+1}$ se factorise par $\faisSO_{2n+1}$, et donc par torsion (Proposition \ref{foncttors_prop}) on obtient de même une factorisation du morphisme $\faisSGamma_q \to \faisorthO_q$ par $\faisSO_q$ pour tout $q$, forme \fppf\ de $\hypq_{2n+1}$.

\begin{prop} \label{secSpinSO_prop}
Pour tout $q$, forme \fppf\ de $\hypq_{2n+1}$, on a le diagramme commutatif de faisceaux \fppf\ aux lignes et aux colonnes exactes 
\begin{equation}\label{diagSpinSOimpair_diag}
\xymatrix{
 & 1 \ar[d] & 1 \ar[d] & 1 \ar[d] \\
1 \ar[r] & \faismu_2 \ar[r] \ar[d] & \faisSpin_{q} \ar[d] \ar[r] & \faisSO_{q} \ar[r] \ar@{=}[d] & 1 \\
1 \ar[r] & \faisGm \ar[r] \ar[d]_{(-)^2} & \faisSGamma_{q}  \ar[r] \ar[d]^{\sn} & \faisSO_{q} \ar[r] \ar[d] & 1 \\
1 \ar[r] & \faisGm \ar@{=}[r] \ar[d] & \faisGm \ar[r] \ar[d] & 1 \\
  & 1  & 1 \\
}
\end{equation}
où les deux morphismes horizontaux vers $\faisSO$ sont ceux définis juste au-dessus. 
De plus, le diagramme pour une forme $q$ est obtenu par torsion par la procédure de la proposition \ref{foncttors_prop} à partir de celui pour $q=\hypq_{2n+1}$.
\end{prop}
\begin{proof}
On montre d'abord le diagramme quand $q=\hypq_{2n+1}$. La première colonne est exacte car tout élément de $\Gamma(S)^{\times}$ est localement un carré pour la topologie \fppf.
La deuxième colonne est exacte par la remarque \ref{snsurj_rema}. 
La deuxième ligne est exacte à gauche. En effet, tout élément de $\faisSGamma$ qui induit l'identité par conjugaison sur $M$ est central dans $\faisCliff_{0}$, et est donc un élément de $\faisGm$ par \ref{CliffStruc_theo}. Sa norme spinorielle est alors son carré, ce qui explique la commutativité du carré en bas à gauche, ainsi que l'exactitude à gauche de la première ligne. Il ne reste qu'à montrer que le morphisme $\faisSpin \to \faisSO$ est un épimorphisme, puisque cela implique bien entendu la même chose pour $\faisSGamma \to \faisSO$. 
Remarquons tout d'abord que localement, pour la topologie Zariski, la proposition \ref{SOCartanDieudonne_prop} assure que tout élément $f$ de $\faisSO_q(S)$ se décompose en produit de réflexions orthogonales $f=\tau_{v_1}\circ \cdots \circ \tau_{v_m}$ par rapport aux vecteurs $v_1,\ldots,v_m$ tels que les $q(v_i)$ sont tous inversibles. Puisque $\tau_{v} = \tau_{\lambda v}$ pour tout $\lambda \in \faisGm(S)$, par extension \fppf, on peut supposer que tous les $v_i$ sont tels que $q(v_i)=-1$. Alors la norme spinorielle de la classe de l'élément $v_1 \otimes \cdots \otimes v_m$ est bien $1$ et il s'envoie bien sur $(-\tau_{v_1})\circ \cdots \circ(-\tau_{v_m})=(-1)^m f$. Soit $m$ est pair, soit $m$ est impair, et puisque $1=\det(f)=(-1)^m$ par le Lemme \ref{detrefl_lemm}, c'est donc que $1=-1$ dans $\faisO_S(S)$. Dans les deux cas, $v_1 \otimes \cdots \otimes v_m$ s'envoie donc sur $f$. 

Enfin, le cas d'un module quadratique général s'obtient par torsion du précédent, en appliquant la proposition \ref{sntordu_prop}, pour les deux colonnes de droite du diagramme. On vérifie sans peine que l'action de $\faisorthO_{\hypq_{2n+1}}$ sur les groupes (abéliens) noyaux $\mu_2$ et $\faisGm$ est triviale et qu'il ne se tordent pas, par conséquent. 
\end{proof}

\begin{prop} \label{secSpinOplus_prop}
Pour tout $q$, forme \fppf\ de $\hypq_{2n}$, on a le diagramme commutatif de faisceaux \fppf\ aux lignes et aux colonnes exactes 
\begin{equation}\label{diagSpinOplus_diag}
\xymatrix{
 & 1 \ar[d] & 1 \ar[d] & 1 \ar[d] \\
1 \ar[r] & \faismu_2 \ar[r] \ar[d] & \faisSpin_{q} \ar[d] \ar[r] & \faisOplus_{q} \ar[r] \ar@{=}[d] & 1 \\
1 \ar[r] & \faisGm \ar[r] \ar[d]_{(-)^2} & \faisSGamma_{q}  \ar[r] \ar[d]^{\sn} & \faisOplus_{q} \ar[r] \ar[d] & 1 \\
1 \ar[r] & \faisGm \ar@{=}[r] \ar[d] & \faisGm \ar[r] \ar[d] & 1 \\
  & 1  & 1 \\
}
\end{equation}
où la ligne du milieu est la suite exacte \eqref{SGamOplus_eq} de la proposition \ref{SGamOplus_prop}.
De plus, le diagramme pour une forme $q$ est obtenu par torsion par la procédure de la proposition \ref{foncttors_prop} à partir de celui pour $q=\hypq_{2n}$.
\end{prop}
\begin{proof}
Ainsi que dans la proposition précédente, on a la commutativité du diagramme en bas à gauche, puis que l'intersection de $\faisGm$ et $\faisSpin_q$ dans $\faisSGamma_q$ est bien $\faismu_2$. Pour vérifier que $\faisSpin_q \to \faisOplus_q$ est un épimorphisme, étant donné qu'on sait déjà que $\faisSGamma_q \to \faisOplus_q$ en est un, il suffit de relever \fppf-localement un point de $\faisOplus_q$, et de le corriger par un élément provenant de $\faisGm$ afin de rendre sa norme spinorielle triviale. Quitte à localiser encore, c'est possible par commutativité du diagramme en bas à gauche, puisque tout élément est \fppf-localement un carré. 

L'affirmation sur les formes tordues s'obtient comme dans la proposition précédente. 
\end{proof}

\section{Connexité}

Traitons maintenant de la connexité géométrique des groupes de nature quadratique introduits précédemment. Rappelons que:

\begin{lemm} \label{Xkbar_lemm}
Soit $X$ un schéma sur un corps $k$ et soit $\bar k$ un extension de $k$ algébriquement close. Si $X_{\bar k}$ est connexe, alors il en est de même de $X$. De plus, si $X$ est localement de type fini sur $k$, alors $X_{\bar k}(\bar k)$ est (très) dense dans l'espace topologique sous-jacent à $X_{\bar k}$, et $X$ est donc connexe si $X_{\bar k}(\bar k)$ l'est. 
\end{lemm}
\begin{proof}
Voir \cite[Ch. IV, Prop. 4.5.1 et 10.4.9]{ega}.
\end{proof}

Soit $n\geq 1$ et soit $(q,M)$ un $\faisO_S$-module quadratique de rang $n$, forme de $\hypq_{n}$. Soit $X_q$ la quadrique affine sur $S$ donnée sur les points par $X_q(T)=\{v \in M(T) \text{ t.q. }q(v)=1 \}$. De même, soit $U_q$ le sous foncteur de $M$ défini par $U_q(T)=\{ v \in M(T) \text{ t.q. } q(v) \in \Gamma(T)^\times \}$. Alors la quadrique $X_q$ est représentable par un schéma affine localement de type fini sur $S$, fermé dans $U_q$ qui est un ouvert affine sur $S$ dans $M$. En effet, par la proposition \ref{reprLocal_prop}, on se ramène au cas où $q$ est hyperbolique, or $X_{\hypq_{2n}}$ est tirée de $\Spec (\ZZ[x_1,\ldots,x_{2n}]/(1-\sum x_ix_{i+1}))$ et $U_{\hypq_{2n}}$ de $\Spec (\ZZ[x_1,\ldots,x_{2n}][(\sum x_ix_{i+1})^{-1}])$. Le cas de rang impair est analogue. Sur tout corps, on vérifie également aisément que $X_q$ est irréductible puisque défini par un polynôme irréductible sur une clôture algébrique, donc connexe, et $U_q$ est irréductible donc connexe puisqu'ouvert non vide dans $\mathbb{A}^n$ qui est irréductible. Par ailleurs, l'action (fidèle) de $\faisGm$ sur $M$ se restreint en une action sur $U_q$.

\begin{lemm} \label{arfrefl_lemm}
Si $q$ est une forme de $\hypq_{2n}$ ($n\geq 1$) et si $v \in M(S)$ est tel que $q(v)$ est inversible, alors la réflexion orthogonale $\tau_v$ par rapport à $v$ est d'invariant de Arf égal à $1 \in (\ZZ/2)(S)$. 
\end{lemm}
\begin{proof}
Puisque $\ZZ/2$ est localement constant, il suffit de le vérifier sur les fibres géométriques, et donc quand $q\simeq \hypq_{2n}$. On considère alors l'application $U_q \to \faisorthO_q$ définie sur les points par $v \mapsto \tau_v$. Puisque $(U_q)_k$ est connexe par le lemme précédent et $(\ZZ/2)_k$ est constitué de deux points fermés, la composition $(U_q)_k \to (\faisorthO_q)_k \to (\ZZ/2)_k$ est constante. Il suffit donc de monter qu'une réflexion bien précise s'envoie sur $1 \in (\ZZ/2)(k)$. Choisissons la réflexion par rapport à $e_1-e_2$, ou $e_1$ et $e_2$ sont une paire de vecteurs engendrant un plan hyperbolique et tels que $q(e_1)=q(e_2)=0$ et $b_q(e_1,e_2)=1$. Elle provient de $(\hypq_{2n})_\ZZ$ et par connexité de $\ZZ$, l'invariant de Arf est constant dessus, on peut donc se contenter de le calculer pour une fibre de caractéristique différente de $2$, auquel cas $\ZZ/2\simeq \mu_2$ par $i \mapsto (-1)^i$, et le résultat découle alors du lemme \ref{detrefl_lemm}. 
\end{proof}

\begin{lemm} \label{pairerefl_lemm}
Tout produit d'un nombre pair de réflexions est dans la composante connexe du neutre de $\faisorthO_q$.
\end{lemm}
\begin{proof}
Par récurrence, il suffit de le montrer pour un produit de deux réflexions, puisque tout élément $g$ de $\faisorthO_q$ permute les composantes connexes, donc $\tau_1\cdots\tau_{2m-2}$ et $(\tau_1 \cdots\tau_{2m-2})(\tau_{2m-1}\tau_{2m})$ sont dans la même composante connexe, si $(\tau_{2m-1} \tau_{2m})$ est dans la composante du neutre.
De plus, on peut supposer la base $S$ connexe, et même vérifier la propriété sur les fibres géométriques. Soient $v_1$ et $v_2$ tels que $q(v_1)$ et $q(v_2)$ sont inversibles. Alors l'application $U_q \to \faisorthO_q$ définie sur les points par $v \mapsto \tau_v \tau_{v_2}$ atteint $\tau_{v_1}\tau_{v_2}$ et $\id=\tau_{v_2}\tau_{v_2}$, ce qui montre le résultat par connexité de $U_q$.
\end{proof}

\begin{theo} \label{Oplusconnexe_theo}
Pour toute forme (étale ou \fppf) $q$ de $\hypq_{2n}$, le groupe algébrique $\faisOplus_q$ est à fibres connexes. 
\end{theo}
\begin{proof}
Par le lemme \ref{Xkbar_lemm}, on se limite à prouver la connexité de l'espace $\faisOplus_q(k)$ lorsque $k$ est algébriquement clos. Tout élément de $\faisorthO_q(k)$ est alors un produit de réflexions par le théorème de Cartan-Dieudonné \ref{Cartan-Dieudonne_theo}, et est dans $\faisOplus_q(k)$ si et seulement si ce nombre est pair par le lemme \ref{arfrefl_lemm}. Il résulte alors du lemme \ref{pairerefl_lemm} que $\faisOplus_q$ est alors exactement la composante connexe du neutre de $\faisorthO_q$.
\end{proof}

\begin{theo} \label{SOconnexe_theo}
Pour toute forme (étale ou \fppf) $q$ de $\hypq_{2n+1}$, le groupe algébrique $\faisSO_{q}$ est à fibres connexes.
\end{theo}
\begin{proof}
On se ramène de nouveau à montrer que $\faisSO_q(k)$ est connexe si $k$ est algébriquement clos et $q=\hypq_{2n+1}$. Par le théorème \ref{Cartan-Dieudsemi_theo}, tout élément $g$ de $\faisorthO_q(k)$ est un produit de $m$ réflexions orthogonales. Si la caractéristique de $k$ est différente de $2$, par le lemme \ref{pairerefl_lemm}, on a $g \in \faisSO_q(k)$ si et seulement si $m$ est pair, ce qui prouve que $\faisSO_q$ coïncide avec la composante connexe du neutre, puisque $\faismu_2$ a deux composantes connexes. Si la caractéristique de $k$ est $2$, on remarque que $\id=\tau_e$ où $e$ est une base de $\fq{1}$ dans la décomposition $\hypq_{2n+1}=\hypq_{2n}\perp \fq{1}$. On peut donc rajouter $\tau_e$ à un produit pour se ramener à un nombre pair de réflexions, et tout $\faisSO_q(k)$ est égal à sa composante connexe du neutre.
\end{proof}

\begin{lemm} \label{quotienttopo_lemm}
Soit $k$ un corps algébriquement clos et soient $G$ schéma en groupes localement de type fini sur $k$ et $F$ un sous-schéma en groupes. Si l'inclusion est  quasi-compacte (automatique si $F$ est affine ou fermé dans $G$) et si le quotient \fppf\ est représenté par un schéma $H$, alors $H(k)\simeq G(k)/F(k)$ comme espaces topologiques.
\end{lemm}
\begin{proof}
C'est évident en tant qu'ensembles sous-jacents, la seule chose à montrer est que la topologie sur $H(k)$ est bien la topologie quotient, autrement dit que l'application $G(k) \to H(k)$ est ouverte, ce qui découle de \cite[Exp. VI, théorème 3.2]{sga3}. 
\end{proof}

\begin{exo} \label{topologie_exo}
Soit $F$ un sous-groupe d'un groupe topologique $G$. Si $G/F$ est connexe et $F$ est inclus dans une composante connexe de $G$, ce qui est en particulier le cas si $F$ est connexe, alors $G$ connexe. 
\end{exo}

\begin{theo} \label{Spinconnexe_theo}
Pour toute forme (étale ou \fppf) $q$ de $\hypq_{2n}$ ou $\hypq_{2n+1}$, les groupes algébriques $\faisSpin_q$ et $\faisSGamma_q$ sont à fibres connexes.
\end{theo}
\begin{proof}
Comme précédemment, il suffit de montrer que $\faisSpin_q(k)$ (resp. $\faisSGamma_q(k)$) est connexe lorsque $k$ est un corps algébriquement clos et donc $q=\hypq_{2n}$ ou $q=\hypq_{2n+1}$.

En rang impair, on utilise alors la surjection $\faisSpin_{2n+1}(k) \to \faisSO_{2n+1}(k)$ (resp. ou $\faisSGamma_{2n+1}(k) \to \faisSO_{2n+1}(k)$) de noyau $\mu_2(k)$ (resp. $\faisGm$) de la proposition \ref{secSpinSO_prop}. Le lemme \ref{quotienttopo_lemm} et l'exercice \ref{topologie_exo}, achèvent la preuve pour $\faisSGamma_q$, et il suffit de montrer que $\mu_2(k)$ est tout entier dans une composante connexe de $\faisSpin_{2n+1}(k)$. Or ce sous-groupe est $\faismu_2 \cdot 1$ au sein de l'algèbre de Clifford. Si $e_1$ et $e_2$ sont deux vecteurs engendrant un plan hyperbolique, tels que $q(e_1)=q(e_2)=0$ et $b_q(e_1,e_2)=1$, il y a une copie de $\faisGm \subseteq \faisSGamma_{2n+1}$ des éléments de la forme $\alpha e_1\otimes e_2 +\alpha^{-1} e_2 \otimes e_1$, puisque $e_1 \otimes e_2 +e_2 \otimes e_1=1$. Or ils sont de norme spinorielle $1$, par un calcul direct. Cette copie de $\faisGm$ (connexe) est donc dans $\faisSpin_{2n+1}$ et contient $\faismu_2$.  

En rang pair, on procède de même, mais en avec la surjection de la proposition \ref{secSpinOplus_prop}.
\end{proof}

\begin{prop} \label{PGOplusconnexe_prop}
Pour toute forme (étale ou \fppf) $q$ de $\hypq_{2n}$, le groupe algébrique $\faisPGOplus_q$ est à fibres connexes.
\end{prop}
\begin{proof}
Cela découle de la connexité des fibres de $\faisOplus_q$ et de la surjection $\faisOplus_q \to \faisPGOplus_q$ de la proposition \ref{secOPGO_prop}.
\end{proof}

\begin{prop} \label{Spconnexe_prop}
Le groupe symplectique $\faisSp_{2n}$ (voir la définition \ref{PGSp_defi} plus bas) est à fibres connexes.
\end{prop}
\begin{proof}
La preuve est similaire à celle de la connexité de $\faisOplus$, mais un peu plus simple: sur un corps $k$, tout élément de $\faisSp_{2n}$ est un produit de transvections symplectiques (voir \cite{Dieu55}), c'est-à-dire de transformations de la forme $\tau_{v,\lambda}:x \mapsto x + \lambda \hypa_{2n}(x,v)v$ et réciproquement, toute transvection symplectique est évidement un élément du groupe symplectique. On connecte alors toute transvection symplectique à l'identité en faisant varier $\lambda$. Puisque $\faisGa$ est connexe, cela montre que toute transvection symplectique est dans la composante connexe de l'identité, ainsi que tout élément du groupe symplectique, par récurrence. 
\end{proof}

\section{Groupes semi-simples de type $B_n$} \label{Bn_sec}

\subsection{Groupe déployé adjoint} \label{adjointBn_sec}

Dans cette section, $(M,q)$ est toujours un $\faisO_S$-module quadratique tel que $M$ est localement libre de type fini (voir la section \ref{modulesquadratiques_sec}), de groupe orthogonal noté $\faisorthO_q$ (définition \ref{O_defi}), et de groupe spécial orthogonal noté $\faisSO_q$ (définition \ref{SO_defi}). Nous supposerons de plus que $(M,q)$ est une forme \fppf\ de $\hypq_{2n}$ ou de $\hypq_{2n+1}$, pour un certain $n\geq 1$.

\begin{prop} \label{dimOSO_prop}
Si $m$ est le rang de $M$, alors les fibres géométriques des groupes algébriques $\faisorthO_q$, $\faisSO_q$ et, si $m$ est pair, $\faisOplus_q$ sont de dimension $m(m-1)/2$ (donc $n(2n-1)$ si $m=2n$ ou $n(2n+1)$ si $m=2n+1$).
\end{prop}
\begin{proof}
Sur un corps algébriquement clos, $q=\hypq_{2n}$ ou $q=\hypq_{2n+1}$. On peut alors suivre la preuve de \cite[\S 23.6]{borelag} mutatis mutandis pour $\faisorthO_q$.
On en déduit le résultat pour les autres groupes par la proposition \ref{dimsec_prop} à l'aide des suites exactes des propositions \ref{seSOOimpair_prop}, \ref{seSOOpair_prop} et \ref{secOplus_prop}. 
\end{proof}

Rappelons que le module sous-jacent à $\hypq_{2n+1}$ est $\faisO_S^{2n+1}$, de base canonique notée $(e_0, e_1,\ldots, e_{2n})$, et la structure quadratique est donnée par l'équation 
$$\hypq_{2n+1}(x_0,x_1,\ldots,x_{2n})=x_0^2+x_1 x_2+\cdots+x_{2n-1}x_{2n}.$$ 
\begin{defi} \label{SOimpairsurZ_defi}
Nous noterons $\faisorthO_{2n+1}$ (resp. $\faisSO_{2n+1}$) au lieu de $\faisorthO_{\hypq_{2n+1}}$ (resp. $\faisSO_{\hypq_{2n+1}}$).
\end{defi}
Calculons maintenant l'algèbre de Lie de $\faisorthO_{2n+1}$ comme une sous-algèbre de $\faisLie_{\faisGL_{2n+1}}$, conformément au lemme \ref{sous_algebre_lemm}. 
En considérant la suite exacte
$$\xymatrix{0\ar[r] & \faisLie_{\faisorthO_{2n+1}}(S) \ar[r] & \faisorthO_{2n+1}(I_S)\ar[r]^-p & \faisorthO_{2n+1}(S)\ar[r] & 1}$$
dont chacun des termes s'inclut dans le terme correspondant de la suite \eqref{LieGLn_eq} de $\faisGL_{2n+1}$, on s'aperçoit que $\faisLie_{\faisorthO_{2n+1}}(S)$ est constituée des éléments de la forme $1+f t$ avec $f\in \setM_{2n+1}(\Gamma(S))$ satisfaisant $(b_h)_{T}(f_{T}(v),v)=0$ pour tout $S$-schéma $T$ et tout $v\in \Gamma(T)^{2n+1}$. Ceci est équivalent à dire que le module bilinéaire $b_{h,f}:\faisO_S^{2n+1}\times \faisO_S^{2n+1}\to \faisO_S$ donné par $b_{h,f}(x,y)=b_h(f(x),y)$ est alterné. 

Pour calculer la dimension des fibres géométriques de cette algèbre de Lie, on peut localiser et donc supposer $S$ affine égal à $\Spec(R)$ avec $R$ local. Notons $B$ la matrice de $f_R$ sur la base canonique de $\faisO_S^{2n+1}(R)$ (resp. $C$ celle de $b_{q_R,f_R}$). 

On a alors, pour tout $j=0,\ldots,2n$, les égalités $C_{0j}=2B_{0j}$, $C_{2i-1,j}=B_{2i,j}$ et $C_{2i,j}=B_{2i-1,j}$ pour tout $i=1,\ldots,n$. Il est clair que $C$ est alternée si et seulement si $C_{ii}=0$ et $C_{ij}=-C_{ji}$ pour tous $i\neq j$. On obtient finalement les équations suivantes:
$$
\left\{\begin{array}{rl}2B_{00}=B_{2i,2i-1}=B_{2i-1,2i}=0 & \text{ pour tout $i=1,\ldots, n$} \\
2B_{0,2j}=-B_{2j-1,0} & \text{ pour tout $j=1,\ldots, n$} \\
2B_{0,2j-1}=-B_{2j,0} & \text{ pour tout $j=1,\ldots, n$} \\
B_{2i-1,2j}=-B_{2j-1,2i} & \text{ pour tout $1\leq i<j\leq n$}\\  
B_{2i,2j}=-B_{2j-1,2i-1} & \text{ pour tout $i,j=1,\ldots,n$}\\ 
B_{2i,2j-1}=-B_{2j,2i-1} & \text{ pour tout $1\leq i<j\leq n$.}
\end{array}\right.
$$
Si $R$ est une algèbre sur un corps $k$ de caractéristique différente de $2$, la dimension de $\faisLie_{\faisorthO_{2n+1}}(R)$ est ainsi $(2n+1)^2-(2n^2+3n+1)=(2n+1)\cdot n$. Elle est égale à la dimension du groupe algébrique $(\faisorthO_{2n+1})_k$ qui est par conséquent lisse par la proposition \ref{sgcorpslisse_prop}. Par contre, si $R$ est sur un corps $k$ de caractéristique $2$, on voit qu'on a seulement $2n^2+3n$ équations linéairement indépendantes. L'algèbre de Lie $\faisLie_{\faisorthO_{2n+1}}(k)$ est ainsi de dimension strictement supérieure à la dimension de $(\faisorthO_{2n+1})_k$. On peut donc en conclure que $\faisorthO_{2n+1}$ n'est pas lisse sur $\ZZ$ puisque ses fibres géométriques ne sont pas toutes lisses.

Si par contre on se restreint au noyau du déterminant, la lissité s'améliore. 

L'algèbre de Lie $\faisLie_{\faisSO_{2n+1}}$ est la sous-algèbre de $\faisLie_{\faisorthO_{2n+1}}$ des matrices de trace nulle. On obtient ainsi une description de l'algèbre de Lie de $\faisSO_{2n+1}(R)$ pour tout anneau $R$ comme étant la sous-algèbre de $\setM_{2n+1}(R)$ des matrices $B_{ij}$ satisfaisant les relations
$$
\left\{\begin{array}{rl}B_{00}=B_{2i,2i-1}=B_{2i-1,2i}=0 & \text{ pour tout $i=1,\ldots, n$} \\
2B_{0,2j}=-B_{2j-1,0} & \text{ pour tout $j=1,\ldots, n$} \\
2B_{0,2j-1}=-B_{2j,0} & \text{ pour tout $j=1,\ldots, n$} \\
B_{2i-1,2j}=-B_{2j-1,2i} & \text{ pour tout $1\leq i<j\leq n$}\\  
B_{2i,2j}=-B_{2j-1,2i-1} & \text{ pour tout $i,j=1,\ldots,n$}\\ 
B_{2i,2j-1}=-B_{2j,2i-1} & \text{ pour tout $1\leq i< j\leq n$.}
\end{array}\right.
$$
On trouve alors une base $\faisLie_{\faisSO_{2n+1}}(R)$ donnée par les matrices
$$
\left\{\begin{array}{rl}E_{0,2i}-2E_{2i-1,0} & \text{ pour tout $i=1,\ldots,n$} \\ 
E_{0,2i-1}-2E_{2i,0} & \text{ pour tout $i=1,\ldots,n$}\\
E_{2i-1,2j-1}-E_{2j,2i} & \text{ pour $i,j=1,\ldots,n$}\\
E_{2i-1,2j}-E_{2j-1,2i} & \text{ pour $1\leq i<j\leq n$}\\
E_{2i,2j-1}-E_{2j,2i-1} & \text{ pour $1\leq i<j\leq n$}.
\end{array}\right.
$$
où $E_{i,j}$ dénote comme d'habitude la matrice élémentaire avec un $1$ en ligne $i$ colonne $j$ et zéro ailleurs. Elle est donc de dimension $n(2n+1)$.

\begin{prop}\label{Oplus_ss_prop}
Le groupe $\faisSO_{2n+1}$ est semi-simple.
\end{prop}
\begin{proof}
On vérifie tout d'abord que $\faisSO_{2n+1}$ est lisse sur $\Spec(\ZZ)$. Au vu de la proposition \ref{lissite_prop}, il suffit de vérifier que les fibres sont connexes, lisses, et de dimension constante. La connexité est prouvée dans le théorème \ref{SOconnexe_theo}. Les fibres sont lisses par la proposition \ref{sgcorpslisse_prop} puisque sur tout corps $k$, sa dimension est bien égale à celle de la fibre $(\faisSO_{2n+1})_k$. Par changement de base, il reste donc lisse sur toute base $S$. Il est de plus affine sur $S$ (voir juste avant \ref{SO_defi}). Pour conclure qu'il est réductif (au sens de \SGAtrois) et même semi-simple, on utilise que les fibres géométriques de $\faisSO_{2n+1}$ sont (semi-)simples par \cite[Theorem 25.10]{bookinv}.
\end{proof}

Pour trouver un tore maximal de $\faisSO_{2n+1}$, on constate que le morphisme défini sur les points par $(\alpha_1,\ldots,\alpha_n)\mapsto (1,\alpha_1,\alpha_1^{-1},\alpha_2,\alpha_2^{-1},\ldots,\alpha_n,\alpha_n^{-1})$ de $\faisGm^n$ vers $\faisDiag_{2n+1}$ a pour image un sous-tore fermé $\faisDiag_{2n+1}$ isomorphe à $\faisGm^n$ et est en fait dans le sous-groupe fermé $\faisSO_{2n+1}$ de $\faisGL_{2n+1}$. On note $\faisSOdiag_{2n+1}$ ce sous-groupe de $\faisSO_{2n+1}$. Par définition,  $\faisSOdiag_{2n+1}$ est un tore déployé de rang $n$ et il est maximal puisque il est maximal sur les fibres géométriques par \cite[\S 23.6]{borelag}. Considérons la base de caractères $\{t_i\vert i=1,\ldots,n\}$ de ce tore donnés par 
$$t_i(1,\alpha_1,\alpha_1^{-1},\alpha_2,\alpha_2^{-1},\ldots,\alpha_n,\alpha_n^{-1})=\alpha_i$$ 
pour $i=1,\ldots,n$. 

Les données d'une base de l'algèbre de Lie comme sous-algèbre de $\faisM_{2n+1}$ et du tore maximal comme un sous-tore du tore maximal usuel de $\faisGL_{2n+1}$ permettent de calculer les racines de $\faisSO_{2n+1}$. Pour ce faire, rappelons tout d'abord la proposition \ref{donneeradicielleGLn_prop} qui montre que le tore des matrices diagonales dans $\faisGL_{2n+1}$ agit sur l'algèbre de Lie via 
$$TE_{ij}T^{-1}=t_i(T)t_j^{-1}(T)E_{ij}.$$
Les espaces propres de la représentation adjointe (et les racines) sont donc les suivants:

\begin{enumerate}
\item Le sous-module de l'algèbre de Lie du tore (les matrices diagonales dans l'algèbre de Lie) de poids trivial. 
\item Le module libre de rang $1$ de base $(E_{0,2i}-2E_{2i-1,0})$ de poids $t_i$ pour $1\leq i\leq n$. 
\item Le module libre de rang $1$ de base $(E_{0,2i-1}-2E_{2i,0})$ de poids $-t_i$ pour $1\leq i\leq n$.
\item Le module libre de rang $1$ de base $(E_{2i-1,2j-1}-E_{2j,2i})$ de poids $t_i-t_j$ pour tout $i,j=1,\ldots,n$ avec $i\neq j$.
\item Le module libre de rang $1$ de base $(E_{2i-1,2j}-E_{2j-1,2i})$ de poids $t_i+t_j$ pour tout $1\leq i<j\leq n$.
\item Le module libre de rang $1$ de base $(E_{2i,2j-1}-E_{2j,2i-1})$ de poids $-t_i-t_j$ pour tout $1\leq i<j\leq n$.
\end{enumerate}

On explicite maintenant pour toute racine $\alpha$ non triviale l'unique morphisme $\faisSOdiag_{2n+1}$-équivariant de groupes algébriques
$$\exp_\alpha:\faisGa\to \faisSO_{2n+1}$$
induisant l'inclusion canonique de l'espace propre associé à $\alpha$ dans son algèbre de Lie (\cite[Exp. XXII, théorème 1.1]{sga3}). On trouve:
\begin{enumerate}
\item Le morphisme $\exp_{t_i}$ défini par 
$\lambda\mapsto Id+\lambda(E_{0,2i}-2E_{2i-1,0}-\lambda E_{2i-1,2i}).$
\item Le morphisme $\exp_{-t_i}$ défini par 
$\lambda\mapsto Id+\lambda(E_{0,2i-1}-2E_{2i,0}-\lambda E_{2i,2i-1}).$
\item Le morphisme $\exp_{t_i-t_j}$ défini par
$\lambda\mapsto Id+\lambda(E_{2i-1,2j-1}-E_{2j,2i}).$
\item Le morphisme $\exp_{t_i+t_j}$ défini par
$\lambda\mapsto Id+\lambda(E_{2i-1,2j}-E_{2j-1,2i}).$
\item Le morphisme $\exp_{-t_i-t_j}$ défini par 
$\lambda\mapsto Id+\lambda(E_{2i,2j-1}-E_{2j,2i-1}).$
\end{enumerate}

\'Etant donnés les racines et les morphismes exponentiels, on calcule les coracines. Associant à chaque racine dans la liste ci-dessus, on trouve dans l'ordre:
\begin{enumerate}
\item Les coracines $2t_i^\vee$ pour tout $1\leq i\leq n$.
\item Les coracines $-2t_i^\vee$ pour tout $1\leq i\leq n$.
\item Les coracines $t_i^\dual-t_j^\dual$ pour tout $i,j=1,\ldots,n$ avec $i\neq j$.
\item Les coracines $t_i^\dual+t_j^\dual$ pour tout $1\leq i<j\leq n$.
\item \label{racines--_item} Les coracines $-t_i^\dual-t_j^\dual$ pour tout $1\leq i<j\leq n$.
\end{enumerate}

Les accouplements sont $(X,Y)\mapsto -XY$ dans le cas \ref{racines--_item}. ci-dessus, et $(X,Y)\mapsto XY$ dans tous les autres cas. 

Nous pouvons ainsi énoncer la proposition suivante:

\begin{prop} \label{donneeradicielleOplus_prop}
La donnée radicielle de $\faisSO_{2n+1}$ relativement au tore maximal déployé $\faisSOdiag_{2n+1}$ est 
\begin{itemize}
\item Le $\ZZ$-module libre $N\simeq \ZZ^n$ des caractères de $\faisSOdiag_{2n+1}$ (dont $t_1,\ldots,t_n$ forment une base); 
\item Le sous-ensemble fini de celui-ci formé des racines $\pm t_i$ pour $1\leq i\leq n$ et $\pm t_i\pm t_j$ pour $1\leq j<i\leq n$;
\item Le $\ZZ$-module libre $N^\dual\simeq\ZZ^n$ des cocaractères (de base duale $t_1^\cdual,\ldots,t_n^\cdual$);
\item Le sous-ensemble fini de celui-ci formé des coracines $\pm 2t_i^\dual$ pour $i=1\leq i\leq n$ et $\pm t_i^\dual\pm t_j^\dual$ pour $1\leq j<i\leq n$.
\end{itemize}
\end{prop}

Finalement, on déduit le résultat suivant des calculs précédents:

\begin{theo}
Pour tout $n\geq 1$, le groupe $\faisSO_{2n+1}$ est déployé, semi-simple et adjoint de type $B_n$.
\end{theo}
\begin{proof}
Le groupe est semi-simple par la proposition \ref{Oplus_ss_prop} et adjoint puisque ses racines engendrent les caractères du tore. On voit que les racines $\{t_1-t_2, t_2-t_3,\ldots,t_{n-1}-t_n,t_n\}$ forment un système de racines simples, et il est facile de vérifier que son diagramme de Dynkin associé est de type $B_n$. 
\end{proof}

\subsection{Groupe déployé simplement connexe}\label{bn_sconnexe}

Soit $(M,q)$ un $\faisO_S$-module quadratique. On dispose de son algèbre de Clifford $\faisCliff_q$ (définition \ref{algCliff_defi}), à laquelle on associe son groupe de Clifford $\faisGamma_{q}$ (définition \ref{groupeClifford_defi}) ainsi que son groupe de Clifford pair $\faisSGamma_{q}$ (définition \ref{groupeCliffordpair_defi}) qui sont donnés sur les $R$-points par 
$$\faisGamma_q(R)=\{g\in \setCliff_{q_R}^\times |\ g(M_R)g^{-1}=M_R\}.$$
et
$$\faisSGamma_q(R)=\{g\in \setCliff_{0,q_R}^\times |\ g(M_R)g^{-1}=M_R\}.$$
On dispose enfin du groupe spinoriel $\faisSpin_q$ associé à $(M,q)$ (définition \ref{PinSpin_defi}) dont les $R$-points sont de la forme
$$\faisSpin_q(R)=\{g\in C_0(M_R,q_R)^\times\vert g(M_R)g^{-1}=M_R\text{ et } g\cdot \invstd(g)=1\}$$
où $\invstd$ est l'involution standard ou canonique. 

\begin{prop} \label{Spindim_prop}
Sur toute fibre géométrique $s$, le groupe algébrique $\faisSpin_{2n}$ (resp. $\faisSpin_{2n+1}$) est de dimension $n(2n-1)$ (resp. $n\cdot(2n+1)$).
\end{prop}
\begin{proof}
On utilise la proposition \ref{dimsec_prop}. La ligne exacte du haut du diagramme de la proposition \ref{secSpinOplus_prop} montre que la dimension est la même que celle de $\faisSO_{2n}$, donnée par la proposition \ref{dimOSO_prop}. On procède de même avec la suite exacte de la proposition \ref{secSpinSO_prop} pour $\faisSpin_{2n+1}$. 
\end{proof}

Supposons maintenant que $q=\hypq_{2n+1}$ et que $S=\Spec(\ZZ)$ et calculons l'algèbre de Lie de $\faisSpin_{2n+1}$. Pour ce faire, on considère la suite exacte \eqref{seLie_eq} 
$$\xymatrix@C=1.5em{0\ar[r] & \faisLie_{\faisSpin_{2n+1}}(R)\ar[r] & \faisSpin_{2n+1}(R[t])\ar[r]^-p & \faisSpin_{2n+1}(R)\ar[r] & 1}$$
pour tout anneau $R$. Un élément de la forme $1+t a$ avec $a\in C_0(M_R,q_R)$ est dans $\faisSpin_{2n+1}(R[t])$ si et seulement si $am-ma\in M$ pour tout $m\in M$ et $a+\invstd(a)=0$. La première condition donne $a=b+\sum_{0\leq i<j\leq 2n}a_{ij}e_i\otimes e_j$ pour des éléments $b,a_{ij}\in R$, et la seconde condition donne $2b+\sum_{i=1}^n a_{2i-1,2i}=0$. Une base de l'algèbre de Lie $\faisLie_{\faisSpin_{2n+1}}(R)$ est ainsi donnée par les éléments suivants:
\begin{itemize}
\item[1.] $1-2e_{2n-1}\otimes e_{2n}$
\item[2.] $e_{2i-1}\otimes e_{2i}-e_{2n-1}\otimes e_{2n}$ pour tout $1\leq i\leq n-1$.
\item[3.] $e_i\otimes e_j$ pour tout $0\leq i<j\leq 2n$ avec $(i,j)\neq (2r-1,2r)$ pour $r=1,\ldots,n$.
\end{itemize}
Cela montre que l'algèbre de Lie est le groupe vectoriel associé à un module libre de rang $n(2n+1)$. On en déduit: 
\begin{prop}\label{Spin_ss_prop}
Le schéma en groupes $\faisSpin_{2n+1}$ est semi-simple.
\end{prop}
\begin{proof}
On sait que les fibres sont connexes (théorème \ref{Spinconnexe_theo}), de dimension constante et semi-simples. Elles sont lisses par la proposition \ref{sgcorpslisse_prop} et les calculs ci-dessus. La proposition \ref{lissite_prop} fournit donc la lissité sur $\Spec(\ZZ)$, puis sur toute autre base par changement de base. Nous savons déjà que $\faisSpin_{2n+1}$ est affine (juste après la définition \ref{PinSpin_defi}). Il est donc réductif et semi-simple au sens de \SGAtrois\ puisqu'à fibres semi-simples (résultat classique qu'on peut obtenir à partir de la semi-simplicité des fibres de $\faisSO_{2n+1}$ et de la suite exacte les reliant). 
\end{proof}

Pour déterminer un tore maximal de $\faisSpin_{2n+1}$, on considère le diagramme de la proposition \ref{secSpinSO_prop} pour la forme $\hypq_{2n+1}$
$$
\xymatrix{
 & 1 \ar[d] & 1 \ar[d] & 1 \ar[d] \\
1 \ar[r] & \faismu_2 \ar[r] \ar[d] & \faisSpin_{2n+1} \ar[d] \ar[r] & \faisSO_{2n+1} \ar[r] \ar@{=}[d] & 1 \\
1 \ar[r] & \faisGm \ar[r] \ar[d]_{(-)^2} & \faisSGamma_{2n+1}  \ar[r] \ar[d] & \faisSO_{2n+1} \ar[r] \ar[d] & 1 \\
1 \ar[r] & \faisGm \ar@{=}[r] \ar[d] & \faisGm \ar[r] \ar[d] & 1 \\
  & 1  & 1 \\
}
$$
dont la seconde ligne, la suite exacte 
$$
\xymatrix{1 \ar[r] & \faisGm \ar[r] & \faisSGamma_{2n+1}  \ar[r]^{\chi}  & \faisSO_{2n+1} \ar[r] & 1 }
$$
nous permet facilement de trouver un tore maximal de $\faisSGamma_{2n+1}$. En effet, soit $\faisSGammadiag_{2n+1}$ la préimage de $\faisSOdiag_{2n+1}$ sous le morphisme $\chi$. Une extension de tores étant un tore, on voit que $\faisSGammadiag_{2n+1}$ est un tore, qui est maximal puisque $\faisSOdiag_{2n+1}$ l'est. On peut expliciter les caractères de $\faisSGammadiag_{2n+1}$, mais on ne se servira que des cocaractères qui décrivent explicitement le tore. On trouve les cocaractères suivants:
$$s_0^\dual:\faisGm\to \faisSGammadiag_{2n+1}$$
défini pour tout schéma $T$ et tout $\alpha\in \Gamma(T)$ par $s_0^\dual(\alpha)=\alpha \cdot 1$ (dans l'algèbre de Clifford)
$$s_1^\dual:\faisGm\to \faisSGammadiag_{2n+1}$$
défini par $s_1^\dual(\alpha)=\alpha e_1\otimes e_2+ e_2\otimes e_1,$
ainsi que pour tout $i=2,\ldots,n$ les morphismes 
$$t_i^\dual:\faisGm\to \faisSGammadiag_{2n+1}$$
où  $t_i^\dual(\alpha)$ est donné  par la formule
{\small$$\alpha e_1\otimes e_2\otimes e_{2i-1}\otimes e_{2i}+e_1\otimes e_2\otimes e_{2i}\otimes e_{2i-1}+ e_2\otimes e_1\otimes e_{2i-1}\otimes e_{2i}+ \alpha^{-1}e_2\otimes e_1\otimes e_{2i}\otimes e_{2i-1}.$$}
Considérons maintenant la suite exacte de groupes de la colonne du milieu dans le diagramme ci-dessus
$$\xymatrix{1\ar[r] & \faisSpin_{2n+1}\ar[r] & \faisSGamma_{2n+1}\ar[r]^-{sn} & \faisGm\ar[r] & 1}$$
où $\sn$ est la norme spinorielle (définition \ref{sn_defi}).
Le noyau de la restriction $\sn:\faisSGammadiag_{2n+1}\to \faisGm$ de la norme spinorielle au tore maximal de $\faisSGamma_{2n+1}$ donne immédiatement un tore maximal de $\faisSpin_{2n+1}$. On le note $\faisSpindiag_{2n+1}$. On a $sn(s_0^\dual(\alpha))=\alpha^2$, $sn(s_1^\dual(\alpha))=\alpha$ et $sn(t_i^\dual(\alpha))=1$ pour tout $i=2,\ldots,n$. Il s'ensuit que le cocaractère $t_1^\dual=2s_1^\dual-s_0^\dual$ (donné par $t_1^\dual(\alpha)=\alpha e_1\otimes e_2+\alpha^{-1} e_2\otimes e_1$) et les cocaractères $t_i^\dual$ ($i=2,\ldots,n$) déterminent $\faisSpindiag_{2n+1}$.

Les caractères de ce tore sont donnés par $t_i(t_j^\dual(\alpha_j))=\delta_{ij}\alpha_j$ pour tout $i,j=1,\ldots,n$.

On calcule maintenant les racines de $\faisSpin_{2n+1}$ par rapport au tore maximal  $\faisSpindiag_{2n+1}$ et les espaces propres dans la représentation adjointe. On trouve pour tout anneau $R$:

\begin{itemize}
\item[1.] L'algèbre de Lie du tore engendré par $R\cdot (1-2e_{2n-1}\otimes e_{2n})$ et par les $R\cdot (e_{2i-1}\otimes e_{2i}-e_{2n-1}\otimes e_{2n})$ pour tout $i=2,\ldots,n$. 
\item[2.] Le sous-module $R\cdot (e_0\otimes e_{1})$ de poids $2t_1+t_2+\ldots+ t_n$.
\item[3.] Le sous-module $R\cdot (e_0\otimes e_2)$ de poids $-2t_1-t_2-\ldots- t_n$.
\item[4.] Le sous-module $R\cdot (e_0\otimes e_{2i-1})$ de poids $t_i$ pour tout $i=2,\ldots,n$.
\item[5.] Le sous-module $R\cdot (e_0\otimes e_{2i})$ de poids $-t_i$ pour tout $i=2,\ldots,n$.
\item[6.] Le sous-module $R\cdot (e_1\otimes e_{2i-1})$ de poids $2t_1+2t_i+\sum_{j=2, j\neq i}^n t_j$ pour tout $i=2\ldots, n$.
\item[7.] Le sous-module $R\cdot (e_2\otimes e_{2i})$ de poids $-2t_1-2t_i-(\sum_{j=2, j\neq i}^n t_j)$ pour tout $i=2,\ldots,n$.
\item[8.] Le sous-module $R\cdot (e_{2i-1}\otimes e_{2j-1})$ de poids $t_i+t_j$ pour tout $2\leq i<j\leq n$.
\item[9.] Le sous-module $R\cdot (e_{2i}\otimes e_{2j})$ de poids $-t_i-t_j$ pour tout $2\leq i<j\leq n$.
\item[10.] Le sous-module $R\cdot (e_1\otimes e_{2i})$ de poids $2t_1+ \sum_{j=2, j\neq i}^n t_j$ pour tout $i=2,\ldots,n$.
\item[11.] Le sous-module $R\cdot (e_2\otimes e_{2i-1})$ de poids $-2t_1- (\sum_{j=2, j\neq i}^n t_j)$ pour tout $i=2,\ldots,n$.
\item[12.] Le sous-module $R\cdot (e_{2i-1}\otimes e_{2j})$ de poids $t_i-t_j$ pour tout $2\leq i<j\leq n$.
\item[13.] Le sous-module $R\cdot (e_{2i}\otimes e_{2j-1})$ de poids $t_j-t_i$ pour tout $2\leq i<j\leq n$.
\end{itemize}

Les morphismes exponentiels sont donnés pour tous les espaces propres par $\lambda\mapsto 1+\lambda(e_i\otimes e_j)$ avec $i<j$. Les racines et les morphismes exponentiels étant connus, il est facile de calculer les coracines. Associant à chaque racine dans la liste ci-dessus sa coracine, on trouve dans l'ordre:

\begin{itemize}
\item[2.] La coracine $t_1^\dual$.
\item[3.] La coracine $-t_1^\dual$.
\item[4.] Les coracines $2t_i^\dual-t_1^\dual$ pour tout $i=2,\ldots,n$.
\item[5.] Les coracines $t_1^\dual-2t_i^\dual$ pour tout $i=2,\ldots,n$..
\item[6.] Les coracines $t_i^\dual$ pour tout $i=2\ldots, n$.
\item[7.] Les coracines $-t_i^\dual$ pour tout $i=2\ldots, n$.
\item[8.] Les coracines $t_i^\dual+t_j^\dual-t_1^\dual$ pour tout $2\leq i<j\leq n$.
\item[9.] Les coracines $t_1^\dual-t_i^\dual-t_j^\dual$ pour tout $2\leq i<j\leq n$.
\item[10.] Les coracines $t_1^\dual-t_i^\dual$ pour tout $i=2,\ldots,n$.
\item[11.] Les coracines $t_i^\dual-t_1^\dual$ pour tout $i=2,\ldots,n$.
\item[12.] Les coracines $t_i^\dual -t_j^\dual$ pour tout $2\leq i<j\leq n$.
\item[13.] Les coracines $t_i^\dual -t_j^\dual$ pour tout $2\leq i<j\leq n$. 
\end{itemize}

Les accouplements sont $(X,Y)\mapsto XY$ dans les cas 2. à 9. et $(X,Y)\mapsto -XY$ dans les cas 10. à 13. 
Les longues listes données ci-dessus permettent de trouver la donné radicielle de $\faisSpin_{2n+1}$.

\begin{prop}\label{donneeradicielleSpin_prop}
La donnée radicielle de $\faisSpin_{2n+1}$ par rapport au tore déployé $\faisSpindiag_{2n+1}$ est 
\begin{itemize}
\item Le $\ZZ$-module libre $N\simeq \ZZ^n$ des caractères de $\faisSpindiag_{2n+1}$ (dont $t_1,\ldots,t_n$ forment une base); 
\item Le sous-ensemble fini de celui-ci formé des racines $\pm t_i$ pour $2\leq i\leq n$, $\pm t_i\pm t_j$ pour $2\leq j<i\leq n$, $\pm (2t_1+t_2+\ldots+ t_n)$, $\pm(2t_1+t_2+\ldots+t_{i-1}+2t_i+t_{i+1}+ \ldots+ t_n)$ pour $i=2,\ldots,n$ et $\pm(2t_1+t_2+\ldots+t_{i-1}+t_{i+1}+ \ldots+ t_n)$ pour $i=2,\ldots,n$;
\item Le $\ZZ$-module libre $N^\dual\simeq\ZZ^n$ des cocaractères (de base duale $t_1^\cdual,\ldots,t_n^\cdual$);
\item Le sous-ensemble fini de celui-ci formé des coracines $\pm t_i^\dual$ pour $i=1,\dots,n$, $\pm (2t_i^\dual-t_1^\dual)$ pour $i=2,\ldots,n$, $\pm (t_i^\dual+t_j^\dual-t_1^\dual)$ pour $2\leq i<j\leq n$ et $\pm(t_i^\dual-t_j^\dual)$ pour $1\leq i<j\leq n$.
\end{itemize}
\end{prop}

Au vu des calculs obtenus, nous pouvons énoncer le résultat suivant:

\begin{theo}
Le groupe $\faisSpin_{2n+1}$ est déployé, semi-simple et simplement connexe de type $B_n$.
\end{theo}

\begin{proof}
Il reste à voir qu'il est simplement connexe de type $B_n$. Le premier point est clair puisque les coracines engendrent les cocaractères. L'ensemble de racines simples $\{2t_1+\sum_{i=3}^n t_i,t_2-t_3,t_3-t_4,\ldots,t_{n-1}-t_n, t_n\}$ montre que le diagramme de Dynkin associé est de type $B_n$.
\end{proof}

On cherche maintenant le centre de $\faisSpin_{2n+1}$. Pour ce faire, rappelons qu'on dispose d'un morphisme de groupes $\chi:\faisSGamma_{2n+1}\to \faisSO_{2n+1}$. La restriction de $\chi$ à $\faisSpin_{2n+1}$ induit un morphisme
$$\faisSpin_{2n+1}\to \faisSO_{2n+1}$$
et un morphisme sur les tores déployés
$$\faisSpindiag_{2n+1}\to \faisSOdiag_{2n+1}$$
dont le noyau est précisément le centre de $\faisSpin_{2n+1}$. Utilisant les descriptions explicites précédentes, on trouve qu'il envoie un élément $(\alpha_1,\ldots,\alpha_n)$ sur $(\alpha_1^2,\alpha_2\alpha_1^2,\ldots,\alpha_n\alpha_1^2)$ (tout deux décomposés sur les bases de cocaractères fournies). Nous avons ainsi démontré la proposition suivante:

\begin{prop} \label{centreSpin_prop}
Le morphisme de groupes algébriques
$$\zeta:\faismu_2\to Z(\faisSpin_{2n+1})$$
défini par $\zeta(\alpha)=\alpha\cdot 1$ (vu dans l'algèbre de Clifford) est un isomorphisme.
\end{prop}


\subsection{Automorphismes}

Puisque le diagramme de Dynkin n'a pas d'automorphisme non trivial, tout automorphisme est intérieur. En d'autres termes, $\faisAut^\gr_{\faisSpin_{2n+1}} =\faisAut^\gr_{\faisSO_{2n+1}}= \faisSO_{2n+1}$.

\subsection{Groupes tordus}

\begin{theo} \label{Bnadjointtordu_theo}
\`A isomorphisme près, les $S$-schémas en groupes réductifs de donnée radicielle déployée identique à celle de $\faisSO_{2n+1}$ ($n\geq 1$), c'est-à-dire de type déployé semi-simple adjoint $B_n$, sont les groupes $\faisSO_{q}$ où $q$ est un $\faisO_S$-module quadratique localement isomorphe à $\hypq_{2n+1}$ pour la topologie étale et de module demi-déterminant trivial de parité de $n$ (c'est-à-dire isomorphe à celui de la forme hyperbolique).
Toutes les formes sont intérieures (et donc strictement intérieures puisque le groupe est adjoint). Pour les formes fortement intérieures, voir le corollaire \ref{Bnadjfortementint_coro}.
\end{theo}
\begin{proof}
Nous aurons besoin du résultat suivant:
\begin{lemm} \label{dettrivSO_lemm}
Par l'équivalence de la proposition \ref{formesmodulesquad_prop} entre $\faisorthO_{2n+1}$-torseurs et formes étales de $\hypq_{2n+1}$, les torseurs provenant de $\faisSO_{2n+1}$ correspondent aux modules quadratiques de module demi-déterminant trivial de parité de $n$. 
\end{lemm}
\begin{proof}
C'est une quasi-tautologie:
L'inclusion de $\faisSO_{2n+1}$ dans $\faisorthO_{2n+1}$ est induite par le foncteur fibré $\QuadDetTrivn{\hypq_{2n+1}}\to \Formes{\hypq_{2n+1}}$. Un objet  $(q,\phi)$, où $\phi$ trivialise le demi-déterminant de $q$, s'envoie sur $q$ qui est donc de demi-déterminant trivial. Réciproquement, si $q$ est de demi-déterminant trivial, c'est qu'il existe une trivialisation $\phi$, et donc $q$ provient de $(q,\phi)$. 
\end{proof}
Le théorème \ref{Bnadjointtordu_theo} découle alors de la proposition \ref{secSpinSO_prop}.
\end{proof}

\begin{rema} \label{qquelconque_rema}
On n'obtiendrait pas plus de groupes, à isomorphisme près, en tordant $\faisSO_{2n+1}$ par tous les torseurs sous $\faisorthO_{2n+1}$, plutôt que ceux provenant de $\faisSO_{2n+1}$. En effet, ces groupes seraient les $\faisSO_q$ avec $q$ forme étale quelconque de $\hypq_{2n+1}$. Or si $(q,M)$ n'a pas un module demi-déterminant trivial, on a tout de même l'isomorphisme $\qDemiDetMorph{} : \Det{M}^{\otimes 2} \simeq \faisO_S$, par construction du module demi-déterminant. Du coup, si on pose $M'=M \otimes \Det{M}$ et $q': M \otimes \Det{M} \to \faisO_S$ en envoyant $m \otimes l$ sur $q(m) \qDemiDetMorph{q}(l)$, on obtient un nouveau module quadratique $(q',M')$. De plus, $\Det{M'}\simeq \Det{M} \otimes \Det{M}^{\otimes 2n+1}$, donc $\Det{M'}\simeq \faisO_S$ par $\qDemiDetMorph{q}^{n+1}$, ce qui trivialise son module demi-déterminant. Par ailleurs, $\faisorthO_{q'}\simeq \faisorthO_q$, en envoyant $f$ sur $f\otimes \det(f)$. Cela induit un isomorphisme $\faisSO_{q'}\simeq \faisSO_q$. 
\end{rema}

Afin de pouvoir énoncer le théorème sur les formes du groupe simplement connexe, nous devons introduire un invariant supplémentaire. 

\begin{lemm} \label{invariantL_lemm}
Soit $(A,P,\psi)$ une $\faisO_S$-algèbre d'Azumaya $A$ munie d'une trivialisation de sa puissance $m$-ième $\psi:A^{\otimes m}\simeq \faisEnd_P$. Si de plus $A$ est triviale dans le groupe de Brauer, alors pour toute trivialisation $(N,\phi:A \simeq \faisEnd_N)$ de $A$, il existe un fibré en droites $L$ et un isomorphisme $\lambda:N^{\otimes m} \simeq P\otimes L$ tels que $\lambda$ soit un morphisme de $\faisEnd_P$-module à gauche lorsqu'on fait agir $\faisEnd_P$ sur $N^{\otimes m}$ par l'isomorphisme 
$$\faisEnd_{N^{\otimes m}}\simeq \faisEnd_N^{\otimes m} \isoby{\phi^{\otimes m}} A^{\otimes m} \isoby{\psi} \faisEnd_P.$$
De manière équivalente $\int_{\lambda}$ coïncide avec $\psi$ lorsqu'on identifie $\faisEnd_{N}^{\otimes m}\simeq \faisEnd_{N^{\otimes m}}$ et $\faisEnd_P \simeq \faisEnd_{P \otimes L}$ par les isomorphismes canoniques, et $A \simeq \faisEnd_N$ par $\phi$. En d'autres termes, la composition
$$A^{\otimes m}\isoby{\phi^{\otimes m}} \faisEnd_N^{\otimes m} \simeq \faisEnd_{N^{\otimes m}} \isoby{\int_\lambda} \faisEnd_{P\otimes L} \simeq \faisEnd_{P}$$
coïncide avec $\psi$.
De plus, la classe de $L$ dans $\Pic(S)/m$ ne dépend que de $(A,P,\psi)$ (et non pas du choix de $N$, $\phi$ ou $L$). 
\end{lemm}
\begin{proof}
On fait de $N^{\otimes m}$ un $\faisEnd_P$-module à gauche. Puis, $P^\dual$ étant un $\faisEnd_P$-module à droite, on considère $L=P^\dual \otimes_{\faisEnd_P} N^{\otimes m}$, dont on vérifie facilement qu'il est un fibré en droites localement lorsque tous les modules sont libres. L'équivalence de Morita $P \otimes P^\dual \simeq \faisEnd_P$ fournit alors l'isomorphisme $\lambda: N^{\otimes m} \simeq P \otimes P^\dual \otimes_{\faisEnd_P} N^{\otimes m} \simeq P \otimes L$. L'exercice de vérifier que les autres conditions sur $\lambda$ dans l'énoncé sont équivalentes est laissé au lecteur.  

Le morphisme $\phi$ et $N$ étant donnés, $(L,\lambda)$ est unique à isomorphisme canonique près puisqu'on peut l'obtenir en tensorisant $P \otimes L$ par $P^{\dual} \otimes_{\faisEnd_P}(-)$ et que $\lambda$ est un morphisme de $\faisEnd_P$-modules.
Si $(N,\phi)$ est changé en $(N',\phi')$, avec pour fibré associé $(L',\lambda')$, alors $\faisEnd_N \simeq \faisEnd_{N'}$ et par le même raisonnement, on montre qu'il existe $K$ fibré en droites tel que $N'\simeq K \otimes N$, d'où un isomorphisme de $\faisEnd_P$-modules à gauche
$$P \otimes L \otimes K^{\otimes m} \simeq N^{\otimes m} \otimes K^{\otimes m} \simeq (N')^{\otimes m} \simeq P \otimes L'$$
et donc $L\otimes K^{\otimes m} \simeq L'$.
\end{proof}

Partant maintenant d'un module quadratique $q$, forme \fppf\ de $\hypq_{2n+1}$. Son algèbre de Clifford paire est munie d'une $2$-trivialisation canonique induite par l'involution standard $\invstd_q$: on a $\faisCliff_{0,q}^{\otimes 2} \simeq \faisCliff_{0,q} \otimes \faisCliff_{0,q}^{\opp} \simeq \faisEnd_{\faisCliff_{0,q}}$. Supposons maintenant qu'elle est triviale dans le groupe de Brauer $\Brauer(S)$. Alors cette algèbre de Clifford est de la forme $\faisEnd_N$ pour un certain $\faisO_S$-module localement libre $N$, et on peut donc lui associer la classe d'un fibré en droites dans $\Pic(S)/2$ par le lemme précédent. Notons cette classe $l_q$. 

\begin{theo} \label{Bnsimplementconnexetordu_theo}
\`A isomorphisme près, les $S$-schémas en groupes réductifs de donnée radicielle déployée identique à celle de $\faisSpin_{2n+1}$ ($n\geq 1$), c'est-à-dire de type déployé semi-simple simplement connexe $B_n$, sont les groupes $\faisSpin_{q}$ où $q$ est un $\faisO_S$-module quadratique localement isomorphe à $\hypq_{2n+1}$ pour la topologie étale, et de module demi-déterminant trivial. Toutes les formes sont intérieures, et les formes strictement intérieures sont les groupes $\faisSpin_q$ avec $q$ d'algèbre de Clifford triviale dans le groupe de Brauer, et dont l'invariant $l_q\in \Pic(S)/2$ est trivial.
\end{theo}
\begin{proof}
Les formes étales de $\faisSpin_{2n+1}$ sont obtenues par torsion sous un torseur sous $\faisSO_{2n+1}=\faisAut_{\faisSpin_{2n+1}}$, situation qui a déjà été examinée pour la première ligne du diagramme \eqref{diagSpinSOimpair_diag}. Toutes les formes sont intérieures puisque $\faisSpin_{2n+1}/\faismu_2=\faisSO_{2n+1}$. Les formes strictement intérieures, donc tordues par un torseur sous $\faisSpin_{2n+1}$ lui-même), s'obtiennent en utilisant la proposition \ref{auttordus_prop} appliquée à la catégorie $\DeplCliffSn{2n+1}$. On trouve donc les groupes d'automorphismes des objets de cette catégorie, qui sont précisément, par construction de ce champ, les groupes $\faisSpin_q$ pour les formes $q$ satisfaisant aux hypothèses de l'énoncé.
\end{proof}

\begin{coro} \label{Bnadjfortementint_coro}
Les formes fortement intérieures de $\faisSO_{2n+1}$, le groupe déployé adjoint de type $B_n$, sont les groupes $\faisSO_q$ où $q$ est d'algèbre de Clifford triviale dans le groupe de Brauer et dont l'invariant $l_q \in Pic(S)/2$ est trivial.
\end{coro}

\begin{rema}
Il n'y a pas d'autre groupe semi-simple de type déployé semi-simple $B_n$ que ceux mentionnés dans les deux théorèmes de cette section, car il n'y a pas d'autre groupe déployé entre $\faisSpin$ et $\faisSO$, le centre du premier étant $\faismu_2$, qui est aussi petit que possible sans être trivial.
\end{rema}

\begin{rema}
De même que dans la remarque \ref{qquelconque_rema}, on n'obtiendrait pas plus de groupes en considérant des modules quadratiques quelconques, formes étales de $\hypq_{2n+1}$, car en définissant $q'$ de la même manière, on aurait $\faisCliff_{0,q'}\simeq \faisCliff_{0,q}$ et $\faisSpin_{q'}\simeq \faisSpin_{q}$. 
\end{rema}

\begin{rema}
On voit au passage, par la description des torseurs sous $\faisSGamma_{2n+1}$, que les groupes $\faisSpin_q$ avec $q$ d'algèbre de Clifford triviale mais pas nécessairement d'invariant $l_q$ trivial, sont ceux tordus par un torseur provenant de $\faisSGamma_{2n+1}$.
\end{rema}


\section{Groupes semi-simples de type $C_n$} \label{Cn_sec}

\subsection{Groupe déployé adjoint}

Soit $n\in \NN$ et $A$ une $\faisO_S$-algèbre d'Azumaya de degré $2n$ (définition \ref{degreAzumaya_defi}) munie d'une involution symplectique $\sigma$ (définition \ref{involutionsymplectique_defi}). 
\begin{defi} \label{PGSp_defi}
Le schéma en groupes \emph{projectif symplectique} $\faisPGSp_{A,\sigma}$ est $\faisAut^{\alginv}_{A,\sigma}$ considéré dans la définition \ref{Autalginv_defi} $\faisAut^{\alginv}_{A,\sigma}$ (représentable par un schéma affine sur $S$ par la proposition \ref{Autalginv_prop}).
\end{defi}
En particulier, soit la $\faisO_S$-algèbre $\faisM_{2n,S}$ et soit le module alterné $\hypa_{2n}$ (définition \ref{hyperbolique_defi}). Ce module induit une involution de première espèce symplectique $\sigma_h$ sur $\faisM_{2n,S}$, et on note $\faisPGSp_{2n,S}=\faisPGSp_{\faisM_{2n,S},\sigma_h}$.

Dans ce contexte, la proposition \ref{Autalginv_prop} nous dit que pour tout $S$-schéma $T$, on a
$$
\faisPGSp_{2n,T}=(\faisPGSp_{2n,S})_T=(\faisPGSp_{2n,\ZZ})_T
$$
et $\faisPGSp_{2n}(R)=\{ \alpha\in \setAut_{R\mathrm{-alg}}(\setM_{2n}(R))\vert \alpha\sigma=\sigma\alpha\}$ pour tout anneau $R$.

Utilisant la suite exacte
$$\xymatrix{0\ar[r] & \faisLie_{\faisPGSp_{2n}}(R)\ar[r] & \faisPGSp_{2n}(R[t])\ar[r]^-p & \faisPGSp_{2n}(R)\ar[r] & 1}$$
on voit facilement que l'algèbre de Lie de $\faisPGSp_{2n}$ est donnée pour tout anneau $R$ par le $R$-module libre 
$$\faisLie_{\faisPGSp_{2n}}(R):=\{\alpha\in M_{2n}(R)\vert \alpha +\sigma(\alpha)\in R\cdot Id\}/R\cdot Id.$$ 
\'Ecrivant $M=(M_{ij})$ pour $1\leq i,j\leq 2n$ avec $M_{ij}\in M_{2n}(R)$, on obtient les égalités 
$$
\left\{\begin{array}{l} M_{22}=\lambda Id- \transp M_{11} \\
 M_{12}= \transp M_{12} \\
 M_{21}= \transp M_{21}
\end{array}\right.
$$
Ceci donne une base de l'algèbre de Lie de la forme
$$
\left\{\begin{array}{rl} E_{ii} & \text{ pour tout $i=1\dots,n$}, \\
E_{ij}-E_{j+n,i+n} & \text{ pour tout $i,j=1,\ldots,n$ avec $i\neq j$}, \\
E_{i,n+j}+E_{j,n+i} & \text{ pour tout $1\leq i<j\leq n$},\\
E_{n+i,j}+E_{n+j,i} & \text{ pour tout $1\leq i<j\leq n$},\\
E_{i,n+i} & \text{ pour tout $1\leq i\leq n$},\\
E_{n+i,i} & \text{ pour tout $1\leq i\leq n$.}
\end{array}\right.
$$
et celle-ci est bien de dimension $n(2n+1)$. En conséquence:

\begin{prop}\label{PGSp_ss_prop}
Pour tout $n\geq 1$, le groupe $\faisPGSp_{2n}$ est semi-simple. 
\end{prop}

\begin{proof}
Pour tout point géométrique $s\in S$, le groupe $(\faisPGSp_{2n})_s$ est connexe (prop. \ref{Spconnexe_prop}) et de dimension $n(2n+1)$ par \cite[VI, \S23.4]{bookinv}. 
La lissité de $\faisPGSp_{2n}$ découle alors immédiatement des propositions \ref{sgcorpslisse_prop} et \ref{lissite_prop} sur $\Spec(\ZZ)$, puis sur toute autre base par changement de base. Le groupe est semi-simple sur les fibres géométriques par \cite[VI, théorème 25.11]{bookinv}. Il est par conséquent semi-simple.
\end{proof}

Pour trouver un tore maximal de $\faisPGSp_{2n}$, on considère le tore des matrices diagonales dans $\faisGL_{\faisM_{2n,\ZZ}}$ et on calcule son intersection avec $\faisPGSp_{2n}$.  Soit donc $N\in GL(M_{2n}(R))$ une matrice diagonale. On suppose que $N$ est donnée par $N(E_{ij})=\alpha_{ij}E_{ij}$ pour tout $1\leq i,j\leq 2n$. Exprimant les conditions pour que cette matrice soit dans $\faisPGSp_{2n}(R)$, on voit que $\alpha_{ij}=\alpha_{1i}^{-1}\cdot\alpha_{1j}$ pour tout $1\leq i,j\leq 2n$ et que $\alpha_{1,i+n}=\alpha_{1,n+1}\cdot\alpha_{1,i}^{-1}$. Il s'ensuit que l'automorphisme $N$ est uniquement déterminé par les $\alpha_{1,i}$ avec $i=2,\ldots, n+1$. On obtient ainsi un homomorphisme injectif de groupes $(R^\times)^n\to \faisPGSp_{2n}(R)$ en associant à $(\lambda_1,\ldots, \lambda_n)$ l'automorphisme de $M_{2n}(R)$ défini par $\alpha_{1j}=\lambda_{j-1}$ pour tout $j=2,\ldots n+1$ et les relations ci-dessus. On obtient ainsi un morphisme de groupes $\faisGm^n\to \faisPGSp_{2n}$ qui est une immersion fermée par \cite[Exp. IX, corollaire 2.5]{sga3}. On note $\faisPGSpdiag_{2n}$ l'image de $\faisGm^n$ dans $\faisPGSp_{2n}$. C'est évidemment un tore déployé, et on s'aperçoit qu'il est maximal puisque maximal sur les fibres géométriques. Les caractères de ce tore sont donnés par $t_i((\alpha_{ij}))=\alpha_{1,i+1}$ pour tout $i=1,\ldots,n$.

Un calcul immédiat donne les racines de $\faisPGSp_{2n}$ par rapport au tore $\faisPGSpdiag_{2n}$ pour tout anneau $R$:

\begin{itemize}
\item[1.] Le sous-module de l'algèbre de Lie du tore de poids trivial.
\item[2.] Le sous-module $R\cdot (E_{ij}-E_{j+n,i+n})$ de poids $t_{i-1}-t_{j-1}$ pour tout $2\leq i,j\leq n$ avec $i\neq j$.
\item[3.] Le sous-module $R\cdot (E_{i1}-E_{n+1,i})$ de poids $t_{i-1}$ pour tout $i=2,\ldots,n$.
\item[4.] Le sous-module $R\cdot (E_{1i}-E_{i+n,1})$ de poids $-t_{i-1}$ pour tout $i=2,\ldots,n$.
\item[5.] Le sous-module $R\cdot (E_{i,n+j}+E_{j,n+i})$ de poids $t_{i-1}+t_{j-1}-t_n$ pour tout $2\leq i<j\leq n$.
\item[6.] Le sous-module $R\cdot (E_{n+i,j}+E_{n+j,i})$ de poids $t_n-t_{i-1}-t_{j-1}$ pour tout $2\leq i<j\leq n$.
\item[7.] Le sous-module $R\cdot (E_{1,n+i}+E_{i,n+1})$ de poids $t_{i-1}-t_n$ pour tout $i=2,\ldots,n$.
\item[8.] Le sous-module $R\cdot (E_{n+1,i}+E_{n+i,1})$ de poids $t_n-t_{i-1}$ pour tout $i=2,\ldots,n$.
\item[9.] Le sous-module $R\cdot E_{i,n+i}$ de poids $2t_{i-1}-t_n$ pour tout $i=2,\ldots,n$.
\item[10.] Le sous-module $R\cdot E_{n+i,i}$ de poids $t_n-2t_{i-1}$ pour tout $i=2,\ldots,n$.
\item[11.] Le sous-module $R\cdot E_{n+1,1}$ de poids $t_n$.
\item[12.] Le sous-module $R\cdot E_{1,n+1}$ de poids $-t_n$.
\end{itemize}

Pour décrire les morphismes exponentiels $\exp:\faisGa\to \faisPGSp_{2n}$ on donne de fa\c con fonctorielle pour tout anneau $R$ et tout $\lambda\in R$ une matrice $A(\lambda)$ et on définit $\exp(\lambda)$ comme étant la conjugaison par $A(\lambda)$. On trouve:
\begin{itemize}
\item[2.] Le morphisme $\exp_{t_{i-1}-t_{j-1}}:R\cdot (E_{ij}-E_{j+n,i+n})\to \faisPGSp_{2n}(R)$ défini par 
$$\lambda\mapsto Id+\lambda(E_{ij}-E_{j+n,i+n}).$$
\item[3.] Le morphisme $\exp_{t_{i-1}}:R\cdot (E_{i1}-E_{n+1,i})\to \faisPGSp_{2n}(R)$ défini par 
$$\lambda\mapsto Id+\lambda (E_{i1}-E_{n+1,i}).$$
\item[4.] Le morphisme $\exp_{-t_{i-1}}:R\cdot (E_{1i}-E_{i+n,1})\to \faisPGSp_{2n}(R)$ défini par 
$$\lambda\mapsto Id+\lambda (E_{1i}-E_{i+n,1}).$$
\item[5.] Le morphisme $\exp_{t_{i-1}+t_{j-1}-t_n}: R\cdot (E_{i,n+j}+E_{j,n+i})\to \faisPGSp_{2n}(R)$ défini par
$$\lambda\mapsto Id+\lambda (E_{i,n+j}+E_{j,n+i}).$$
\item[6.] Le morphisme $\exp_{t_n-t_{i-1}-t_{j-1}}: R\cdot (E_{n+i,j}+E_{n+j,i})\to \faisPGSp_{2n}(R)$ défini par
$$\lambda\mapsto Id+\lambda (E_{n+i,j}+E_{n+j,i}).$$
\item[7.] Le morphisme $\exp_{t_{i-1}-t_n}: R\cdot (E_{1,n+i}+E_{i,n+1})\to \faisPGSp_{2n}(R)$ défini par
$$\lambda\mapsto Id+\lambda (E_{1,n+i}+E_{i,n+1}).$$
\item[8.] Le morphisme $\exp_{t_n-t_{i-1}}: R\cdot (E_{n+1,i}+E_{n+i,1})\to \faisPGSp_{2n}(R)$ défini par
$$\lambda\mapsto Id+\lambda (E_{n+1,i}+E_{n+i,1}).$$
\item[9.] Le morphisme $\exp_{2t_{i-1}-t_n}: R\cdot E_{i,n+i}\to \faisPGSp_{2n}(R)$ défini par
$$\lambda\mapsto Id+\lambda E_{i,n+i}.$$
\item[10.] Le morphisme $\exp_{t_n-2t_{i-1}}: R\cdot E_{n+i,i}\to \faisPGSp_{2n}(R)$ défini par
$$\lambda\mapsto Id+\lambda E_{n+i,i}.$$
\item[11.] Le morphisme $\exp_{t_n}: R\cdot E_{n+1,1}\to \faisPGSp_{2n}(R)$ défini par
$$\lambda\mapsto Id+\lambda E_{n+1,1}.$$
\item[12.] Le morphisme $\exp_{-t_n}: R\cdot E_{1,n+1}\to \faisPGSp_{2n}(R)$ défini par
$$\lambda\mapsto Id+\lambda E_{1,n+1}.$$
\end{itemize}

Ces données permettent de trouver immédiatement les coracines, selon la liste suivante:
\begin{itemize}
\item[2.] Les coracines $t_{i-1}^\dual-t_{j-1}^\dual$ pour tout $2\leq i,j\leq n$ avec $i\neq j$.
\item[3.] Les coracines $2t_n^\dual+2t_{i-1}^\dual+\sum_{j=2,j\neq i}^n t_{j-1}^\dual$ pour tout $i=2,\ldots,n$.
\item[4.] Les coracines $-2t_n^\dual-2t_{i-1}^\dual-\sum_{j=2,j\neq i}^n t_{j-1}^\dual$ pour tout $i=2,\ldots,n$.
\item[5.] Les coracines $t_{i-1}^\dual+t_{j-1}^\dual$ pour tout $2\leq i<j\leq n$.
\item[6.] Les coracines $-t_{i-1}^\dual-t_{j-1}^\dual$ pour tout $2\leq i<j\leq n$.
\item[7.] Les coracines $-2t_n^\dual-\sum_{j=2,j\neq i}^n t_{j-1}^\dual$ pour tout $i=2,\ldots,n$.
\item[8.] Les coracines $2t_n^\dual+\sum_{j=2,j\neq i}^n t_{j-1}^\dual$ pour tout $i=2,\ldots,n$.
\item[9.] Les coracines $t_{i-1}^\dual$ pour tout $i=2,\ldots,n$.
\item[10.] Les coracines $-t_{i-1}^\dual$ pour tout $i=2,\ldots,n$.
\item[11.] La coracine $2t_n^\dual+\sum_{j=1}^{n-1}t_j^\dual$.
\item[12.] La coracine $-2t_n^\dual-\sum_{j=1}^{n-1}t_j^\dual$.
\end{itemize}

Le résultat suivant est donc démontré:

\begin{prop} \label{donneeradiciellePGSp_prop}
La donnée radicielle de $\faisPGSp_{2n}$ relativement au tore maximal déployé $\faisPGSpdiag_{2n}$ est 
\begin{itemize}
\item Le $\ZZ$-module libre $N\simeq \ZZ^n$ des caractères de $\faisPGSp_{2n}$ (dont $t_1,\ldots,t_n$ forment une base); 
\item Le sous-ensemble fini de celui-ci formé des racines $\pm(t_{i-1}-t_{j-1})$ et des racines $\pm(t_{i-1}+t_{j-1}-t_n)$ pour $2\leq i<j\leq n$, ainsi que $\pm(t_{i-1}-t_n)$ et $\pm(2t_{i-1}-t_n)$ pour $i=2,\ldots,n$ et $\pm t_{i}$ pour $i=1,\ldots n$;
\item Le $\ZZ$-module libre $N^\dual\simeq\ZZ^n$ des cocaractères (de base duale $t_1^\cdual,\ldots,t_n^\cdual$);
\item Le sous-ensemble fini de celui-ci formé des coracines $\pm t_{i-1}^\dual\pm t_{j-1}^\dual$ pour $2\leq j<i\leq n$, ainsi que pour tout $i=2,\ldots,n$ les coracines $\pm t_{i-1}^\dual$, $\pm(2t_n^\dual+2t_{i-1}^\dual+\sum_{j=2,j\neq i}^n t_{j-1}^\dual)$ et $\pm(2t_n^\dual+\sum_{j=2,j\neq i}^n t_{j-1}^\dual)$, et finalement les coracines $\pm(2t_n^\dual+\sum_{j=1}^{n-1}t_j^\dual)$.
\end{itemize}
\end{prop}

\begin{theo}
Pour tout $n\geq 1$, le groupe $\faisPGSp_{2n}$ est déployé, semi-simple et adjoint de type $C_n$.
\end{theo}

\begin{proof}
Le groupe est semi-simple par la proposition \ref{PGSp_ss_prop} et déployé. Il est adjoint puisque ses racines engendrent les caractères du tore. Enfin, le système de racines simples $\{t_1-t_2,t_2-t_3,\ldots,t_{n-1}-t_n,t_n\}$ montre que le diagramme de Dynkin associé est de type $C_n$.
\end{proof}


\subsection{Groupe déployé simplement connexe}

Soit $n\in\NN$ et $A$ une $\faisO_S$-algèbre d'Azumaya de degré $2n$ munie d'une involution symplectique $\sigma$. On considère la représentation
\[
\rho:\faisGL_{1,A}\to \faisGL_A
\]
donnée par $\rho(a)(m)=\sigma(a)\cdot m\cdot a$ pour tout $T\to S$, tout $a\in A(T)^\times$ et tout $m\in A(T)$. 
\begin{defi}\label{SpsurZ_defi}
Le schéma en groupes \emph{symplectique} $\faisSp_{A,\sigma}$ est défini comme la préimage sous $\rho$ du stabilisateur de $Id$ dans $\faisGL_A$. Le lemme \ref{produit_fibre_lemm} et la proposition \ref{stab_prop} montrent que $\faisSp_{A,\sigma}$ est représentable et qu'il est un sous-groupe fermé de $\faisGL_{1,A}$, donc affine et de type fini sur $S$.
\end{defi}
En particulier, on peut considérer la paire $(\faisM_{2n,S},\sigma_h)$ et on note $\faisSp_{2n,S}$ est le groupe $\faisSp_{\faisM_{2n,S},\sigma_h}$. 
De manière analogue à la section précédente, on voit que pour tout schéma $T$
\[
\faisSp_{2n,T}=(\faisSp_{2n,S})_T=(\faisSp_{2n,\ZZ})_T.
\]
Pour tout anneau $R$, on obtient aussi la description des $R$-points de $\faisSp_{2n}$ via la formule
$$\faisSp_{2n}(R)=\{ M\in \setGL_{2n}(R)\phantom{i}\vert \phantom{i} \transp MHM=H\}$$
où $H$ est la matrice de $a^h_{2n}$ dans la base canonique de $R^{n}\oplus (R^n)^\vee$. 
\medskip

On détermine maintenant son algèbre de Lie comme sous-algèbre de Lie de $\faisLie_{\faisGL_{\faisM_{2n,\ZZ}}}=\alL_{\faisM_{2n,\ZZ}}$ (voir prop. \ref{LieGLn_prop}) en accord avec le lemme \ref{sous_algebre_lemm}. La suite exacte
$$\xymatrix{1\ar[r] & \faisLie_{\faisSp_{2n}}(R)\ar[r] & \faisSp_{2n}(R[t])\ar[r] & \faisSp_{2n}(R)\ar[r] & 1}$$
montre que $1+t M\in \faisLie_{\faisSp_{2n}}(R)$ si et seulement si $\transp MH+HM=0$. \'Ecrivant $M=(M_{ij})$ pour $1\leq i,j\leq 2$ avec $M_{ij}\in M_n(R)$, on obtient les égalités 
$$
\left\{\begin{array}{l} M_{22}=-\transp M_{11} \\
M_{12}=\transp M_{12} \\
M_{21}=\transp M_{21}
\end{array}\right.
$$
Il s'ensuit que $\faisLie_{\faisSp_{2n}}(R)$ est libre de dimension $n^2+n(n+1)/2+n(n+1)/2=n(2n+1)$ sur $R$ avec une base donnée par les matrices
$$
\left\{\begin{array}{rl} E_{ij}-E_{j+n,i+n} & \text{ pour tout $1\leq i,j\leq n$}, \\
E_{i,n+j}+E_{j,n+i} & \text{ pour tout $1\leq i<j\leq n$},\\
E_{n+i,j}+E_{n+j,i} & \text{ pour tout $1\leq i<j\leq n$},\\
E_{i,n+i} & \text{ pour tout $1\leq i\leq n$},\\
E_{n+i,i} & \text{ pour tout $1\leq i\leq n$.}
\end{array}\right.
$$
En conséquence:

\begin{prop}\label{Sp_ss_prop}
Le groupe $\faisSp_{2n}$ est semi-simple sur $\Spec\ZZ$.
\end{prop}

\begin{proof}
Pour tout point géométrique $s\in\Spec\ZZ$, le groupe $(\faisSp_{2n})_s$ est connexe de dimension $n\cdot (2n+1)$, et (semi-)simple par \cite[\S 23.3]{borelag}. 
Puisque l'algèbre de Lie sur $s$ est de même dimension, les propositions \ref{sgcorpslisse_prop} et \ref{lissite_prop} montrent que le groupe est lisse, et donc réductif et semi-simple au sens de \SGAtrois.
\end{proof}

Pour trouver un tore maximal de $\faisSp_{2n}$, on considère le tore $\faisSpdiag_{2n}$ donné par les matrices de la forme $diag(\alpha_1,\ldots,\alpha_n,\alpha_1^{-1},\ldots,\alpha_n^{-1})$. C'est un tore maximal puisqu'il est maximal sur les fibres géométriques (\cite[\S 23.3]{borelag}). Les caractères du tore sont donnés par $t_i(diag(\alpha_1,\ldots,\alpha_n,\alpha_1^{-1},\ldots,\alpha_n^{-1}))=\alpha_i$ pour tout $1\leq i\leq n$. Le fait que $\faisSp_{2n}$ soit un sous-groupe de $\faisGL_{2n}$ et que $\faisSpdiag_{2n}$ soit un sous tore du tore maximal des matrices diagonales dans $\faisGL_{2n}$ donne facilement les racines de $\faisSp_{2n}$. On trouve pour tout anneau $R$ la liste suivante:

\begin{itemize}
\item[1.] Le sous-module de l'algèbre de Lie du tore de poids trivial.
\item[2.] Le sous-module $R\cdot (E_{ij}-E_{j+n,i+n})$ de poids $t_i-t_j$ pour tout $1\leq i,j\leq n$ avec $i\neq j$.
\item[3.] Le sous-module $R\cdot (E_{i,n+j}+E_{j,n+i})$ de poids $t_i+t_j$ pour tout $1\leq i<j\leq n$.
\item[4.] Le sous-module $R\cdot (E_{n+i,j}+E_{n+j,i})$ de poids $-t_i-t_j$ pour tout $1\leq i<j\leq n$.
\item[5.] Le sous-module $R\cdot E_{i,n+i}$ de poids $2t_i$ pour tout $1\leq i\leq n$.
\item[6.] Le sous-module $R\cdot E_{n+i,i}$ de poids $-2t_i$ pour tout $1\leq i\leq n$.
\end{itemize}
Pour toute racine $\alpha$, l'unique morphisme de groupes algébriques
$$\exp_\alpha:\faisGa\to \faisSp_{2n}$$
induisant l'inclusion canonique de l'espace propre associé à $\alpha$ dans son algèbre de Lie (\cite[Exp. XXII, théorème 1.1]{sga3}) est comme ci-dessous:
\begin{itemize}
\item[1.] Le morphisme $\exp_{t_i-t_j}:R\cdot (E_{ij}-E_{j+n,i+n})\to \faisSp_{2n}(R)$ défini par 
$$\lambda\mapsto Id+\lambda(E_{ij}-E_{j+n,i+n}).$$
\item[2.] Le morphisme $\exp_{t_i+t_j}:R\cdot (E_{i,n+j}+E_{j,n+i})\to \faisSp_{2n}(R)$ défini par 
$$\lambda\mapsto Id+\lambda(E_{i,n+j}+E_{j,n+i}).$$
\item[3.] Le morphisme $\exp_{-t_i-t_j}: R\cdot (E_{n+i,j}+E_{n+j,i})\to \faisSp_{2n}(R)$ défini par
$$\lambda\mapsto Id+\lambda(E_{n+i,j}+E_{n+j,i}).$$
\item[4.] Le morphisme $\exp_{2t_i}: R\cdot E_{i,n+i}\to \faisSp_{2n}(R)$ défini par
$$\lambda\mapsto Id+\lambda E_{i,n+i}.$$
\item[5.] Le morphisme $\exp_{-2t_i}:R\cdot E_{n+i,i}\to \faisSp_{2n}(R)$ défini par
$$\lambda\mapsto Id+\lambda E_{n+i,i}.$$
\end{itemize}

Ces données nous permettent de calculer facilement les coracines. Dans tous les cas ci-dessous, l'accouplement de la remarque \ref{signeAccouplement_rema} est donné par $(X,Y)\mapsto -XY$. Les coracines associées aux racines décrites ci-dessus sont dans l'ordre:

\begin{itemize}
\item[2.] Les coracines $t_i^\dual-t_j^\dual$ pour tout $1\leq i,j\leq n$ avec $i\neq j$.
\item[3.] Les coracines $t_i^\dual+t_j^\dual$ pour tout $1\leq i<j\leq n$.
\item[4.] Les coracines $-t_i^\dual-t_j^\dual$ pour tout $1\leq i<j\leq n$.
\item[5.] Les coracines $t_i^\dual$ pour tout $1\leq i\leq n$.
\item[6.] Les coracines $-t_i^\dual$ pour tout $1\leq i\leq n$.
\end{itemize}

En résumé la donnée radicielle de $\faisSp_{2n}$ est la suivante:

\begin{prop} \label{donneeradicielleSp_prop}
La donnée radicielle de $\faisSp_{2n}$ relativement au tore maximal déployé $\faisSpdiag_{2n}$ est 
\begin{itemize}
\item Le $\ZZ$-module libre $N\simeq \ZZ^n$ des caractères de $\faisSp_{2n}$ (dont $t_1,\ldots,t_n$ forment une base); 
\item Le sous-ensemble fini de celui-ci formé des racines $\pm2t_i$ pour $1\leq i\leq n$ et $\pm t_i\pm t_j$ pour $1\leq j<i\leq n$;
\item Le $\ZZ$-module libre $N^\dual\simeq\ZZ^n$ des cocaractères (de base duale $t_1^\cdual,\ldots,t_n^\cdual$);
\item Le sous-ensemble fini de celui-ci formé des coracines $\pm t_i^\dual$ pour $i=1\leq i\leq n$ et $\pm t_i^\dual\pm t_j^\dual$ pour $1\leq j<i\leq n$.
\end{itemize}
\end{prop}

Le résultat suivant découle immédiatement de cette proposition:

\begin{theo}
Pour tout $n\geq 1$, le groupe $\faisSp_{2n}$ est déployé, semi-simple et simplement connexe de type $C_n$. 
\end{theo}

\begin{proof}
La proposition \ref{Sp_ss_prop} montre déjà que le groupe est semi-simple. Il est simplement connexe puisque la proposition ci-dessus montre que les coracines engendrent les cocaractères. Il est de type $C_n$ puisque les racines simples $\{t_1-t_2,t_2-t_3,\ldots,t_{n-1}-t_n,2t_n\}$ donnent un diagramme de Dynkin de ce type.
\end{proof}

Pour calculer le centre de $\faisSp_{2n}$, on définit un morphisme de groupes $\chi:\faisSp_{2n}\to \faisPGSp_{2n}$ en associant à $M\in \faisSp_{2n}(R)$ l'automorphisme intérieur $\int_M$ de $M_{2n}(R)$ pour tout anneau $R$. Ce morphisme induit un morphisme $\chi:\faisSpdiag_{2n}\to \faisPGSpdiag_{2n}$ dont le noyau est le centre de $\faisSp_{2n}$. Explicitement, on trouve
$$\chi(diag(\alpha_1,\ldots,\alpha_n,\alpha_1^{-1},\ldots,\alpha_n^{-1}))=(\alpha_{ij})$$
avec $\alpha_{1i}=\alpha_1\cdot\alpha_i^{-1}$ pour $i=2,\ldots,n$ et $\alpha_{1,n+1}=\alpha_1^2$. Au niveau des caractères des tores $\faisSpdiag_{2n}$ et $\faisPGSpdiag_{2n}$, on obtient $t_i\mapsto t_1-t_{i+1}$ pour $i=1,\ldots,n-1$ et $t_n\mapsto 2t_1$. On en déduit que le centre de $\faisSp_{2n}$ est le groupe multiplicatif associé au groupe $\ZZ/2$. Ainsi:

\begin{prop} \label{centreSp_prop}
Le morphisme de groupes algébriques
$$\zeta:\faismu_2\to Z(\faisSp_{2n})$$
défini pour tout anneau $R$ et tout $\alpha\in \faismu_2(R)$ par $\zeta(\alpha)=diag(\alpha,\ldots,\alpha)$ est un isomorphisme.
\end{prop}

Identifions maintenant les torseurs sous $\faisSp_{2n}$. Pour ce faire, on considère la $S$-catégorie fibrée $\Sympl{2n}$ dont la fibre en $T$ est la catégorie dont les objets sont les $\faisO_T$-modules alternés (non dégénérés) de rang $2n$ et les morphismes sont les isométries. Il est facile de vérifier que $\Sympl{2n}$ est un champ en groupoïdes pour la topologie étale ou bien \fppf. 

\begin{prop}\label{Sp2ntors_prop}
Le foncteur $(M,b)\mapsto \faisIso_{(\faisO_s^{2n},a^h_{2n}),(M,b)}$ définit une équivalence de catégories fibrées 
\[
\Sympl{2n} \isoto \Tors{\faisSp_{2n}}.
\]
du champ en groupoïdes des modules alternés (non dégénérés) de rang $2n$ vers le champ en groupoïdes des $\faisSp_{2n}$-torseurs pour la topologie étale ou \fppf. 
\end{prop}

\begin{proof}
Par la proposition \ref{tordusformes_prop}, il suffit de comprendre les formes Zariski ou étale du module alterné $(\faisO_S^{2n},a^h_{2n})$. Faisons-le dans le cas Zariski. Soit donc $(M,b)$ un module alterné de rang $2n$. La proposition \ref{loclibre_prop} montre que $M$ est une forme Zariski de $\faisO_S^{2n}$ et il suffit de voir que Zariski localement $b$ est isométrique à $a^h_{2n}$.  Pour ce faire, on suppose que $S=\Spec(R)$ est affine et on remarque que tout sous-module libre de rang $1$ d'un module symplectique $(M,b)$ est forcément inclus dans son orthogonal puisque une forme alternée sur un module de rang $1$ est triviale. On a donc une décomposition $(M,b)=(M^\prime,b^\prime)\perp a^h_{2}$ avec $(M^\prime,b^\prime)$ alterné. Quitte à localiser $R$, on peut à nouveau supposer que $M^\prime$ a un sous-facteur libre de rang $1$. Par induction, on voit que tout module alterné est localement de la forme $a^h_2\perp\ldots \perp a^h_2$. On conclut en remarquant que les modules alternés $a^h_2\perp\ldots \perp a^h_2$ et $a^h_{2n}$ sont isométriques. 
\end{proof}


\subsection{Automorphismes}

Le diagramme de Dynkin $C_n$ n'a pas d'automorphisme non trivial. Il s'ensuit que tout automorphisme est intérieur et par conséquent
\[
\faisAut^\gr_{\faisSp_{2n}} =\faisAut^\gr_{\faisPGSp_{2n}}= \faisPGSp_{2n}.
\]

\subsection{Groupes tordus}

\begin{theo}\label{Cnadjointtordu_theo}
Les $S$-schémas en groupes de type déployé semi-simple adjoint $C_n$ ($n\geq 1$) sont à isomorphisme près les groupes $\faisPGSp_{A,\sigma}$ où $A$ est une $\faisO_S$-algèbre d'Azumaya de degré $2n$ et $\sigma$ est une involution de première espèce symplectique sur $A$. Toutes les formes sont intérieures (et donc strictement intérieures puisque le groupe est adjoint). Pour les formes fortement intérieures, voir le corollaire \ref{Cnadjfortmentint_coro}.
\end{theo}

\begin{proof}
Au vu de la proposition \ref{auttordus_prop}, on sait que tout torseur $P$ sous $\faisAut^{\alginv}_{\faisM_{2n},\sigma_h}=\faisPGSp_{2n}$ est de la forme $\faisAut_{P\contr {\faisPGSp_{2n}} (\faisM_{2n},\sigma_h)}$. Il suffit donc de démontrer que les formes étales de $(\faisM_{2n},\sigma_h)$ sont précisément les Algèbres d'Azumaya de degré $2n$ munies d'une involution symplectique $\sigma$. Soit donc une telle paire $(A,\sigma)$. Par la proposition \ref{invlocale_prop}, il suffit de démontrer que tout module symplectique $(M,b)$ est étale localement de la forme $(\faisO_S^{2n},a^h_{2n})$. C'est précisément la proposition \ref{Sp2ntors_prop}.
\end{proof}

\begin{theo}\label{Cnsimplementconnexe_theo}
Les $S$-schémas en groupes réductifs de donnée radicielle déployée identique à $\faisSp_{2n}$ (pour $n\geq 1$) sont les groupes $\faisSp_{A,\sigma}$ où $A$ est une $\faisO_S$-algèbre d'Azumaya de degré $2n$ et $\sigma$ est une involution symplectique. Toutes les formes sont intérieures et les formes strictement intérieures sont les groupes $\faisSp_b=\faisSp_{\faisEnd_M,\sigma_b}$ où $M$ est un $\faisO_S$-module localement libre de rang $2n$ muni d'un isomorphisme alterné $b$ et $\sigma_b$ est l'involution associée à $b$.
\end{theo}

\begin{proof}
Les formes étales de $\faisSp_{2n}$ sont obtenues par torsion sous un $\faisPGSp_{2n}=\faisAut_{\faisSp_{2n}}$-torseur. Argumentant comme dans la proposition \ref{auttordus_prop}, on s'aperçoit que le morphisme naturel
\[
P\contr{\faisPGSp_{2n}}\faisSp_{2n}\to \faisSp_{P\contr{\faisPGSp_{2n}} (\faisM_{2n},\sigma_h)}
\]
est un isomorphisme pour tout $\faisPGSp_{2n}$-torseur $P$. La preuve du théorème ci-dessus identifie les algèbres à involution $P\contr{\faisPGSp_{2n}} (\faisM_{2n},\sigma_h)$ avec les algèbres $(A,\sigma)$ de type voulu. De plus, toutes le formes sont intérieures puisque $\faisSp_{2n}/\faismu_2=\faisPGSp_{2n}$. 

Pour les formes strictement intérieures, on utilise la proposition \ref{Sp2ntors_prop} qui identifie les torseurs sous $\faisSp_{2n}$ aux modules alternés $(M,b)$. Lorsqu'on tord $\faisSp_{2n}$ par un tel torseur, on obtient bien un groupe de la forme $\faisSp_{\faisEnd_M,\sigma_b}$. 
\end{proof}

\begin{coro} \label{Cnadjfortmentint_coro}
Les formes fortement intérieures de $\faisPGSp_{2n}$, le groupe adjoint déployé de type $C_n$, sont les $\faisPGSp_b=\faisPGSp_{\faisEnd_M,\sigma_b}$ où $M$ est un $\faisO_S$-module localement libre de rang $2n$ muni d'un isomorphisme alterné $b$, d'involution symplectique associée $\sigma_b$. 
\end{coro}


\section{Groupes semi-simples de type $D_n$} \label{Dn_sec}

\subsection{Groupe déployé adjoint}\label{dn_adjoint}

Soit $(A,\sigma,f)$ une paire quadratique au sens de la définition \ref{pairesquad_defi}. Rappelons que le groupe des automorphismes $\faisPGO_{(A,\sigma,f)}$ d'une telle paire est représentable par un schéma affine sur $S$ par la proposition \ref{PGOrepresentable_prop}. On dispose de plus de l'invariant de Arf
\[
\faisPGO_{(A,\sigma,f)}\to \ZZ/2
\]
dont le noyau est le groupe $\faisPGOplus_{(A,\sigma,f)}$ qui est de ce fait représentable (définition \ref{PGOplus_defi}). On dispose enfin d'une suite exacte de faisceau étales (proposition \ref{secPGOplus_prop})
\[
1\to \faisPGOplus_{(A,\sigma,f)}\to \faisPGO_{(A,\sigma,f)}\to \ZZ/2\to 0
\]
En particulier, pour la paire quadratique $(\faisM_{2n},\invadj_{2n},f_{2n})$ définie avant le corollaire \ref{pairequadformes_coro}, ces schémas en groupes sont notés $\faisPGO_{2n}$ et $\faisPGOplus_{2n}$.
\medskip

A partir de maintenant, nous supposerons $n\geq 2$.

Pour calculer l'algèbre de Lie de $\faisPGOplus_{2n}$, on prend le chemin des écoliers. Rappelons tout d'abord qu'on dispose pour tout schéma $S$, toute $\faisO_S$-algèbre d'Azumaya $A$ et toute involution orthogonale $\sigma$ d'un $\faisO_S$-module localement libre $\faisSym_{A,\sigma}$ et d'un groupe $\faisGO_{(A,\sigma,f)}$ pour toute paire quadratique $(A,\sigma,f)$. En particulier, on obtient ainsi un groupe 
\[
\faisGO_{2n}:=\faisGO_{(\faisM_{2n},\invadj_{2n},f_{2n})}
\]
qui est représentable sur toute base $S$ et donc également sur $\Spec(\ZZ)$.

Fixons maintenant une base $(e_1,\ldots,e_{2n})$ de $\faisO_S^{2n}$ sur laquelle le module hyperbolique $\hypq_{2n}$ est donné par
$$\hypq_{2n}(x_1,\ldots,x_{2n}) = x_1 x_2 +\cdots +x_{2n-1} x_{2n}.$$
\'Ecrivant $\int_A$ pour la conjugaison par une matrice $A$, une description explicite de $\faisGO_{2n}$ est donnée sur tout anneau $R$ par
$$\faisGO_{2n}(R)=\{ A\in GL_{2n}(R)\text{ t.q. $\eta_{2n}\circ\int_A=\int_A\circ\eta_{2n}$ et $f_{2n}\circ \int_A=f_{2n}$}\}.$$
La condition $\eta_{2n}\circ\int_A=\int_A\circ\eta_{2n}$ s'exprime par 
$$H^{-1}\transp (A^{-1}BA)H=A^{-1}H^{-1}\transp B HA$$ 
pour tout $B\in M_{2n}(R)$, où $H=\sum_1^n (E_{i,i+1}+E_{i+1,i})$ est symétrique. Le centre de $M_{2n}(R)$ étant constitué des matrices scalaires, on en déduit que $\eta_{2n}\circ\int_A=\int_A\circ\eta_{2n}$ si et seulement si $\transp (A^{-1})HA^{-1}H^{-1}$ est une matrice scalaire. Ceci nous permet de calculer l'algèbre de Lie $\faisLie_{\faisGO_{2n}}$. Utilisant la suite exacte
$$\xymatrix@C=1.5em{0\ar[r] & \faisLie_{\faisGO_{2n}}(R)\ar[r] & \faisGO_{2n}(R[t])\ar[r]^-p & \faisGO_{2n}(R)\ar[r] & 1,}$$
on voit que pour tout $B\in M_{2n}(R)$ on a $Id+t B\in \faisLie_{\faisGO_{2n}}(R)$ si et seulement si $B+\eta_{2n}(B)\in R\cdot Id$ et $f_{2n}(CB-BC)=0$ pour tout élément symétrique $C\in \mathrm{Sym}(M_{2n}(R),\eta_{2n})$. 
La première condition donne pour $B=(b_{ij})$ les relations
\begin{itemize}
\item[1.] $b_{2j,2i}=-b_{2i-1,2j-1}$ pour tout $1\leq i,j\leq n$ avec $i\neq j$.
\item[2.] $b_{2j-1,2i}=-b_{2i,2j-1}$ pour tout $1\leq i,j\leq n$.
\item[3.] $b_{2j,2i-1}=-b_{2i,2j-1}$ pour tout $1\leq i,j\leq n$.
\item[4.] $b_{2i-1,2i-1}+b_{2i,2i}=b_{11}+b_{22}$ pour tout $2\leq i\leq n$.
\end{itemize}

Pour expliciter un peu ces relations, on donne une base de $\mathrm{Sym}(M_{2n}(R),\eta_{2n})$. Utilisant la description explicite de la matrice $H$, on obtient la base suivante (comme sous-module de $\mathrm{M}_{2n}(R)$)
\begin{itemize}
\item[1.] $E_{2i-1,2i}$ pour tout $i=1,\ldots,n$.
\item[2.] $E_{2i,2i-1}$ pour tout $i=1,\ldots,n$.
\item[3.] $E_{2i-1,2j}+E_{2j-1,2i}$ pour tout $1\leq i<j\leq n$.
\item[4.] $E_{2i-1,2j-1}+E_{2j,2i}$ pour tout $1\leq i\leq j\leq n$.
\item[5.] $E_{2i,2j}+E_{2j-1,2i-1}$ pour tout $1\leq i<j\leq n$.
\item[6.] $E_{2i,2j-1}+E_{2j,2i-1}$ pour tout $1\leq i<j\leq n$.
\end{itemize} 

On voit que la condition $b_{2j-1,2i}=-b_{2i,2j-1}$ pour tout $1\leq i,j\leq n$ ci-dessus donne les relations supplémentaires
\begin{itemize}
\item[5.] $b_{2i,2i-1}=0$ pour tout $1\leq i\leq n$.
\item[6.] $b_{2i-1,2i}=0$ pour tout $1\leq i\leq n$.
\end{itemize}

Ceci montre que $\faisLie_{\faisGO_{2n}}(R)$ est le $R$-module libre dont une base est donnée par les éléments suivants:
\begin{itemize}
\item[1.] $E_{2i-1,2j-1}-E_{2j,2i}$ pour tout $1\leq i,j\leq n$ avec $i\neq j$.
\item[2.] $E_{2i-1,2j}-E_{2j-1,2i}$ pour tout $1\leq i<j\leq n$.
\item[3.] $E_{2i,2j-1}-E_{2j,2i-1}$ pour tout $1\leq i<j\leq n$.
\item[4.] $E_{2i-1,2i-1}-E_{2i,2i}$ pour tout $2\leq i\leq n$.
\item[5.] $E_{11}+\sum_{i=2}^n E_{2i,2i}$.
\item[6.] $E_{22}+\sum_{i=2}^n E_{2i,2i}$.
\end{itemize}

Par ailleurs, la suite exacte de groupes
\[
1\to \faisGm\to  \faisGO_{2n}\to  \faisPGO_{2n}\to  1
\]
de la proposition \ref{secGOPGO_prop} montre que $\faisLie_{\faisPGO_{2n}}$ est le quotient de $\faisLie_{\faisGO_{2n}}$ par $\faisLie_{\faisGm}=\faisGa$. On trouve ainsi une base de $\faisLie_{\faisPGO_{2n}}(R)$ pour tout anneau $R$ formée des éléments (seul le point 4. diffère de la liste ci-dessus):
\begin{itemize}
\item[1.] $E_{2i-1,2j-1}-E_{2j,2i}$ pour tout $1\leq i,j\leq n$ avec $i\neq j$.
\item[2.] $E_{2i-1,2j}-E_{2j-1,2i}$ pour tout $1\leq i<j\leq n$.
\item[3.] $E_{2i,2j-1}-E_{2j,2i-1}$ pour tout $1\leq i<j\leq n$.
\item[4.] $E_{2i-1,2i-1}-E_{2i,2i}$ pour tout $2\leq i\leq n-1$.
\item[5.] $E_{11}+\sum_{i=2}^n E_{2i,2i}$.
\item[6.] $E_{22}+\sum_{i=2}^n E_{2i,2i}$.
\end{itemize}
Elle est donc de dimension $n(2n-1)$.
Par ailleurs, $\faisPGOplus_{2n}$ a la même algèbre de Lie que $\faisPGO_{2n}$ comme le montre la proposition \ref{secPGOplus_prop}. On en déduit immédiatement le résultat suivant:

\begin{prop}\label{PGOplus_ss_prop}
Le groupe $\faisPGOplus_{2n}$ est semi-simple. 
\end{prop}
\begin{proof}
Par la proposition \ref{dimsec_prop}, et en utilisant la suite exacte de la proposition \ref{secOPGO_prop}, on a que sur tout corps $k$, $\faisPGOplus_{2n}$ est bien de dimension identique à celle de $\faisOplus_{2n}$, c'est-à-dire $n(2n-1)$. Puisque c'est bien la dimension de l'algèbre de Lie sur $k$ et qu'il est à fibres connexes par la proposition \ref{PGOplusconnexe_prop}, il suit des propositions \ref{sgcorpslisse_prop} et \ref{lissite_prop} que le groupe $\faisPGOplus$ sur $\ZZ$ est lisse. Par changement de base, il le reste sur tout autre base $S$. Ses fibres géométriques sont simples par \cite[Theorem 25.12, p. 359]{bookinv}. C'est donc bien un groupe algébrique réductif semi-simple au sens de \SGAtrois.
\end{proof}

Pour trouver un tore maximal de $\faisPGOplus_{2n}$, on remarque d'abord qu'un tore est connexe et qu'il suffit par conséquent de trouver un tore maximal de $\faisPGO_{2n}$ pour arriver à nos fins. Pour ce faire, on utilise à nouveau la suite exacte 
\[
1\to \faisGm\to  \faisGO_{2n}\to  \faisPGO_{2n}\to  1.
\]
Si $T$ est un tore maximal de $\faisGO_{2n}$ contenant le centre $\faisGm$ de $\faisGO_{2n}$, alors le quotient est un tore maximal de $\faisPGO_{2n}$. On considère le morphisme
$$\tau:\faisGm^{n+1}\to \faisGO_{2n}$$
défini sur les points en envoyant $(\alpha_0,\alpha_1,\ldots,\alpha_n)$ sur 
$$\tau(\alpha_0,\alpha_1,\ldots,\alpha_n)=\mathrm{diag}(\alpha_0\alpha_1,\alpha_1^{-1},\alpha_0\alpha_2,\alpha_2^{-1},\ldots,\alpha_0\alpha_n,\alpha_n^{-1})$$
(matrice diagonale dans $\faisGL_{2n}$, dont on vérifie qu'elle est bien dans $\faisGO_{2n}$).
Il est clair que $\tau$ est un morphisme de groupes et une immersion fermée par \cite[Exp. IX, corollaire 2.5]{sga3}. On note $\faisGOdiag_{2n}$ ce tore. 

\begin{lemm}\label{tore_maximal_dn_adjoint_lemm}
Le tore $\faisGOdiag_{2n}$ est maximal. 
\end{lemm}

\begin{proof}
La proposition \ref{secOGO_prop} montre qu'on dispose d'une suite exacte
$$\xymatrix{1\ar[r] & \faisO_{2n}\ar[r] & \faisGO_{2n}\ar[r]^-\eta & \faisGm\ar[r] & 1}$$
qui induit une suite exacte de tores
$$\xymatrix{1\ar[r] & N\ar[r] & \faisGOdiag_{2n}\ar[r]^-\eta & \faisGm\ar[r] & 1,}$$
où $N$ est le tore des matrices de la forme $\mathrm{diag}(\alpha_1,\alpha_1^{-1},\alpha_2,\alpha_2^{-1},\ldots,\alpha_n,\alpha_n^{-1})$. Ce tore $N$ étant maximal sur toute fibre géométrique par \cite[\S 23.4 et 23.6]{borelag}, on voit que $\faisGOdiag_{2n}$ est lui-aussi maximal.
\end{proof}

On trouve finalement un tore maximal de $\faisPGOplus_{2n}$ en considérant le quotient $\faisPGOplusdiag_{2n}$ de $\faisGOdiag_{2n}$ par le centre $\faisGm$. Explicitement, $\faisPGOplusdiag_{2n}(R)$ est l'ensemble des conjugaisons par les matrices de la forme 
$$\mathrm{diag}(\alpha_0\alpha_1,\alpha_1^{-1},\alpha_0\alpha_2,\alpha_2^{-1},\ldots,\alpha_0\alpha_{n-1},\alpha_{n-1}^{-1},\alpha_0,1)$$ 
avec $\alpha_i\in R^\times$. Les caractères de ce tore admettent une base donnée par 
$$t_i(\mathrm{diag}(\alpha_0\alpha_1,\alpha_1^{-1},\alpha_0\alpha_2,\alpha_2^{-1},\ldots,\alpha_0\alpha_{n-1},\alpha_{n-1}^{-1},\alpha_0,1))=\alpha_i$$
pour tout $i=0,\ldots,n-1$. Un calcul direct donne la liste suivante de racines de $\faisPGOplus_{2n}$ pour tout anneau $R$:
\begin{itemize}
\item[1.] L'algèbre de Lie du tore engendrée par les éléments $E_{11}+\sum_{i=2}^n E_{2i,2i}$, $E_{22}+\sum_{i=2}^n E_{2i,2i}$ et $E_{2i-1,2i-1}-E_{2i,2i}$ pour tout $2\leq i\leq n-1$.
\item[2.] Le sous-module $R\cdot (E_{2i-1,2j-1}-E_{2j,2i})$ de poids $t_i-t_j$ pour tout $1\leq i,j\leq n-1$ avec $i\neq j$.
\item[3.] Le sous-module $R\cdot (E_{2i-1,2n-1}-E_{2n,2i})$ de poids $t_i$ pour tout $i=1,\ldots,n-1$.
\item[4.] Le sous-module $R\cdot (E_{2n-1,2i-1}-E_{2i,2n})$ de poids $-t_i$ pour tout $i=1,\ldots,n-1$.
\item[5.] Le sous-module $R\cdot (E_{2i-1,2j}-E_{2j-1,2i})$ de poids $t_0+t_i+t_j$ pour tout $1\leq i<j\leq n-1$.
\item[6.] Le sous-module $R\cdot (E_{2i,2j-1}-E_{2j,2i-1})$ de poids $-t_0-t_i-t_j$ pour tout $1\leq i<j\leq n-1$.
\item[7.] Le sous-module $R\cdot (E_{2i-1,2n}-E_{2n-1,2i})$ de poids $t_0+t_i$ pour tout $i=1,\ldots,n-1$.
\item[8.] Le sous-module $R\cdot (E_{2i,2n-1}-E_{2n,2i-1})$ de poids $-t_0-t_i$ pour tout $i=1,\ldots,n-1$.
\end{itemize}

Les morphismes exponentiels 
$$\exp_\alpha:\faisGa\to \faisPGOplus_{2n}$$
induisant l'inclusion canonique de l'espace propre associé à $\alpha$ dans son algèbre de Lie (\cite[Exp. XXII, théorème 1.1]{sga3}) sont comme ci-dessous:
\begin{itemize}
\item[2.] Le morphisme $exp_{t_i-t_j}:R\cdot (E_{2i-1,2j-1}-E_{2j,2i})\to \faisPGOplus_{2n}$ défini par 
$$\lambda\mapsto Id+\lambda(E_{2i-1,2j-1}-E_{2j,2i}).$$
\item[3.] Le morphisme $exp_{t_i}:R\cdot (E_{2i-1,2n-1}-E_{2n,2i})\to \faisPGOplus_{2n}$ défini par 
$$\lambda\mapsto Id+\lambda(E_{2i-1,2n-1}-E_{2n,2i}).$$
\item[4.] Le morphisme $exp_{-t_i}:R\cdot (E_{2n-1,2i-1}-E_{2i,2n})\to \faisPGOplus_{2n}$ défini par 
$$\lambda\mapsto Id+\lambda(E_{2n-1,2i-1}-E_{2i,2n}).$$
\item[5.] Le morphisme $exp_{t_0+t_i+t_j}:R\cdot (E_{2i-1,2j}-E_{2j-1,2i})\to \faisPGOplus_{2n}$ défini par 
$$\lambda\mapsto Id+\lambda(E_{2i-1,2j}-E_{2j-1,2i}).$$
\item[6.] Le morphisme $exp_{-t_0-t_i-t_j}:R\cdot (E_{2i,2j-1}-E_{2j,2i-1})\to \faisPGOplus_{2n}$ défini par 
$$\lambda\mapsto Id+\lambda(E_{2i,2j-1}-E_{2j,2i-1}).$$
\item[7.] Le morphisme $exp_{t_0+t_i}:R\cdot (E_{2i-1,2n}-E_{2n-1,2i})\to \faisPGOplus_{2n}$ défini par 
$$\lambda\mapsto Id+\lambda(E_{2i-1,2n}-E_{2n-1,2i}).$$
\item[8.] Le morphisme $exp_{-t_0-t_i}:R\cdot (E_{2i,2n-1}-E_{2n,2i-1})\to \faisPGOplus_{2n}$ défini par 
$$\lambda\mapsto Id+\lambda(E_{2i,2n-1}-E_{2n,2i-1}).$$
\end{itemize}

Utilisant ces morphismes, on trouve facilement la liste de coracines suivante:
\begin{itemize}
\item[2.] Les coracines $t_i^\dual-t_j^\dual$ pour tout $1\leq i,j\leq n-1$ avec $i\neq j$.
\item[3.] Les coracines $2t_i^\dual-2t_0^\dual+\sum_{j=1, j\neq i}^{n-1} t_j^\dual$ pour tout $i=1,\ldots,n-1$.
\item[4.] Les coracines $2t_0^\dual-2t_i^\dual-\sum_{j=1, j\neq i}^{n-1} t_j^\dual$ pour tout $i=1,\ldots,n-1$.
\item[5.] Les coracines $t_i^\dual+t_j^\dual$ pour tout $1\leq i<j\leq n-1$.
\item[6.] Les coracines $-t_i^\dual-t_j^\dual$ pour tout $1\leq i<j\leq n-1$.
\item[7.] Les coracines $2t_0^\dual-\sum_{j=1, j\neq i}^{n-1} t_j^\dual$ pour tout $i=1,\ldots,n-1$.
\item[8.] Les coracines $-2t_0^\dual+\sum_{j=1, j\neq i}^{n-1} t_j^\dual$ pour tout $i=1,\ldots,n-1$.
\end{itemize}

Les accouplements de la remarque \ref{signeAccouplement_rema} sont donnés dans les cas 2. à 4. par $(X,Y)\mapsto -XY$ et dans les cas 5. à 8. par $(X,Y)\mapsto XY$.

Nous avons ainsi démontré le résultat suivant:

\begin{prop}
Pour $n\geq 2$, la donnée radicielle de $\faisPGOplus_{2n}$ par rapport au tore maximal déployé $\faisPGOplusdiag_{2n}$ est
\begin{itemize}
\item Le $\ZZ$-module libre $N\simeq \ZZ^n$ des caractères de $\faisPGOplus_{2n}$ (dont $t_0,\ldots,t_{n-1}$ forment une base); 
\item Le sous-ensemble fini de celui-ci formé des racines $t_i-t_j$ pour $1\leq i,j\leq n-1$ avec $i\neq j$, $\pm t_i$ pour $i=1,\ldots,n-1$, $\pm(t_0+t_i+t_j)$ pour $1\leq i<j\leq n-1$ et $\pm(t_0+t_i)$ pour $i=1,\ldots,n-1$;
\item Le $\ZZ$-module libre $N^\dual\simeq\ZZ^n$ des cocaractères (de base duale $t_0^\cdual,\ldots,t_{n-1}^\cdual$);
\item Le sous-ensemble fini de celui-ci formé des coracines $t_i^\dual-t_j^\dual$ pour $1\leq i,j\leq n-1$ avec $i\neq j$, $\pm(2t_i^\dual-2t_0^\dual+\sum_{j=1, j\neq i}^{n-1} t_j^\dual)$ pour $i=1,\ldots,n-1$, $\pm(t_i^\dual+t_j^\dual)$ pour $1\leq i<j\leq n-1$ et $\pm(2t_0^\dual-\sum_{j=1, j\neq i}^{n-1} t_j^\dual)$ pour tout $i=1,\ldots,n-1$.
\end{itemize}
\end{prop}

\begin{theo}
Pour tout $n\geq 2$, le groupe $\faisPGOplus_{2n}$ est déployé, semi-simple et adjoint de type $D_n$ (avec la convention que $D_2=A_1 \times A_1$). 
\end{theo}

\begin{proof}
Le groupe est semi-simple par la proposition \ref{PGOplus_ss_prop} et déployé. Il est adjoint puisque ses racines engendrent les caractères du tore. Enfin, il est de type $D_n$ comme le montre le système de racines simples $\{t_0+t_{n-1},t_1-t_2,t_2-t_3,\ldots,t_{n-2}-t_{n-1},t_{n-1}\}$ (lorsque $n=2$, il s'agit de $\{t_0+t_1,t_1\}$).
\end{proof}


\subsection{Groupe déployé simplement connexe}

Ce cas est très similaire au cas traité dans la section \ref{bn_sconnexe}. On entrera donc moins dans les détails que dans les autres sections. 

Soit $(M,q)$ un module quadratique régulier, avec $M$ localement libre de rang constant $2n$. On lui associe son schéma en groupes spinoriel $\faisSpin_q$ (définition \ref{PinSpin_defi}) qui est affine sur $S$. En particulier: 
\begin{defi} \label{SpinsurZ_defi}
Si $M=\faisO_S^{2n}$ et $q=\hypq_{2n}$ on note $\faisSpin_{2n}$ le groupe $\faisSpin_{\hypq_{2n}}$. 
\end{defi}
Il est clair que ce groupe est défini sur $\Spec(\ZZ)$ et on l'étudie comme d'habitude sur cette base. Pour tout anneau $R$, on a la description suivante des $R$-points 
$$\faisSpin_{2n}(R)=\{g\in C_{0,h}(R)^\times\vert g(M\otimes R)g^{-1}=M\otimes R\text{ et } g\cdot \invstd(g)=1\}$$
où $\sigma$ est l'involution standard sur l'algèbre de Clifford paire $C_{0,h}$ du module quadratique $\hypq_{2n}$.

Le calcul de la section \ref{bn_sconnexe} montre que pour tout anneau $R$, l'algèbre de Lie $\faisLie_{\faisSpin_{2n}}$ est le $R$-module libre engendré par les éléments
\begin{itemize}
\item[1.] $1-2e_{2n-1}\otimes e_{2n}$
\item[2.] $e_{2i-1}\otimes e_{2i}-e_{2n-1}\otimes e_{2n}$ pour tout $1\leq i\leq n-1$.
\item[3.] $e_i\otimes e_j$ pour tout $1\leq i<j\leq 2n$ avec $(i,j)\neq (2r-1,2r)$ pour $r=1,\ldots,n$.
\end{itemize}

On a donc:

\begin{prop}\label{dn_Spin_ss_prop}
Le groupe $\faisSpin_{2n}$ est semi-simple.
\end{prop}
\begin{proof}
Le raisonnement est identique à celui de la proposition \ref{Spin_ss_prop}, mutatis mudandis: La dimension des fibres de $\faisSpin_{2n}$ a été calculée dans la proposition \ref{Spindim_prop}, et leur connexité dans le théorème \ref{Spinconnexe_theo}.
\end{proof}

Pour trouver un tore maximal de $\faisSpin_{2n}$, on considère les cocaractères
$$t_1^\dual:\faisGm\to \faisSpin_{2n}$$
défini pour tout anneau $R$ et tout $\alpha_0\in R^\times$ par $t_1^\dual(\alpha_1)=\alpha_1 e_1\otimes e_2+\alpha_1^{-1}e_2\otimes e_1$,
ainsi que pour tout $i=2,\ldots,n$ 
$$t_i^\dual:\faisGm\to \faisSpin_{2n}$$
donnés  par $t_i^\dual(\alpha_i)=$
{\small$$\alpha_ie_1\otimes e_2\otimes e_{2i-1}\otimes e_{2i}+e_1\otimes e_2\otimes e_{2i}\otimes e_{2i-1}+ e_2\otimes e_1\otimes e_{2i-1}\otimes e_{2i}+ \alpha_i^{-1}e_2\otimes e_1\otimes e_{2i}\otimes e_{2i-1}.$$}
Le sous-groupe de $\faisSpin_{2n}$ engendré par ces cocaractères est un tore, qui est maximal par les mêmes arguments que dans la section \ref{bn_sconnexe}. On le note $\faisSpindiag_{2n}$. Les caractères de ce tore sont donnés par $t_i(t_j^\dual(\alpha_j))=\delta_{ij}\alpha_j$ pour tout $1\leq i,j\leq n$.

Les racines sont données pour tout anneau $R$ par la liste suivante:
\begin{itemize}
\item[1.] L'algèbre de Lie du tore engendré par $R\cdot (1-2e_{2n-1}\otimes e_{2n})$ et par les $R\cdot (e_{2i-1}\otimes e_{2i}-e_{2n-1}\otimes e_{2n})$ pour tout $i=1,\ldots,n-1$.
\item[2.] Le sous-module $R\cdot (e_1\otimes e_{2i-1})$ de poids $2t_1+2t_i+\sum_{j=2,j\neq i}^n t_j$ pour tout $i=2,\ldots,n$. 
\item[3.] Le sous-module $R\cdot (e_2\otimes e_{2i})$ de poids $-2t_1-2t_i-\sum_{j=2,j\neq i}^n t_j$ pour tout $i=2,\ldots,n$.  
\item[4.] Le sous-module $R\cdot (e_1\otimes e_{2i})$ de poids $2t_1+\sum_{j=2,j\neq i}^n t_j$ pour tout $i=2,\ldots,n$. 
\item[5.] Le sous-module $R\cdot (e_2\otimes e_{2i-1})$ de poids $-2t_1-\sum_{j=2,j\neq i}^n t_j$ pour tout $i=2,\ldots,n$. 
\item[6.] Le sous-module $R\cdot (e_{2i-1}\otimes e_{2j-1})$ de poids $t_i+t_j$ pour tout $2\leq i<j\leq n$.
\item[7.] Le sous-module $R\cdot (e_{2i}\otimes e_{2j})$ de poids $-t_i-t_j$ pour tout $2\leq i<j\leq n$.
\item[8.] Le sous-module $R\cdot (e_{2i-1}\otimes e_{2j})$ de poids $t_i-t_j$ pour tout $2\leq i<j\leq n$.
\item[9.] Le sous-module $R\cdot (e_{2i}\otimes e_{2j-1})$ de poids $t_j-t_i$ pour tout $2\leq i<j\leq n$.
\end{itemize}

On vérifie immédiatement que les morphismes exponentiels $\faisGa\to \faisSpin_{2n}$ prévus par \cite[Exp. XXII, théorème 1.1]{sga3} sont dans tous les cas de la forme $\lambda\mapsto 1+\lambda e_i\otimes e_j$. Ceci nous permet de calculer les coracines. On trouve pour tout anneau $R$ les coracines énumérée ci-dessous:

\begin{itemize}
\item[2.] Les coracines $t_i^\dual$ pour tout $i=2,\ldots,n$. 
\item[3.] Les coracines $-t_i^\dual$ pour tout $i=2,\ldots,n$. 
\item[4.] Les coracines $t_1^\dual-t_i^\dual$ pour tout $i=2,\ldots,n$.
\item[5.] Les coracines $t_i^\dual-t_1^\dual$ pour tout $i=2,\ldots,n$.
\item[6.] Les coracines $t_i^\dual+t_j^\dual-t_1^\dual$ pour tout $2\leq i<j\leq n$.
\item[7.] Les coracines $t_1^\dual-t_i^\dual-t_j^\dual$ pour tout $2\leq i<j\leq n$.
\item[8.] Les coracines $t_i^\dual-t_j^\dual$ pour tout $2\leq i<j\leq n$.
\item[9.] Les coracines $t_j^\dual-t_i^\dual$ pour tout $2\leq i<j\leq n$.
\end{itemize} 

Dans tous les cas, les accouplements prévus par la remarque \ref{signeAccouplement_rema} sont donnés par $(X,Y)\mapsto XY$. Finalement:

\begin{prop}\label{donneeradicielleSpindn_prop}
La donnée radicielle de $\faisSpin_{2n}$ par rapport au tore maximal déployé $\faisSpindiag_{2n}$ est
\begin{itemize}
\item Le $\ZZ$-module libre $N\simeq \ZZ^n$ des caractères de $\faisSpin_{2n}$ (dont $t_1,\ldots,t_{n}$ forment une base); 
\item Le sous-ensemble fini de celui-ci formé des racines $\pm(2t_1+2t_i+\sum_{j=2,j\neq i}^n t_j)$ pour $i=2,\ldots,n$, $\pm(2t_1+\sum_{j=2,j\neq i}^n t_j)$ pour $i=2,\ldots,n$ et $\pm t_i\pm t_j$ pour $2\leq i<j\leq n$;
\item Le $\ZZ$-module libre $N^\dual\simeq\ZZ^n$ des cocaractères (de base duale $t_1^\cdual,\ldots,t_{n}^\cdual$);
\item Le sous-ensemble fini de celui-ci formé des coracines $\pm t_i^\dual$ pour $i=2,\ldots,n$, $\pm(t_1^\dual-t_i^\dual)$ pour $i=2,\ldots,n$, $\pm(t_i^\dual+t_j^\dual-t_1^\dual)$ pour $2\leq i<j\leq n$ et $\pm(t_i^\dual-t_j^\dual)$ pour $2\leq i<j\leq n$.
\end{itemize}
\end{prop}

\begin{theo}
Pour tout $n\geq 2$, le groupe $\faisSpin_{2n}$ est déployé, semi-simple et simplement connexe de type $D_n$ (avec la convention $D_2=A_1 \times A_1$).
\end{theo}

\begin{proof}
Le groupe est semi-simple par la proposition \ref{dn_Spin_ss_prop} et déployé. Il est semi-simple puisque ses racines engendrent les caractères du tore. Enfin, il est de type $D_n$ comme on s'en convainc en contemplant le système de racines simples $\{2t_1+\sum_{i=3}^n t_i,t_2-t_3,t_3-t_4,\ldots, t_{n-1}-t_n, t_{n-1}+t_n\}$ quand $n\geq 3$ ou $\{2t_1,2t_1+2t_2\}$ quand $n=2$.

\end{proof}

On calcule maintenant le centre de $\faisSpin_{2n}$. Pour ce faire, on dispose d'un diagramme de morphisme de schémas en groupes 
$$
\xymatrix{
\faisSpin_{2n} \ar[r] \ar[dr] & \faisOplus_{2n} \ar[d] \ar[r] & \faisPGOplus_{2n} \ar[d] \\
 & \faisorthO_{2n} \ar[r] & \faisPGO_{2n}
}
$$
constitué de plusieurs flèches naturelles, définies dans les propositions \ref{secSpinOplus_prop} et \ref{secOPGO_prop}.
On dispose ainsi d'un morphisme
$$\chi:\faisSpin_{2n}\to \faisPGOplus_{2n}$$
qui induit une bijection entre ensembles de racines et de coracines, ainsi qu'un morphisme de tores $\chi:\faisSpindiag_{2n}\to \faisPGOplusdiag_{2n}$ dont le noyau est le centre de $\faisSpin_{2n}$. Identifiant les tores à $\faisGm^n$ au moyen des cocaractères, on trouve 
$$\chi(\alpha_1,\ldots,\alpha_n)=(\alpha_n^2,\alpha_1^2\cdot\prod_{i=2}^{n-1}\alpha_i,\alpha_2\alpha_n^{-1},\alpha_3\alpha_n^{-1},\ldots,\alpha_{n-1}\alpha_n^{-1}).$$
Ainsi $(\alpha_1,\ldots,\alpha_n)\in \ker\chi$ si et seulement si $\alpha_i=\alpha_n$ pour tout $i=2,\ldots,n-1$, $\alpha_n^2=1$ et $\alpha_1^2\alpha_n^{n-2}=1$. On a ainsi la proposition suivante:

\begin{prop} \label{centreSpin2n_prop}
Si $n$ est pair, le centre de $\faisSpin_{2n}$ est isomorphe à $\faismu_2\times \faismu_2$. Si $n$ est impair, alors le centre est isomorphe à $\faismu_4$.
\end{prop}


\subsection{Automorphismes}\label{automorphismes_dn}

On calcule maintenant les automorphismes du groupe $\faisPGOplus_{2n}$, adjoint de type $D_n$. On suppose que $n\neq 1$ ou $4$ (nous ne traiterons pas la trialité ici). 

Pour commencer, remarquons que le groupe $\faisPGO_{2n}$ agit sur $\faisPGOplus_{2n}$ par conjugaison puisque on a une suite exacte de groupes (proposition \ref{secGOPGO_prop})
$$\xymatrix{1 \ar[r] & \faisPGOplus_{2n} \ar[r] & \faisPGO_{2n} \ar[r] & \faisconst \ar[r] & 0}$$
On note $\phi:\faisPGO_{2n}\to \faisAut_{\faisPGOplus_{2n}}$ le morphisme de groupe induit par la conjugaison.

\begin{theo} \label{PGOplusAut_theo}
Pour $n\neq 1$ ou $4$, on a un isomorphisme de suites exactes 
$$\xymatrix{1 \ar[r] & \faisPGOplus_{2n} \ar[r] \ar@{=}[d] & \faisPGO_{2n} \ar[r] \ar[d]^{\simeq}_{\phi} & \faisconst \ar[r] \ar[d]^{\simeq} & 1 \\
1 \ar[r] & \faisPGOplus_{2n} \ar[r] & \faisAut_{\faisPGOplus_{2n}} \ar[r] & \faisAutExt_{\faisPGOplus_{2n}} \ar[r] & 1 }$$
\end{theo}

\begin{proof}
Il suffit d'exhiber, si $R$ est un anneau local, n'importe quel élément de $\faisPGO_{2n}(R)$ qui permute non trivialement les racines simples de $\faisPGOplus_{2n}$ (et qui est alors automatiquement d'invariant de Arf non trivial). Soit par exemple $A=E_{2n-1,2n}+E_{2n,2n-1}+\sum_{i=1}^{2n-2} E_{ii}\in GL_{2n}(R)$. La conjugaison $\int_A$ par $A$ est un automorphisme de $M_{2n}(R)$ préservant l'involution $\eta_{2n}$ et l'homomorphisme $f_{2n}$. Ainsi $\int_A\in \faisPGO_{2n}(R)$. Un calcul direct montre que $\int_A$ préserve le tore $\faisGOdiag_{2n}$ et son quotient $\faisPGOplusdiag_{2n}$ et échange les deux racines simples $t_0+t_{n-1}$ et $t_{n-1}$ de $\faisPGOplus_{2n}$. 
\end{proof}


\subsection{Groupes tordus}
Soit $n\geq 5$ un entier.

\begin{theo} \label{Dnadjointtordus_theo}
\`A isomorphisme près, les $S$-schémas en groupes réductifs de donnée radicielle déployée identique à celle de $\faisPGOplus_{2n}$ ($n\geq 5$), c'est-à-dire de type déployé semi-simple adjoint $D_n$, $n\geq 5$, sont les groupes $\faisPGOplus_{A,\sigma,f}$ où $(A,\sigma,f)$ est une paire quadratique de degré $2n$ au sens de la définition \ref{pairesquad_defi}.
Les formes intérieures sont les groupes $\faisPGO_{A,\sigma,f}$ tels que le centre de l'algèbre de Clifford paire de $(A,\sigma,f)$ est isomorphe à $\faisO_S \times \faisO_S$. Toutes les formes intérieures sont strictement intérieures puisque le groupe est adjoint. Pour les formes fortement intérieures, voir le corollaire \ref{Dnadjfortementint_coro}.
\end{theo}
\begin{proof}
Le groupe des automorphismes de $\faisPGOplus_{2n}$ est $\faisPGO_{2n}$ par le théorème \ref{PGOplusAut_theo}. Pour les formes quelconques, sachant que le champ des $\faisPGO_{2n}$-torseurs est équivalent à $\PairesQuadn{2n}$, on applique donc la proposition \ref{secPGOplus_prop}. Pour les formes intérieures, sachant que le champ des $\faisPGOplus_{2n}$-torseurs est équivalent au champ $\PairesQuadArfTrivn{2n}$, on vérifie que l'image essentielle par le foncteur oubli $\PairesQuadArfTrivn \to \PairesQuadn{2n}$ est évidemment constituée des objets $(A,\sigma,f)$ tels qu'il existe une trivialisation du centre $Z_{0,A,\sigma,f}$, et on applique la proposition \ref{foncttors_prop}, point \ref{fonctobjet_item}.
\end{proof}

Il nous faut maintenant introduire, comme dans le cas $B_n$, un invariant supplémentaire pour les modules quadratiques réguliers d'algèbre de Clifford triviale. La construction est identique au cas $B_n$, mais en utilisant l'algèbre de Clifford toute entière au lieu de sa partie paire. Elle est munie d'une $2$-trivialisation canonique $\faisCliff_q^{\otimes 2} \simeq \faisCliff_q\otimes\faisCliff_q^\opp \simeq \faisEnd_{\faisCliff_q}$ où le premier isomorphisme est induit par l'involution standard $\invstd_q$. Par le lemme \ref{invariantL_lemm}, on obtient donc une classe $l_q$ dans $\Pic(S)/2$.

\begin{theo} \label{Dnsctordus_theo}
\`A isomorphisme près, les $S$-schémas en groupes réductifs de donnée radicielle déployée identique à celle de $\faisSpin_{2n}$ ($n\geq 5$), c'est-à-dire de type déployé semi-simple simplement connexe $D_n$, $n\geq 5$, sont les groupes $\faisSpin_{A,\sigma,f}$ où $(A,\sigma,f)$ est une paire quadratique de degré $2n$ au sens de la définition \ref{pairesquad_defi}.
Les formes intérieures sont les groupes $\faisSpin_{A,\sigma,f}$ tels que le centre de l'algèbre de Clifford paire de $(A,\sigma,f)$ est isomorphe à $\faisO_S \times \faisO_S$. Les formes strictement intérieures sont les groupes $\faisSpin_q$ où $q$ est un module quadratique régulier de rang $2n$, d'algèbre de Clifford paire triviale dans le groupe de Brauer et d'invariant $l_q \in \Pic(S)/2$ trivial.
\end{theo}
\begin{proof}
Les formes quelconques sont des tordues de $\faisSpin_{2n}$ par un torseur sous le groupe d'automorphismes qui est $\faisPGO_{2n}$ comme dans le cas du groupe adjoint. Le résultat découle alors de la proposition \ref{torsionSpin_prop}. Les formes intérieures sont donc celles des torseurs sous le groupe adjoint $\faisPGOplus_{2n}$, et on les obtient donc par image essentielle de $\PairesQuadArfTrivn{2n} \to \PairesQuadn{2n}$. Enfin, les formes strictement intérieures sont celles qui proviennent de $\faisSpin_{2n}$ lui-même, qu'on obtient comme images essentielles par le foncteur $\DeplCliffSn{2n} \to \PairesQuadn{2n}$ qui envoie un objet $((q,M),N,\phi,\psi)$ vers la paire quadratique $(\faisEnd_M,\invadj_q,f_q)$. Par construction de $\DeplCliffSn{2n}$, ce sont bien les $(\faisEnd_M,\invadj_q,f_q)$ pour les $q$ dont l'algèbre de Clifford est triviale dans le groupe de Brauer, et dont l'invariant $l_q$ est trivial.
\end{proof}

\begin{coro} \label{Dnadjfortementint_coro}
Les formes fortement intérieures du groupe $\faisPGOplus_{2n}$, le groupe adjoint déployé de type $B_n$, sont les groupes $\faisPGOplus_q$ (définition \ref{PGOplus_defi}) où $q$ est un module quadratique régulier de rang $2n$, d'algèbre de Clifford paire triviale dans le groupe de Brauer et d'invariant $l_q\in Pic(S)/2$ trivial.
\end{coro}

\begin{rema}
Il n'est pas difficile de voir que dans le cas intérieur, les résultats des théorèmes \ref{Dnadjointtordus_theo} et \ref{Dnsctordus_theo} restent valables quand $n\leq 4$.
\end{rema}

\section{Tables}

Pour les tables \ref{GroupesTordus_table} et \ref{Automorphismes_table}, en type $A_n$, l'involution $(A,\nu)$ est  unitaire avec $A$ de degré $n+1$ sur son centre. En type $B_n$, le module quadratique $q$ est une forme étale de l'hyperbolique de rang $2n+1$. En type $C_n$, l'involution $(A,\sigma)$ est symplectique, avec $A$ de degré $2n$. En type $D_n$, la paire quadratique $(A,\sigma,f)$ est avec $A$ de degré $2n$.

Pour la table \ref{GroupesTordusInterieurs_table}, en type $A_n$, le $\faisO_S$-module localement libre $M$ est de rang $n+1$ et de déterminant trivial. En type $D_n$ (resp. $B_n$), l'algèbre de Clifford de $(A,\sigma,f)$ est de centre trivial $\faisO_S\times \faisO_S$ et le module quadratique $\tilde{q}$ est une forme étale de la forme hyperbolique de rang $2n$ (resp. $2n+1$), dont l'algèbre de Clifford (resp. paire) est triviale dans le groupe de Brauer, et dont l'invariant $l_{\tilde{q}}$ est trivial (voir \ref{Dnsctordus_theo} et \ref{Bnsimplementconnexetordu_theo}). En type $C_n$, l'isomorphisme alterné $b$ est sur un $\faisO_S$-module localement libre de rang $2n$. 

\begin{table}[h!]
\setlength\belowcaptionskip{.5ex}
\caption{Groupes de Chevalley (sur $\ZZ$) simples simplement connexes ou adjoints par type de Dynkin}
\begin{center}
\setlength{\tabcolsep}{0.1ex}
\begin{tabular}[h]{l@{\hspace{1.5ex}}l@{\hspace{1.5ex}}l@{\hspace{1.5ex}}l}
\toprule
type & simplement connexe & adjoint & centre du groupe s.c. \\
\midrule
$A_n$ & $\faisSL_{n+1}$ \eqref{SLn_defi} & $\faisPGL_{n+1}$ \eqref{PGLn_defi} & $\faismu_{n+1}$ \eqref{centreSL_prop} \\
$B_n$ & $\faisSpin_{2n+1}$ \eqref{PinSpin_defi} & $\faisSO_{2n+1}$ \eqref{SOimpairsurZ_defi} & $\faismu_2$ \eqref{centreSpin_prop} \\
$C_n$ & $\faisSp_{2n}$ \eqref{SpsurZ_defi} & $\faisPGSp_{2n}$ \eqref{PGSp_defi} & $\faismu_2$ \eqref{centreSp_prop} \\
$D_n$ & $\faisSpin_{2n}$ \eqref{SpinsurZ_defi} & $\faisPGOplus_{2n}$ \eqref{PGO_defi} & \doublecell{$\faismu_2\times \faismu_2$ si $n$ pair \eqref{centreSpin2n_prop} \\ $\faismu_4$ si $n$ impair (ibid)} \\
\bottomrule
\end{tabular}
\end{center}
\label{groupesChevalley_table}
\end{table}

\begin{table}[h!]
\setlength\belowcaptionskip{0ex}
\caption{Groupes tordus de type déployé simple et simplement connexe ou adjoint}
\begin{center}
\setlength{\tabcolsep}{.1ex}
\begin{tabular}[h]{l@{\hspace{.5ex}}lll}
\toprule
type & simpl. connexe & adjoint & centre du s.c. \\
\midrule
$A_n$ & $\faisSU_{A,\nu}$ \eqref{formesSLn_theo} & $\ker (\faisAut_{A,\nu}\hspace{-.5ex} \to \hspace{-.5ex}\faisAut_{Z(A),\nu})$ \eqref{adjointtordu_prop} & $\faismu_{n+1}$ \\
$B_n$ & $\faisSpin_q$ \eqref{Bnsimplementconnexetordu_theo} & $\faisSO_q$ \eqref{Bnadjointtordu_theo} & $\faismu_2$ \\
$C_n$ & $\faisSp_{A,\sigma}$ \eqref{Cnsimplementconnexe_theo} & $\faisPGSp_{A,\sigma}$ \eqref{Cnadjointtordu_theo} & $\faismu_2$ \\
\doublecell{$D_n$ \\ $\scriptstyle n\neq1,4$} & $\faisSpin_{A,\sigma,f}$ \eqref{Dnsctordus_theo} & $\faisPGOplus_{A,\sigma,f}$ \eqref{Dnadjointtordus_theo} & \doublecell{$\faismu_2\times \faismu_2$ ($n$ pair) \\ $\faismu_4$ ($n$ impair)} \\
\bottomrule
\end{tabular}
\end{center}
\label{GroupesTordus_table}
\end{table}

\begin{table}[h!]
\setlength\belowcaptionskip{.5ex}
\caption{Automorphismes des groupes simples adjoints ou simplement connexes}
\begin{center}
\begin{tabular}[h]{llll}
\toprule
type & déployé & général & extérieurs \\
\midrule
$A_n$ & $\faisAut_{(\faisM_n\times \faisM_n),\tau}$ \eqref{PGLAut_theo} & $\faisAut_{A,\nu}$ & $\ZZ/2$ si $n>1$ \\
$B_n$ & $\faisSO_{2n+1}$ & $\faisSO_q$ & $0$ \\
$C_n$ & $\faisPGSp_{2n}$ & $\faisPGSp_{A,\sigma}$ & $0$ \\
$D_{n}$ $\scriptstyle n\neq 1,4$ & $\faisPGO_{2n}$ \eqref{PGOplusAut_theo} & $\faisPGO_{A,\sigma,f}$ & $\ZZ/2$ \\
\bottomrule
\end{tabular}
\end{center}
\label{Automorphismes_table}
\end{table}

\begin{table}[h!]
\setlength\belowcaptionskip{.5ex}
\caption{Groupes tordus intérieurs et fortement intérieurs de type déployé simple simplement connexe ou adjoint}
\begin{center}
\begin{tabular}{lllll}
\toprule
type & \multicolumn{2}{c}{simplement connexe} & \multicolumn{2}{c}{adjoint} \\ 
 & int. & fort. int. & int. & fort. int. \\
\midrule
$A_n$ & $\faisSL_{1,A}$ & $\faisSL_M$ & $\faisPGL_A$ & $\faisPGL_M$ \\
$B_n$ & $\faisSpin_q$ & $\faisSpin_{\tilde{q}}$ & $\faisSO_q$ & $\faisSO_{\tilde{q}}$ \\
$C_n$ & $\faisSp_{A,\sigma}$ & $\faisSp_b$ & $\faisPGSp_{A,\sigma}$ & $\faisPGSp_{b}$ \\
$D_n$ & $\faisSpin_{A,\sigma,f}$ & $\faisSpin_{\tilde{q}}$ & $\faisPGOplus_{A,\sigma,f}$ & $\faisPGOplus_{\tilde{q}}$ \\
\bottomrule
\end{tabular}
\end{center}
\label{GroupesTordusInterieurs_table}
\end{table}

\clearpage

\begin{table}[h!]
\setlength\belowcaptionskip{.5ex}
\caption{Champ équivalent au champ $\Tors{G}$ des torseurs sous le groupe $G$}
\begin{center}
\begin{tabular}{llll}
\toprule
$G$ & Champ & Topologie & \\
\midrule
$\faismu_n$ & $\nTriv{n}$ & \fppf & \eqref{torseursfppfmun_prop} \\
$\faismu_2$ & $\nTriv{2}=\ModDet$ & \fppf & \eqref{moddet_defi} \\
$\mathcal{S}_n$ & $\Etn{n}$ & ét ou \fppf & \eqref{SntorsEtn_prop} \\
$\faisGL_n$ & $\Vecn{n}$ & ét ou \fppf & \eqref{GLntors_prop} \\
$\faisGL_{1,A}$ & $\AVecn{A^{\opp}}{1}$ & ét ou \fppf & \eqref{GL1Atorseurs_prop} \\
$\faisSL_{n}$ & $\Vectrivdetn{n}$ & ét ou \fppf & \eqref{tordusformes_prop} \\
$\faisPGL_n$ & $\Azumayan{n}$ & ét ou \fppf & \eqref{PGLAtorseurs_prop} \\
$\faisPGL_A$ & $\Azumayan{n}$ & ét ou \fppf & \eqref{PGLAtorseurs_prop} \\
$\faisO_{2n}$ & $\RegQuadn{2n}$ & ét ou \fppf & \eqref{formesmodulesquad_prop} \\ 
$\faisO_{2n+1}$ & $\Formes{\hypq_{2n+1}}$ & ét ou \fppf & \eqref{formesmodulesquad_prop} \\ 
$\faisSO_{2n+1}$ & $\QuadDetTrivn{2n+1}$ & ét ou \fppf & \eqref{SOtorseursimpairs_prop} \\ 
$\faisOplus_{2n}$ & $\RegQuadArfTrivn{2n}$ & ét ou \fppf & \eqref{Oplustorseurs_prop} \\
$\faisGamma_{2n}$ & $\GrDeplCliff{2n}$ & ét ou \fppf & \eqref{Gamma2ntors_prop} \\
$\faisSGamma_{2n}$ & $\DeplCliff{2n}$ & ét ou \fppf & \eqref{SGammapairtorseurs_prop} \\
$\faisPGO_{2n}$ & $\PairesQuadn{2n}$ & ét ou \fppf & \eqref{PGOtorseurs_prop} \\
$\faisPGO_{A,\sigma,f}$ & $\PairesQuadn{2n}$ & ét ou \fppf & \eqref{PGOtorseurs_prop} \\
$\faisPGOplus_{2n}$ & $\PairesQuadArfTrivn{2n}$ & ét ou \fppf & \eqref{PGOplustorseurs_prop} \\
$\faisSp_{2n}$ & $\Sympl{2n}$ & ét ou \fppf & \eqref{Sp2ntors_prop} \\
$\faisSpin_{2n}$ & $\DeplCliffSn{2n}$ & ét ou \fppf & \eqref{SpinTorsImpai_prop} \\
$\faisSpin_{2n+1}$ & $\DeplCliffSn{2n+1}$ & ét ou \fppf & \eqref{SpinTorsPair_prop} \\
\bottomrule
\end{tabular}
\end{center}
\label{Champs_table}
\end{table}

\newpage

\def\refname{Références}

\end{document}